\documentclass[a4paper]{amsart}
\usepackage[english]{babel}
\usepackage[utf8]{inputenc}
\usepackage[T1]{fontenc}
\usepackage{csquotes} 
\usepackage{lmodern}
\usepackage{amssymb}
\usepackage{float}
\usepackage{amscd}
\usepackage{amsthm}
\usepackage[top=2cm, bottom=2cm, left=2cm, right=2cm,twoside=false]{geometry}
\usepackage{array}
\usepackage{longtable}
\usepackage{mathtools}
\usepackage{graphicx}
\usepackage{url} 
\usepackage{wrapfig}
\usepackage{color}
\usepackage[usenames,x11names]{xcolor}
\usepackage[all]{xy}
\usepackage{enumitem}
\setlist[1]{leftmargin=*}
\setlist[enumerate,1]{label=(\alph*)}
\setlist[enumerate,2]{label=(\small{\arabic*}), ref=(\alph{enumi}.\small{\arabic*})}
\newlist{steps}{enumerate}{1}
\setlist[steps]{label=\bfseries{\arabic*.} , ref=\bfseries{\arabic*}}

\usepackage{changepage}

\usepackage{multirow}
\usepackage{units} 
\usepackage{yfonts}
\usepackage{mathrsfs}
\usepackage[export]{adjustbox}
\usepackage{etoolbox}

\usepackage[font=small, tableposition=below]{caption}
\usepackage{subcaption} 
\usepackage{hyperref} 
\hypersetup{
	linkcolor=DodgerBlue3, 
	citecolor  = teal,
	urlcolor   = DodgerBlue4,
	colorlinks = true, 
	hyperfootnotes =false,
	bookmarksdepth=3,
	linktoc=all	
}
\newcommand{\hlabel}[1]{\phantomsection\label{#1}}

\newcommand{\red}[1]{\textcolor{red}{#1}}  \newcommand{\blue}[1]{\textcolor{blue}{#1}}

\newcommand{\rA}{\mathrm{A}}
\newcommand{\rD}{\mathrm{D}}
\newcommand{\rE}{\mathrm{E}}

\usepackage[scr=boondoxo]{mathalfa}
\renewcommand{\ll}{\mathscr{l}}
\newcommand{\cc}{\mathscr{c}}
\newcommand{\qq}{\mathscr{q}}
\newcommand{\pp}{\mathscr{p}}
\newcommand{\hh}{\mathscr{h}}
\newcommand{\rr}{\mathscr{r}}
\newcommand{\kk}{\mathscr{k}}
\newcommand{\ngr}{\mathscr{N}}

\newcommand{\lts}[2]{\mathcal{T}_{#1}(-\log #2)}
\newcommand{\etop}{e_{\mathrm{top}}}

\renewcommand{\to}{\longrightarrow}
\newcommand{\map}{\dashrightarrow}
\newcommand{\into}{\hookrightarrow}
\newcommand{\Aut}{\operatorname{Aut}}
\newcommand{\Sing}{\operatorname{Sing}}
\newcommand{\coker}{\operatorname{coker}}
\newcommand{\R}{\mathbb{R}}
\newcommand{\F}{\mathbb{F}}
\newcommand{\N}{\mathbb{N}}
\newcommand{\C}{\mathbb{C}}
\newcommand{\Z}{\mathbb{Z}}
\newcommand{\Q}{\mathbb{Q}}
\renewcommand{\P}{\mathbb{P}}

\newcommand{\cO}{\mathcal{O}}
\newcommand{\cT}{\mathcal{T}}
\newcommand{\cS}{\mathcal{S}}

\renewcommand{\tilde}{\widetilde}
\renewcommand{\hat}{\widehat}
\renewcommand{\epsilon}{\varepsilon}
\renewcommand{\phi}{\varphi}
\renewcommand{\theta}{\vartheta}

\renewcommand{\d}{\partial}
\newcommand{\id}{\mathrm{id}}
\newcommand{\pr}{\mathrm{pr}}

\newcommand{\am}{_{\mathrm{am}}}

\newcommand{\de}{\coloneqq}
\newcommand{\cp}[1]{^{(#1)}} 

\newcommand{\NS}{\operatorname{NS}}
\newcommand{\Exc}{\operatorname{Exc}}

\newcommand{\Bk}{\operatorname{Bk}}
\newcommand{\redd}{_{\mathrm{red}}}
\newcommand{\rev}[1]{#1^{t}}
\newcommand{\reg}{^{\mathrm{reg}}}

\renewcommand{\leq}{\leqslant}
\renewcommand{\geq}{\geqslant}

\newcommand{\hor}{_{\mathrm{hor}}}
\renewcommand{\vert}{_{\mathrm{vert}}}

\newcommand{\Sec}{\Xi}

\newcommand{\GAut}{\Aut_{\textnormal{graph}}}

\theoremstyle{plain}
\newtheorem{thm}{Theorem}[section]
\newtheorem*{thm*}{Theorem}
\newtheorem{cor}[thm]{Corollary}

\theoremstyle{definition}
\newtheorem{dfn}[thm]{Definition}
\newtheorem{lem}[thm]{Lemma}
\newtheorem{prop}[thm]{Proposition}
\newtheorem{ex}[thm]{Example}

\newtheorem{conjecture}[thm]{Conjecture}
\newtheorem{notation}[thm]{Notation}
\newtheorem{rem}[thm]{Remark}
\newtheorem{constr}[thm]{Construction} 

\newtheorem{conf}{Configuration} 

\theoremstyle{remark}

\newtheorem*{claim*}{Claim}
\newtheorem*{rem*}{Remark}

\makeatletter \def\subsection{\@startsection{subsection}{3}
	\z@{.5\linespacing\@plus.7\linespacing}{.5\linespacing}
	{\bfseries\itshape}} 

\renewcommand\paragraph{\@startsection{paragraph}{4}%
	\z@{.5\linespacing\@plus.7\linespacing}{-.5\linespacing}%
	{\normalfont\bfseries}}

\renewenvironment{proof}[1][\proofname]{
	\par\pushQED{\qed}\normalfont
	\topsep6\p@\@plus6\p@\relax
	\trivlist\item[\hskip\labelsep\bfseries#1\@addpunct{.}]
	\ignorespaces}{
	\popQED\endtrivlist\@endpefalse} 

\@namedef{subjclassname@2020}{\textup{2020} Mathematics Subject Classification}
\makeatother

\renewcommand{\bar}{\overline}
\newcounter{constr}
\newcommand{\nextconstr}{\refstepcounter{constr}$\langle\mathfrak{\arabic{constr}}\rangle$}

\newcommand{\tref}[2]{#1_{\mbox{\tiny{\ref{#2}}}}}
\newcommand{\tst}[1]{#1_{\mathfrak{\langle * \rangle}}}
\newcommand{\rst}{\mathfrak{\langle * \rangle}}

\newcommand{\tick}{\checkmark}
\newcommand{\bl}[1]{\pi_{#1}}

\newcommand{\FZa}{\mathcal{FZ}_{1}}  
\newcommand{\cA}{\mathcal{A}} 
\newcommand{\cB}{\mathcal{B}} 
\newcommand{\cC}{\mathcal{C}} 
\newcommand{\cD}{\mathcal{D}} 
\newcommand{\cE}{\mathcal{E}}
\newcommand{\cF}{\mathcal{F}}
\newcommand{\cG}{\mathcal{G}}
\newcommand{\cP}{\mathcal{P}}
\newcommand{\cORa}{\mathcal{OR}_{1}}
\newcommand{\cORb}{\mathcal{OR}_{2}}
\newcommand{\Qa}{\mathcal{Q}_{4}}
\newcommand{\Qb}{\mathcal{Q}_{3}}
\newcommand{\FZb}{\mathcal{FZ}_{2}}
\newcommand{\FE}{\mathcal{FE}}
\newcommand{\cH}{\mathcal{H}}
\newcommand{\cI}{\mathcal{I}}
\newcommand{\cJ}{\mathcal{J}}
\newcommand{\cN}{\mathcal{N}}

\usepackage{xspace}
\newcommand{\QHP}{$\Q$-homology plane\xspace}
\newcommand{\QHPs}{$\Q$-homology planes\xspace}
\newcommand{\ZHP}{$\Z$-homology plane\xspace}
\newcommand{\ZHPs}{$\Z$-homology planes\xspace}

\newcommand{\Bs}{\operatorname{Bs}} 

\usepackage{diagbox}
\newcommand{\any}{\mbox{any}}
\newcommand{\none}{ \times}

\newcommand{\tables}{Tables \ref{table:smooth_0}--\ref{table:C**}\xspace}

\usepackage{tikz}
\usetikzlibrary{calc}
\usetikzlibrary{intersections}
\tikzset{
	partial ellipse/.style args={#1:#2:#3}{
		insert path={+ (#1:#3) arc (#1:#2:#3)}
	},
	thick/.style=      {line width=1.6pt},
	add/.style args={#1 and #2}{
		to path={%
			($(\tikztostart)!-#1!(\tikztotarget)$)--($(\tikztotarget)!-#2!(\tikztostart)$)%
			\tikztonodes},add/.default={.2 and .2}}
}
\usepackage{tikz-cd}

\begin{document}

\title[Smooth $\Q$-homology planes satisfying the Negativity Conjecture]{Smooth $\Q$-homology planes\\ satisfying the Negativity Conjecture}
\begin{abstract}
	A complex algebraic surface $S$ is a \emph{\QHP} if $H_{i}(S,\Q)=0$ for $i>0$.  
	The \emph{Negativity Conjecture} of Palka asserts that $\kappa(K_{X}+\tfrac{1}{2}D)=-\infty$, where $(X,D)$ is a log smooth completion of $S$.
	
	We give a complete description of smooth \QHPs satisfying the Negativity Conjecture. We restrict our attention to those of log general type, as otherwise their geometry is well understood. We show that, as conjectured by tom Dieck and Petrie, they can be arranged in finitely many discrete series, each obtained in a uniform way from an arrangement of lines and conics on $\P^{2}$. We infer that these surfaces satisfy the Rigidity Conjecture of Flenner and Zaidenberg; and a conjecture of Koras, which asserts that $\#\Aut(S)\leq 6$.
\end{abstract}
	\author{Tomasz Pe{\l}ka}
\address{Institute of Mathematics, Polish Academy of Sciences, Śniadeckich 8, 00-656 Warsaw, Poland}	
\address{Institute of Mathematics, University of Warsaw, Banacha 2, 02-097 Warsaw, Poland}
\email{tpelka@mimuw.edu.pl}
\thanks{This research project was funded by the National Science Centre, Poland, grant No.\ 2015/18/E/ST1/00562.}

\maketitle

\section{Introduction}

We work with complex algebraic varieties. A smooth surface $S$ is a \emph{$\Q$-homology plane} (\emph{$\Q$HP}, for short) if $H_{i}(S;\Q)=0$ for $i>0$. All $\Q$HPs are affine \cite[2.5]{Fujita-noncomplete_surfaces} and rational \cite{GuPrad-rationality_3(smooth_II)}. Being in many ways similar to the affine plane, they gained attention as a testing ground for many problems concerning $\C^{2}$ \cite[3.3, 5.2]{Miy-recent_dev}; as well as a source of exotic structures on $\C^{n}$ \cite{Zaid-exotic_structures}. The classification of $\Q$-, or even $\Z$-homology planes is still a challenging problem. 

We note that in the literature, $\Q$HPs are usually allowed to be singular \cite{GM_k<2,Palka-recent_progress_Qhp}. Since we approach the classification problem only in the smooth case, for convenience we include smoothness in the definition. 

Structure theorems for affine surfaces have led to a complete description of \QHPs $S$ whose logarithmic Kodaira-Iitaka dimension $\kappa(S)$ is not maximal, i.e.\ $\kappa(S)<2$; see Proposition \ref{prop:kappa<2} for a summary. Hence we restrict our attention to those with $\kappa(S)=2$, i.e.\ \emph{of log general type}. 
 
Here, no structure theorem is known. Nonetheless, many authors have constructed various series of examples, see Section \ref{sec:literature}. These examples suggest certain rigidity, which can be expressed by means of the following, open conjectures. For definitions of $\lts{X}{D}$ and a core graph see Sections \ref{sec:rig} and \ref{sec:log_surfaces}, respectively.

\begin{conjecture}\label{conj:classical}
	Let $S$ be a smooth $\Q$-homology plane of log general type, and let $(X,D)$ be a minimal log smooth completion of $S$. Then
	\begin{enumerate}
		\item \label{item:tDP} (Tom Dieck--Petrie Conjecture \cite[1.1]{tDieck_optimal-curves}) There is a birational morphism $\tau\colon X\to \P^2$ such that  $\tau_{*}D$ is a sum of lines and conics.
		\item \label{item:strong_rig} (Flenner--Zaidenberg Rigidity Conjecture \cite[1.3]{Zaid-open_MONTREAL_problems}, cf.\ \cite{FZ-deformations}) $H^{i}(\lts{X}{D})=0$ for all $i\geq 0$.
		\item \label{item:finiteness} (Finiteness Conjecture, cf.\ \cite[1.6]{Zaid-open_MONTREAL_problems}) The number of possible core graphs for $D$ is finite.
		\item \label{item:automorphisms} (Koras Conjecture) The automorphism group of $S$ has at most six elements. 
	\end{enumerate}
\end{conjecture}

In fact, even a weaker rigidity conjecture, asserting that $\chi(\lts{X}{D})=0$, is open, see  \cite[1.4]{Zaid-open_MONTREAL_problems}. By \cite[4.3(i)]{Palka-minimal_models}, the number $\chi(\lts{X}{D})$ is equal to $h^0(2K_{X}+D)$, see Lemma \ref{lem:rig}\ref{item:rig_weak} (for yet another interpretation of $\chi(\cT(-\log D))$ see \cite[4.7]{tDieck_optimal-curves}). Hence from the birational geometry viewpoint, a natural generalization of \cite[1.4]{Zaid-open_MONTREAL_problems} is the following conjecture, which holds for all known examples.

\begin{conjecture}[{Palka Negativity Conjecture \cite[4.7]{Palka-minimal_models}}]\label{conj:negativity}
	Let $(X,D)$ be a log smooth completion of a \QHP. Then $\kappa(K_{X}+\tfrac{1}{2}D)=-\infty$.
\end{conjecture}
We note that $\kappa(K_{X}+\tfrac{1}{2}D)$ depends only on the isomorphism class of the surface $X\setminus D$, see  \cite[2.10]{Palka_MMP}.

The main result of this article is the classification Theorem \ref{CLASS}, where for every $\Q$-homology plane satisfying the Negativity Conjecture \ref{conj:negativity}, we explicitly construct the morphism $\tau$ from tom Dieck--Petrie Conjecture \ref{conj:classical}\ref{item:tDP}. As a consequence, we obtain the following theorem.

\begin{thm}\label{thm:conjectures}
	The Negativity Conjecture \ref{conj:negativity} implies all parts of Conjecture \ref{conj:classical}.
\end{thm}

Moreover, in Corollaries \ref{cor:uniq} and \ref{cor:aut} we provide quantitative answers to Conjectures \ref{conj:classical}\ref{item:finiteness} and \ref{item:automorphisms}.

We remark that the original Finiteness Conjecture, formulated by Zaidenberg in  \cite[1.6]{Zaid-open_MONTREAL_problems}, asked if the number of Eisenbud--Neumann diagrams of $D$ is finite. A positive answer to this question follows from the work of Tono \cite[4.4]{Tono-number_of_cusps}, cf.\  \cite{Orevkov_Tono}, who bounded the number of twigs of $D$ by $17$ (see Section \ref{sec:log_surfaces} for definitions). Its stronger version, Conjecture \ref{conj:classical}\ref{item:finiteness}, was proved by Palka  \cite[1.3]{Palka-minimal_models} in case when $X\setminus D$ is a complement of a rational cuspidal curve, i.e.\ $X\setminus D\cong \P^2\setminus \bar{E}$ for some curve $\bar{E}$ which is homeomorphic to $\P^1$ (in the Euclidean topology). In this case, Negativity Conjecture \ref{conj:negativity} implies the  Rigidity Conjecture \ref{conj:classical}\ref{item:strong_rig} \cite[1.6(d)]{PaPe_delPezzo}.

The main tool used to prove all the above results is the theory of almost minimal models \cite{Palka-minimal_models, Palka_MMP}, applied to the pair $(X,\tfrac{1}{2}D)$. This method is also a key ingredient of the proof of the Coolidge--Nagata Conjecture \cite{KoPa-CooligeNagata2} and other structure theorems for rational cuspidal curves \cite{KoPa_four-cusps}. 
\smallskip
%

	Theorem \ref{CLASS} is stated in terms of of an algorithm proposed by tom Dieck and Petrie  \cite{DiPe-hp_announcement_and_survey,tDieck_optimal-curves}. Before we formulate it in Definition \ref{def:TDP}, we explain the key observation that motivates it. Let $(X,D)$ be a log smooth completion of a $\Q$HP. Assume there is a $(-1)$-curve $A\not \subseteq D$ meeting $D$ normally in two points, exactly one of which lies on a $(-2)$-tip $T$ of $D$, see Figure \ref{fig:intro}. Let $\sigma\colon X\to \bar{X}$ be a contraction of $A$, and let $\bar{A}=\sigma_{*}T$, $\bar{D}=\sigma_{*}D-\bar{A}$. Then $(\bar{X},\bar{D})$ is a log smooth completion  of a $\Q$HP (with a minor exception if $d(\bar{D})=0$, see Lemma \ref{lem:expansions}), and $\bar{A}$ is a $(-1)$-curve meeting $\bar{D}$ normally, twice. If $\bar{A}$ meets a $(-2)$-tip of $\bar{D}$ then we repeat this process, otherwise we contract $\bar{A}$ and stop. Performing this construction for other $A$'s as above yields a morphism $(X,D)\to(X',D')$, called \emph{expansion} in Definition \ref{def:expansion-intro}.
	
	Conversely, given a $(-1)$-curve $A\subseteq X$ meeting $D$ twice, let $\tilde{\sigma}\colon \tilde{X}\to X$ be a blowup at one of the points of $A\cap D$, and let $\tilde{D}=\tilde{\sigma}^{-1}_{*}(D+A)$, $\tilde{A}=\Exc\tilde{\sigma}$, $\tilde{T}=\tilde{\sigma}^{-1}_{*}A$. Again, up to a minor exception,  $(\tilde{X},\tilde{D})$ is a log smooth completion of a $\Q$HP, and $\tilde{A}\not\subseteq \tilde{D}$ is a $(-1)$-curve meeting $\tilde{D}$ in two points, exactly one of which lies on a $(-2)$-tip $\tilde{T}$ of $\tilde{D}$. Therefore, we can iterate this process to get a \emph{tower} of $\Q$-homology planes, each obtained by an expansion from a fixed \emph{basic} pair $(X',D')$.

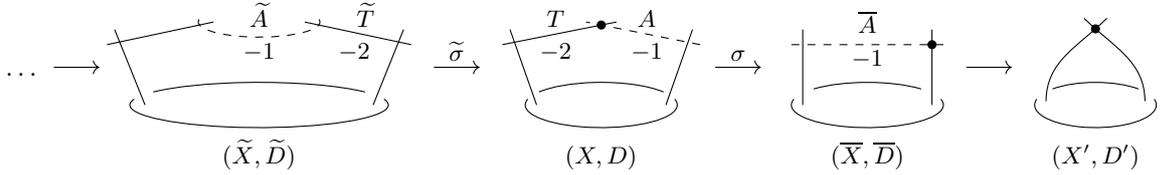
\begin{figure}[ht]
	\centering
\begin{tikzpicture}
	\path[use as bounding box] (-3,-0.8) rectangle (12,1.2);	
	\begin{scope}
		\node at (-3.1,0.4) {$\dots$};
		\draw [->] (-2.7,0.5) -- (-2.1,0.5);		
		\draw (1.5,0) -- (1.9,1);
		\draw (-1.5,0) -- (-1.9,1);
		\draw (-2,0.8) -- (-0.6,1.1);
		\draw (2,0.8) -- (0.6,1.1);
		\draw (0,0) [partial ellipse=35:145:1.7 and .3];
		\draw (0,0) [partial ellipse=-195:15:1.7 and .3];
		\draw[dashed] (0,1.1) [partial ellipse=-190:10:.8 and .2];
		\node at (0,1.2) {\small{$\tilde{A}$}};
		\node at (0,0.7) {\small{$-1$}};
		\node at (1.4,1.2) {\small{$\tilde{T}$}};
		\node at (1.25,0.7) {\small{$-2$}};
		\draw [->] (2.3,0.5) -- (2.9,0.5);
		\node[text height=1] at (2.6,0.6) {\small{$\tilde{\sigma}$}};
		\node[text height=1] at (0,-0.8) {\small{$(\tilde{X},\tilde{D})$}};	
	\end{scope}
	\begin{scope}[shift={(4.5,0)}]
		\draw (0.85,0) -- (1.2,1);
		\draw (-0.85,0) -- (-1.2,1);
		\draw[name path = T] (-1.3,0.8) -- (0.2,1.1);
		\draw[dashed, name path=A] (1.3,0.8) -- (-0.2,1.1);
		\path [name intersections={of=A and T,by=E}];
		\filldraw (E) circle (1.5pt);
		\draw (0,0) [partial ellipse=35:145:1 and .3];
		\draw (0,0) [partial ellipse=-195:15:1 and .3];
		\node at (0.6,1.15) {\small{$A$}};
		\node at (0.6,0.7) {\small{$-1$}};
		\node at (-0.6,1.15) {\small{$T$}};
		\node at (-0.6,0.7) {\small{$-2$}};
		\draw [->] (1.5,0.5) -- (2.1,0.5);
		\node[text height=1] at (1.8,0.6) {\small{$\sigma$}};
		\node[text height=1] at (0,-0.8) {\small{$(X,D)$}};	
	\end{scope}
	\begin{scope}[shift={(8,0)}]
		\draw[name path=R] (0.85,0) -- (0.85,1);
		\draw (-0.85,0) -- (-0.85,1);
		\draw[dashed, name path=A] (-1,0.8) -- (1,0.8);
		\path [name intersections={of=A and R,by=E}];
		\filldraw (E) circle (1.5pt);
		\draw (0,0) [partial ellipse=45:135:1 and .3];
		\draw (0,0) [partial ellipse=-195:15:1 and .3];
		\node at (0,1.1) {\small{$\bar{A}$}};
		\node at (0,0.6) {\small{$-1$}};
		\draw [->] (1.3,0.5) -- (1.9,0.5);
		\node[text height=1] at (0,-0.8) {\small{$(\bar{X},\bar{D})$}};
	\end{scope}
	\begin{scope}[shift={(11,0)}]
		\draw[name path=L]  (0.65,0.05) to[out=90,in=-45] (-0.15,1.15);
		\draw[name path=R] (-0.65,0.05) to[out=90,in=-135] (0.15,1.15);
		\draw (0,0) [partial ellipse=45:135:0.8 and .3];
		\draw (0,0) [partial ellipse=-195:15:0.8 and .3];
		\path [name intersections={of=L and R,by=E}];
		\filldraw (E) circle (1.5pt);
		\node[text height=1] at (0,-0.8) {\small{$(X',D')$}};		
	\end{scope}
\end{tikzpicture}
	\caption{Expansion produces a tower of \QHPs, see Definition \ref{def:expansion-intro}.}
	\label{fig:intro}
\end{figure}	
	
	The main point of using almost MMP is that it naturally produces curves $A$ as above. Indeed, let $\psi\colon (X,\tfrac{1}{2}D)\to (X_{\min},\tfrac{1}{2}D_{\min})$ be an MMP run. It follows from \cite[5.2]{Palka_MMP} that every curve $A\not \subseteq D$ contracted by $\psi$ is a $(-1)$-curve meeting $D$ normally, twice, so it  is of the above type. This way, we get an expansion $(X,D)\to (X',D')$, where $(X',D')$ is a minimal log resolution of $(X_{\min},D_{\min})$, see Lemma \ref{lem:MMP}.
	
	Assume $\kappa(K_{X}+\tfrac{1}{2}D)=-\infty$, so $(X_{\min},\tfrac{1}{2}D_{\min})$ is a log Mori fiber space over some base $B$. If $\dim B=1$ then $X\setminus D$ is $\C^{**}$-fibered, cf.\ \cite[4.5(4)]{Palka-minimal_models}. This case was settled in \cite{MiySu-Cstst_fibrations_on_Qhp}, see Section \ref{sec:Cstst}. If $\dim B=0$ then \cite[5.5]{Palka_MMP} implies that $X_{\min}$ is a canonical del Pezzo surface of rank one, see Lemma \ref{lem:R}. Such surfaces are classified in \cite{Furushima}. To infer Conjecture \ref{conj:classical}\ref{item:tDP}, we will transform $D_{\min}\subseteq X_{\min}$ into  a planar configuration.
	\smallskip
	
	In order to state Theorem \ref{CLASS}, we will now describe the above procedure in more precise terms. We adapt terminology from \cite{DiPe-hp_and_alg_curves}. Some examples constructed this way appeared earlier e.g.\ in \cite{MiySu-Cstst_fibrations_on_Qhp}, where what we call \emph{expansion} is called an \emph{oscillating sequence of blowing-ups}.

	\begin{dfn}[Expansion]\label{def:expansion-intro}
		Let $D'$ be an snc divisor on a smooth projective surface $X'$. Fix an integer $n\geq 0$. For $i\in \{1,\dots, n\}$, fix a node $x_{i}\in \Sing D'$ and coprime positive integers $u_{i},w_{i}\in \N$. Let $U_i,W_i$ be the components of $D'$ meeting at $x_{i}$. Put $\cC=((U_{1},W_{1};x_1),\dots, (U_n,W_n;x_n))$ and $\boldsymbol{v}=(\frac{u_1}{w_1},\dots,\frac{u_n}{w_n})\in \Q_{>0}^{n}$.
		
		An \emph{expansion over $(X',D')$ with centers $\cC$ and weights $\boldsymbol{v}$} is a morphism $\varphi\colon (X,D)\to (X',D')$ with the following properties. First,  $\Exc\varphi=\sum_{i=1}^{n}T_{i}$, where each $T_{i}\de \varphi^{-1}(x_i)$ is a chain with a unique $(-1)$-curve $A_{i}$. Second, the multiplicity of $A_i$ in $\varphi^{*}U_{i}$, $\varphi^{*}W_{i}$ equals $u_{i}$ and $w_{i}$, respectively. Third, $D=(\varphi^{*}D')\redd-\sum_{i=1}^{n}A_{i}$.
	
		A \emph{tower over $\cC$} is the set of pairs $(X,D)$ obtained by an expansion with centers $\cC$ and varying weights.
	\end{dfn}

	\begin{dfn}[tom Dieck--Petrie algorithm, see Example \ref{ex:7}]\label{def:TDP}
		Let $\pp\subseteq \P^{2}$ be a reduced configuration of lines and conics, and let $P$ be a subset of $\Sing \pp$. Denote by $\pi\colon X'\to \P^{2}$ the minimal log resolution and by $E'$ be the sum of $(-1)$-curves in $\pi^{-1}(P)$. Put $D'=(\pi^{*}\pp)\redd-E'$. On $(X',D')$, fix $\cC, \boldsymbol{v}$ as in Definition \ref{def:expansion-intro}, and let $(X,D)\to (X',D')$ be the expansion with centers $\cC$ and weights $\boldsymbol{v}$.
		
		In this case, we say that the pair $(X,D)$ is \emph{obtained via tom Dieck--Petrie algorithm} from $(\pp,P,\cC,\boldsymbol{v})$.
	\end{dfn}

With these preparations at hand, we can now state our main classification result.

\begin{thm}\label{CLASS}
	Let $S$ be a smooth \QHP of log general type and let $(X,D)$ be a minimal log smooth completion of $S$. Assume that $S$ satisfies the Negativity Conjecture \ref{conj:negativity}, that is, 
	\begin{equation*}
		\kappa(K_{X}+\tfrac{1}{2}D)=-\infty.
	\end{equation*}
	Then $(X,D)$ is obtained via tom Dieck--Petrie algorithm from data listed in one of the rows of \tables.
		
	Conversely, any pair $(X,D)$ obtained this way is a minimal log smooth completion of a \QHP of log general type satisfying the Negativity Conjecture \ref{conj:negativity}.
\end{thm}

We denote by $\rst$ the tower obtained from the $*$-th row of \tables. For each pair $(X,D)$ in $\rst$, the complement $X\setminus D$ is a $\Q$HP, unless the corresponding weight $\boldsymbol{v}$ is excluded in the second-to-last column. By a slight abuse of notation, we will refer by $\rst$ to the set of those $\Q$HPs, too.

Configurations $\pp\subseteq \P^2$ from Definition \ref{def:TDP} are constructed in Section \ref{sec:constructions}. There is $39$ of them, and they satisfy $\deg\pp\leq 11$. Theorem \ref{CLASS} asserts that each $\Q$HP belongs to a tower over one of these $39$ configurations. 
The corresponding graphs of $D'$ are shown in Figures \ref{fig:nodal-cubic-P2_un}--\ref{fig:Cstst}. To obtain $D$, one needs to remove from $D'$ the points of $\cC$, and replace each of them by $T_{i}-A_{i}$ as in Definition \ref{def:expansion-intro}, i.e.\ by at most two twigs of $D$. Given Theorem \ref{CLASS}, this proves parts \ref{item:core} and \ref{item:EN} of the following corollary. We write $\operatorname{core}(D)$ for $D$ minus the sum of its twigs.

\begin{cor}[Finiteness]\label{cor:uniq}
	Let $(X,D)$ be a minimal log smooth completion of a smooth $\Q$-homology plane of log general type, satisfying the Negativity Conjecture \ref{conj:negativity}. Then 
	\begin{enumerate}
		\item\label{item:core} $D$ has at most $10$ twigs, $\#\operatorname{core}(D)\leq 7$, and $C\cdot (D-C)\leq 5$ for every component $C$ of $D$.
		\item\label{item:EN} The core graph (hence the Eisenbud-Neumann diagram) of $D$ has at most $17$ vertices. 
		\item\label{item:n} Up to an isomorphism, there are exactly $\ngr\leq 4$ smooth $\Q$-homology planes $\tilde{S}$ of log general type; satisfying Negativity Conjecture \ref{conj:negativity}, such that, denoting by $(\tilde{X},\tilde{D})$ a minimal log smooth completion of $\tilde{S}$, the weighted graphs of $\tilde{D}$ and $D$ are the same. The number $\ngr$ is listed in \tables. 
%
%
	\end{enumerate}
\end{cor}

To prove Corollary \ref{cor:uniq}\ref{item:n}, we note that each weighted graph of $D$ from Theorem \ref{CLASS} is uniquely determined by the initial data from \tables. Moreover, since each blowup within the construction of $(X,D)$ is centered at the common point of specific components of the boundary, these combinatorial data determine the isomorphism class of $(X,D)$, up to certain symmetries of that pair. Thus to infer \ref{item:n}, we need to understand those symmetries. This is done in Section \ref{sec:uniqueness}, where we prove the following corollary, too.  

\begin{cor}[Automorphisms]\label{cor:aut}
	Let $S$ be a \QHP of log general type. Assume that $S$ satisfies the Negativity Conjecture \ref{conj:negativity}. Then $\Aut(S)$ is isomorphic to a subgroup of the symmetric group $S_{3}$. In case  $\Aut(S)\cong S_{3}$, the surface $S$ is isomorphic to the complement of the planar tricuspidal quartic given by
	\begin{equation*}
	(xy)^{2}+(yz)^{2}+(zx)^{2}=2xyz(x+y+z).
	\end{equation*}
	In general, if the group $\Aut(S)$ is nontrivial then $S$ is isomorphic to one of the surfaces in Proposition \ref{prop:aut}.
\end{cor}

It turns out that all but one $\Q$HPs from Proposition \ref{prop:aut} have already appeared in the literature, see \cite{tDieck_symmetric_hp} and Section \ref{sec:literature}. The exception is the surface \ref{def:F2n0}, which is a complement of a certain new four-cuspidal curve on $\P(1,1,2)$, and admits an involution constructed in Example \ref{ex:F2n0_Aut}. In general, no absolute bound on $\#\Aut(S)$ is known; even in case of \ZHPs, for which $\Aut(S)$ is cyclic by \cite[Theorem 2.2]{GKMR-singular_Zhp}.
\smallskip

The following simple consequence of Theorem \ref{CLASS} is proved in Section \ref{sec:constructions}.

\begin{cor}\label{cor:k}
	Every \QHP of log general type satisfying Conjecture \ref{conj:negativity} is  defined over a number field of degree at most $4$ over $\Q$. This field is specified in column \enquote{$\kk$} of \tables.
\end{cor}

We conclude with an observation pointed out by Marco Golla. Part \ref{item:exceptions} shows that Corollary \ref{cor:uniq}\ref{item:n} gives a lot of interesting examples of non-isomorphic, but diffeomorphic $\Q$-homology planes of log general type. 

\begin{cor}[see Proposition \ref{prop:Marco}]\label{cor:Marco}
	Let $\cS$ be an $\ngr$-tuple from Corollary \ref{cor:uniq}\ref{item:n}. That is, $\cS$ is a set of isomorphism classes of all $\Q$-homology planes of log general type, satisfying Negativity Conjecture \ref{conj:negativity}, whose minimal log smooth completions $(X,D)$ share a given weighted graph of $D$. Then the following holds.
	\begin{enumerate}
		\item All surfaces in $\cS$ are defined over a number field $\kk$, and the Galois group of $\kk$ acts transitively on $\cS$. In particular, the fundamental groups of all surfaces in $\cS$ have the same  profinite completions.
		\item\label{item:exceptions} If $\cS$ is contained in towers 
\ref{def:A1A2_c=1}--\ref{def:A1A2_C2C3-node},  \ref{def:A1A2_q-nnc}--\ref{def:A1A2_q-cn_31}, 
		\ref{def:A1A2_2n1c}, 
		\ref{def:A1A2_2c1n}, 
		\ref{def:F2_n1-node}, 
		\ref{def:F2_n1-cusp-hor_41},
		\ref{def:F2-5_cn}, 
		\ref{def:nodal-cubic-P2_un}, 
		\ref{def:F2n2-nodal} or 
		\ref{def:F2_n2-transversal}
		then $\cS$ contains a pair of non-isomorphic surfaces which are complex-conjugate, hence diffeomorphic.
	\end{enumerate}
\end{cor}

%
%
%

\begin{rem}
	The assumption $\kappa(S)=2$ is crucial for all the above corollaries. Indeed, almost all \QHPs of non-general type admit a $\C^1$- or a  $\C^{*}$-fibration, see Proposition \ref{prop:kappa<2}. With such a structure at hand, one can easily deform its degenerate fibers to construct arbitrarily high-dimensional families of $\Q$HPs with the same weighted boundary graph; see \cite[4.16, 6.9]{FZ-deformations}, \cite[5.1]{Palka-recent_progress_Qhp} or \cite{Tono_1cusp_with_kod_1}.
\end{rem}

The article is organized as follows. In Section \ref{sec:preliminaries}, we recall some basic definitions and notation. Next, in Section \ref{sec:expansions} we list the properties of expansions $(X,D)\to (X',D')$ used in Theorem \ref{CLASS}. In particular, in Lemma \ref{lem:expansions} we show how to check if the resulting pair $(X,D)$ is a \QHP; and in Lemma \ref{lem:MMP} we relate it with an almost MMP run described in \cite{Palka_MMP}. 
Before we delve into the proof of Theorem \ref{CLASS}, in Section \ref{sec:strategy} we take a break to illustrate our strategy on particular examples; and in Section \ref{sec:literature} we explain how to recover some well-known $\Q$-homology planes from our list.

To prove Theorem \ref{CLASS}, we study a minimal model $(X_{\min},\tfrac{1}{2}D_{\min})$  of $(X,\tfrac{1}{2}D)$. The assumption $\kappa(K_{X}+\tfrac{1}{2}D)=-\infty$ implies that $(X_{\min},\tfrac{1}{2}D_{\min})$ is a log Mori fiber space: either over a point, or over a curve. The first case is treated in Section \ref{sec:classification}: this is the main part of the proof. In the second case, the assumption $\kappa=2$ implies that $S$ is $\C^{**}$-fibered, see Lemma \ref{lem:R}\ref{item:MFS_point}. Such $\Q$HPs are described in \cite{MiySu-Cstst_fibrations_on_Qhp}, see Table   \ref{table:C**}. Nonetheless, in Section \ref{sec:Cstst} we give a self-contained proof of Theorem \ref{CLASS} in the $\C^{**}$-fibered case. This is motivated by the fact that, as observed in \cite{Su-hp_having_rational_pencils}, the proof in \cite{MiySu-Cstst_fibrations_on_Qhp} can be much simplified using a later result of \cite{GM-Affine_lines_on_Qhp} that $\Q$HPs of log general type contain no curves isomorphic to $\C^{1}$.

 Planar divisors used in Theorem \ref{CLASS} are constructed in Section \ref{sec:constructions}.  
In Section \ref{sec:uniqueness} we prove Corollaries \ref{cor:uniq}\ref{item:n} and \ref{cor:aut} by describing the automorphisms of our surfaces and their boundary graphs. In Section \ref{sec:rig} we prove the remaining part of Theorem \ref{thm:conjectures}, that is, we show that surfaces in Theorem \ref{CLASS} satisfy the Rigidity Conjecture \ref{conj:classical}\ref{item:strong_rig}. To do this, we follow the argument in \cite[Proposition 2.2]{FlZa_cusps_d-3} showing rigidity for complements of certain rational cuspidal curves. Eventually, in Section \ref{sec:k=2} we check that 
our $\Q$HPs are indeed of log general type.

\paragraph{Acknowledgments}

The article contains results of my PhD thesis, written at the Institute of Mathematics of the Polish Academy of Sciences. I would like to thank my advisor, Karol Palka, for numerous conversations and reading preliminary versions of the thesis. I would also like to thank Torgunn K.\ Moe, Marco Golla and Mikhail Zaidenberg for helpful discussions. I am grateful to the referees for valuable suggestions.

\setcounter{tocdepth}{1}
\tableofcontents
	
\section{Preliminaries}\label{sec:preliminaries}

\subsection{Divisors on surfaces}\label{sec:log_surfaces}

In order to settle the notation for the remaining part of the article, we now briefly recall some standard definitions from the theory of log surfaces. For details see e.g.\ \cite{Fujita-noncomplete_surfaces}.

Let $X$ be a smooth projective surface. We denote by $\NS(X)$ the N\'eron-Severi group of divisors on $X$ modulo numerical equivalence; and write $\NS_{\Q}(X)=\NS(X)\otimes_{\Z}\Q$. By a \emph{curve} we mean irreducible and reduced variety of dimension $1$. A curve $C\subseteq X$ such that $C\cong \P^{1}$ and $C^{2}=n$ is called an \emph{$n$-curve}. By adjunction formula, an $n$-curve $C$ satisfies $K_{X}\cdot C=-2-n$, where $K_{X}$ is the canonical divisor on $X$.

Let $D$ be a reduced divisor on $X$.  By a \emph{component} of $D$ we mean an \emph{irreducible} component. We denote their number by $\#D$. A \emph{subdivisor} of $D$ is an effective divisor $T$ such that $D-T$ is effective. Its \emph{branching number} is $\beta_{D}(T)=T\cdot (D-T)$. A component $C$ of $D$ is a \emph{tip} of $D$ if $\beta_{D}(C)\leq 1$, and is \emph{branching} in $D$ if $\beta_{D}(C)\geq 3$.

A point of $p\in D$ is a \emph{simple normal crossing} (\emph{snc}) point if $p$ lies on exactly two different components of $D$, which meet transversally at $p$. We say that $D$ is an \emph{snc divisor} if all its singular points are snc. Note that two components of an snc divisor can meet in more than one point. 

Let $D$ be a connected snc divisor. If $\beta_{D}(T)\leq 2$ for every component $T$ of $D$, then $D$ is called a \emph{chain} in case at least one of these inequalities is strict and \emph{circular} otherwise. If $D$ has no circular subdivisor then $D$ is a \emph{tree}. 
We say that $D$ is \emph{rational} if all its components are. 

Let $T$ be a chain. Its components $T_{1},\dots, T_{r}$ can be ordered in such a way that $T_{i}\cdot T_{i+1}=1$ for $i\in \{1,\dots, m-1\}$. 
If $T$ is rational, we say that it is of \emph{type} $[-T_{1}^{2},\dots, -T_{r}^{2}]$, and often abuse notation by writing $T=[-T_{1}^{2},\dots, -T_{r}^{2}]$. We write $(m)_{k}$ for the sequence consisting of an integer $m$ repeated $k$ times.

Let again $D$ be a reduced divisor, and let $D_{1},\dots, D_{s}$ be its components. We say that $D$ is \emph{negative definite} if its intersection matrix $[D_{i}\cdot D_{j}]_{1\leq i,j\leq s}$ is. The \emph{discriminant} of $D$ is $d(D)\de \det[-D_{i}\cdot D_{j}]_{1\leq i,j\leq s}$. By convention, $d(0)\de 1$. For elementary properties of discriminants we refer to \cite[Section 3]{Fujita-noncomplete_surfaces}.

A \emph{twig} of $D$ is a nonzero subchain $T\subseteq D$ containing a tip of $D$, such that $\beta_{D}(C)\leq 2$ for every component $C$ of $T$. If a twig $T\subseteq D$ is not a connected component of $D$, we order it so that its first component is a tip of $D$, so the last component of $T$ meets $D-T$. A \emph{$(-2)$-twig} is a twig whose all components are $(-2)$-curves. We say that a twig (or a $(-2)$-twig) of $D$ is \emph{maximal} if it is not properly contained in any other twig (or a $(-2)$-twig) of $D$. Note that a maximal $(-2)$-twig is not necessarily a maximal twig.
\smallskip

 Let $B$ be the sum of branching components of $D$, and let $T$ be the sum of twigs of $D$. The \emph{core graph} (respectively, the \emph{Eisenbud-Neumann diagram}) of $D$ is an unweighted graph, whose vertices are components of $D-T$ (resp.\  $B$) and connected components of $T$ (resp.\ $D-B$); and two vertices are connected by an edge whenever the corresponding subdivisors of $D$ are distinct and meet each other. In other words, the core graph is obtained from the graph of $D$ by forgetting the weights and contracting each twig to a single vertex; and EN-diagram is  obtained by further contracting all chains of non-branching components.
 
 Note that the core graph does not encode self-intersection numbers of the components of $D$, or lengths of the twigs. In fact, for surfaces in Theorem \ref{CLASS} those numbers can be arbitrarily big. This is why Conjecture \ref{conj:classical}\ref{item:finiteness} is formulated in terms of core graphs.
\smallskip

Let $T$ be a rational twig of $D$ whose components have self-intersection at most $-2$. The \emph{bark} of $T$ in $D$, denoted by $\Bk_{D}T$ is the unique effective $\Q$-subdivisor of $T$ such that $T'\cdot \Bk_{D}T=T'\cdot (K_{X}+D)=\beta_{D}(T')-2$ for every component $T'$ of $T$, see \cite[6.12]{Fujita-noncomplete_surfaces}. It exists because $T$ is negative definite. We write $\Bk T=\Bk_{D}T$ if $D$ is clear from the context. If $T_1,\dots T_r$ are disjoint twigs of $D$, we put $\Bk T=\sum_{i}\Bk T_i$. If $T=T\cp{1}+\dots+T\cp{r}$ is an ordered $(-2)$-twig such that $T\cp{r}$ meets $D-T$, then a direct computation, see \cite[II.3.3.3]{Miyan-OpenSurf}, gives
\begin{equation}\label{eq:bark_of_a_2-twig}
\Bk T=\sum_{i=1}^r\frac{r-i+1}{r+1}T\cp{i}.
\end{equation}

A $(-1)$-curve $L\subseteq D$ is called \emph{superfluous} if it contracts to a smooth or snc point of the image of $D$. Equivalently, $L$ is superfluous if $\beta_{D}(L)=1$ or $\beta_{D}(L)=2$ and $L$ meets two different components of $D-L$. An snc divisor is \emph{snc-minimal} if it has no superfluous $(-1)$-curves.

Let $\sigma\colon X\to X'$ be a birational morphism between smooth projective surfaces. The \emph{rank} of $\sigma$, denoted by $\rho(\sigma)$, is the difference of Picard ranks $\rho(X)-\rho(X')$. A \emph{center} of $\sigma$ is a base point of $\sigma^{-1}$, i.e.\ a point $p\in X'$ such that $\sigma$ is not an isomorphism in any neighborhood of the preimage of $p$. In this case, $\sigma^{-1}(p)$ is a rational tree with at least one $(-1)$-curve. We denote by $\Bs\sigma^{-1}\subseteq X'$ the (finite) set of centers of $\sigma$.

A morphism of pairs $\sigma\colon (X,D)\to (X',D')$ is a morphism $\sigma\colon X\to X'$ such that $\sigma_{*}D=D'$.
\smallskip

Let $S$ be a smooth surface. A surjective morphism $f\colon S\to B$ onto a smooth curve is called a $\P^1$- (respectively,  $\C^{(n*)}$-) \emph{fibration} if there is a nonempty Zariski-open subset $U\subseteq B$ such that for every $b\in U$ the fiber $f^{*}(b)$ is isomorphic to $\P^1$ (respectively, $\C^{1}\setminus \{p_{1},\dots, p_{n}\}$ for some pairwise distinct $p_{1},\dots, p_{n}$, depending on $b$). A fiber without this property is called \emph{degenerate}. Any fiber of a $\P^1$-fibration can be contracted to a $0$-curve by iterated contractions of $(-1)$-curves. Therefore, such fibers are easy to understand, see \cite[Section 4]{Fujita-noncomplete_surfaces}.

A curve $C$ on $S$ is called \emph{vertical} if its image $f(C)$ is a point; otherwise $C$ is called \emph{horizontal}. For every divisor $T$ on $S$, we have a unique decomposition $T=T\vert+T\hor$, where all components of $T\vert$ and $T\hor$ are vertical and horizontal, respectively. We call $T\vert$ and $T\hor$ the \emph{vertical part} and the \emph{horizontal part} of $T$.

\subsection{Log smooth completions}
A \emph{log smooth surface} is a pair $(X,D)$ consisting of a smooth projective surface $X$ and an snc-divisor $D$ on $X$. Any smooth surface $S$ admits a \emph{log smooth completion}, that is, a log smooth pair $(X,D)$ such that $X\setminus D\cong S$. A log smooth completion is \emph{minimal} if $D$ is snc-minimal.

The following lemma is a well known consequence e.g.\ of \cite[Corollary 3.36]{FKZ-weighted-graphs}. Part \ref{item:k=2} allows us to speak of \emph{the} log smooth completion of a surface of log general type.

\begin{lem}\label{lem:unique_completion} 
	Let $S$ be a smooth surface and let $(X,D)$ be some minimal log smooth completion of $S$. Assume that one of the following conditions holds.
	\begin{enumerate}
		\item\label{item:k=2} The surface $S$ is of log general type.
		\item\label{item:graph} Every non-branching rational component of $D$ has negative self-intersection number.
	\end{enumerate}
	Then $(X,D)$ is the unique minimal log smooth completion of $S$. In particular, every automorphism of $S$ extends to an automorphism of $(X,D)$.\end{lem}
\begin{proof}
	First, we show that condition \ref{item:k=2} implies \ref{item:graph}. Suppose that $D$ has a non-branching, rational component $C$ with $C^2\geq 0$. Choose a point $p\in C\cap (D-C)$, or any $p\in C$ in case $C\cap (D-C)=\emptyset$. Blow up over $p$ until the proper transform of $C$, say $\hat{C}$, becomes a $0$-curve. Denote the resulting surface by $\hat{X}$, and let $\hat{D}$ be the reduced total transform of $D$. Then $\beta_{\hat{D}}(\hat{C})\leq 2$, so the $\P^1$-fibration induced by $|\hat{C}|$ restricts to a $\P^1$-, $\C^{1}$- or $\C^{*}$-fibration of $\hat{X}\setminus \hat{D}=S$. The Iitaka easy addition theorem gives $\kappa(S)<2$, which proves the claim.
	
	Therefore, we can assume condition \ref{item:graph}. Suppose $S$ admits another minimal log smooth completion, say $(\tilde{X},\tilde{D})$. Let $(\hat{X},\hat{D})$ be a log smooth completion of $S$ which dominates both, i.e.\ there are birational morphisms $\phi\colon (\hat{X},\hat{D})\to (X,D)$ and $\tilde{\phi}\colon (\hat{X},\hat{D})\to (\tilde{X},\tilde{D})$. We choose $(\hat{X},\hat{D})$ in such a way that the Picard rank $\rho(\hat{X})$ is minimal possible. Then the first exceptional curve of $\tilde{\phi}$, say $L$, is a superfluous $(-1)$-curve in $\hat{D}$ and is not contracted by $\phi$. Since $\phi(L)$ cannot be a superfluous $(-1)$-curve in $D$, the morphism $\phi$ is not an isomorphism near $L$. Hence $\phi(L)^2\geq L^2+1=0$. Since $\Exc\phi\subseteq \hat{D}$ and $D=\phi_{*}\hat{D}$ is snc,  we have $\beta_{D}(\phi(L))\leq \beta_{\hat{D}}(L)\leq 2$. Therefore, $\phi(L)$ is a non-branching, rational component of $D$ with $\phi(L)^2\geq 0$; a contradiction.
\end{proof}

\subsection{Structure theorems for $\Q$-homology planes}\label{sec:QHP}

Structure theorems for smooth $\Q$HPs of non general type ($\kappa<2$) can be summarized as follows.

\begin{prop}\label{prop:kappa<2}
	Any smooth \QHP $S$ is affine \cite[2.5]{Fujita-noncomplete_surfaces} and rational \cite{GuPrad-rationality_3(smooth_II)}. Moreover, 
	\begin{enumerate} 
		\item\label{item:k=-infty} \cite[III.1.3.2]{Miyan-OpenSurf} If $\kappa(S)=-\infty$ then $S$ has a $\C^{1}$-fibration.
		\item\label{item:k=0,1} \cite[III.1.7.1]{Miyan-OpenSurf} Assume that $\kappa(S)\in \{0,1\}$. Then either $S$ has a $\C^{*}$-fibration, or $S$ has a $\C^{**}$-fibration and is one of the Fujita surfaces $Y\{3,3,3\}$, $Y\{2,4,4\}$ or $Y\{2,3,6\}$ \cite[8.64]{Fujita-noncomplete_surfaces} (cf.\ \cite[III.4.4.2, III.4.4.3]{Miyan-OpenSurf}). 
	\end{enumerate}
\end{prop}

Similar results for singular $\Q$HPs hold, too, see \cite{Palka-classification1_Qhp} and \cite{Palka-recent_progress_Qhp} for a survey.
\smallskip

The main tool to study surfaces of log general type ($\kappa=2$) is the logarithmic Bogomolov-Miyaoka-Yau inequality \cite{KobNakSak-ball_quotients,Langer_log_Euler_and_appl}, see \cite[Corollary 2.5]{Palka-exceptional_Qhp}. One of its consequences is the following result, which will be of key importance for us. The main idea of its proof appeared in \cite{MiTs-lines_on_qhp} (cf.\  \cite[II.3.11]{Miyan-OpenSurf}), where it was shown that a smooth $\Q$HP contains no curve isomorphic to $\C^{1}$.

\begin{lem}[\cite{GM-Affine_lines_on_Qhp}]\label{lem:no_lines}
	Let $S$ be a smooth \QHP of log general type. Then $S$ contains no topologically contractible curves.
\end{lem}

\subsection{Del Pezzo surfaces of rank one with $\rA_{k}$-singularities}

	In Lemma \ref{lem:MMP}\ref{item:smoothness}, we will see that the underlying surface $X_{\min}$ of the minimal model of $(X,\tfrac{1}{2}D)$ is a canonical del Pezzo surface of Picard rank one. Such surfaces were classified in \cite{Furushima}, cf.\ \cite{AN-delPezzo,Ye}. We will use the following part of this classification.

\begin{lem}[{\cite[Theorem 2]{Furushima}\footnote{There are some misprints in Table II loc.\ cit. In case $\rA_{2}+\rA_{5}$ there should be $(P_{1}^{3},P_{2}^{4})$ and $3P_{1}\equiv 0, 3P_{2}\equiv 0$ in the third and fourth column, respectively: this can be seen from the required singularity types. The cubic $\tilde{D}_{1}$ in cases $2\rA_1+2\rA_3$ and $4\rA_{2}$ should be nodal: indeed, the equations in  the fourth column have a solution if $\tilde{D}_{1}\reg\cong \mathbb{G}_{m}$, not $\mathbb{G}_{a}$. Moreover, instead of $\rD_{4}+\rA_{3}$ and $\rE_{7}+\rA_{2}$ there should be $\rD_{4}+3\rA_{1}$ and $\rE_{6}+\rA_{2}$, respectively, cf.\ \cite{Ye}.}}]\label{lem:Furushima}
	Let $\bar{Y}\not\cong \P^2,\P(1,1,2)$ be a del Pezzo surface of rank one. Assume that each singular point of $\bar{Y}$ is of type $\rA_{k}$ for some $k\geq 1$.  
	Let $\alpha\colon Y \to \bar{Y}$ be the minimal resolution. Then there is a birational morphism $\theta\colon Y\to \P^{2}$ such that $(\Exc\alpha+\Exc\theta)\redd$ has one of the graphs in Figure \ref{fig:Furushima}, where:
	\begin{itemize}
		\item $\Exc \alpha$ is the union of solid lines (both fine and bold): they are $(-2)$-curves,
		\item $\Exc\theta$ is the union of fine lines (both solid and dashed): the dashed ones represent $(-1)$-curves.
	\end{itemize}
\end{lem}

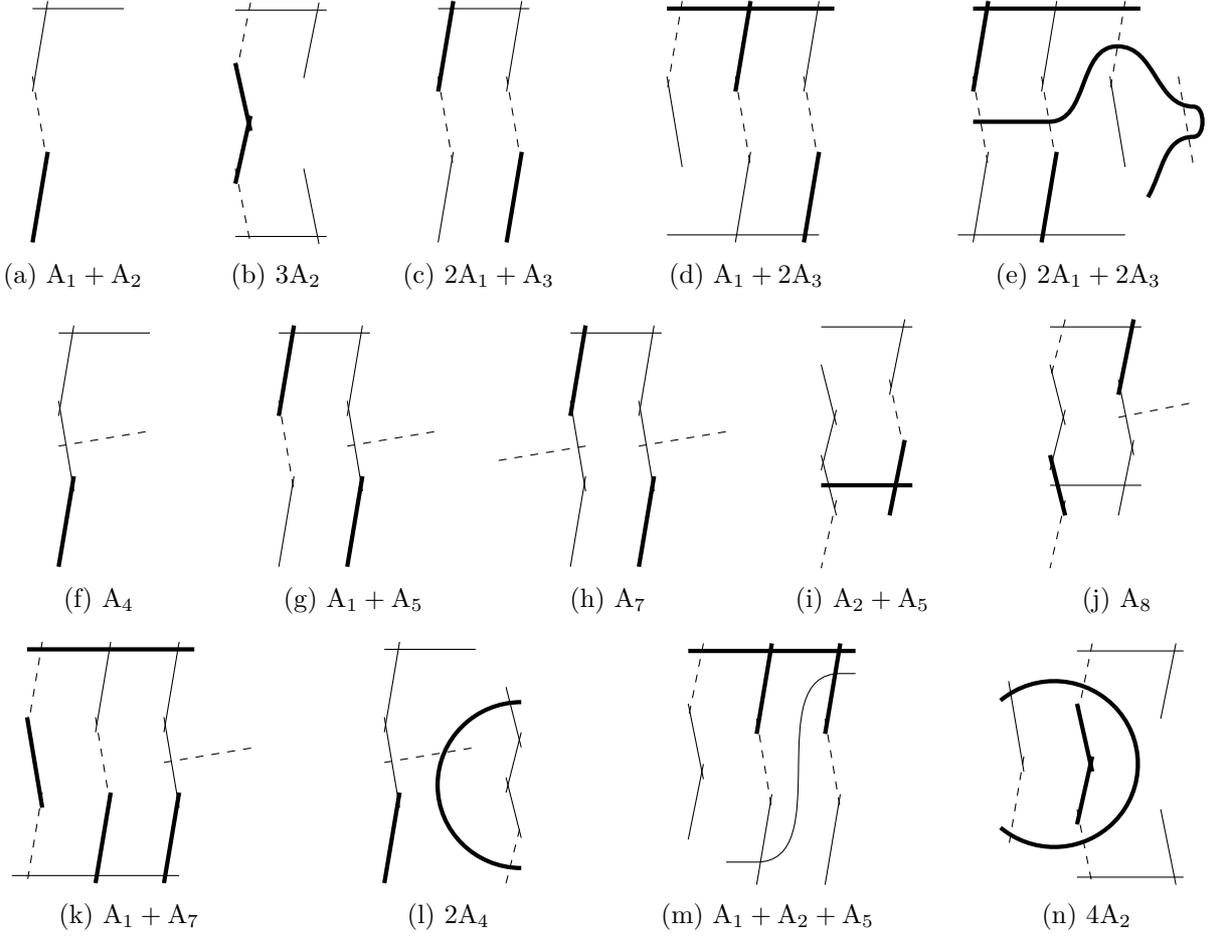
\begin{figure}[ht]
	\subcaptionbox{$\rA_1+\rA_2$ \label{fig:A1+A2}}
	[.15\linewidth]
	{
	\begin{tikzpicture}
		\draw (0,3) -- (1.2,3);
		\draw (0.2,3.1) -- (0,1.9);
		\draw[dashed] (0,2.1) -- (0.2,0.9);
		\draw[thick] (0.2,1.1) -- (0,-0.1);
	\end{tikzpicture}
	}
	\subcaptionbox{$3\rA_2$ \label{fig:3A2}}
	[.15\linewidth]
	{
	\begin{tikzpicture}
		\draw (0,3) -- (1.2,3);
		\draw[dashed] (0.2,3.1) -- (0,2.1);
		\draw[thick] (0,2.3) -- (0.2,1.4);
		\draw[thick] (0.2,1.6) -- (0,0.7);
		\draw[dashed] (0,0.9) -- (0.2,-0.1);
		\draw (1.1,3.1) -- (0.9,2.1);
		\draw (0.9,0.9) -- (1.1,-0.1);
		\draw (0,0) -- (1.2,0);
	\end{tikzpicture}
	}
	\subcaptionbox{$2\rA_1+\rA_3$ \label{fig:2A1+A3}}
	[.15\linewidth]{
	\begin{tikzpicture}
		\draw (0,3) -- (1.2,3);
		\draw[thick] (0.2,3.1) -- (0,1.9);
		\draw[dashed] (0,2.1) -- (0.2,0.9);
		\draw (0.2,1.1) -- (0,-0.1);
		\draw (1.1,3.1) -- (0.9,1.9);
		\draw[dashed] (0.9,2.1) -- (1.1,0.9);
		\draw[thick] (1.1,1.1) -- (0.9,-0.1);		
	\end{tikzpicture}	
	}
	\subcaptionbox{$\rA_1+2\rA_3$ \label{fig:A1+2A3}}
	[.25\linewidth]{
	\begin{tikzpicture}
		\draw[thick] (0,3) -- (2.2,3);
		\draw (0,0) -- (2,0);
		\draw[dashed] (0.2,3.1) -- (0,1.9);
		\draw (0,2.1) -- (0.2,0.9);
		\draw[thick] (1.1,3.1) -- (0.9,1.9);
		\draw[dashed] (0.9,2.1) -- (1.1,0.9);
		\draw (1.1,1.1) -- (0.9,-0.1);		
		\draw (2,3.1) -- (1.8,1.9);
		\draw[dashed] (1.8,2.1) -- (2,0.9);
		\draw[thick] (2,1.1) -- (1.8,-0.1);			
	\end{tikzpicture}	
	}		
	\subcaptionbox{$2\rA_1+2\rA_3$ \label{fig:2A1+2A3}}
		[.25\linewidth]{
		\begin{tikzpicture}
			\draw[thick] (0,3) -- (2.2,3);
			\draw (-0.2,0) -- (2,0);
			\draw[thick] (0.2,3.1) -- (0,1.9);
			\draw[dashed] (0,2.1) -- (0.2,0.9);
			\draw (0.2,1.1) -- (0,-0.1);
			\draw (1.1,3.1) -- (0.9,1.9);
			\draw[dashed] (0.9,2.1) -- (1.1,0.9);
			\draw[thick] (1.1,1.1) -- (0.9,-0.1);		
			\draw[dashed] (2,3.1) -- (1.8,1.9);
			\draw (1.8,2.1) -- (2,0.9);
			\draw[dashed] (2.7,2.1) -- (2.9,0.9);
			\draw[thick] (0,1.5) -- (1,1.5) to[out=0,in=180] (1.9,2.5) to[out=0,in=180] (2.9,1.7)
			to[out=0,in=0] 
			(2.9,1.3) to[out=180,in=60] (2.3,0.5);		
		\end{tikzpicture}	
		}\\
	\bigskip

	\subcaptionbox{$\rA_4$ \label{fig:A4}}
	[.19\linewidth]
	{
		\begin{tikzpicture}
			\draw (0,3) -- (1.2,3);
			\draw (0.2,3.1) -- (0,1.9);
			\draw (0,2.1) -- (0.2,0.9);
			\draw[thick] (0.2,1.1) -- (0,-0.1);
			\draw[dashed] (0,1.5) -- (1.2,1.7);
		\end{tikzpicture}
	}	
	\subcaptionbox{$\rA_1+\rA_5$ \label{fig:A1+A5}}
	[.19\linewidth]{
		\begin{tikzpicture}
			\draw (0,3) -- (1.2,3);
			\draw[thick] (0.2,3.1) -- (0,1.9);
			\draw[dashed] (0,2.1) -- (0.2,0.9);
			\draw (0.2,1.1) -- (0,-0.1);
			\draw (1.1,3.1) -- (0.9,1.9);
			\draw (0.9,2.1) -- (1.1,0.9);
			\draw[thick] (1.1,1.1) -- (0.9,-0.1);
			\draw[dashed] (0.9,1.5) -- (2.1,1.7);		
		\end{tikzpicture}	
	}
	\subcaptionbox{$\rA_7$ \label{fig:A7}}
	[.19\linewidth]{
		\begin{tikzpicture}
			\draw (0,3) -- (1.2,3);
			\draw[thick] (0.2,3.1) -- (0,1.9);
			\draw (0,2.1) -- (0.2,0.9);
			\draw (0.2,1.1) -- (0,-0.1);
			\draw (1.1,3.1) -- (0.9,1.9);
			\draw (0.9,2.1) -- (1.1,0.9);
			\draw[thick] (1.1,1.1) -- (0.9,-0.1);
			\draw[dashed] (0.9,1.5) -- (2.1,1.7);	
			\draw[dashed] (0.2,1.5) -- (-1,1.3);	
		\end{tikzpicture}	
	}	
	\subcaptionbox{$\rA_2+\rA_5$ \label{fig:A2+A5}}
	[.19\linewidth]{
		\begin{tikzpicture}
			\draw (0,3) -- (1.2,3);
			\draw (0,2.5) -- (0.2,1.7);
			\draw (0.2,1.9) -- (0,1.1);
			\draw (0,1.3) -- (0.2,0.5);
			\draw[dashed] (0.2,0.7) -- (0,-0.2);
			\draw (1.1,3.1) -- (0.9,2.1);
			\draw[dashed] (0.9,2.3) -- (1.1,1.3);
			\draw[thick] (1.1,1.5) -- (0.9,0.5);
			\draw[thick] (0,0.9) -- (1.2,0.9);
		\end{tikzpicture}	
	}	
	\subcaptionbox{$\rA_8$ \label{fig:A8}}
	[.19\linewidth]{
		\begin{tikzpicture}
			\draw (0,3) -- (1.2,3);
			\draw[dashed] (0.2,3.1) -- (0,2.3);
			\draw (0,2.5) -- (0.2,1.7);
			\draw (0.2,1.9) -- (0,1.1);
			\draw[thick] (0,1.3) -- (0.2,0.5);
			\draw[dashed] (0.2,0.7) -- (0,-0.2);
			\draw[thick] (1.1,3.1) -- (0.9,2.1);
			\draw (0.9,2.3) -- (1.1,1.3);
			\draw (1.1,1.5) -- (0.9,0.5);
			\draw[dashed] (0.9,1.8) -- (1.9,2);
			\draw (0,0.9) -- (1.2,0.9);
		\end{tikzpicture}	
	} \\
	\bigskip

	\subcaptionbox{$\rA_1+\rA_7$ \label{fig:A1+A7}}
	[.24\linewidth]{
		\begin{tikzpicture}
			\draw[thick] (0,3) -- (2.2,3);
			\draw (-0.2,0) -- (2,0);
			\draw[dashed] (0.2,3.1) -- (0,1.9);
			\draw[thick] (0,2.1) -- (0.2,0.9);
			\draw[dashed] (0.2,1.1) -- (0,-0.1);
			\draw (1.1,3.1) -- (0.9,1.9);
			\draw[dashed] (0.9,2.1) -- (1.1,0.9);
			\draw[thick] (1.1,1.1) -- (0.9,-0.1);
			\draw (2,3.1) -- (1.8,1.9);
			\draw (1.8,2.1) -- (2,0.9);
			\draw[thick] (2,1.1) -- (1.8,-0.1);
			\draw[dashed] (1.8,1.5) -- (3,1.7);
		\end{tikzpicture}	
	}	
	\subcaptionbox{$2\rA_4$ \label{fig:2A4}}
	[.24\linewidth]	{
		\begin{tikzpicture}
			\draw (0,3) -- (1.2,3);
			\draw (0.2,3.1) -- (0,1.9);
			\draw (0,2.1) -- (0.2,0.9);
			\draw[thick] (0.2,1.1) -- (0,-0.1);
			\draw[dashed] (0,1.5) -- (1.2,1.7);
			\draw[dashed] (1.6,-0.1)--(1.8,0.7);
			\draw (1.8,0.5) -- (1.6,1.3);
			\draw (1.6,1.1) -- (1.8,1.9);
			\draw (1.8,1.7) -- (1.6,2.5); 
			\draw[thick] (1.8,1.2) [partial ellipse= 90 : 270 : 1.1 and 1.1];
		\end{tikzpicture}
	}
	\subcaptionbox{$\rA_1+\rA_2+\rA_5$ \label{fig:A1+A2+A5}}
	[.24\linewidth]{
		\begin{tikzpicture}
			\draw[thick] (0,3) -- (2.2,3);
			\draw[dashed] (0.2,3.1) -- (0,2.1);
			\draw (0,2.3) -- (0.2,1.3);
			\draw (0.2,1.5) -- (0,0.5);
			\draw[thick] (1.1,3.1) -- (0.9,1.9);
			\draw[dashed] (0.9,2.1) -- (1.1,0.9);
			\draw (1.1,1.1) -- (0.9,-0.1);
			\draw[thick] (2,3.1) -- (1.8,1.9);
			\draw[dashed] (1.8,2.1) -- (2,0.9);
			\draw (2,1.1) -- (1.8,-0.1);
			\draw (0.5,0.2)-- (0.9,0.2) to[out=0,in=180] (2,2.7) -- (2.2,2.7);
		\end{tikzpicture}	
	}			
	\subcaptionbox{$4\rA_2$ \label{fig:4A2}}
	[.24\linewidth]{
		\begin{tikzpicture}
			\draw (0,3) -- (1.4,3);
			\draw[dashed] (0.2,3.1) -- (0,2.1);
			\draw[thick] (0,2.3) -- (0.2,1.4);
			\draw[thick] (0.2,1.6) -- (0,0.7);
			\draw[dashed] (0,0.9) -- (0.2,-0.1);
			\draw (1.3,3.1) -- (1.1,2.1);
			\draw (1.1,0.9) -- (1.3,-0.1);
			\draw (0,0) -- (1.4,0);
			\draw (-0.9,2.6) -- (-0.7,1.4);
			\draw[dashed] (-0.7,1.6) -- (-0.9,0.4);
			\draw[thick] (-0.3,1.5) [partial ellipse= 130 : -130 : 1.1 and 1.1];						
		\end{tikzpicture}
	}
	\caption{Del Pezzo surfaces of rank one with $\rA_{k}$ singularities, see Lemma \ref{lem:Furushima}}
	\label{fig:Furushima}
\end{figure}

\subsection{Double points on curves}\label{sec:delta-inv}

In Lemma \ref{lem:MMP}\ref{item:mult_2}, we will see that the singularities of $D_{\min}$ lie in $X_{\min}\reg$ and have multiplicity two. It will be convenient not to resolve them. Instead, we will use their $\delta$-invariants, cf.\   \cite[p.\ 151]{Wall_singular_curves}.
\smallskip

Let $p$ be a smooth point of a surface and let $\chi$ be a curve germ of multiplicity $2$ at $p$. Let $r_p$ be the number of branches of $\chi$ at $p$: we call $(p\in \chi)$ a \emph{cusp} if $r_p=1$ and \emph{node} if $r_p=2$. In some local coordinates $(x,y)$ at $p$, $\chi$ is given by $\{x^{2}=y^{2\delta_{p}-r_{p}+2}\}$ for some $\delta_{p}\geq 1$, where $\delta_{p}$ is the \emph{$\delta$-invariant} of $p\in \chi$. If $\delta_{p}=1$ then $p\in \chi$ is called \emph{ordinary} (it can be a cusp or a node). Note that $\delta_{p}-1$ is the minimal number of blowups at $p$ after which the the proper transform of $\chi$ is nc. For the minimal log resolution, one needs two additional blowups if $p\in \chi$ is a cusp and one if $p\in \chi$ is a non-ordinary node, cf.\ \cite[Figure 20]{Hartshorne_AG}. 

If $C$ is a curve whose all singular points have multiplicity two, we put $\delta_{C}=\sum_{p\in \Sing C}\delta_{p}$. For planar curves, we have the following observations, Lemma \ref{lem:P2-curves}\ref{item:P2-adjunction} being a consequence of the adjunction formula.

\begin{lem}[Planar curves with double points]\label{lem:P2-curves}
	Let $\cc\subseteq \P^{2}$ be a planar curve of degree $d$. Let $p\in\cc$ be a point of multiplicity two and let $\ll_{p}$ be a line tangent to $\cc$ at $p$. Then
	\begin{enumerate}
		\item\label{item:P2-tangent} $2<(\ll_{p}\cdot \cc)_{p}\leq d$ and either  $(\ll_{p}\cdot\cc)_{p}=2k$ for some $k\leq \delta_{p}$ or $(\ll_{p}\cdot\cc)_{p}\geq 2\delta_{p}+1$.
		\item\label{item:P2-adjunction} If $\cc$ is rational and all singularities of $\cc$ have multiplicity two then $\delta_{\cc}=\tfrac{1}{2}(d-1)(d-2)$.
	\end{enumerate}	
\end{lem}

\begin{lem}\label{lem:conics}
	Let $\cc_{1},\cc_{2}\subseteq \P^{2}$ be two conics and let $\ll$ be their common tangent line. 
	\begin{enumerate}
		\item \label{item:conics_22} If $(\cc_{1}\cdot \cc_{2})_{p}=(\cc_{1}\cdot \cc_{2})_{q}=2$ for some $p,q \in \P^{2}$, $p\neq q$, then $p\in \ll$ or $q\in \ll$.
		\item \label{item:conics_4} If $(\cc_{1}\cdot \cc_{2})_{p}=4$ for some $p\in \P^{2}$ then $p\in\ll$.
	\end{enumerate}
\end{lem}
\begin{proof}
	Suppose the contrary, so $\ll\cap \cc_{1}\neq \ll\cap \cc_{2}$. Let $\ll'$ be the line joining $p$ and $q$ in case \ref{item:conics_22} and the line tangent to $\cc_{1}$ at $p$ in case \ref{item:conics_4}. Then $2\ll'$ is a member of the pencil of conics generated by $\cc_{1}$ and $\cc_{2}$. This pencil induces a map $\ll\to \P^{1}$ of degree $2$, ramified at $\ll\cap\cc_{1}$, $\ll\cap \cc_{2}$ and $\ll\cap \ll'$. This is impossible by the Hurwitz formula; a contradiction.
\end{proof}

\subsection{Hirzebruch surfaces}\label{sec:Fm}

In order to settle the notation, we list some basic properties of Hirzebruch surfaces, cf.\ \cite[V.2]{Hartshorne_AG}.

	For $m\geq 0$ \emph{the $m$-th Hirzebruch surface} is $\F_{m}\de \P(\cO_{\P^{1}}(m)\oplus \cO_{\P^{1}})$. Assume $m>0$. Let $F$ be a fiber of the unique $\P^{1}$-fibration $\pi_{\F_{m}}\colon \F_{m}\to \P^{1}$. Then $\F_{m}$ contains a unique curve of negative self-intersection number, which we call the \emph{negative section} and denote by $\Sec_{m}$. It satisfies $\Sec_m\cong \P^{1}$, $\Sec_m\cdot F=1$ and $\Sec_m^{2}=-m$. In case $m=0$, we have $\F_{m}=\P^{1}\times \P^{1}$, and we let $F$, $\Sec_0$ be a fiber of each projection.
	
	For any $m\geq 0$, the group $\NS(\F_{m})$ is a free abelian group generated by the classes of $\Sec_m$ and a fiber $F$. A divisor $G$ on $\F_{m}$ is said to be \emph{of type} $(a,b)$ for some $a,b\in \Z$ if $G\cdot \Sec_m=a$ and $G\cdot F=b$. In this case, $G\equiv (a+bm)F+b\Sec_{m}$, and $G$ is effective if and only if $G=\Sec_m$ or $a,b\geq 0$. The canonical divisor $K_{\F_{m}}$ is of type $(m-2,-2)$. It follows from the adjunction formula that any divisor of type $(0,1)$ is an $m$-curve. We call it a \emph{positive section} of $\F_{m}$.

We have the following consequence of the adjunction formula, similar to Lemma \ref{lem:P2-curves}\ref{item:P2-adjunction}.

\begin{lem}[Curves with double points on $\F_{m}$]\label{lem:F2-adjunction} 
	If $C\subseteq \F_{m}$ is a rational curve of type $(a,b)$ and all singularities of $C$ have multiplicity $2$ then  $\delta_{C}=(b-1)(a+\tfrac{m}{2}b-1)$.
\end{lem}

We now introduce some additional notation for $m=2$. It will be useful in the case when $X_{\min}$ is the quadric cone $\P(1,1,2)$, i.e.\ the image of $\F_2$ after contracting $\Sec_2$.

 Fix $p\in \F_{2}\setminus \Sec_2$, and let $F_p$ be the fiber containing $p$. Define a birational map $\sigma_{p}\colon \F_{2}\map \P^{2}$ as a blowup at $p$ followed by the contraction of the proper transform of $F_{p}+\Sec_2$. Write $\Bs\sigma_{p}^{-1}=\{\hat{p}\}$, so $\hat{p}=\sigma(\Sec_2)$. The image of the exceptional curve is a line, we denote it by $\hat\ll_{p}$. For every line $\ll$ not containing $\hat{p}$, its preimage $(\sigma_{p}^{-1})_{*}\ll$ is a positive section passing through $p$. Similarly, for every positive section $H$ not containing $p$, its image $(\sigma_{p})_{*}H$ is a conic tangent to $\hat{\ll}_{p}$ at $\hat{p}$.

	Using the map $\sigma_{p}$, it is easy to see that for any two points $p,q\in \F_{2}\setminus\Sec_2$ there is a pencil of positive sections passing through $p$, $q$. Moreover, if $(p,\chi)$ is a cusp (i.e.\ an irreducible analytic curve germ) of multiplicity sequence $(\mu_{1},\mu_{2},\dots)$ such that $(\chi\cdot F_{p}) _{p}=\mu_{1}$, then there is a unique positive section $L_{\chi}$ such that $(\chi\cdot L_{\chi})_{p}\geq \mu_{1}+\mu_{2}+\mu_{3}$; namely the proper transform of the line tangent to $(\sigma_{p})_{*}\chi$. Such $L_{\chi}$ will be called a \emph{tangent section} to $\chi$.

	\section{Expansions and almost minimal models}\label{sec:expansions}
	
		In this section, we list some useful properties of expansions. In particular, in Lemma \ref{lem:expansions} we show how to check when the pair $(X,D)$ constructed by an expansion is a log smooth completion of a $\Q$HP. In Lemma \ref{lem:MMP}, we relate expansions with an almost minimalization of $(X,\tfrac{1}{2}D)$ in the sense of \cite{Palka_MMP}.
	
		Recall from Definition \ref{def:expansion-intro} that an \emph{expansion} with centers $(x_i;U_i,W_i)$ and weights $v_{i}\in \Q_{>0}$, $i=1,\dots,n$ is a birational morphism $\varphi\colon (X,D)\to (X',D')$ with $\Bs\varphi^{-1}=\{x_1,\dots,x_n\}$, such that $\varphi^{-1}(x_i)$ is a chain determined by $v_i$. We will usually keep the choice of $x_i\in U_i\cap W_i$ implicit; i.e.\ we will call $\phi$ an expansion \emph{with centers $(U_i,W_i)$ and weights $v_i$}, or simply \emph{at} $(U_i,W_i;v_i)$. Note, however, that the intersection $U_i\cap W_i$ can contain more than one center of expansion. In this case, the above list will include $(U_i,W_i)$ more than once.
		
	We begin with a simple remark, which follows directly from Definition \ref{def:expansion-intro}. 
		
			\begin{rem}
				\label{rem:expansion} Let $\phi$ be an expansion with one center $(x;U,W)$ and weight $v=\frac{u}{w}$ for some coprime $u,w\in \N$.
				\begin{enumerate}
					\item  $\Exc\phi=[W',1,U']$. Its first tip meets $\phi^{-1}_{*}U$, last tip meets $\phi^{-1}_{*}W$, and  $d(U')=u$, $d(W')=w$. 
					\item\label{item:v_integer} If $v=1$ then  $\phi$ is a single blowup. If $v\in \N$ then $\Exc\phi=[1,(2)_{v-1}]$ and its first tip meets $\phi^{-1}_{*}U$.
					\item\label{item:1/v} The pair $((x;U,W),v)$ determines the morphism $\phi$ uniquely. On the other hand, pairs $((x;U,W),v)$ and $((x;W,U),v^{-1})$ define the same expansion.
					\item Let $\chi$ be a smooth germ meeting the exceptional $(-1)$-curve of $\phi$ at a general point, transversally. Then $x\in \phi_{*}\chi$ is a cusp described in coordinates $U$, $W$ by a single HN-pair $\binom{u}{w}$ cf.\ \cite[\S 2D]{PaPe_Cstst-fibrations_singularities}.
				\end{enumerate}
			\end{rem}
	
	\begin{lem}\label{lem:sequence_of_expansions}
		Let $\phi_{i}\colon (X_{i},D_{i})\to (X_{i+1},D_{i+1})$ be an expansion, $i\in \{1,\dots,n\}$. If $D_{1}$ is connected, then $\phi_{n}\circ\dots\circ \phi_{1}\colon (X_1,D_1)\to(X_{n+1},D_{n+1})$ is an expansion, too.
	\end{lem}	
	\begin{proof}
		Since $D_1$ is connected, so are its images $D_{i}$. Hence by induction, we can assume $n=2$ and that $\phi_{1}$ has only one center, say $p$. Put $E_{i}=\Exc\phi_{i}$, $i\in \{1,2\}$. If $p\not\in E_2$ then $\Exc \phi_2\circ\phi_1=E_1+(\phi_{1}^{-1})_{*}E_2$, so $\phi_2\circ\phi_1$ is an expansion. Suppose $p\in E_2$. Since $p$ is a center of $\phi_1$, $p\in \Sing D_2$. Since $E$ is contained in the sum of twigs of $D_2$, we have $\{p\}=T\cap (D_2-T)$ for some twig $T$ of $D_2$. Then $(\phi_{1}^{-1})_{*}T$ and $(\phi_{1}^{-1})_{*}(D_2-T)$ lie in different connected components of $D_1$, a contradiction.
	\end{proof}
	
	Recall that expansions with fixed centers and varying weights produce a \emph{tower} of pairs $(X,D)$, hence of surfaces $S\de X\setminus D$. 
	The following Lemma \ref{lem:expansions}, which is well known, gives a practical way to decide if a surface $S$ in the tower is a \QHP, and shows that in this case \enquote{most} of them are. For an explicit application of this criterion see Example \ref{ex:7}.
	
	For the readers' convenience, we give a full proof of Lemma \ref{lem:expansions}, mostly following \cite[Theorems B, 3.13]{DiPe-hp_and_alg_curves}. Similar results appear e.g.\ in \cite{MiySu-Cstst_fibrations_on_Qhp}, \cite{DiPe-hp_announcement_and_survey}, \cite[1.10]{GKMR-singular_Zhp} or  \cite[3.1(vii)]{Palka-classification1_Qhp}, see also   \cite{Mumford-surface_singularities}.
	\smallskip
	
	To formulate this result, we introduce the following notation. For a divisor $W$ on a surface $Y$, we denote by $\Z[W]$  the free abelian group generated by the components of $W\redd$. It admits a natural map $r_{W}\colon \Z[W]\to \NS(Y)$, and, for any subdivisor $V$ of $W$, a projection $\pr_{V}\colon \Z[W]\to \Z[V]$.

	\begin{lem}
			\label{lem:expansions}
	Let $(X',D')$ be a log smooth pair. Let $\phi\colon (X,D)\to (X',D')$ be an expansion with $n$ centers, and let  $A=(\phi^{*}D')\redd-D\subseteq \Exc\phi$. Define a $\Z$-linear map
	\begin{equation}\label{eq:m}
		m=\pr_{A}\circ \phi^{*} \colon \ker r_{D'}\to \Z[A].
	\end{equation}
	The surface $S\de X\setminus D$ is a \QHP if and only if the following conditions hold:
	\begin{enumerate}
		\item\label{item:D'_spans} $X'$ is rational and the components of $D'$ generate $\NS_{\Q}(X')$,
		\item\label{item:D_rat-tree} $D$ is a rational tree, or equivalently: all components of $D'$ are rational, and the dual graph of $D'$ becomes a tree after removing the edges corresponding to the centers of $\phi$.
		\item\label{item:expansions_det} $m\otimes_{Z}\Q$ is an isomorphism, or equivalently: 
$
	n=\#D'-\rho(X') 
$ and 
$
\det m\neq 0
$, 
where we view $m$ as an element of $\operatorname{End}(\Z^n)$.		
	\end{enumerate}
	Furthermore, if $S$ is a \QHP then $H_{i}(S;\Z)=0$ for $i\geq 2$, and
	\begin{equation}\label{eq:expansion_H1}
		\#H_{1}(S;\, \Z)=|d(D)|^{1/2}=|\!\det m| \cdot \#\coker r_{D'}.
	\end{equation} 
	\end{lem}	
\begin{proof}
	Put $S'=X'\setminus D'$, $E=\Exc\phi$. By the universal coefficient theorem, $S$ is a $\Q$HP if and only if $\tilde{H}^{*}(S;\Q)=0$, in which case $H_{1}(S)\cong H^{2}(S)$; where we skip the coefficient ring if it is $\Z$.
	
	Smooth $\Q$HPs are rational by \cite{GuPrad-rationality_3(smooth_II)}, so we can assume that $X'$ and $X$ are rational. Hence $H_{1}(X)=H_{3}(X)=0$, and by the exponential sequence, we get an isomorphism $\NS(X)=H_{2}(X)$, mapping a curve to its fundamental class. By Mayer-Vietoris, $H_{2}(D)=\Z[D]$. 
	By Lefschetz duality $H_{i}(X,D)=H^{4-i}(S)$. The homology exact sequence of the pair $(X,D)$ gives $H^{3}(S)=\tilde{H}_{0}(D)$,  $H^{i}(S)=0$ for $i\geq 4$, and 
 an exact sequence
	\begin{equation}\label{eq:pair}
		\begin{tikzcd}
		0\ar[r] & H^{1}(S) \ar[r] & \Z[D]\ar[r, "r_{D}"] & \NS(X)\ar[r] & H^{2}(S) \ar[r] & H_{1}(D)\ar[r] & 0.
		\end{tikzcd}
	\end{equation}
	 
	Assume that $S$ is a $\Q$HP. Then the sequence \eqref{eq:pair} shows that $\Z[D]$ is a subgroup of $\NS(X)$, of finite index $\#H^{2}(S)$. The intersection product, which is a perfect pairing on $X$, has discriminant $|d(D)|$ when restricted to $\Z[D]$, so $\#H^{2}(S)=|d(D)|^{1/2}$, which proves the first equality in the formula \eqref{eq:expansion_H1}. Moreover, the sequence \eqref{eq:pair} gives $H_{1}(D;\Q)=\tilde{H}_{0}(D;\Q)=0$, i.e.\ $D$ is a rational tree. So if $S$ is a $\Q$HP then condition \ref{item:D_rat-tree} holds.
	
	Assume now that \ref{item:D_rat-tree} holds, i.e.\ $D$ is a rational tree. Then $H_{1}(D)=0$. By induction on the Picard rank $\rho(\phi)\de \rho(X)-\rho(X')$, it is easy to see that the map $(\phi_{*},\pr)$ gives isomorphisms $\NS(X)=\NS(X')\oplus \Z[E]$ and $\Z[D]=\Z[D']\oplus \Z[E-A]$. Thus we get a commutative diagram with exact rows 
	
	\begin{equation*}
		\begin{tikzcd}
		0 \ar[r] & \Z[E-A]\ar[r] \ar[d] &  \Z[D] \ar[r, "\phi_{*}"] \ar[d, "r_{D}"] & \Z[D'] \ar[r] \ar[d, "r_{D'}"] & 0 \\
		0 \ar[r] & \Z[E]\ar[r] &  \NS(X) \ar[r, "\phi_{*}"] & \NS(X') \ar[r] & 0 
		\end{tikzcd}
	\end{equation*}
  	The exact sequence \eqref{eq:pair} shows that $\ker r_{D}=H^{1}(S)$. Since $H_{1}(D)=0$, we get $\coker r_{D}=H^{2}(S)$. Snake lemma yields an exact sequence 
	\begin{equation*}
		\begin{tikzcd}
		0 \ar[r] & H^{1}(S) \ar[r] & \ker r_{D'} \ar[r, "m"] & \Z[A] \ar[r] & H^{2}(S) \ar[r] & \coker r_{D'} \ar[r] & 0.
		\end{tikzcd}
	\end{equation*}
	A diagram chase shows that the connecting map is indeed $m$. We conclude that, if \ref{item:D_rat-tree} holds, $S$ is a $\Q$HP if and only if $\coker r_{D'}\otimes_{\Z}\Q=0$ and $m\otimes_{\Z}\Q$ is an isomorphism, i.e.\ \ref{item:D'_spans} and \ref{item:expansions_det} hold. Under these conditions, we get  $\#H_{1}(S)=\#H^{2}(S)=\#\coker m \cdot \#\coker r_{D'} =|\! \det m | \cdot \#\coker r_{D'}$, i.e.\ the second equality in \eqref{eq:expansion_H1}. 
\end{proof}

\begin{rem}[{cf.\ \cite[III.4.2.1]{Miyan-OpenSurf}}]\label{rem:D_rat-tree}
	Let $(X,D)$ be a log smooth pair. Lemma \ref{lem:expansions} applied to $\phi=\id$ shows that $X\setminus D$ is a $\Q$HP if and only if $X$ is rational, $D$ is a rational tree, and components of $D$ freely  generate $\NS_{\Q}(X)$. 
	
	Indeed, in this case condition \ref{item:D'_spans} asserts that $X$ is rational and the components of $D$ generate $\NS_{\Q}(X)$, and condition \ref{item:D_rat-tree} asserts that $D$ is a rational tree. To interpret condition \ref{item:expansions_det}, note that since $A=0$, the group $\Z[A]$ is trivial, so the map $m\otimes_{\Z} \Q\colon \ker r_{D}\otimes_{\Z}\Q\to \Z[A]\otimes_{\Z}\Q$ is an isomorphism if and only if $\ker r_{D} \otimes \Q$ is trivial, too. Thus condition \ref{item:expansions_det} asserts that the components of $D$ are linearly independent in $\NS_{\Q}(X)$, as claimed.
	
	Moreover, if $S$ is a $\Q$HP with a log smooth completion $(X,D)$, then  $H_{i}(S;\Z)=0$ for $i\geq 2$, and $\# H_{1}(S;\Z)=|d(D)|^{1/2}$ by formula \eqref{eq:expansion_H1}. Thus if $|d(D)|=1$, i.e.\ if $S$ is a $\Z$HP, then by Whitehead theorem $S$ is contractible if and only if $\pi_{1}(S)=\{1\}$. Explicit computations of $\pi_{1}(S)$ can be complicated, see e.g.\ \cite{Aguilar} for $\Q$HPs obtained from arrangements of lines.
\end{rem}

\begin{rem}\label{rem:m_comuptation}
	Let us show how to compute the map $m$ from the formula \eqref{eq:m}, assuming that the condition  \ref{lem:expansions}\ref{item:D'_spans} holds. Choose a basis $(\mathcal{R}_{j})$ of $\ker r_{D'}\subseteq \Z[D']$, and for a component $C$ of $D'$, let $c_{j}(C)$ be the coefficient of $C$ in $\mathcal{R}_{j}$. Say that $A_{i}$ is the exceptional $(-1)$-curve of an expansion at  $(U_{i},W_{i};u_i/w_i)$, $\gcd(u_i,w_i)=1$. Then by Definition \ref{def:expansion-intro}, the matrix of $m$ with respect to bases $(\mathcal{R}_{j})$ and $(A_{i})$ is $[m_{ij}]$, where 
	\begin{equation*}
		m_{ij}=u_{i}\cdot c_{j}(U_{i})+w_{i}\cdot c_{j}(W_{i}).
	\end{equation*}
	 In particular, $m(C)=0$ for every component $C$ of $D'$ which does not pass through a center.
\end{rem}

\begin{rem}In Section \ref{sec:classification}, the pair $(X',D')$ will be a log resolution of $(X_{\min},D_{\min})$, where $X_{\min}$ is a del Pezzo surface, hence it is rational; $\rho(X_{\min})=1$ and $D_{\min}\neq 0$; see Lemma \ref{lem:R}\ref{item:MFS_point}. Thus in this case, the map $r_{D_{\min}}\otimes_{\Z} \Q$ is surjective, which implies  \ref{lem:expansions}\ref{item:D'_spans}.
	
In fact, we will see that if $n=\#D_{\min}-1\geq 1$, the components of $D_{\min}$ generate $\NS(X_{\min})$, hence $r_{D'}$ is surjective over $\Z$, too. Formula \eqref{eq:expansion_H1} in this case gives  $\#H_{1}(X;S)=|\! \det m|$. In the $\C^{**}$-fibered case, treated in Section \ref{sec:Cstst}, the map $r_{D'}$ will turn out to be surjective, too, except for the tower \ref{def:C**_3}, where $\#\coker r_{D'}=3$.
\end{rem}

Formula \eqref{eq:expansion_H1} shows that $\#H_{1}(\, \cdot \, ;\Z)$ varies a lot within a tower. More precisely, we have the following.

\begin{lem}[{cf.\ \cite[p.\ 859]{DiPe-hp_and_alg_curves}}]\label{lem:ZHP}
Fix a log smooth pair $(X',D')$ and $n\geq 1$ centers on $D'$. Let $\cT$ be the corresponding tower. Assume that for some $(X,D)\in\mathcal{T}$, the surface $S\de X\setminus D$ is a \ZHP. Then for every $k\in \Z_{\geq 1}$, $\mathcal{T}$ contains infinitely many pairs $(X,D)$ such that $S$ is a $\Q$HP with $\#H_{1}(S;\Z)=k$.
\end{lem}
\begin{proof}
	Fixing all but one centers and weights, we can assume $n=1$. Remark \ref{rem:m_comuptation} shows that there are $a,b\in \Z$ such that a surface $S$ obtained by expansion with weight $u/w$, $\gcd(u,w)=1$, satisfies $\#H_{1}(S)=au+bw$. Since the tower $\cT$ contains a $\Z$-homology plane,  there are positive integers $u_0,w_0$ such that $au_0+bw_0=1$. It follows that, say, $a>0$ and $b\leq 0$. Fix positive, coprime integers $k,l\in \N$ and put $u=ku_0-lb$, $w=kw_0+la$. Then $u,w> 0$; $au+bw=k$ and $\gcd(u,w)=1$. To see the last statement, note that $\gcd(u,w)|k$ since $au+bw=k$, so $\gcd(u,w)|l$ by definition of $u,w$, and therefore  $\gcd(u,w)=1$ since $l$ was chosen coprime to $k$, as claimed. We conclude that an expansion with weight $u/w$ gives a $\Q$HP $S$ with $\#H_{1}(S)=k$. Varying $l$ we get infinitely many such  weights. By Remark \ref{rem:expansion}\ref{item:1/v}, they give non-isomorphic pairs $(X,D)$.
\end{proof}

\begin{rem} If the surfaces $S=X\setminus D$ in Lemma \ref{lem:ZHP} are of log general type, then by Lemma \ref{lem:unique_completion} each of them uniquely determines, up to an isomorphism, its minimal log smooth completion $(X,D)$. Hence for each $k\in \Z_{\geq 1}$, we get infinitely many pairwise non-isomorphic $\Q$HPs $S$ with $\#H_{1}(S;\Z)=k$.
\end{rem}

In order to view a pair $(X,D)$ as an expansion $(X,D)\to (X',D')$, we need find a $(-1)$-curve $A\subseteq X$ which can serve as the exceptional curve of the first blowdown. Such $A$ are distinguished as follows.

\begin{dfn}[Bubble]\label{def:bubble}
	Let $D$ be a connected reduced divisor on a smooth projective surface $X$. We say  that a $(-1)$-curve $A\subseteq X$ is a \emph{bubble}  on $(X,D)$ if
	\begin{equation*}
		A\not\subseteq D,\quad A\cdot D=2\quad \mbox{and }A \mbox{ meets }D\mbox{ in two points, lying on different components of }D.
	\end{equation*}
	In particular, if $A$ is a bubble then $A$ is a superfluous $(-1)$-curve in $A+D$. 
\end{dfn}

\begin{figure}[htbp]
	\begin{tikzpicture}
	\draw (0,0) -- (1.2,0.2);
	\draw (1,0.2) -- (2.2,0);
	\draw (2,0) -- (3.2,0.2);
	\draw[dashed] (3,0.2) -- (4.6,0);
	\draw (4.4,0) -- (5.6,0.2);
	\draw (5.4,0.2) -- (6.6,0);
	\node at (0.6,0.3) {\small{$U$}};
	\node at (1.4,-0.1) {\small{$-3$}};
	\node at (2.6,-0.1) {\small{$-2$}};
	\node at (3.8,0.3) {\small{$A$}};			
	\node at (3.7,-0.1) {\small{$-1$}};	
	\node at (5,-0.1) {\small{$-3$}};
	\node at (6.1,0.3) {\small{$V$}};
	\draw (-0.2,0) [partial ellipse= 30 : 350 : .6 and .3];	
	\draw (6.8,0) [partial ellipse= 150 : -170 : .6 and .3];
	\draw[->] (8,0)	-- (8.6,0);
	\draw[name path = U] (10,0) -- (11.2,0.2);
	\draw[name path = V] (11,0.2) -- (12.2,0);
	\path [name intersections={of=U and V,by=E}];
	\filldraw (E) circle (1.5pt);
	\node at (10.6,0.3) {\small{$U$}};
	\node at (11.7,0.3) {\small{$V$}};
	\draw (9.8,0) [partial ellipse= 30 : 350 : .6 and .3];	
	\draw (12.4,0) [partial ellipse= 150 : -170 : .6 and .3];													
\end{tikzpicture}
	\caption{A bubble $A$ and an expansion with weight $\frac{3}{5}$ and center $(U,V)$.}
	\label{fig:exp}
\end{figure}
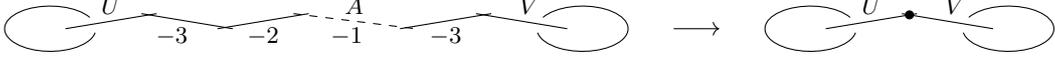	

Assume that some log smooth completion $(\tilde{X},\tilde{D})$ of $S$ is obtained by an expansion $\tilde{\psi}\colon (\tilde{X},\tilde{D})\to (X',D')$. The following observation shows that, in case which is of most interest for us, $\tilde{\psi}$ factors through the \emph{minimal} log smooth completion of $S$. Therefore, even if we allowed arbitrary log smooth completions in Theorem \ref{CLASS}, we would not get a shorter list. 

\begin{lem}\label{lem:other-completion}
	Let $S$ be a smooth surface of log general type which admits no $\C^{**}$-fibration. Let $\tau\colon (\tilde{X},\tilde{D})\to(X,D)$ be a birational morphism between log smooth completions of $S$; and let $\tilde{A}$ be a bubble on $(\tilde{X},\tilde{D})$. Then $\tau(\tilde{A})$ is a bubble on $(X,D)$.
\end{lem}
\begin{proof}
	Suppose the contrary. Then $\tau$ is not an isomorphism in any neighborhood of $\tilde{A}$. By induction on $\rho(\tau)$, we can assume that $\tilde{A}$ meets a $(-1)$-curve $U\subseteq \Exc \tau$, so $\tilde{A}\cdot U=1$ by Definition \ref{def:bubble}. Since $D=\tau_{*}\tilde{D}$ is snc, $U$ is superfluous in $D$. The linear system of $U+\tilde{A}=[1,1]$ induces a $\P^1$-fibration $\pi$ of $\tilde{X}$ whose fiber $F$ satisfies $F\cdot \tilde{D}=(U+\tilde{A})\cdot \tilde{D}=\beta_{\tilde{D}}(U)+U^2+\tilde{A}\cdot \tilde{D}\leq 2-1+2=3$. Since $\kappa(S)=2$, the Iitaka easy addition theorem implies that $F\cdot \tilde{D}=3$, so $\pi|_{S}$ is a $\C^{**}$-fibration, a contradiction.
\end{proof}

We explain the effect of expansions on Kodaira dimensions. Lemma \ref{lem:kappa_expansion}\ref{item:kappa_1/2_expansion} shows that $\kappa(K_{X}+\tfrac{1}{2}D)$ is preserved. In fact, by \cite[Lemma 5.9(2),(4)]{Palka_MMP} a bubble $A$ on $(X,D)$ is,  usually,  almost log exceptional for $(X,\tfrac{1}{2}D)$, i.e.\ some MMP run for $(X,\tfrac{1}{2}D)$ contracts $A$, see  \cite[Definitions 3.6, 3.17]{Palka_MMP}. The very design of MMP requires $\kappa(K_{X}+\tfrac{1}{2}D)$ to be preserved, see \cite[Corollary 3.2, Remark 3.16]{Palka_MMP}. 

On the other hand, MMP for $(X,D)$ does not see $A$: in fact, by Lemma \ref{lem:no_lines} and \cite[Lemmas 4.3, 4.5]{Palka_MMP} it contracts only twigs of $D$. And indeed, expansions may change $\kappa(K_{X}+D)$. However, Lemma \ref{lem:kappa_expansion}\ref{item:kappa_expansion} shows that if some expansion produces a surface of log general type, then all further ones do so, too. In particular, in order to verify that $\Q$HPs in Theorem \ref{CLASS} are of log general type, which we do in Proposition \ref{prop:k=2}, we only need to consider expansions of low Picard rank; usually one.

\begin{lem}[$\kappa$ does not drop within a tower]\label{lem:kappa_expansion}
	Let $\phi\colon(X,D)\to (X',D')$ be an expansion. Then
	\begin{enumerate}
		\item\label{item:kappa_1/2_expansion} $\kappa(K_{X}+\tfrac{1}{2}D)=\kappa(K_{X'}+\tfrac{1}{2}D')$,
		\item\label{item:kappa_expansion} Assume that $\tilde{\phi}\colon (\tilde{X},\tilde{D})\to (X',D')$ is an expansion with the same centers as $\phi$, such that $\tilde\phi=\phi\circ\eta$ for some birational morphism $\eta$. Then $\kappa(\tilde{X}\setminus \tilde{D})\geq \kappa(X\setminus D)$.
	\end{enumerate}
\end{lem}
\begin{proof}
	By induction, we can assume that $\phi$ has only one center. Let $A\subseteq \Exc\phi$ be the bubble.
	
	\ref{item:kappa_1/2_expansion} For any integers $a,b\geq 1$, we have $A\cdot (a(2K_{X}+D)+bA)=-b<0$, so $A$ is in the fixed part of $|a(2K_{X}+D)+bA|$. By induction on $b$, we conclude that $|a(2K_{X}+D)|$ and $|a(2K_{X}+D+A)|$ define the same map, which contracts $A$. Hence  $\kappa(K_{X}+\tfrac{1}{2}D)=\kappa(K_{X}+\tfrac{1}{2}(D+A))=\kappa(K_{X'}+\tfrac{1}{2}D')$ by induction on $\rho(\phi)$.
	
	\ref{item:kappa_expansion} Put $E=\Exc\eta$. By induction, we can assume that $E=[1]$. The point $\eta(E)$ lies in $A\cap D\subseteq D\reg$, so $K_{\tilde{X}}+\tilde{D}=(\eta^{*}K_{X}+E)+\eta^{-1}_{*}(D+A)=\eta^{*}(K_{X}+D)+\eta^{-1}_{*}A$, which gives an injective map $|a(K_{X}+D)|\into |a(K_{\tilde{X}}+\tilde{D})|$ for all $a\geq 1$.
\end{proof}	

	We will now show how the theory of almost minimal models \cite{Palka_MMP} recovers all $\Q$HPs by expansions from particularly simple pairs. The remaining task will be to relate those pairs with planar configurations.
	\smallskip
	
	Let $(X,D)$ be the minimal log smooth completion of a $\Q$HP of log general type. Let $A$ be a bubble on $(X,D)$. Contract $A$ and all superfluous $(-1)$-curves in the subsequent images of $D$, passing through the image of $A$; and repeat this until the image of $(X,D)$ has no bubbles. The process must end since $\rho(X)$ drops at each step. Since $D$ is connected,  Lemma \ref{lem:sequence_of_expansions} implies that the resulting morphism 
	\begin{equation*}
		\psi\colon (X,D)\to (X',D')
	\end{equation*}
	is an expansion. Let $\alpha'\colon (X',\tfrac{1}{2}D')\to (X_{\min},\tfrac{1}{2}D_{\min})$ be a log MMP run. Let $\upsilon \colon X'\to X\am$ be its \emph{almost minimalization} \cite[Definition 3.1]{Palka_MMP}, that is,  a $K_{X'}$-MMP run over $X_{\min}$. Then $\alpha'=\alpha\circ \upsilon$ for some morphism $\alpha\colon X\am\to X_{\min}$. The characterization of (almost) log exceptional curves on $(X,\tfrac{1}{2}D)$, given in \cite[\S 5]{Palka_MMP} implies the following.
	
\begin{lem}\label{lem:MMP}
	Let $(X,D)$ be the minimal log smooth completion of a $\Q$HP $S$ of log general type. Let
	\begin{equation}\label{eq:MMP}
	\begin{tikzcd}
		(X,D) \ar[r, "\psi"] & (X',D')\ar[r, "\upsilon"] &(X\am,D\am)\ar[r, "\alpha"] &(X_{\min},D_{\min})
	\end{tikzcd}
	\end{equation}
	be as above. Let $\Delta\am$ be the sum of all maximal $(-2)$-twigs of $D\am$, and $R\am=D\am-\Delta\am$. Then 
	\begin{enumerate}
	\item\label{item:Si_affine} $S'\de X'\setminus D'$ is an affine open subset of $S$,
	\item\label{item:Si_no-lines} $\kappa(S')=2$ and $S'$ contains no topologically contractible curves,	
	\item\label{item:alpha_peeling} $\Exc\alpha=\Delta\am$ and $\alpha^{*}(2K_{X_{\min}}+D_{\min})=2K_{X\am}+D\am-\Bk\Delta\am$, 	
	\item\label{item:smoothness} $X\am$ is smooth, and each singularity of $X_{\min}$ is  canonical cyclic (type $\rA_{k}$),
	\item\label{item:D_squeezed} $\Exc\upsilon\subseteq D'$, and every $(-1)$-curve $U$ in $D\am$ satisfies $\beta_{D\am}(U)\geq 4$ or $\beta_{R\am}(U)\geq 3$,
	\item\label{item:no_bubbles} There are no bubbles on $(X',D')$ or $(X\am,D\am)$,
	\item\label{item:mult_2} All singularities of $D\am$ and $D_{\min}$ are of multiplicity $2$; and $\Sing D_{\min}\subseteq X_{\min}\reg$.
	\end{enumerate}
\end{lem}
\begin{proof}
	\ref{item:Si_affine} By \cite[2.5]{Fujita-noncomplete_surfaces}, $S$ is affine. Since $S'$ equals $S$ minus a sum of bubbles, $S'$ is affine, too. 
	
	\ref{item:Si_no-lines} This follows from \ref{item:Si_affine} and Lemma \ref{lem:no_lines}.

	\ref{item:alpha_peeling} By \cite[Definition 3.10]{Palka_MMP}, $\alpha$ is a pure peeling, so \ref{item:alpha_peeling} follows from \cite[Corrolary 5.5]{Palka_MMP}. Indeed, by \ref{item:Si_affine} $D\am$ has no negative definite connected components, so the divisor $\Gamma+\Lambda$ defined in loc.\ cit.\ is zero in this case.
	
	\ref{item:smoothness} By construction, the surface $X'$ is smooth, so the surface $X\am$, which is an outcome of an  $K_{X'}$-MMP run is smooth, too. By \ref{item:alpha_peeling}, all singularities of $X_{\min}$ are images of $(-2)$-chains in $\Delta\am$, so they are of type $\rA_{k}$.
	
	\ref{item:D_squeezed}, \ref{item:no_bubbles} Let $\bar{\upsilon}\colon (X',\tfrac{1}{2}D')\to (\bar{X}\am,\tfrac{1}{2}\bar{D}\am)$ be an $\alpha'$-squeezing, that, is, a contraction of all $(-1)$-curves in $D'$ and its images which are contracted by $\alpha'=\alpha\circ\upsilon$. By construction, $(X',D')$ has no bubbles, so $(\bar{X}\am,\bar{D}\am)$ has no bubbles either. Indeed, suppose $\bar{A}$ is a bubble on $(\bar{X}\am,\bar{D}\am)$. Since $D'$ is snc-minimal and has no negative definite connected components, we have  $\Bs\bar{\upsilon}^{-1}\subseteq \Sing \bar{D}\am$. In particular, $\bar{\upsilon}$ is an isomorphism in a neighborhood of $A'\de \bar{\upsilon}^{-1}_{*}A$. Thus $A'$ is a bubble on $(X',D')$; a contradiction.
	
	Suppose $\bar{\upsilon} \neq \upsilon$. Then $(\bar{X}\am,\tfrac{1}{2}\bar{D}\am)$ contains an almost log exceptional curve $A$. By \cite[Lemma 5.9]{Palka_MMP}, $A$ is a $(-1)$-curve such that either $A\cdot \bar{D}\am\leq 1$, which is impossible by \ref{item:Si_no-lines}, or $A$ is a bubble on $(\bar{X}\am,\bar{D}\am)$, which is impossible by the above claim. Thus $\bar{\upsilon}=\upsilon$, so $(X\am,D\am)$ has no bubbles and $\Exc\upsilon\subseteq D'$. The second part of \ref{item:D_squeezed} follows from the characterization of squeezing in \cite[Lemma 5.2]{Palka_MMP}.
	
	\ref{item:mult_2} This follows from \cite[Remark 5.7]{Palka_MMP}.
\end{proof}

In our case, minimality of $(X_{\min},\tfrac{1}{2}D_{\min})$ implies the following.
	
\begin{lem}\label{lem:R}
	We keep the notation and assumptions from Lemma \ref{lem:MMP}. Assume further that $S$ satisfies the Negativity Conjecture \ref{conj:negativity} and admits no $\C^{**}$-fibration. Then
	\begin{enumerate}
		\item\label{item:MFS_point} $\rho(X_{\min})=1$ and $-(K_{X_{\min}}+\tfrac{1}{2}D_{\min})$ is ample.
		\item\label{item:Xmin-del-Pezzo} $X_{\min}$ is a del Pezzo surface of rank one, with only $\rA_{k}$-type singularities.
		\item\label{item:R}	$\#R\am=n+1$, and $R\am$ has exactly $2n$ nodes.	
	\end{enumerate}
\end{lem}
\begin{proof}
	\ref{item:MFS_point} By assumption, $\kappa(K_{X}+\tfrac{1}{2}D)=-\infty$, so $\kappa(K_{X'}+\tfrac{1}{2}D')=-\infty$ by Lemma \ref{lem:kappa_expansion}\ref{item:kappa_1/2_expansion}. Log abundance in dimension two, cf.\ \cite[p.\ 127]{Matsuki_MMP_intro}, shows that $K_{X_{\min}}+\tfrac{1}{2}D_{\min}$ is not nef. So there is a curve $F\subseteq X_{\min}$ such that $F\cdot (2K_{X_{\min}}+D_{\min})<0$. By the contraction theorem \cite[1-3-5]{Matsuki_MMP_intro}, $F$ is contracted by some morphism $\phi\colon X_{\min}\to Z$ with $\rho(\phi)=1$. Because $(X_{\min},\tfrac{1}{2}D_{\min})$ is minimal, $F^2\geq 0$, so $\dim Z\leq 1$. Suppose $\dim Z=1$. Then $F^2=0$, so $F\cdot (2K_{X_{\min}}+D_{\min})<0$ implies that $F\cong \P^1$ and $F\cdot D_{\min}\leq 3$. Since $\kappa(X_{\min}\setminus D_{\min})=2$, we have $F\cdot D_{\min}=3$ by the Iitaka easy addition theorem, so $\phi$ pulls back to a $\C^{**}$-fibration of $S$; a contradiction. Therefore, $Z$ is a point, which implies \ref{item:MFS_point}.
	
	\ref{item:Xmin-del-Pezzo} This follows from \ref{item:MFS_point} and Lemma \ref{lem:MMP}\ref{item:smoothness}.
	
	\ref{item:R} By Lemma \ref{lem:MMP}\ref{item:alpha_peeling}, $R\am=\alpha^{-1}_{*}D_{\min}$, so by \ref{item:MFS_point}, $\#R\am-1=\#D_{\min}-\rho(X_{\min})=\#D+n-\rho(X)=n$ since $\NS_{\Q}(X)=\Z[D]\otimes \Q$ by Remark \ref{rem:D_rat-tree}. Thus $\#R\am=n+1$. To prove the second equality, we introduce the following notation. For a reduced divisor $B\neq 0$, let $\nu_{B}$ be the number of nodes of $B$, and let $e_{B}=\nu_{B}-\#B$. Clearly, $e_{B}$ does not change under snc-minimization or a contraction of a tip of $B$. It is easy to see that resolving a double point does not change $e_{B}$, either. In turn, adding a bubble to $B$ increases $e_{B}$ by one. We get $e_{R\am}=e_{D\am}=e_{D'}=e_{(\psi^{*}D')\redd}=e_{D}+n=-1+n$ since $D$ is a tree. Now $\nu_{R\am}=\#R\am+e_{R\am}=2n$.
\end{proof}

	We conclude this preliminary section by a simple application of Lemma \ref{lem:MMP}, which shows that the Negativity Conjecture \ref{conj:negativity} is closely related to another variant of tom Dieck--Petrie Conjecture \ref{conj:classical}\ref{item:tDP}.

	Let $(X,D)$ be a minimal log smooth completion of a $\Q$HP $S$. We say that a divisor $B$ on a surface $Y$ is \emph{optimal} for $S$ if there is an expansion $(X,D)\to (Y',B')$, where $(Y',B')\to (Y,B)$ is a minimal log resolution, see \cite{tDieck_optimal-curves}. In other words, $(X,D)$ is obtained by the tom Dieck--Petrie algorithm  with $P=\emptyset$, but with $(\P^2,\pp)$ replaced by $(Y,B)$.
	Thus $D_{\min}\subseteq X_{\min}$ in Lemma \ref{lem:MMP} is an optimal divisor for $S$. 
	
	Tom Dieck proposed the following conjecture, see \cite[Conjecture 1.2]{tDieck_optimal-curves}. 
	\begin{conjecture}\label{conj:optimal}
	A smooth \ZHP of log general type has an optimal divisor on $\P^2$ or $\F_{m}$.
	\end{conjecture}
	We remark that the definition of an optimal divisor in \cite{tDieck_optimal-curves} does not require the log smooth completion $(X,D)$ of $S$ to be minimal. But Lemma \ref{lem:other-completion} shows that we can add this assumption if $S$ is not $\C^{**}$-fibered. In the other case, it is easy to construct an optimal divisor for $S$ on $\P^2$ or on $\F_0=\P^1\times \P^1$: we will do so in Section \ref{sec:Cstst}, see Figure \ref{fig:Cstst} and \cite{MiySu-Cstst_fibrations_on_Qhp}. Therefore, Conjecture \ref{conj:optimal} holds for $\C^{**}$-fibered $\Z$HPs, and thus is equivalent to  \cite[Conjecture 1.2]{tDieck_optimal-curves}. 
	
	Assume that $S$ has no $\C^{**}$-fibration, and satisfies the Negativity Conjecture \ref{conj:negativity}. Then the optimal divisor $D_{\min}\subseteq X_{\min}$ is as in Lemma \ref{lem:R}\ref{item:Xmin-del-Pezzo}: hence \emph{$S$ admits an optimal divisor on a del Pezzo surface of rank one, with at most $\rA_{k}$ singularities}.
	
	In fact, in Proposition \ref{prop:U} we will see that $X_{\min}$ can be chosen to be $\P^2$, $\P(1,1,2)$, or $\P(1,2,3)$. We conclude that Negativity Conjecture implies the following slight modification of Conjecture \ref{conj:optimal}.
	
	\begin{cor}
		Every smooth \QHP of log general type, satisfying the Negativity Conjecture \ref{conj:negativity}, admits an optimal divisor on $\P^2$, $\F_0$, $\F_2$ or $\P(1,2,3)$.
	\end{cor}

	The singular surface $\P(1,2,3)$ cannot be replaced by $\F_m$. Indeed, if $S$ is as in \ref{def:A1A2_c=3}--\ref{def:A1A2_2c1n} or \ref{def:nodal-cubic-P2_un}, then $X_{\min}=\P(1,2,3)$, and our construction shows that the only bubbles on $(X,D)$ are those contracted by $(X,D)\to (X_{\min},D_{\min})$. Thus Conjecture \ref{conj:optimal} fails for $\Z$HPs in these towers.

\section{General strategy and notation}\label{sec:strategy}

In order to settle the notation for Section \ref{sec:classification}, we now review our strategy of the proof of Theorem \ref{CLASS}, and explain it on a typical example.

Recall that the tom Dieck--Petrie algorithm, introduced in Definition \ref{def:TDP}, takes as an input: a configuration of lines and conics $\pp\subseteq \P^{2}$, its minimal resolution $\pi\colon X'\to \P^2$ and a subdivisor $D'\subseteq (\pi^{*}\pp)\redd$; and constructs from it a pair $(X,D)$ by an expansion $\psi\colon (X,D)\to (X',D')$.

We need to show that, conversely, for any $\Q$HP satisfying the assumptions of Theorem \ref{CLASS}, its minimal log smooth completion $(X,D)$ is constructed this way. To do this, we take for $\psi$ the expansion from Lemma \ref{lem:MMP}. By \cite{MiySu-Cstst_fibrations_on_Qhp} or by Proposition \ref{prop:C**}, we can assume that $S$ has no $\C^{**}$-fibrations, so $(X_{\min},D_{\min})$ is as in Lemma \ref{lem:R}. In particular, $\rho(X_{\min})=1$, and $n\de \#D_{\min}-1$ is the number of centers of $\psi$. 

To recover $\pi\circ \psi$, we find $\#\pp-n-1$ curves on $X'$ such that adding them to $D'$ gives a divisor with the same weighted graph as $\pi^{*}\pp$, hence contractible to a configuration with the same graph as $\pp$, on a smooth, projective rational surface of Picard rank $\rho(X_{\min})+(\#D'-\#D_{\min})-(\#(\pi^{*}\pp)\redd-\#\pp)=1$, that is, on $\P^{2}$. The graph of $\pp$ does not always determine $\pp\subseteq \P^2$ uniquely, nonetheless, if one $\pp$ realizing such graph is allowed in Theorem \ref{CLASS}, then so are all of them, see Section \ref{sec:constructions}. Thus we get $\pi\circ\psi$ as required. 

A major part of the proof consists of transforming given planar configurations (e.g.\ $D_{\min}$ if $X_{\min}=\P^2$) to configurations of lines and conics using Cremona maps. For this, we introduce the following notation. 

\begin{notation}\label{not:P2}
	Curves on $\P^{2}$ will be denoted by lowercase script letters $\cc_{\star},\ll_{\star},\dots$ with some subscripts $\star$. Their proper transforms on $X'$ will be denoted by corresponding uppercase letters $C_{\star},L_{\star},\dots$. For two points $p,q\in \P^{2}$, the line joining them will be denoted by $\ll_{pq}$. For a smooth curve germ $(p,\chi)$ on $\P^{2}$ the line tangent to $\chi$ at $p$ will be denoted by $\ll_{\chi,p}$ or simply $\ll_{p}$ if $\chi$ is clear from the context. 
\end{notation}

\begin{notation}\label{not:bl}
For a point $r$ on a smooth surface we denote by $\pi_{r}$ a blow up at $r$. Fix a birational morphism $\sigma$ such that $\Bs\sigma^{-1}=\{r\}$, and $\sigma^{-1}(r)$ contains a unique $(-1)$-curve. Then the components of $\sigma^{-1}(r)$ are naturally ordered as exceptional curves of subsequent blowups in the decomposition of $\sigma$. We denote them by  $E_{r},E_{r}',\dots$, so $E_{r}$ is the $(-1)$-curve. We write $U_{r},U_{r}',\dots$ for the same curves in the reverse order.
\end{notation}

\begin{rem}\label{rem:notation-graphs}
	Notation \ref{not:P2} and \ref{not:bl} is also used in Figures \ref{fig:nodal-cubic-P2_un}--\ref{fig:Cstst}, which show graphs of $\pi^{*}\pp$ encountered in the course of the proof. There, Notation \ref{not:bl} refers always to the morphism $X'\to \P^{2}$, not to $X'\to X_{\min}$.
	
	Dashed lines in these figures correspond to components of $E'\de (\pi^{*}\pp)\redd-D'$. By definition, they are $(-1)$-curves not contained in $D'$, and meet $D'$ normally. According to Notation \ref{not:bl}, we denote them by $E_p$, where the sub-index refers to a point $p\in \P^2$ which is the image of $E_p$. 
	
	Sometimes, introducing too many dashed lines could obscure the figure. In such cases, instead of drawing a curve $E_p$, we just mark the points where it meets $D'$ with some fixed symbol, say \enquote{$\bullet$}, and add a comment \enquote{$\bullet \mapsto p$}. For example, in the top-right picture of Figure \ref{fig:7} below, the symbol \enquote{$\bullet$} refers to a curve $E_{p}$ meeting $D'$ in $L_1$, $L_2$ and $C_2$. Similarly, curves $E_{p_1}$ and $E_{r}$ meet $D'$ in points marked by \enquote{$\boldsymbol{\diamond}$} and \enquote{$\boldsymbol{\circ}$}, respectively.
\end{rem}

\begin{ex}\label{ex:7}
	In order to illustrate the tom Dieck--Petrie algorithm, we follow it closely in the case of tower \ref{def:F2_n1-cusp}. For another example of this kind see \cite[p. 132]{tDieck_optimal-curves}, where tower \ref{def:F2n2-cuspidal} is constructed (from a planar divisor slightly different than ours).
	
	We take a planar divisor $\pp=\cc_{1}+\cc_{2}+\ll_{1}+\ll_{2}+\ll_{p'}$ from Configuration \ref{conf:F2_n1-cusp}, that is, $\ll_{1}$, $\ll_{2}$, $\ll_{p'}$ are lines tangent to the conic $\cc_{1}$; and $\cc_{2}$ is a conic tangent to $\cc_{1}$ with multiplicity $2$, passing through $\ll_{1}\cap \ll_{2}=\{p\}$, $\ll_{1}\cap \cc_{1}=\{p_{1}\}$ and tangent to $\ll_{p'}$ at $\ll_{2}\cap \ll_{p'}=\{r\}$, see Figure \ref{fig:7}. We take $P=\{p,p_1,r\}$.
	
	Let $\pi\colon X'\to\P^2$ be the minimal resolution of $\pp$, and let $E'\de E_p+E_{p_1}+E_r$ be the sum of $(-1)$-curves over $P$. The divisor $D'\de (\pi^{*}\pp)\redd-E'$ is shown in the top-right of Figure \ref{fig:7}. It meets $E'$ in points denoted by $\bullet$, $\boldsymbol{\diamond}$ and $\boldsymbol{\circ}$ (this notation is explained in Remark \ref{rem:notation-graphs}). We use Notation \ref{not:P2} for the components of $D'$.
	
	Now, $C_1+C_2+E_q$ is a unique circular subdivisor of $D'$: thus an expansion $(X,D)\to (X',D')$ such that $D$ is a tree must have a unique center, which is a node of $C_1+C_2+E_q$. Each of these choices leads to a pair $(X,D)$ with a different graph of $D$. 
	\smallskip
	
	We now use Lemma \ref{lem:expansions} to check when $S=X\setminus D$ is a \QHP. We compute that the map $r_{D'}\colon \Z[D']\to \NS(X')$ is surjective, and $\ker r_{D'}$ is generated by one relation, say $\mathcal{R}$. By Remark \ref{rem:m_comuptation}, $m(\mathcal{R})=m(\mathcal{R}_0)$, where $\mathcal{R}_0\de \pr_{C_{1}+C_2+E_q}(\mathcal{R})$. We compute that, up to a sign:
	\begin{equation*}
	\mathcal{R}_0=C_{1}-6C_{2}-10E_{q}.
	\end{equation*}
	Write the weight of our expansion as $\frac{u}{w}$, $\gcd(u,w)=1$.  Consider the center $(C_{1},C_{2})$. Then  $m(\mathcal{R}_0)=u-6w$, see Remark \ref{rem:m_comuptation}. Hence $S$ is a $\Q$HP if and only if $\frac{u}{w}\neq 6$. In this case formula \eqref{eq:expansion_H1} gives $\#H_{1}(S;\Z)=|u-6w|$, so we get infinitely many $\Z$HPs, for weights $6\pm \tfrac{1}{w}$, $w\in \N$. For weight $1$ we get $\#H_{1}(S;\Z)=5$: as we will see in Example \ref{ex:reduction} below, this $S$ is in fact a complement of a planar quintic.
	
	Similarly, for $(C_{1},E_{q})$ we have $m(\mathcal{R}_0)=u-10w$, so $S$ is a $\Q$HP if and only if $u/w\neq 10$, and we get infinitely many $\Z$HPs. But for $(C_{2},E_{q})$, $m(\mathcal{R}_0)=6u+10w\geq 16$, so all $S$'s are $\Q$HPs and none is a $\Z$HP.
\end{ex}
\begin{figure}[htbp]
\centering
				\begin{tikzpicture}
				\begin{scope}
					\draw (0.2,0) -- (0,1.4);
					\draw (0,1.2) -- (0.2,2.6);
					\draw (0.2,2.4) -- (0,4);
					\draw (0,3.8) -- (0.2,5.2);
					\node at (-0.2,0.6) {\small{$-2$}};
					\node at (-0.2,2) {\small{$-1$}};
					\node at (-0.2,3.2) {\small{$-3$}};
					\node at (-0.2,4.6) {\small{$-2$}};			
				\end{scope}
				\begin{scope}[shift={(1.4,0)}]
					\draw (0.2,0) -- (0,1.4);
					\draw (0,1.2) -- (0.2,2.6);
					\draw (0.2,2.4) -- (0,4);
					\draw (0,3.8) -- (0.2,5.2);
					\node at (-0.2,0.6) {\small{$-2$}};
					\node at (-0.2,2) {\small{$-1$}};
					\node at (-0.2,3.2) {\small{$-3$}};
					\node at (-0.2,4.6) {\small{$-2$}};			
				\end{scope}
				\begin{scope}[shift={(2.8,0)}]
					\draw (0.2,0) -- (0,1.4);
					\draw (0,1.2) -- (0.2,2.6);
					\draw (0.2,2.4) -- (0,4);
					\draw (0,3.8) -- (0.2,5.2);
					\node at (-0.2,0.6) {\small{$-2$}};
					\node at (-0.2,2) {\small{$-1$}};
					\node at (-0.2,3.2) {\small{$-3$}};
					\node at (-0.2,4.6) {\small{$-2$}};			
				\end{scope}	
				\draw (-0.2,1.7) -- (4.4,1.7);
				\draw[dashed] (2.6,3) -- (3.2,3) to[out=0,in=90] (4.2,2) -- (4.2,1.5);
				\node at (4,3) {\small{$-1$}};
				\node at (3.7,2.6) {\small{$A$}};
				\node at (3.5,1.9) {\small{$-5$}};
				\draw[->] (5,2.8) -- (7,2.8);
				\node at (6,3) {$\psi$};
				\node at (6,2.6) {\tiny{expansion}};
			\begin{scope}[shift={(8,0)}]
				\draw (0.2,0) -- (0,1.4);
				\draw (0,1.2) -- (0.2,2.6);
				\draw (0.2,2.4) -- (0,4);
				\draw (0,3.8) -- (0.2,5.2);
				\node at (-0.25,4.6) {\small{$-2$}};
				\node at (0.4,4.6) {\small{$L_{1}$}};
				\node at (-0.25,3.2) {\small{$-3$}};
				\node at (0.45,3.2) {\small{$L_{p'}$}};
				\node at (-0.2,2) {\small{$-1$}};
				\node at (0.45,2) {\small{$E_{p'}$}};
				\node at (-0.2,0.6) {\small{$-2$}};
				\node at (0.45,0.6) {\small{$U_{p'}$}};
				\node at (0.17,5) {\large{$\bullet$}};
				\node at (0.12,4.6) {\large{$\boldsymbol{\diamond}$}};
				\node at (0.1,3.2) {\large{$\boldsymbol{\circ}$}};			
			\end{scope}
			\begin{scope}[shift={(9.4,0)}]
				\draw (0.2,0) -- (0,1.4);
				\draw (0,1.2) -- (0.2,2.6);
				\draw (0.2,2.4) -- (0,4);
				\draw (0,3.8) -- (0.2,5.2);
				\node at (-0.25,4.6) {\small{$-2$}};
				\node at (0.4,4.6) {\small{$U_{r}$}};
				\node at (-0.25,3.2) {\small{$-3$}};
				\node at (0.4,3.2) {\small{$L_2$}};
				\node at (-0.2,2) {\small{$-1$}};
				\node at (0.45,2) {\small{$E_{p_2}$}};
				\node at (-0.2,0.6) {\small{$-2$}};
				\node at (0.45,0.6) {\small{$U_{p_2}$}};
				\node at (0.12,4.6) {\large{$\boldsymbol{\circ}$}};
				\node at (0.1,3.2) {\large{$\bullet$}};
			\end{scope}
			\begin{scope}[shift={(10.8,0)}]
				\draw (0.2,0) -- (0,1.4);
				\draw (0,1.2) -- (0.2,2.6);
				\draw (0.2,2.4) -- (0,4);
				\draw (0,3.8) -- (0.2,5.2);
				\node at (-0.25,4.6) {\small{$-2$}};
				\node at (0.5,4.6) {\small{$U_{p_1}$}};
				\node at (-0.25,3.2) {\small{$-2$}};
				\node at (0.4,3.2) {\small{$C_2$}};
				\node at (-0.2,2) {\small{$-1$}};
				\node at (0.35,2) {\small{$E_q$}};
				\node at (-0.2,0.6) {\small{$-2$}};
				\node at (0.35,0.6) {\small{$U_{q}$}};
				\node at (0.12,4.6) {\large{$\boldsymbol{\diamond}$}};
				\node at (0.1,3.2) {\large{$\bullet$}};
				\node at (0.07,3.5) {\large{$\boldsymbol{\circ}$}};													
			\end{scope}	
		\begin{scope}[shift={(8,0)}]
			\draw (-0.2,1.7) -- (3.8,1.7) to[out=0,in=0] (3.8,2.8) -- (2.6,2.8);
			\node at (3.8,2.2) {\small{$-4$}};
			\node at (4.4,2.2) {\small{$C_1$}};
			\node at (0.8,1.7) {\large{$\boldsymbol{\diamond}$}};
			\node[right] at (4.2,4) {\small{$\bullet\mapsto p$}};
			\node[right] at (4.2,3.7) {\small{$\boldsymbol{\diamond}\mapsto p_1$}};
			\node[right] at (4.2,3.4) {\small{$\boldsymbol{\circ}\mapsto r$}};
		\end{scope}
	\draw[->] (2,0.2) -- (2,-0.8);
	\node[left] at (2,-0.4) {$\tau$};
	\draw[->] (10,0.2) -- (10,-0.8);
	\node[left] at (10,-0.4) {$\pi$};
	\node[right] at (10,-0.25) {\tiny{minimal}};
	\node[right] at (10,-0.55) {\tiny{log resolution}};
	\begin{scope}[shift={(0.5,-4)}]
				\draw[semithick] (0,0) to[out=30,in=-90] (1.5,2.6) to[out=-90,in=150] (3,0) to[out=150,in=30] (0,0);
				\draw[dashed] (1.5,2.8) -- (1.5,0.2);
				\node at (2.8,0.5) {\small{$\mathcal{Q}_3$}};
	\end{scope}
	\begin{scope}[shift={(9.9,-3.3)}]
	\draw[name path=c1] (0,0) circle (1);
	\coordinate (P) at (0,2); 
	\node at ($(P)+(0,0.3)$) {\small{$p$}};
	\node at (P.center) {$\large{\bullet}$};
	\coordinate (R) at (1.732,-1);
	\coordinate (R') at (-1.732,-1);
	\node at ($(R)+(0.1,0.2)$) {\small{$r$}}; 
	\node at (R.center) {$\boldsymbol{\circ}$};
	\coordinate (P') at (0,-1);
	\node at ($(P')+(0,0.2)$) {\small{$p'$}}; \filldraw (P') circle (0.04);
	\coordinate (Q) at (0,1);
	\node at ($(Q)+(0,0.2)$) {\small{$q$}}; \filldraw (Q) circle (0.04);
	\coordinate (P1) at (-0.866,0.5);
	\node at ($(P1)+(-0.3,0)$) {\small{$p_1$}};
	\node at (P1.center) {$\boldsymbol{\diamond}$};
	\coordinate (P2) at (0.866,0.5);
	\node at ($(P2)+(0.3,0)$) {\small{$p_2$}}; \filldraw (P2) circle (0.04);
	\node at (-1.5,-0.2) {\small{$\ll_{1}$}};
	\node at (1.5,-0.2) {\small{$\ll_{2}$}};
	\node at (-0.7,-0.2) {\small{$\cc_{1}$}};
	\draw[add= 0.1 and 0.05, name path = L1] (P) to (R');
	\draw[add= 0.1 and 0.05, name path = L2] (P) to (R);
	\draw[add= 0.1 and 0.3, name path = LP'] (R') to (R);
	\node at (2.6,-0.8) {\small{$\ll_{p'}$}};
	\draw (P) to[out=180,in=135] (P1) to[out=-45,in=-120]
	($(Q)+(-0.4,-0.4)$)
	to[out=60,in=180]
	(Q) 
	to[out=0,in=100]
	($(Q)+(0.4,-0.4)$) 
	to[out=-80,in=180] (R) to[out=0,in=0] (P);
	\node at ($(P2)+(1.3,0)$) {\small{$\cc_{2}$}};
	\end{scope}
			\end{tikzpicture}
	\caption{Surface obtained in \ref{def:F2_n1-cusp} by an expansion with center $(C_{1},C_{2})$ and weight one becomes the complement of the tricuspidal quintic $\Qb$.}
	\label{fig:7}
\end{figure}
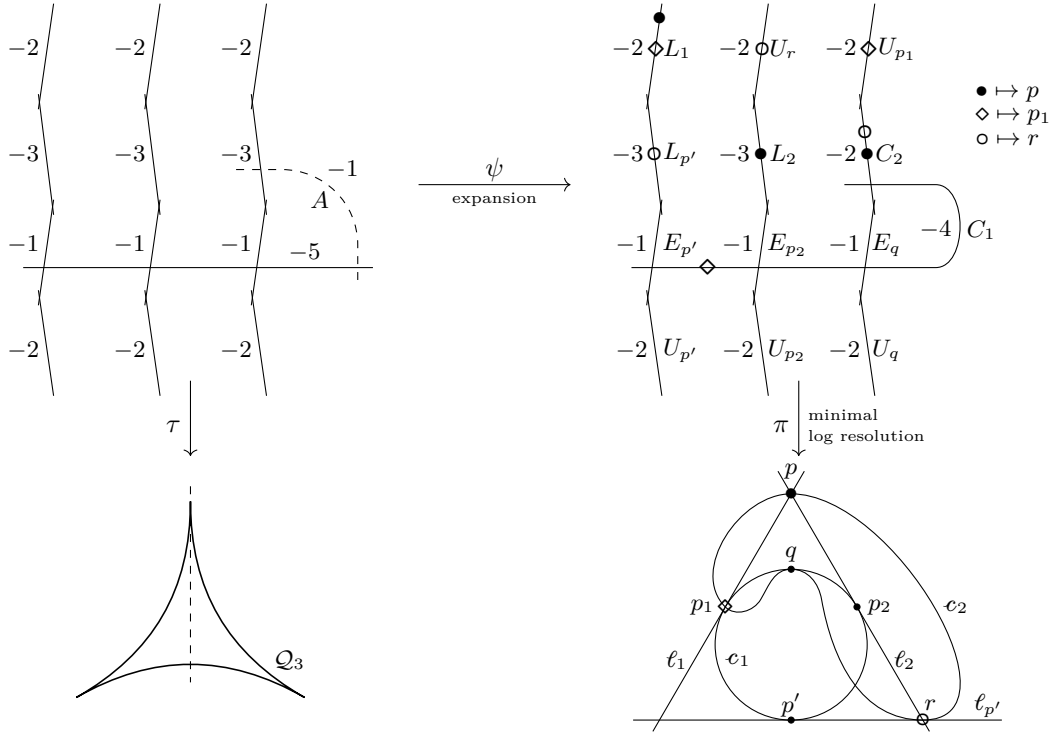

\begin{ex}
	[Different minimal models in a tower]\label{ex:reduction}
	In Lemma \ref{lem:MMP}\ref{item:no_bubbles}, we have chosen $(X',D')$ with no bubbles. This choice is not necessary for the application of MMP as in Lemma \ref{lem:MMP}. However, it will be very useful to restrict the possible types of $(X_{\min},D_{\min})$. To see this, we consider $S$ be as in Example \ref{ex:7} above.
	
	The expansion $\psi$ maps $(X,D)$ to $(X',D')$ shown in the top-right of Figure \ref{fig:7}. Then the almost minimalization  $\upsilon\colon (X',D')\to (X\am,D\am)$ contracts  $D'-(C_{1}+C_{2}+U_{p_1})$, so $X\am=\F_2$, with negative section $\Delta\am=\upsilon(U_{p_1})$, fiber $\upsilon(C_2)$ and a $3$-section $\upsilon(C_1)$. Thus $X_{\min}=\P(1,1,2)$. 
	
	Consider a particular case when $\psi$ is centered at $(C_1,C_2;1)$, so $\Exc\psi$ is the bubble $A$ in Figure \ref{fig:7}. 
	Let $\tau$ be the contraction of $D-\psi^{-1}_{*}C_1$. Then $\tau_{*}D$ is the tricuspidal quintic $\Qb\subseteq \P^2$ from  \cite[2.3.10.8]{Namba_geometry_of_curves}, see Table \ref{table:rcc}. The pair $(\P^2,\tfrac{1}{2}\Qb)$ satisfies Lemma \ref{lem:R}\ref{item:MFS_point}, so we could have taken $(X_{\min},\tfrac{1}{2}D_{\min})=(\P^2,\tfrac{1}{2}\Qb)$ for our minimal model of $(X,\tfrac{1}{2}D)$. In fact, $\tau$ is the unique MMP run for $(X,\tfrac{1}{2}D)$ as in \cite[Lemma 3.11]{Palka_MMP}. Nonetheless, in Lemma \ref{lem:MMP} we have decided to contract $A$ anyway. 
	
	Note that the image of $A$ is visible on the minimal model $(\P^{2},\tfrac{1}{2}\Qb)$: it is the line tangent to $\Qb$ at one of its cusps. This suggests that the lack of bubbles restricts the geometry of $D_{\min}$. It is indeed the case, and we will use it frequently, see claims \eqref{eq:F2_delta} or \eqref{eq:P2_ordinary}. For instance, claim \eqref{eq:P2_ordinary} -- roughly speaking -- shows that whenever $(X',D')$ has no bubbles, lines like $\tau(A)$ should lie in $D_{\min}$. Moreover, a lot of potential bubbles is produced by Lemma \ref{lem:Furushima}, see Figure \ref{fig:Furushima}: we will exploit them in Proposition \ref{prop:U} to restrict the singularity types of $X_{\min}$.
\end{ex}	

\begin{ex}[Corollaries \ref{cor:uniq}\ref{item:n} and \ref{cor:aut} for \ref{def:F2_n1-cusp}]\label{ex:7_uniq} 
	We will now outline the proof of Corollary \ref{cor:uniq}\ref{item:n} for surfaces considered in Example \ref{ex:7}; that is, we explain why the weighted graph of $D$ uniquely determines the isomorphism class of $(X,D)$, hence of $S$. We will make this explanation more rigorous in Example \ref{ex:7_bis}, after introducing suitable notation. For another illustration of this argument see Example \ref{ex:33}.
	
	First, note that the computation of Configuration \ref{conf:F2_n1-cusp} in Section \ref{sec:constructions} shows that, up to a projective equivalence, there are exactly two planar divisors $\pp\subseteq \P^{2}$ as in Example \ref{ex:7}, see Proposition \ref{prop:conf_uniq}\ref{item:conf_number}.
	
	Since the centers of the blowups in the decomposition of $\pi$ are determined uniquely by $\pp$, we get a one-to-one correspondence between isomorphism classes of pairs $(\P^{2},\pp)$ and $(X',D'+E')$, \emph{each with fixed order of the boundary components}. For $(\P^2,\pp)$, this order is uniquely determined by the combinatorics of $\pp$. But the weighted graph $D'+E'$ has an involution $\iota$ 
	interchanging the subchains $[2,3,1,2]$ of $D'$, see Figure \ref{fig:7}. 
	The two possible orders of the components of $D'+E'$ correspond to two classes of projective equivalence of $\pp\subseteq \P^{2}$. We conclude that $(X',D'+E')$ is actually unique up to an isomorphism. Now since the centers of the blowups in the decomposition of $\psi'$ are determined uniquely by $(X',D'+E')$, we infer that $(X,D)$ is unique, as claimed.
	 
	The graph of $D$ has an involution coming from $\iota$, but we have seen that it is not induced by automorphism of $S$. However, for an expansion at $(C_{1},C_{2};1)$, shown in Figure \ref{fig:7}, the graph of $D$ has another automorphism, of order three, which does correspond to an automorphism of $S$. Indeed, we have seen in Example \ref{ex:reduction} that in this case $S$ is the complement of the tricuspidal quintic $\Qb$, so $\Aut(S)=\Z_{3}$ is generated by the rotation of $\P^{2}$, see   \cite[Proposition 2.9]{tDieck_optimal-curves}, \cite[Remark 4.8]{PaPe_delPezzo} or Example \ref{ex:7_bis}.
\end{ex}

\section{Comparison with some results in the literature} \label{sec:literature}
We now comment how some well-known examples fit into our classification. As in Section \ref{sec:expansions}, we denote by $\rst$ the tower obtained in Theorem \ref{CLASS} from $*$-th row of \tables. 
\smallskip

The first example of a \emph{topologically contractible} surface which is not isomorphic to $\C^{2}$ was given by Ramanujam \cite[\S 3]{Ramanujam_C2}. We recover it in \ref{def:C**_2} by expansions with centers $(L_{pp'},C_1)$, $(L_{pp'},C_1)$, $(E_{p},L_{2})$ and weights $(\frac{1}{2},1,1)$. Originally, this surface was distinguished from $\C^{2}$ using the fundamental group at infinity. In fact, by loc.\ cit.\ $\C^{2}$ is the only contractible surface for which this group is trivial. Another way to distinguish it is to show that its Kodaira dimension is two. It was done by Gurjar and Miyanishi in \cite[\S 4]{GM_k<2}. In loc.\ cit, another, similar example  of a contractible surface of log general type was constructed, we recover it in  \ref{def:P2n2_cuspidal} by expansions with centers $(E_{p'},C_{1})$, $(E_{p_2},C_{1})$ and weights $(1,1)$, cf.\  \cite[3.15]{DiPe-hp_and_alg_curves}.

Some generalizations of Ramanujam's construction are given in \cite{Sugie_Ram}. In the Appendix to loc.\ cit, Miyanishi and Sugie constructed a series of \ZHPs of log general type. We recover them in  \ref{def:C**_1} by expansions with centers  $(L_{1},L_{2})$, $(E_{q_{1}},L_{q_{1}q_{2}})$, $(L_{1}',L_{1})$, $(L_{2}',L_{2})$ and weights $(m^{-1},n-2,r+1,(n-m)^{-1})$, where $r=m(n-1)$ or $m(n-1)-2$. By Theorem 2 loc.\ cit, the surfaces with $r=m(n-1)$ are contractible.	
\smallskip

As an illustration of their algorithm, in  \cite{tDieck_optimal-curves}  tom Dieck and Petrie constructed $14$ towers  $A,B,\dots,N$. In our classification, they are called \ref{def:C**_1}, \ref{def:P2n3}, \ref{def:C**_2}, \ref{def:P2n2_cuspidal}, \ref{def:F2_n2-tangent}, \ref{def:F2_n2-transversal}, \ref{def:nodal-cubic_P2}, \ref{def:F2n2-cuspidal}, \ref{def:P2n1-cuspidal}, \ref{def:F2-hor-ccc-41}, \ref{def:P2n1-nodal}, \ref{def:F2_n1-cusp}, \ref{def:F2n2-nodal}, \ref{def:C**_3}, respectively. Proposition 2.2 loc.\ cit.\ shows that for towers $A$, $B$, $C$, the initial planar divisor $\pp\subseteq \P^{2}$ can be an arrangement of lines. It is not so e.g.\ for the tower $I$, see \cite{tDieck_symmetric_hp}.  All line arrangements leading to $\Q$HPs of log general type are classified in \cite{DiPe_line-arragements}; those consisting of lines and one conic were studied in \cite{Neusel}. 
Unfortunately, exhaustive lists in each case are rather long, which makes it difficult to study the resulting $\Q$HPs directly. Therefore, we will not attempt to compare them with Theorem \ref{CLASS} (which shows that, assuming Negativity Conjecture \ref{conj:negativity}, a much shorter list of planar arrangements is sufficient). 
\smallskip

A plenty of examples of \QHPs is given by surfaces $S=\P^2\setminus \bar{E}$, where $\bar{E}$ is a \emph{rational cuspidal curve}, i.e.\ a curve homeomorphic to $\P^1$ in the Euclidean topology. Classification of such curves, up to a projective equivalence, is still an open problem, with many interesting connections to low-dimensional topology and singularity theory, see e.g.\ \cite{Bobadilla-LHN_Cuspidal,KoPa_four-cusps} and references there. 

Those rational cuspidal curves whose complements are of log general type and satisfy the Negativity Conjecture \ref{conj:negativity} were classified by Palka and the author in \cite{PaPe_Cstst-fibrations_singularities,PaPe_delPezzo}. Table \ref{table:rcc} shows how to recover\footnote{
	To be precise, it shows how to construct the pair $(X,D)$, where $\pi\colon (X,D)\to (\P^{2},\bar{E})$ is the minimal log resolution. This pair is the minimal log smooth completion of $\P^{2}\setminus \bar{E}$ in all cases except $\cA$-$\cD$ with $\gamma=1$, where one needs to contract a superfluous $(-1)$-curve $\pi^{-1}_{*}\bar{E}\subseteq D$. And indeed, the respective weights from Table \ref{table:rcc} are not admissible (see Table \ref{table:C**}). In these cases, the minimal log smooth completion is obtained in  \ref{def:C**_1} by expansion with centers $(L_{1},L_{2})$, $(E_{p_{1}'},L_{p_{1}'p_{2}'})$, $(L_{1},E_{p_{1}'})$, $(L_{2},E_{p_{2}'})$ and weights as follows. In case $\cA$: $(1,\tfrac{2}{2s-1},p-1,p-1+\tfrac{1}{s})$, $\cB$: $(1,\tfrac{2}{2s-1},\tfrac{s}{ps+p-1},p-1)$, $\cC$: $(1,\tfrac{1}{s},p-1,p-1+\tfrac{1}{2s+1})$ and $\cD$: $(1,\tfrac{1}{s},p-1+\tfrac{4}{2s+1},p-1)$. A similar adjustment is made in the proof of Proposition \ref{prop:C**}.
	\label{note:rcc}
} these complements from Theorem \ref{CLASS}. Note that the towers listed there have $\ngr=1$, that is, by Corollary \ref{cor:uniq}\ref{item:n}, they are uniquely determined by the weighted graph of $D$. On the other hand, for the complements of rational cuspidal curves, that graph depends only on the topology of the cusps. It therefore follows from Theorem \ref{CLASS} that any rational cuspidal curve whose complement is of log general type and satisfies the Negativity Conjecture, is uniquely determined, up to a projective equivalence, by the topological type of its singularities. In \cite[\S 4]{PaPe_delPezzo}, this result is deduced by combining known results about rational cuspidal curves. The proof of Corollary \ref{cor:uniq}\ref{item:n}, given in Section \ref{sec:uniqueness}, provides an independent, uniform argument.

\begin{rem}[$\Z$HPs]\label{rem:ZHP}
	Those towers from Theorem \ref{CLASS} which contain \ZHPs are marked by \enquote{$\tick$} in the column \enquote{$\Z$} of the respective table. By Lemma \ref{lem:ZHP}, every such tower contains infinitely many \QHPs $S$ with $\# H_{1}(S;\Z)=k$ for any positive integer $k$. Therefore, the tom Dieck--Petrie approach is well suited for the study of \QHPs, but not of \ZHPs in particular.
\end{rem}

\section{Proof of Theorem \ref{CLASS}}\label{sec:classification}

Let $S$ be a smooth \QHP, and let $(X,D)$ be the minimal log smooth completion of $S$. To prove Theorem \ref{CLASS} we can, and will, assume that
	\begin{equation}\label{eq:assumption}
		\kappa(S)=2,\quad 
		\kappa(K_{X}+\tfrac{1}{2}D)=-\infty \mbox{ and } S \mbox{ admits no $\C^{**}$-fibration.}
\end{equation}

Indeed, if $S$ is $\C^{**}$-fibered then Theorem \ref{CLASS} follows from \cite{MiySu-Cstst_fibrations_on_Qhp}, see Section \ref{sec:Cstst}. Let
\begin{equation*}
	\begin{tikzcd}
		(X,D) \ar[r, "\psi"] & (X',D')\ar[r, "\upsilon"] &(X\am,D\am)\ar[r, "\alpha"] &(X_{\min},D_{\min})
	\end{tikzcd}
\end{equation*}
be as in Lemma \ref{lem:MMP}. By definition, $\psi$ is an expansion with, say, $n$ centers. By Lemma \ref{lem:MMP}\ref{item:no_bubbles}
	\begin{equation}\label{eq:no_bubbles}
		\mbox{there are no bubbles on $(X',D')$ and on $(X\am,D\am)$.}
	\end{equation}
	By assumption \eqref{eq:assumption} and Lemma \ref{lem:R}\ref{item:MFS_point}, $-(2K_{X_{\min}}+D_{\min})$ is ample and $\rho(X_{\min})=1$. As in Lemma \ref{lem:MMP}, write 
	\begin{equation}\label{eq:Delta_R}
		D\am=R\am+\Delta\am,\quad\mbox{where $\Delta\am$ is the sum of all $(-2)$-twigs of $D\am$,}
	\end{equation}
	Then by Lemma \ref{lem:MMP}\ref{item:alpha_peeling}, for any curve $U\subseteq X\am$ not contained in $\Delta\am$ we have 
	\begin{equation}\label{eq:ampleness}
		0> \alpha(U)\cdot (2K_{X_{\min}}+D_{\min})=U\cdot (2K_{X\am}+D\am-\Bk\Delta\am).
	\end{equation}
	Furthermore, by Lemma \ref{lem:R}\ref{item:R} 
	\begin{equation}\label{eq:R}
	\#R\am=n+1,\quad \mbox{and}\quad R\am \mbox{ has exactly } 2n \mbox{ nodes}.
	\end{equation} 

In this section, we show that all \QHPs satisfying assumption \eqref{eq:assumption} are arranged in 39 towers obtained by the tom Dieck--Petrie algorithm from arrangements $\pp\subseteq \P^{2}$ of lines and conics, listed in Section \ref{sec:constructions}.

Conversely, by Lemma \ref{lem:kappa_expansion}\ref{item:kappa_1/2_expansion}, any $\Q$HP whose log smooth completion $(X,D)$ is obtained by an expansion $(X,D)\to(X',D')$ from a minimal log resolution of $(X_{\min},D_{\min})$ as in Lemma \ref{lem:R}\ref{item:MFS_point}, automatically satisfies $\kappa(K_{X}+\tfrac{1}{2}D)=-\infty$. Thus our $\Q$HPs satisfy the Negativity Conjecture \ref{conj:negativity}. In Proposition \ref{prop:k=2}, we will show they are of log general type, and thus complete the proof of Theorem \ref{CLASS}.

Note that the assumption \eqref{eq:assumption} and the Iitaka easy addition theorem imply that
\begin{equation}\label{eq:assumption_fibr}
	X_{\min}\setminus D_{\min}=X\am\setminus D\am \mbox{ admits no  $\C^{1}$, $\C^{*}$- or $\C^{**}$-fibration.}
\end{equation}

\subsection{Possible types of $X_{\min}$}\label{sec:exclude}
By assumption \eqref{eq:assumption}, the pair $(X_{\min},\tfrac{1}{2}D_{\min})$ is as in Lemma \ref{lem:R}\ref{item:MFS_point}. In particular, $-K_{X_{\min}}$ is ample, $\rho(X_{\min})=1$ and $X_{\min}$ has at most $\rA_{k}$-singularities, so $X_{\min}\cong \P^2,\P(1,1,2)$, or $X_{\min}$ is as in Lemma \ref{lem:Furushima}. In the latter case, Figure \ref{fig:Furushima} shows some $(-1)$-curves on $X\am$. In Proposition \ref{prop:U} below, we will see that unless $X_{\min}\cong \P(1,2,3)$, at least one of these curves is a bubble on $(X\am,D\am)$, contrary to condition \eqref{eq:no_bubbles}. Later, we will see that lack of bubbles restricts the geometry of $D\am$, too, see Example \ref{ex:reduction}.

\begin{prop}
\label{prop:U}
Let $(X_{\min},\tfrac{1}{2}D_{\min})$ be as above. 
Write $D\am=R\am+\Delta\am$ as in formula \eqref{eq:Delta_R}. Then 
\begin{enumerate}
	\item\label{item:Xmin} The surface $X_{\min}$ is either: $\P^{2}$, or a quadric cone $\P(1,1,2)$, or $\P(1,2,3)$.
	\item\label{item:A1A2} Assume that $X_{\min}\cong \P(1,2,3)$ and let $U\subseteq X\am$ be the proper transform of the line joining its singular points, see Figure \ref{fig:A1+A2}. Then $U\not\subseteq D\am$, $U\cdot R\am=1$ and $U$ meets $\Delta\am$ in tips of $D\am$.
\end{enumerate}
\end{prop}
\begin{proof}
Assume $X_{\min}\not\cong \P^2,\P(1,1,2)$. By Lemma \ref{lem:R}\ref{item:MFS_point}, $X_{\min}$ is one of the surfaces in Lemma \ref{lem:Furushima}, and $-(2K_{X_{\min}}+D_{\min})$ is ample. Let $U\subseteq X\am$ be a $(-1)$-curve. The inequality \eqref{eq:ampleness} and the adjunction formula give  
$
0>U\cdot (2K_{X\am}+R\am+\Delta\am-\Bk\Delta\am)=-2+U\cdot R\am+U\cdot (\Delta\am-\Bk\Delta\am), 
$
so 
\begin{equation}\label{eq:U}
U\cdot R\am<2-U\cdot (\Delta\am-\Bk\Delta\am)\leq 2,\quad \mbox{in particular}\quad U\cdot R\am\leq 1.
\end{equation}

\begin{claim*}\label{cl:U}
	Assume that $U$ meets exactly one connected component $T$ of $\Delta\am$. The following \emph{cannot} hold:
	\begin{enumerate}
		\item\label{item:UR=1} $U\cdot T=1$ and $U$ does not meet a tip of $T$.
		\item\label{item:loops} $U\cdot T=2$, $\#T\geq 2$ and $U$ meets both tips of $T$.
	\end{enumerate}
\end{claim*}
\begin{proof}
	Suppose \ref{item:UR=1} or \ref{item:loops} holds. Then $U\not\subseteq R\am$, since $R\am$ meets $T$ once, in a tip. Consider case \ref{item:UR=1}. Then $U\cdot R\am\geq 1$ since by Lemma \ref{lem:MMP}\ref{item:Si_no-lines} $X\am\setminus D\am$ contains no contractible curves. By the formula  \eqref{eq:U}, $U\cdot R\am=1$, hence $U$ is a bubble on $(X\am,D\am)$, contrary to the condition \eqref{eq:no_bubbles}. In case \ref{item:loops}, formula \eqref{eq:bark_of_a_2-twig} gives $U\cdot \Bk\Delta\am=1$, so $U\cdot R\am=0$ by formula \eqref{eq:U}, and again $U$ is a bubble; a contradiction.
\end{proof}
	Conditions \ref{item:UR=1} and \ref{item:loops} above exclude $X_{\min}$ shown in Figures \ref{fig:A4}--\ref{fig:A8} and \ref{fig:A1+A7}--\ref{fig:4A2}, respectively. 
	
	Suppose $X_{\min}$ is of type $2\rA_1+2\rA_3$, see Figure \ref{fig:2A1+2A3}. Then $\theta_{*}D\am=\qq+\ll_1+\ll_2+\ll_3$, where: $\qq$ is a cubic with a node, say $q\in\qq$, $\ll_1$ is a line tangent to $\qq$ at its inflection point, say $p_1$, and $\ll_2$, $\ll_3$ are lines such that for some $p_2,p_3\in \qq\reg$ we have $p_1\in \ll_3$, $p_3\in\ll_2$ and $(\qq\cdot \ll_2)_{p_2}=(\qq\cdot \ll_3)_{p_3}=2$, see \cite[p.\ 15]{Furushima}. Now the proper transform of the line joining $q$ with $p_3$ satisfies condition \ref{item:loops} above; a contradiction.
	
	Suppose that $X_{\min}$ is of one of the types 
$3\rA_2$, 
	$2\rA_1+\rA_3$ or $\rA_1+2\rA_3$, see Figures \ref{fig:3A2}--\ref{fig:A1+2A3}. Then $\Delta\am$ has a connected component $T=[2,2,2]$ or $[2,2]$ such that each of its tips, say $T^{\pm}$ meet a $(-1)$-curve, say  $U^{\pm}\subseteq \Exc\theta$. Since $R\am$ meets only one tip of $T$, say $T^{-}$, we have $U^{+}\not\subseteq R\am$. Formula \eqref{eq:bark_of_a_2-twig} gives $U^{+}\cdot \Bk T=1-\tfrac{1}{\# T+1}$. Looking at Figures \ref{fig:3A2}--\ref{fig:A1+2A3} again, we see that $U^{+}\cdot \Delta\am=2$, and $U^{+}\cdot \Bk(\Delta\am-T)\geq \tfrac{1}{\#T+1}$, so $U^{+}\cdot (\Delta\am-\Bk\Delta\am)\geq 1$. By formula \eqref{eq:U}, $U^{+}\cdot R\am=0$, so $U^{+}$ is a bubble; a contradiction with condition  \eqref{eq:no_bubbles}.
	
	Thus $X_{\min}$ is of type $\rA_1+2\rA_2$, i.e.\ $X_{\min}\cong \P(1,2,3)$. Let $U$ be the $(-1)$-curve in Figure \ref{fig:A1+A2}. Suppose some $(-2)$-tip of $D\am$ does not meet $U$. Then $U\cdot \Bk\Delta\am=\tfrac{5}{6}$, so $U\cdot R\am\leq 0$ by formula \eqref{eq:U}. If $U\subseteq R\am$ then $\beta_{R\am}(U)=U\cdot R\am-U^2\leq 1$ and $\beta_{D\am}(U)=\beta_{R\am}(U)+U\cdot \Delta\am\leq 3$, contrary to   Lemma \ref{lem:MMP}\ref{item:D_squeezed}. Hence $U\not\subseteq R\am$, so $U\cdot R\am=0$ and thus $U$ is a bubble on $(X\am,D\am)$; a contradiction with condition \eqref{eq:no_bubbles}.
	
	We have shown that $U$ meets both $(-2)$-tips of $D\am$, in particular, $U\not\subseteq D\am$. Now $U\cdot R\am\leq 1$ by formula \eqref{eq:U}, so in fact $U\cdot R\am=1$ by condition \eqref{eq:no_bubbles}, as claimed.
\end{proof}

Proposition \ref{prop:U} shows that, under our assumptions, $X_{\min}$ is $\P(1,2,3)$, $\P(1,1,2)$, or $\P^{2}$. We will treat these cases in Sections \ref{sec:A1A2}, \ref{sec:F2} and \ref{sec:P2}, respectively. Our aim is to restrict the possible divisors $D_{\min}$ using condition \eqref{eq:no_bubbles}, and to transform the remaining ones to $\pp\subseteq \P^{2}$, constructed in Section \ref{sec:constructions}.  If $X_{\min}\not\cong \P^{2}$ we shall encounter some cases with $n=0$, i.e.\ when $(X,\tfrac{1}{2}D)$ is already almost minimal. We call them \emph{sporadic}, since their towers consist of single elements. It is convenient to study them together in Section \ref{sec:sporadic}.

\subsection{Case \texorpdfstring{$X_{\min}\cong \P(1,2,3)$}{Xmin=P(1,2,3)}}\label{sec:A1A2}
 In this case, $\Delta\am$ consists of two twigs, $T_1=[2]$ and $T_2=[2,2]$. Let $U\subseteq X\am$ be as in Proposition \ref{prop:U}\ref{item:A1A2}, that is, $U$ is a $(-1)$-curve, $U\not\subseteq D\am$, $U\cdot R\am=1$, $U\cdot\Delta\am=2$ and $U$ meets both $(-2)$-tips of $D\am$. As in Lemma \ref{lem:Furushima}, let $\theta\colon X\am\to \P^{2}$ be the contraction of $U+T_2=[1,2,2]$, see Figure \ref{fig:A1+A2}. Now $\ll\de \theta_{*}T_1$ is a line and meets $\rr\de\theta_{*}R\am$ in two points, say $p$, $q$; $(\ll\cdot\rr)_{p}=1$ and $q\in \rr$ is an ordinary node with one branch  tangent to $\ll$ with multiplicity $3$. In particular, $\deg \rr= 5$. By formula \eqref{eq:R}, we have $n=\#\rr-1\leq 2$. 
\medskip

Consider the case $n=2$. Then $\rr$ is a sum of a cubic $\qq$ and two lines $\ll_{p}\ni p$ and $\ll_{q}\ni q$. Write $\{s\}=\Sing\qq$. Suppose $s\in \qq$ is a cusp. Then $\#\ll'\cap \qq\geq 2$ for every line $\ll'\neq \ll$, so $\#(\qq\cap(\ll_{p}\cup\ll_{q})\setminus \ll) \geq 3$. In fact, the equality holds since $\rr$ has $2n=4$ nodes off $\ll$, and one of them is $\ll_{p}\cap\ll_{q}$. Hence $\ll_{q}$ is tangent to $\qq$ at some $r\neq q$. The projection from $q$ gives a morphism $(\bl{q}^{-1})_{*}\qq\to \P^{1}$ of degree $2$, ramified at $\bl{q}^{-1}(r)$, $\bl{q}^{-1}(s)$ and at the point infinitely near to $q$; a contradiction with the Hurwitz formula. 

Hence $s\in \qq$ is a node. The fact that $\rr$ has four nodes off $\ll$ implies now that $(\ll_{p}\cdot \qq)_{r}=3$ for some $r\neq q,s$. Perform a quadratic Cremona map centered at $q$, $r$, $s$, that is, blow up these points and contract the proper transforms of the lines joining them. Denote the images of
$\qq$, $\ll$, $\ll_{p}$, $\ll_{q}$ and the exceptional curves over
$q$, $r$, $s$ by 
$\cc_{1}$, $\ll_{1}$, $\ll_{2}$, $\ll$ and $\ll_{1}'$, $\ll_{2}'$, $\ll_{r_{1}r_{2}}$. Then  $\cc_{1}+\ll_{1}+\ll_{2}+\ll_{1}'+\ll_{2}'+\ll_{r_{1}r_{2}}+\ll$, and hence $S$, is as in  \ref{def:nodal-cubic-P2_un}. The restrictions on weights of expansions come from Lemma \ref{lem:expansions}.

	\begin{figure}[htbp]
			\begin{tikzpicture}
			\draw (0,-0.2) -- (-0.2,1.2);
			\node at (-0.4,0.4) {\small{$-2$}};
			\draw (-0.2,1) -- (0.1,2.6);
			\node at (-0.35,1.8) {\small{$-1$}};
			\node at (0.2,1.8) {\small{$E_s$}};
			\draw (0,2.5) -- (1.8,2.9);
			\node at (1,2.5) {\small{$-2$}};
			\node at (1,2.9) {\small{$L$}};
			\node at (1.3,2.75) {$\bullet$};	
			\node at (1.3,1.5) {$\bullet$};		
			\draw (1.6,2.9) -- (3.4,2.5);
			\node at (2.5,2.5) {\small{$-2$}};
			\node at (2.6,2.9) {\small{$L_2$}};
			\draw (0.4,2) -- (0.8,4.3);
			\node at (0.4,3.8) {\small{$-2$}};
			\node at (0.95,3.8) {\small{$L_1'$}};
			\draw (0.7,4.2) -- (2.1,4.4);
			\node at (1.6,4.15) {\small{$-2$}};
			\draw[dashed] (1.9,4.4) -- (4,4.1) to[out=0,in=90] (4.6,1.3);
			\node at (4.35,1.8) {\small{$-1$}};
			\node at (4.95,1.8) {\small{$E_{p_1}$}};
			\draw (3.3,-0.2) -- (3.1,1.2);
			\node at (2.9,0.4) {\small{$-2$}};
			\draw (3.1,1) -- (3.3,2.6);
			\node at (2.9,1.8) {\small{$-1$}};
			\node at (3.5,1.8) {\small{$E_{p_2}$}};	
			\draw (3,2.4) -- (3.3,4.3);
			\node at (2.9,3.8) {\small{$-2$}};
			\node at (3.5,3.8) {\small{$L_1$}};
			\node at (3.1,3) {$\bullet$};
			\draw[dashed] (2.1,2.6) to[out=90,in=0] (2.1,3.3) to[out=180,in=0] (0.6,3.3) to[out=180,in=90] (-1.7,2.6) -- (-1.8,1.9);
			\node at (-1.4,2.6) {\small{$E_{r_2}$}};
			\node at (-2,2.6) {\small{$-1$}};
			\draw (3.4,-0.1) -- (1.6,0.3);
			\node at (2.2,0.4) {\small{$-2$}};
			\node at (2.3,-0.2) {\small{$L_2'$}};
			\draw (-3.8,1.5) -- (4.8,1.5);
			\node at (-3.5,1.75) {\small{$C_1$}};
			\node at (-3.5,1.25) {\small{$-3$}};
			\draw (-1.8,1.4) [partial ellipse=190:-10 : 0.8 and 0.8];
			\node at (-2,1.7) {\small{$L_{r_1 r_2}$}};
			\node at (-2.7,1.9) {\small{$-1$}};
			\node at (-1.14,1.8) {$\bullet$};
			\draw[dashed] (0.3,2.1) -- (1.4,2.3) to[out=0,in=90] (1.7,0.1);
			\node at (1.6,1.8) {\small{$-1$}};
			\node at (2.2,1.8) {\small{$E_{p'}$}};
			\node at (-2.5,0.5) {\small{$\bullet\mapsto r_1$}};	
		\end{tikzpicture}
		\caption{The graph of $D'$ in case $X_{\min}\cong \P(1,2,3)$, $n=2$. In this case, $S$ is as in \ref{def:nodal-cubic-P2_un}.}
		\label{fig:nodal-cubic-P2_un}
	\end{figure}
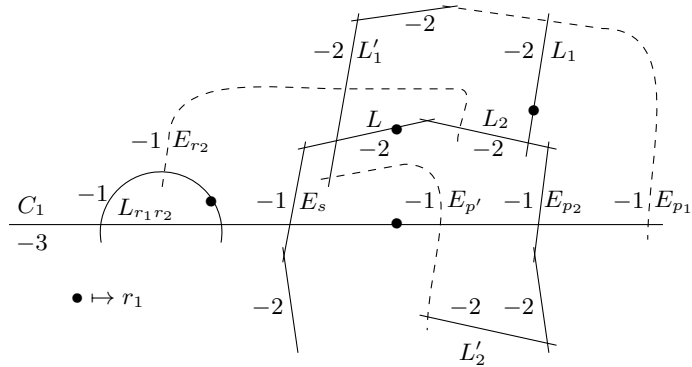
\medskip

Consider the case $n=1$. Then $\rr$ has two components, say $\qq$, $\cc$, and $\deg\qq\geq 3$. Now $\rr$ has $2n=2$ nodes off $\ll$, so $\qq$ has $\nu\leq 1$ nodes off $\ll$ and  meets $\cc\setminus \{q\}$ in $2-\nu$ points. Put $\tilde{D}''=(\theta\circ\upsilon)^{*}(\rr+\ll)\subseteq X'$. Using Notation \ref{not:P2}, we have  $\Delta\am=L+U_{q}+U_{q}'$, $L=[2]$, $U_{q}'+U_{q}=[2,2]$. The $(-1)$-curve $E_{q}$ meets $L$, $U_{q}'$ and $R$, once each. We need to find a divisor $D''$ containing $\tilde{D}''$, whose weighted graph is the same as the one of the preimage, on the minimal log resolution, of $\pp\subseteq \P^{2}$ as in Table \ref{table:smooth_1}.
\smallskip 

Consider the case when $\qq$ is a cubic, so $\cc$ is a conic. Then $\ll$ meets $\cc$ transversally at $p$, $q$ and $(\ll\cdot \qq)_{q}=3$. Let $s$ be the singular point of $\qq$. We have $\#(\cc\cap \qq)=3-\nu$, so $(\cc\cdot \qq)_{r}\geq 3$ for some $r\neq q$.  Let $\ll_{r}$ be the line tangent to $\qq$ at $r$. Then $L_{qr},L_{rs},L_{sq},L_{r}\subseteq X'$ are disjoint $(-1)$-curves and 
\begin{itemize}
	\item $L_{qr}\cdot \tilde{D}''=3$ and $L_{qr}$ meets $\tilde{D}''$ in $Q$, $U_{q}$ and $U_{r}$,
	\item $L_{rs}\cdot \tilde{D}''=4$ and $L_{rs}$ meets $\tilde{D}''$ in $C$, $L$, $U_{r}$ and $U_{s}$,
	\item $L_{sq}\cdot \tilde{D}''=3$ and $L_{sq}$ meets $\tilde{D}''$ in $C$, $U_{s}$ and $U_{q}$,
	\item $L_{r}\cdot \tilde{D}''=3$ and $L_{r}$ meets $\tilde{D}''$ in $Q$, $L$ and $U_{r}'$.
\end{itemize} 

Assume that $s\in \qq$ is a node. Then $\cc\cap \qq=\{q,r\}$ and $(\cc\cdot \qq)_{r}=5$. It follows that $D''=\tilde{D}''+L_{qr}+L_{rs}+L_{sq}$, and hence $S$, is as in  \ref{def:A1A2_C2C3-node}: indeed, the curves 
$Q$, $C$, $L$, $U_{q}$, $U_{r}$, $U_{s}$, $L_{qr}$, $L_{rs}$, $L_{sq}$, $E_{q}$
defined above correspond there to 
$C_{1}$, $C_{2}$, $L_{2}$, $L_{2}'$, $L_{1}'$, $L_{qt}$, $E_{p'}$, $E_{q}$, $E_{t}$, $E_{p_{2}}$,
respectively. 
Assume that $s\in \qq$ is a cusp. Then $\cc\cap \qq=\{q,r,r'\}$ for some $r'\neq q,r$ and $(\cc\cdot \qq)_{r}\in \{3,4\}$. It follows that $D''=\tilde{D}''+L_{rs}+L_{r}$, and hence $S$, is as in  \ref{def:A1A2_C2C3-cusp-41} in case $(\cc\cdot \qq)_{r}=4$ and  \ref{def:A1A2_C2C3-cusp-32} in case $(\cc\cdot \qq)_{r}=3$. Indeed, the curves 
$Q$, $L$, $C$, $U_{r}'$, $U_{s}$, $L_{rs}$, $L_{r}$, $E_{q}$ 
defined above correspond there to 
$C_{1}$, $C_{2}$, $C_{3}$, $L_{qp'}$, $L_{1}$, $E_{q}$, $E_{p'}$, $E_{p_{2}}$, 
respectively.

We claim that in case \ref{def:A1A2_C2C3-cusp-41}, we can take the weight of expansion at $(C_1,C_3)$ to be different from $1$: remaining restrictions in Table \ref{table:smooth_1} come from Lemma \ref{lem:expansions}. Suppose it is $1$. We keep Notation \ref{not:P2}, \ref{not:bl} for Configuration \ref{conf:31}\ref{item:A1A2_C2C3-cusp-41} defining \ref{def:A1A2_C2C3-cusp-41}, and use it for curves on $X$, too. We will find disjoint $(-1)$-curves $A_1,A_2,G\subseteq X$ such that $A_i\cdot D=2$, $G\cdot D=3$; $A_1$ meets $C_1,U_{p_2}$; $A_2$ meets $C_1,E_{r}'$; and $G$ meets $L_1$, $C_2$, $E_r$. Then after contraction of $A_1$, $A_2$ the image of $D+G$ is as in Figure \ref{fig:F2n2-cuspidal}, with the image of $G$ playing the role of $E_{q}$. Hence $S$ is as in \ref{def:F2n2-cuspidal}; with expansions at $(C_1,E_{p_2};1)$ and $(C_1,E_r;1)$, and therefore we can exclude such $S$ from \ref{def:A1A2_C2C3-cusp-41}. 

Write $\cc_1\cap \cc_3=\{p',p_2,r,s\}$. Then $A_1$ is the proper transform of the unique conic passing through $p_2,p',s$ and tangent to $\ll_1$ at $q$, see Figure \ref{fig:A1A2_C2C3-cusp-41_conf}. Consider a $\P^1$-fibration induced by $|2U_{p_2}+2A_1+U_{p_2}'+C_3|$. The horizontal part $D\hor$ of $D$ consists of a $2$-section $C_1$ and $1$-sections $E_r$, $C_2$. Let $F_{T}$ be the fiber containing $T\de L_{qs}+U_{q}+L_{qr}+U_{r}=[1,2,2,2]$. Since $T$ meets $C_1$, $C_2$ and $E_r$, we have $1=(F_{T}-T)\cdot D\hor=(F_{T}-T)\cdot C_1$. In particular, $F_{T}$ is disjoint from $D-D\hor-T$; and since every component of $F_{T}-T$ meets $T$ at most once, Lemma \ref{lem:no_lines} implies that $A_2\de F_{T}-T$ has the required properties. Arguing similarly for the fiber $F$ containing $L_{1}+E_{p_1}+E_{p_1}'=[3,1,2]$, we infer the existence of $G$; which ends the proof of the claim.

		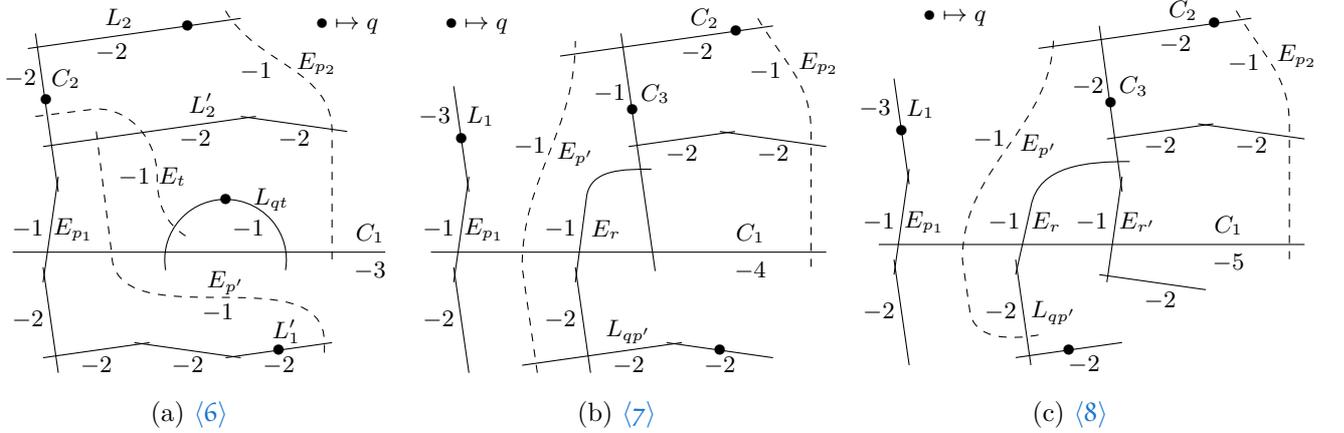
\begin{figure}[htbp]
			\begin{tabular}{ccc}
				\begin{subfigure}[b]{0.3\textwidth}
						\begin{tikzpicture}
						\draw (0,-0.2) -- (1.4,0);
						\node at (0.7,-0.3) {\small{$-2$}};
						\draw (1.2,0) -- (2.6,-0.2);
						\node at (1.9,-0.3) {\small{$-2$}};
						\draw (2.4,-0.2) -- (3.8,0);
						\node at (3.2,0.15) {\small{$L_{1}'$}};
						\node at (3.1,-0.3) {\small{$-2$}};
						\node at (3.1,-0.1) {$\bullet$};
						\draw (0.2,-0.4) -- (0,1);
						\node at (-0.2,0.3) {\small{$-2$}};
						\draw (0,0.8) -- (0.2,2.2);
						\node at (-0.2,1.5) {\small{$-1$}};
						\node at (0.4,1.5) {\small{$E_{p_1}$}};
						\draw (0.2,2) -- (-0.1,4.1);
						\node at (-0.3,3.5) {\small{$-2$}};
						\node at (0.3,3.5) {\small{$C_2$}};	
						\node at (0.04,3.22) {$\bullet$};
						\draw (0,2.6) -- (2.8,3);
						\node at (2.1,3.15) {\small{$L_{2}'$}};
						\node at (2,2.7) {\small{$-2$}};
						\draw (-0.2,3.9) -- (2.6,4.3);
						\node at (1,4.3) {\small{$L_{2}$}};
						\node at (0.9,3.85) {\small{$-2$}};
						\node at (1.9,4.2) {$\bullet$};
						\draw (2.6,3) -- (4,2.8);
						\node at (3.3,2.7) {\small{$-2$}};
						\draw[dashed] (2.4,4.4) to[out=-60,in=90] (3.8,2.8) -- (3.8,1);
						\node at (2.8,3.6) {\small{$-1$}};
						\node at (3.6,3.7) {\small{$E_{p_2}$}};				
						\draw[dashed] (-0.1,3) -- (0.7,3.1) to[out=0,in=90] (1.5,2.2) to[out=-90,in=150] (1.9,1.4);
						\node at (1.2,2.2) {\small{$-1$}};
						\node at (1.7,2.2) {\small{$E_{t}$}};		
						\draw[dashed] (0.7,2.8) -- (0.9,1.2) to[out=-90,in=90] (3.7,0) -- (3.7,-0.2);
						\node at (2.4,0.8) {\small{$E_{p'}$}};
						\node at (2.3,0.4) {\small{$-1$}};
						\draw (-0.4,1.2) -- (4.5,1.2);
						\node at (4.3,1.45) {\small{$C_1$}};
						\node at (4.3,0.95) {\small{$-3$}};
						\draw (2.4,1.1) [partial ellipse=190:-10 : 0.8 and 0.8];	
						\node at (3,1.9) {\small{$L_{qt}$}};
						\node at (2.7,1.5) {\small{$-1$}};
						\node at (2.4,1.9) {$\bullet$};
						\node at (4,4.2) {\small{$\bullet\mapsto q$}};			
					\end{tikzpicture}
					\caption{\ref{def:A1A2_C2C3-node}}
					\label{fig:A1A2_C2C3-node}
				\end{subfigure}
				&
				\begin{subfigure}[b]{0.32\textwidth}
						\begin{tikzpicture}
						\draw (-1.4,-0.3) -- (0.7,0);
						\node at (0,-0.3) {\small{$-2$}};
						\draw (0.5,0) -- (1.9,-0.2);
						\node at (1.2,-0.3) {\small{$-2$}};
						\node at (0,0.15) {\small{$L_{qp'}$}};
						\node at (1.2,-0.1) {$\bullet$};
						\draw (-0.5,-0.4) -- (-0.7,1);
						\node at (-0.9,0.3) {\small{$-2$}};
						\draw (-0.7,0.8) -- (-0.55,1.95) to[out=80,in=180] (0.3,2.3); 
						\node at (-0.9,1.5) {\small{$-1$}};
						\node at (-0.3,1.5) {\small{$E_{r}$}};
						\draw (0.35,0.95) -- (-0.1,4.1);
						\node at (-0.25,3.3) {\small{$-1$}};
						\node at (0.35,3.3) {\small{$C_3$}};	
						\node at (0.05,3.08) {$\bullet$};
						\draw (0,2.6) -- (1.4,2.8);
						\node at (0.7,2.5) {\small{$-2$}};
						\draw (-0.9,3.8) -- (1.9,4.2);
						\node at (1,4.3) {\small{$C_{2}$}};
						\node at (0.9,3.85) {\small{$-2$}};
						\node at (1.41,4.13) {$\bullet$};
						\draw (1.2,2.8) -- (2.6,2.6);
						\node at (1.9,2.5) {\small{$-2$}};
						\draw[dashed] (1.7,4.3) to[out=-60,in=90] (2.4,2.8) -- (2.4,1);
						\node at (1.8,3.6) {\small{$-1$}};
						\node at (2.5,3.7) {\small{$E_{p_2}$}};				
						\draw[dashed] (-1.2,-0.4) -- (-1.4,1) to[out=90,in=-90] (-0.7,4);
						\node at (-1.3,2.6) {\small{$-1$}};
						\node at (-0.7,2.5) {\small{$E_{p'}$}};		
						\draw (-2.1,-0.4) -- (-2.3,1);
						\node at (-2.5,0.3) {\small{$-2$}};
						\draw (-2.3,0.8) -- (-2.1,2.2); 
						\node at (-2.5,1.5) {\small{$-1$}};
						\node at (-1.9,1.5) {\small{$E_{p_1}$}};
						\draw (-2.1,2) -- (-2.3,3.4);
						\node at (-2.55,3) {\small{$-3$}};
						\node at (-1.95,3) {\small{$L_{1}$}};
						\node at (-2.2,2.7) {$\bullet$};		
						\draw (-2.6,1.2) -- (2.6,1.2);
						\node at (1.6,1.45) {\small{$C_1$}};
						\node at (1.6,0.95) {\small{$-4$}};
						\node at (-2,4.2) {\small{$\bullet\mapsto q$}};			
					\end{tikzpicture}
					\caption{\ref{def:A1A2_C2C3-cusp-41}}
					\label{fig:A1A2_C2C3-cusp-41}
				\end{subfigure}
				&
				\begin{subfigure}[b]{0.34\textwidth}
						\begin{tikzpicture}
						\draw (-1.2,-0.3) -- (0.2,-0.1);
						\node at (-0.3,-0.4) {\small{$-2$}};
						\node at (-0.5,-0.2) {$\bullet$};
						\draw (-1,-0.4) -- (-1.2,1);
						\node at (-1.4,0.3) {\small{$-2$}};
						\node at (-0.7,0.3) {\small{$L_{qp'}$}};		
						\draw (-1.2,0.8) -- (-1,1.7) to[out=80,in=180] (0.3,2.3); 
						\node at (-1.35,1.5) {\small{$-1$}};
						\node at (-0.8,1.5) {\small{$E_{r}$}};
						\draw (0.2,1.9) -- (-0.1,4.1);
						\node at (-0.25,3.3) {\small{$-2$}};
						\node at (0.35,3.3) {\small{$C_3$}};	
						\node at (0.05,3.08) {$\bullet$};
						\draw (0.2,2.1) -- (0,0.7);
						\node at (-0.2,1.5) {\small{$-1$}};
						\node at (0.4,1.5) {\small{$E_{r'}$}};
						\draw (-0.1,0.8) -- (1.3,0.6);
						\node at (0.7,0.45) {\small{$-2$}};
						\draw (0,2.6) -- (1.4,2.8);
						\node at (0.7,2.5) {\small{$-2$}};
						\draw (-0.9,3.8) -- (1.9,4.2);
						\node at (1,4.3) {\small{$C_{2}$}};
						\node at (0.9,3.85) {\small{$-2$}};
						\node at (1.41,4.13) {$\bullet$};
						\draw (1.2,2.8) -- (2.6,2.6);
						\node at (1.9,2.5) {\small{$-2$}};
						\draw[dashed] (1.7,4.3) to[out=-60,in=90] (2.4,2.8) -- (2.4,1);
						\node at (1.8,3.6) {\small{$-1$}};
						\node at (2.5,3.7) {\small{$E_{p_2}$}};				
						\draw[dashed] (-0.9,0) to[out=190,in=-80] (-1.8,0.3) -- (-1.9,1) to[out=90,in=-90] (-0.7,4);
						\node at (-1.55,2.6) {\small{$-1$}};
						\node at (-0.9,2.5) {\small{$E_{p'}$}};		
						\draw (-2.6,-0.4) -- (-2.8,1);
						\node at (-3,0.3) {\small{$-2$}};
						\draw (-2.8,0.8) -- (-2.6,2.2); 
						\node at (-3,1.5) {\small{$-1$}};
						\node at (-2.4,1.5) {\small{$E_{p_1}$}};
						\draw (-2.6,2) -- (-2.8,3.4);
						\node at (-3.05,3) {\small{$-3$}};
						\node at (-2.45,3) {\small{$L_{1}$}};
						\node at (-2.7,2.7) {$\bullet$};		
						\draw (-3,1.2) -- (2.6,1.2);
						\node at (1.6,1.45) {\small{$C_1$}};
						\node at (1.6,0.95) {\small{$-5$}};
						\node at (-2,4.2) {\small{$\bullet\mapsto q$}};			
					\end{tikzpicture}
					\caption{\ref{def:A1A2_C2C3-cusp-32}}
					\label{fig:A1A2_C2C3-cusp-32}
				\end{subfigure}
			\end{tabular}
			\caption{Cases when $X_{\min}\cong \P(1,2,3)$, $n=1$, $\qq$ is a cubic and $\cc$ is a conic.}
			\label{fig:A1A2_32}
			\end{figure}

\smallskip
Consider now the case when $\deg\qq=4$, so $\cc$ is a line. By Lemma \ref{lem:P2-curves}\ref{item:P2-adjunction}, $\delta_{\qq}=3$. 
If $a\neq b$ are two singular points of $\qq$ then the line $\ll_{ab}$ joining them meets $\qq$ only at $a$, $b$ with multiplicity $2$, hence $L_{ab}^{2}=-1$ and $L_{ab}$ meets $C$, $U_{a}$ and $U_{b}$. 
 If $q\in \{a,b\}$ then $L_{ab}\cdot \tilde{D}''=3$, otherwise  $L_{ab}\cdot \tilde{D}''=4$ and $L_{ab}$ meets $L$.

Consider the case when $\qq$ has three ordinary singular points. The projection from $q$ induces a morphism $(\bl{q}^{-1})_{*}\qq\to\P^{1}$ of degree two if $q\not\in \cc$ and three otherwise, ramified at the point infinitely near to $q$ and at $\bl{q}^{-1}(\Sing \qq\setminus \{q\})$. Since $\etop((\bl{q}^{-1})_{*}\qq)=2-\nu\geq 1$, it follows from the Hurwitz formula that, say, $\Sing \qq=\{q,a,b\}$, $a\in \qq$ is a node and $b\in \qq$ is a cusp. In particular, $\cc$ meets $\qq$ in a single point, with multiplicity $4$. Therefore, $D''\de \tilde{D}''+L_{ab}+L_{bq}+L_{qa}$, and hence $S$, is as in  \ref{def:A1A2_q-nnc}. Indeed, the curves 
$Q$, $C$, $L$, $U_{q}$, $U_{a}$, $U_{b}$, $L_{ab}$, $L_{bq}$, $L_{qa}$, $E_{q}$
defined above correspond there to 
$C_{1}$, $C_{2}$, $L_{1}$, $L_{p_{1}q_{2}}$, $L_{pr}$, $L_{2}$, $E_{p}$, $E_{q_{2}}$, $E_{r}$, $E_{p_{1}}$, respectively.

	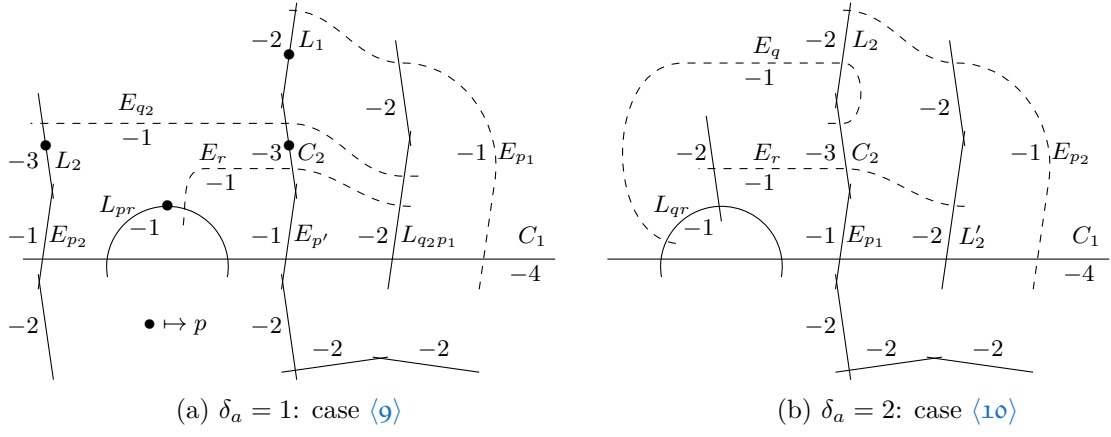
\begin{figure}[htbp]
			\begin{tabular}{cc}
				\begin{subfigure}[b]{0.45\textwidth}
						\begin{tikzpicture}
						\draw (-2.6,-0.4) -- (-2.8,1);
						\node at (-3,0.3) {\small{$-2$}};
						\draw (-2.8,0.8) -- (-2.6,2.2); 
						\node at (-3,1.5) {\small{$-1$}};
						\node at (-2.4,1.5) {\small{$E_{p_2}$}};
						\draw (-2.6,2) -- (-2.8,3.4);
						\node at (-3.0,2.5) {\small{$-3$}};
						\node at (-2.4,2.5) {\small{$L_{2}$}};
						\node at (-2.7,2.7) {$\bullet$};		
						\draw (-3,1.2) -- (4,1.2);
						\node at (3.7,1.5) {\small{$C_1$}};
						\node at (3.6,1) {\small{$-4$}};
						\draw (-1.1,1.1) [partial ellipse=190:-10 : 0.8 and 0.8];
						\node at (-1.75,1.9) {\small{$L_{pr}$}};
						\node at (-1.4,1.6) {\small{$-1$}};
						\node at (-1.1,1.9) {$\bullet$};
						\draw (0.6,-0.4) -- (0.4,1);
						\node at (0.2,0.3) {\small{$-2$}};
						\draw (0.4,0.8) -- (0.6,2.2); 
						\node at (0.2,1.5) {\small{$-1$}};
						\node at (0.8,1.5) {\small{$E_{p'}$}};
						\draw (0.6,2) -- (0.4,3.4);
						\node at (0.2,2.6) {\small{$-3$}};
						\node at (0.8,2.6) {\small{$C_{2}$}};
						\node at (0.5,2.7) {$\bullet$};		
						\draw (0.4,3.2) -- (0.6,4.6);
						\node at (0.2,4.1) {\small{$-2$}};
						\node at (0.8,4.1) {\small{$L_1$}};
						\node at (0.5,3.9) {$\bullet$};
						\draw (0.4,-0.3) -- (1.8,-0.1);
						\node at (1,0) {\small{$-2$}};
						\draw (1.6,-0.1) -- (3,-0.3);
						\node at (2.4,0) {\small{$-2$}};		
						\draw (1.8,0.8) -- (2.1,2.9); 
						\node at (1.6,1.5) {\small{$-2$}};
						\node at (2.35,1.5) {\small{$L_{q_{2}p_{1}}$}};
						\draw (2.1,2.7) -- (1.9,4.1);
						\node at (1.7,3.2) {\small{$-2$}};
						\draw[dashed] (2.2,2.3) -- (2,2.3) to[out=180,in=0] (0.5,3) -- (-2.9,3);
						\node at (-1.5, 3.25) {\small{$E_{q_{2}}$}};		
						\node at (-1.5, 2.8) {\small{$-1$}};
						\draw [dashed] (0.5,4.5) to[out=0,in=180] (2,3.8) to[out=0,in=80] (3.2,2.2) -- (3,0.8);  
						\node at (2.9,2.6) {\small{$-1$}};
						\node at (3.5,2.6) {\small{$E_{p_1}$}};
						\draw[dashed] (2.1,1.9) -- (2.1,1.9) to[out=180,in=0] (0.5,2.4) -- (-0.6,2.4) to[out=180,in=80] (-0.9,1.6);
						\node at (-0.5, 2.6) {\small{$E_{r}$}};		
						\node at (-0.4, 2.2) {\small{$-1$}};
						\node at (-1,0.3) {\small{$\bullet\mapsto p$}};				
					\end{tikzpicture}
					\caption{$\delta_{a}=1$: case \ref{def:A1A2_q-nnc}}
					\label{fig:A1A2_q-nnc}
				\end{subfigure}
				&
				\begin{subfigure}[b]{0.45\textwidth}
						\begin{tikzpicture}
						\draw (-2.6,1.2) -- (4,1.2);
						\node at (3.7,1.5) {\small{$C_1$}};
						\node at (3.6,1) {\small{$-4$}};
						\draw (-1.1,1.1) [partial ellipse=190:-10 : 0.8 and 0.8];
						\node at (-1.75,1.9) {\small{$L_{qr}$}};
						\node at (-1.4,1.6) {\small{$-1$}};
						\draw (-1.1,1.7) -- (-1.3,3.1);
						\node at (-1.5,2.6) {\small{$-2$}};
						\draw[dashed] (2.1,1.9) -- (2.1,1.9) to[out=180,in=0] (0.5,2.4) -- (-1.4,2.4);
						\node at (-0.5, 2.6) {\small{$E_{r}$}};		
						\node at (-0.6, 2.2) {\small{$-1$}};
						\draw (0.6,-0.4) -- (0.4,1);
						\node at (0.2,0.3) {\small{$-2$}};
						\draw (0.4,0.8) -- (0.6,2.2); 
						\node at (0.2,1.5) {\small{$-1$}};
						\node at (0.8,1.5) {\small{$E_{p_1}$}};
						\draw (0.6,2) -- (0.4,3.4);
						\node at (0.2,2.6) {\small{$-3$}};
						\node at (0.8,2.6) {\small{$C_{2}$}};;		
						\draw (0.4,3.2) -- (0.6,4.6);
						\node at (0.2,4.1) {\small{$-2$}};
						\node at (0.8,4.1) {\small{$L_2$}};
						\draw (0.4,-0.3) -- (1.8,-0.1);
						\node at (1,0) {\small{$-2$}};
						\draw (1.6,-0.1) -- (3,-0.3);
						\node at (2.4,0) {\small{$-2$}};		
						\draw (1.8,0.8) -- (2.1,2.9); 
						\node at (1.6,1.5) {\small{$-2$}};
						\node at (2.2,1.5) {\small{$L_{2}'$}};
						\draw (2.1,2.7) -- (1.9,4.1);
						\node at (1.7,3.2) {\small{$-2$}};
						\draw [dashed] (0.5,4.5) to[out=0,in=180] (2,3.8) to[out=0,in=80] (3.2,2.2) -- (3,0.8); 
						\node at (2.9,2.6) {\small{$-1$}};
						\node at (3.5,2.6) {\small{$E_{p_2}$}};
						\draw [dashed] (0.3,3) -- (0.5,3) to[out=0,in=0] (0.5,3.8) -- (0.3,3.8) to[out=180,in=0] (-1.6,3.8)
						to[out=180,in=90] (-2.4,2.6) to[out=-90,in=180] (-1.6,1.4);
						\node at (-0.5, 4) {\small{$E_{q}$}};		
						\node at (-0.6, 3.6) {\small{$-1$}};
					\end{tikzpicture}
					\caption{$\delta_{a}=2$: case \ref{def:A1A2_q-nc}}
					\label{fig:A1A2_q-nc}
				\end{subfigure}
			\end{tabular}	
		\caption{Cases when $X_{\min}\cong \P(1,2,3)$, $n=1$, $a\in \qq$ is a node and $\delta_{a}\leq 2$.}
		\label{fig:A1A2_a-node}
	\end{figure}

Consider now the case when $a\in\qq$ has $\delta_{a}\in \{2,3\}$, so $a\neq q$. Denote by $\ll_{a}$ the line tangent to $\qq$ at $a$. Lemma \ref{lem:P2-curves}\ref{item:P2-tangent} gives $(\ll_{a}\cdot \qq)_{a}=4$. Hence $L_{a}^{2}=-1$, $L_{a}\cdot \tilde{D}''=3$ and $L_{a}$ meets $C$, $L$ and $U_{a}'$. 

	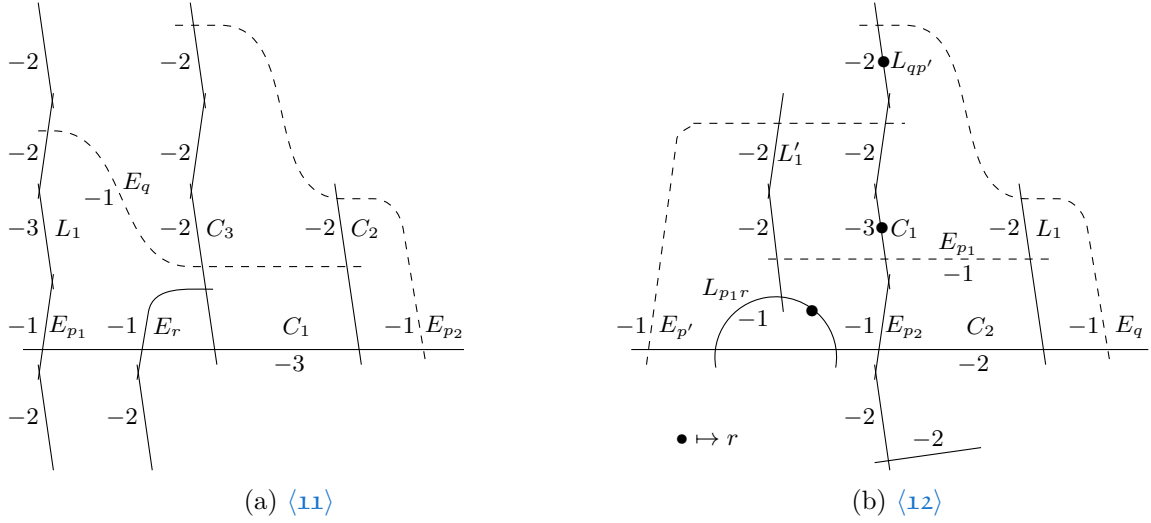
\begin{figure}[htbp]
		\begin{tabular}{cc}
			\begin{subfigure}[b]{0.45\textwidth}
			\begin{tikzpicture}
				\draw (-2.6,-0.4) -- (-2.8,1);
				\node at (-3,0.3) {\small{$-2$}};
				\draw (-2.8,0.8) -- (-2.6,2.2); 
				\node at (-3,1.5) {\small{$-1$}};
				\node at (-2.4,1.5) {\small{$E_{p_1}$}};
				\draw (-2.6,2) -- (-2.8,3.4);
				\node at (-3.0,2.8) {\small{$-3$}};
				\node at (-2.4,2.8) {\small{$L_{1}$}};
				\draw (-2.8,3.2) -- (-2.6,4.6);
				\node at (-3.0,3.8) {\small{$-2$}};
				\draw (-2.6,4.4) -- (-2.8,5.8);	
				\node at (-3.0,5) {\small{$-2$}};	
				\draw (-1.3,-0.4) -- (-1.5,1);
				\node at (-1.7,0.3) {\small{$-2$}};
				\draw (-1.5,0.8) -- (-1.35,1.7) to[out=80,in=180] (-0.5,2); 
				\node at (-1.7,1.5) {\small{$-1$}};
				\node at (-1.1,1.5) {\small{$E_{r}$}};								
				\draw (-0.45,1) -- (-0.8,3.4);
				\node at (-1.0,2.8) {\small{$-2$}};
				\node at (-0.4,2.8) {\small{$C_{3}$}};
				\draw (-0.8,3.2) -- (-0.6,4.6);
				\node at (-1.0,3.8) {\small{$-2$}};
				\draw (-0.6,4.4) -- (-0.8,5.8);	
				\node at (-1.0,5) {\small{$-2$}};
				\draw[dashed] (-2.8,4.1) -- (-2.6,4.1) to[out=0,in=180] (-0.8,2.3) -- (1.5,2.3);
				\node at (-2,3.2) {\small{$-1$}};
				\node at (-1.5,3.4) {\small{$E_q$}};
				\draw[dashed] (-1,5.5) -- (-0.4,5.5) to[out=0,in=180] (1.2,3.2) -- (1.6,3.2) to[out=0,in=100] 		(2,2.8) -- (2.3,1);
				\node at (1.95,1.5) {\small{$-1$}};
				\node at (2.55,1.5) {\small{$E_{p_2}$}};	
				\draw (1.45,1) -- (1.1,3.4);
				\node at (0.9,2.8) {\small{$-2$}};
				\node at (1.5,2.8) {\small{$C_{2}$}};			
				\draw (-3,1.2) -- (2.8,1.2);
				\node at (0.6,1.5) {\small{$C_1$}};
				\node at (0.5,1) {\small{$-3$}};
			\end{tikzpicture}	
				\caption{\ref{def:A1A2_3-cusp}}
				\label{fig:A1A2_3-cusp}
			\end{subfigure}
			&
			\begin{subfigure}[b]{0.45\textwidth}
			\begin{tikzpicture}
				\draw (-4,1.2) -- (2.8,1.2);
				\node at (0.6,1.5) {\small{$C_2$}};
				\node at (0.5,1) {\small{$-2$}};
				\draw (-2.1,1.1) [partial ellipse=190:-10 : 0.8 and 0.8];
				\node at (-1.62,1.7) {\large{$\bullet$}};
				\node at (-2.75,2) {\small{$L_{p_{1} r}$}};
				\node at (-2.4,1.6) {\small{$-1$}};
				\draw (-2,1.7) -- (-2.2,3.4);
				\node at (-2.4,2.8) {\small{$-2$}};
				\draw (-2.2,3.2) -- (-2,4.6);
				\node at (-2.4,3.8) {\small{$-2$}};
				\node at (-1.9,3.8) {\small{$L'_1$}};
				\node at (-4,1.5) {\small{$-1$}};
				\node at (-3.4,1.5) {\small{$E_{p'}$}};				
				\draw[dashed] (-3.8,1) -- (-3.4,4) to[out=80,in=0] (-3.2,4.2) -- (-0.4,4.2);
				\node at (0.3,2.6) {\small{$E_{p_1}$}};
				\node at (0.3,2.2) {\small{$-1$}};
				\draw[dashed] (-2.2,2.4) -- (1.5,2.4);
				\draw (-0.8,-0.3) -- (0.6,-0.1);
				\node at (-0.1,0) {\small{$-2$}};				
				\draw (-0.6,-0.4) -- (-0.8,1);
				\node at (-1,0.3) {\small{$-2$}};
				\draw (-0.8,0.8) -- (-0.6,2.2); 
				\node at (-1,1.5) {\small{$-1$}};
				\node at (-0.4,1.5) {\small{$E_{p_2}$}};			
				\draw (-0.6,2) -- (-0.8,3.4);
				\node at (-1.0,2.8) {\small{$-3$}};
				\node at (-0.4,2.8) {\small{$C_{1}$}};
				\node at (-0.7,2.8) {\large{$\bullet$}};
				\draw (-0.8,3.2) -- (-0.6,4.6);
				\node at (-1.0,3.8) {\small{$-2$}};
				\draw (-0.6,4.4) -- (-0.8,5.8);	
				\node at (-1.0,5) {\small{$-2$}};
				\node at (-0.3,5) {\small{$L_{qp'}$}};
				\node at (-0.68,5) {\large{$\bullet$}};
				\draw[dashed] (-1,5.5) -- (-0.4,5.5) to[out=0,in=180] (1.2,3.2) -- (1.6,3.2) to[out=0,in=100] 		(2,2.8) -- (2.3,1);
				\node at (1.95,1.5) {\small{$-1$}};
				\node at (2.55,1.5) {\small{$E_{q}$}};	
				\draw (1.45,1) -- (1.1,3.4);
				\node at (0.9,2.8) {\small{$-2$}};
				\node at (1.5,2.8) {\small{$L_{1}$}};
				\node at (-3,0) {\small{$\bullet\mapsto r$}};										
			\end{tikzpicture}
				\caption{\ref{def:A1A2_3-node}}
				\label{fig:A1A2_3-node}
			\end{subfigure}
		\end{tabular}
		\caption{Cases when $X_{\min}\cong \P(1,2,3)$, $n=1$ and $\delta_{a}=3$.}
		\label{fig:A1A2_3}
	\end{figure}

Consider the case $\delta_{a}=3$, so $\Sing \qq=\{a\}$. Assume that $a\in \qq$ is a cusp. Then $\cc$ meets $\qq\setminus \{q\}$ in two points, with multiplicities $2$, $1$. Thus $D''\de \tilde{D}''+L_{a}$, and hence $S$, is as in  \ref{def:A1A2_3-cusp}. Indeed, the curves 
$Q$, $L$, $C$, $U_{a}''$, $L_{a}$, $E_{q}$
defined above correspond there to
$C_{1}$, $C_{2}$, $C_{3}$, $L_{1}$, $E_{q}$, $E_{p_2}$,
respectively.  
Assume that $a\in \qq$ is a node, so $\cc$ meets $\qq\setminus \{q\}$ in one point, with multiplicity $3$.   
The pencil of conics tangent to $\qq$ at $a$, $q$  contains a unique member, say $\cc_{a}$, such that $(\cc_{a}\cdot \qq)_{a}=6$. The only degenerate members of this pencil are $\ll_{a}+\ll$ and $2\ll_{aq}$, so $\cc_{a}$ is smooth. Thus $C_{a}\subseteq X'$ is a $(-1)$-curve, $C_{a}\cdot \tilde{D}''=3$ and $C_{a}$ meets $\tilde{D}''$ in $U_{q}'$, $C$ and $U_{a}'$. Now $D''\de \tilde{D}''+L_{a}+L_{aq}+C_{a}$, and hence $S$, is as in  \ref{def:A1A2_3-node}. Indeed, the curves 
$Q$, $C$, $L$, $U_{q}'$, $U_{a}$, $U_{a}''$, $L_{a}$, $L_{aq}$, $C_{a}$, $E_{q}$
defined above correspond there to 
$C_{2}$, $C_{1}$, $L_{1}$, $L_{qp'}$, $L_{1}'$, $L_{rp_{1}}$, $E_{p_{1}}$, $E_{p'}$, $E_{r}$, $E_{q}$, respectively.

		\begin{figure}[htbp]
	\begin{tabular}{cc}
		\begin{subfigure}[b]{0.45\textwidth}
						\begin{tikzpicture}
				\draw (-2.6,-0.4) -- (-2.8,1);
				\node at (-3,0.3) {\small{$-2$}};
				\draw (-2.8,0.8) -- (-2.6,2.2); 
				\node at (-3,1.5) {\small{$-1$}};
				\node at (-2.4,1.5) {\small{$E_{p_2}$}};
				\draw (-2.6,2) -- (-2.8,3.4);
				\node at (-3.0,2.8) {\small{$-3$}};
				\node at (-2.4,2.8) {\small{$L_{2}$}};
				\node at (-2.7,2.8) {\large{$\bullet$}};
				\draw (-2.8,3.2) -- (-2.6,4.6);
				\node at (-3.0,3.8) {\small{$-2$}};	
				\draw (-1.3,-0.4) -- (-1.5,1);
				\node at (-1.7,0.3) {\small{$-2$}};
				\draw (-1.5,-0.3) -- (-0.1,-0.1);
				\node at (-0.8,0) {\small{$-2$}};
				\draw (-1.5,0.8) -- (-1.35,1.7) to[out=80,in=180] (-0.5,2); 
				\node at (-1.7,1.5) {\small{$-1$}};
				\node at (-1.1,1.5) {\small{$E_{p'}$}};								
				\draw (-0.45,1) -- (-0.8,3.4);
				\node at (-1.0,2.8) {\small{$-2$}};
				\node at (-0.4,2.8) {\small{$C_{2}$}};
				\node at (-0.7,2.8) {\large{$\bullet$}};
				\draw (-0.8,3.2) -- (-0.6,4.6);
				\node at (-1.0,3.8) {\small{$-2$}};
				\node at (-0.4,3.8) {\small{$L_1$}};
				\node at (-0.7,3.8) {\large{$\bullet$}};
				\draw[dashed] (-2.8,4.1) -- (-2.6,4.1) to[out=0,in=180] (-0.8,2.3) -- (1.5,2.3);
				\node at (-2,3.2) {\small{$-1$}};
				\node at (-1.5,3.4) {\small{$E_r$}};
				\draw[dashed] (-1,4.4) -- (1.4,4.4) to[out=0,in=100] 	(2,4.2) -- (2.4,1);
				\node at (2.05,1.5) {\small{$-1$}};
				\node at (2.65,1.5) {\small{$E_{p_1}$}};	
				\draw (1.45,1) -- (1.1,3.4);
				\node at (0.9,2.8) {\small{$-2$}};
				\node at (1.6,2.8) {\small{$L_{rp_1}$}};
				\draw (1.1,3.2) -- (1.3,4.6);
				\node at (0.9,3.8) {\small{$-2$}};			
				\draw (-3,1.2) -- (2.8,1.2);
				\node at (0.6,1.5) {\small{$C_1$}};
				\node at (0.5,1) {\small{$-3$}};
				\node at (2,0) {\small{$\bullet\mapsto p$}};					
			\end{tikzpicture}
			\caption{\ref{def:A1A2_q-cn_31}}
			\label{fig:A1A2_q-cn_31}
		\end{subfigure}
		&
		\begin{subfigure}[b]{0.45\textwidth}
			\begin{tikzpicture}
				\draw (-2.6,-0.4) -- (-2.8,1);
				\node at (-3,0.3) {\small{$-2$}};
				\draw (-2.8,0.8) -- (-2.6,2.2); 
				\node at (-3,1.5) {\small{$-1$}};
				\node at (-2.4,1.5) {\small{$E_{p_2}$}};
				\draw (-2.6,2) -- (-2.8,3.4);
				\node at (-3,2.9) {\small{$-3$}};
				\node at (-2.5,2.9) {\small{$L_{2}$}};
				\draw (-2.8,3.2) -- (-2.6,4.6);
				\node at (-3.0,3.8) {\small{$-2$}};	
				\node at (-2.67,4.1) {\large{$\bullet$}};
				\draw (-1,-0.4) -- (-1.2,1);
				\node at (-1.4,0.3) {\small{$-2$}};
				\draw (-1.2,0.8) -- (-1,2.2); 
				\node at (-1.4,1.5) {\small{$-1$}};
				\node at (-0.8,1.5) {\small{$E_{p'}$}};
				\draw (1,-0.4) -- (0.8,1);
				\node at (0.6,0.3) {\small{$-2$}};
				\draw (0.8,0.8) -- (1,2.2); 
				\node at (0.6,1.5) {\small{$-1$}};
				\node at (1.2,1.5) {\small{$E_{p''}$}};
				\draw (-1.2,2.1) -- (1.2,2.1);
				\node at (0.2,1.9) {\small{$-3$}};
				\node at (0.2,2.3) {\small{$C_2$}};
				\node at (0.7,2.1) {\large{$\bullet$}};
				\draw (-0.6,2) -- (-0.8,3.4);
				\node at (-1,2.9) {\small{$-2$}};
				\node at (-0.5,2.9) {\small{$L_{1}$}};
				\draw[dashed] (-2.8,2.6) -- (-0.5,2.6) to[out=0,in=100] (-0.3,2.35) -- (-0.25,2);
				\node at (-1.8,2.4) {\small{$-1$}};
				\node at (-1.8,2.8) {\small{$E_p$}};
				\draw (2.4,0.8) -- (2.6,2.2); 
				\node at (2.2,1.5) {\small{$-2$}};
				\node at (2.85,1.45) {\small{$L_{rp_1}$}};
				\node at (2.54,1.7) {\large{$\bullet$}};
				\draw (2.6,2) -- (2.4,3.4);
				\node at (2.2,2.7) {\small{$-2$}};
				\draw[dashed] (-1,3.2) -- (3,3.2) to[out=0,in=100] 	(3.2,3) -- (3.4,1.1);
				\node at (1,3) {\small{$-1$}};
				\node at (1,3.4) {\small{$E_{p_1}$}};	
				\draw (-3,1.2) -- (3.5,1.2);
				\node at (0,1.4) {\small{$C_1$}};
				\node at (-0.1,1) {\small{$-4$}};
				\node at (2,0) {\small{$\bullet\mapsto r$}};					
			\end{tikzpicture}
			\caption{\ref{def:A1A2_q-cn_22}}
			\label{fig:A1A2_q-cn_22}
		\end{subfigure}
	\end{tabular}
	\caption{Cases when $X_{\min}\cong \P(1,2,3)$, $n=1$, $a$ is a cusp with $\delta_{a}=2$, and $q\not\in\cc$.}
	\label{fig:A1A2_2s}
\end{figure}
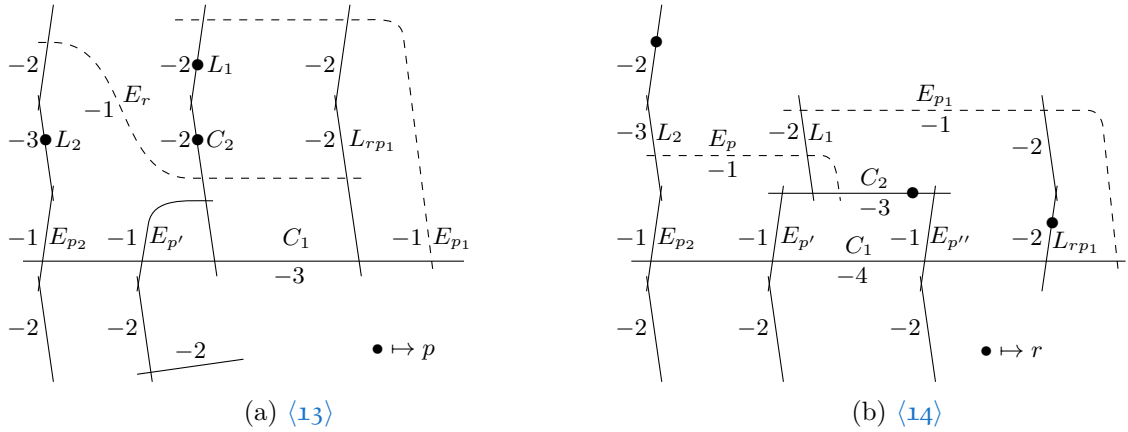

Consider now the case $\delta_{a}=2$ and $q\in \qq$ is a node, i.e.\ $q\not\in\cc$. Put $D''= \tilde{D}''+L_{a}+L_{aq}$. Assume further that $a$ is a node. Then $\cc$ meets $\qq$ in a single point. In this case, $D''$, and hence $S$, is as in  \ref{def:A1A2_q-nc}. Indeed,  the curves 
$Q$, $C$, $L$, $U_{q}$, $U_{a}'$, $L_{a}$, $L_{aq}$, $E_{q}$
defined above correspond there to
$C_{1}$, $C_{2}$, $L_{2}$, $L_{2}'$, $L_{qr}$, $E_{q}$, $E_{r}$, $E_{p_{2}}$, 
respectively. 
Assume now that $a\in \qq$ is a cusp, so $\qq$ meets $\cc$ in two points, with multiplicities $3$, $1$ or $2$, $2$. Then $D''$, and hence $S$, is as in  \ref{def:A1A2_q-cn_31} in the first case and \ref{def:A1A2_q-cn_22} in the second case. Indeed, the curves 
$Q$, $C$, $L$, $U_{q}$, $U_{a}'$, $L_{a}$, $L_{aq}$, $E_{q}$  
defined above correspond there to
$C_{1}$, $C_{2}$, $L_{1}$, $L_{rp_{1}}$, $L_{2}$, $E_{p}$, $E_{r}$, $E_{p_{1}}$, 
respectively. Note that in the second case we can assume that the expansion has center on $E_{p'}$: indeed, after renaming $p'$ and $p''$ we get (non-isomorphic) \QHPs with the same boundary graphs.

We claim that in case \ref{def:A1A2_q-cn_31}, we can assume that the weight of the expansion is not $1$. Suppose it is $1$. We keep the notation from Configuration \ref{conf:A1A2_q-cn}\ref{item:A1A2_q-cn_31} used in \ref{def:A1A2_q-cn_31}. Put $\tilde{D}=D+E_{p_1}+E_r\subseteq X$, see Figure \ref{fig:A1A2_q-cn_31}. Write $\cc_1\cap\cc_2=\{p',q\}$, and let $\ll_{qp}$, $\ll_{qp_1}$ be the lines joining the point $q$ with $p$ and $p_1$, respectively. Their proper transforms  $L_{qp_1},L_{qp}\subseteq X$ are $(-1)$-curves such that $L_{qp_1}\cdot \tilde{D}=3$, $L_{qp}\cdot \tilde{D}=2$; $L_{qp_1}$ meets $U_{p_1}$, $L_2$, and $C_2$; and $L_{qp}$ meets $C_1$ and $L_{rp_1}$. After the contraction of $L_{qp}$, the image of $\tilde{D}+L_{qp_1}$ is as in Figure \ref{fig:F2_n1-node}, with the image of $L_{qp_1}$ playing the role of $E_p$. Hence $S$ is as in \ref{def:F2_n1-node}, with expansion at $(C_1,C_2;1)$. Therefore, we can exclude this $S$ from the tower \ref{def:A1A2_q-cn_31}, as claimed.
\smallskip

		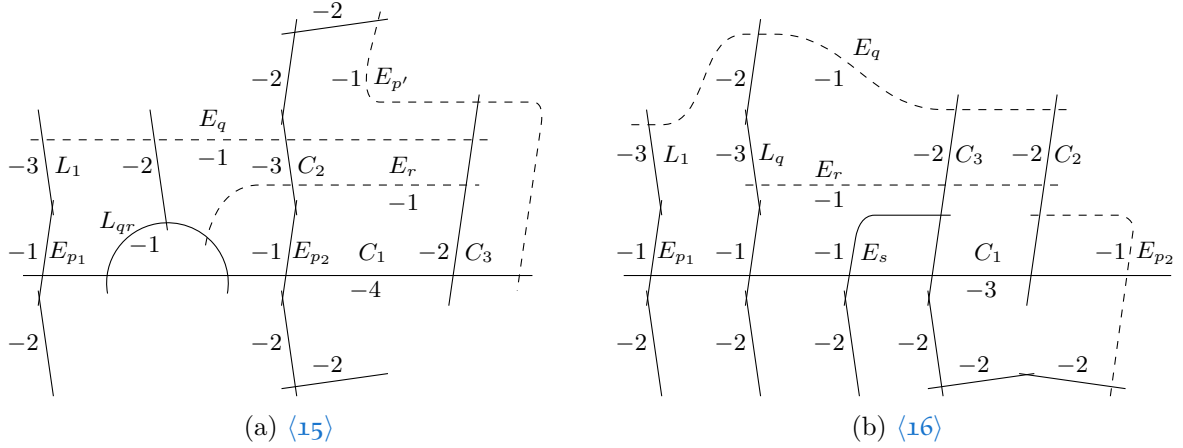
\begin{figure}[htbp]
	\begin{tabular}{cc}
		\begin{subfigure}[b]{0.45\textwidth}
			\begin{tikzpicture}
				\draw (-2.6,-0.4) -- (-2.8,1);
				\node at (-3,0.3) {\small{$-2$}};
				\draw (-2.8,0.8) -- (-2.6,2.2); 
				\node at (-3,1.5) {\small{$-1$}};
				\node at (-2.4,1.5) {\small{$E_{p_1}$}};
				\draw (-2.6,2) -- (-2.8,3.4);
				\node at (-3.0,2.65) {\small{$-3$}};
				\node at (-2.4,2.65) {\small{$L_{1}$}};		
				\draw (-3,1.2) -- (3.7,1.2);
				\node at (1.6,1.5) {\small{$C_1$}};
				\node at (1.5,1) {\small{$-4$}};
				\draw (-1.1,1.1) [partial ellipse=190:-10 : 0.8 and 0.8];
				\node at (-1.75,1.9) {\small{$L_{qr}$}};
				\node at (-1.4,1.6) {\small{$-1$}};
				\draw (-1.1,1.8) -- (-1.33,3.4);
				\node at (-1.5,2.65) {\small{$-2$}};
				\draw (0.6,-0.4) -- (0.4,1);
				\node at (0.2,0.3) {\small{$-2$}};
				\draw (0.4,0.8) -- (0.6,2.2); 
				\node at (0.2,1.5) {\small{$-1$}};
				\node at (0.8,1.5) {\small{$E_{p_2}$}};
				\draw (0.6,2) -- (0.4,3.4);
				\node at (0.2,2.65) {\small{$-3$}};
				\node at (0.8,2.65) {\small{$C_{2}$}};	
				\draw (0.4,3.2) -- (0.6,4.6);
				\node at (0.2,3.8) {\small{$-2$}};
				\draw (0.4,-0.3) -- (1.8,-0.1);
				\node at (1,0) {\small{$-2$}};
				\draw (0.4,4.4) -- (1.8,4.6);
				\node at (1,4.7) {\small{$-2$}};
				\draw (2.6,0.8) -- (3,3.6); 
				\node at (2.4,1.5) {\small{$-2$}};
				\node at (3,1.5) {\small{$C_3$}};
				\node at (1.25,3.8) {\small{$-1$}};
				\node at (1.85,3.8) {\small{$E_{p'}$}};
				\draw[dashed] (1.7,4.7) -- (1.55,4.1) to[out=-100,in=180] (1.7,3.5)-- (3.7,3.5)
				to[out=0,in=80] (3.8,3.1) -- (3.5,1);
				\draw[dashed] (-2.9,3) -- (3.2,3);
				\node at (-0.5,3.25) {\small{$E_{q}$}};
				\node at (-0.5,2.75) {\small{$-1$}};
				\draw[dashed] (-0.6,1.6) to[out=80,in=180] (0.1,2.4) -- (3,2.4);
				\node at (2,2.65) {\small{$E_{r}$}};
				\node at (2,2.15) {\small{$-1$}};
			\end{tikzpicture}
			\caption{\ref{def:A1A2_2n1c}}
			\label{fig:A1A2_2n1c}
		\end{subfigure}
		&
		\begin{subfigure}[b]{0.45\textwidth}
		\begin{tikzpicture}
			\draw (-3.9,-0.4) -- (-4.1,1);
			\node at (-4.3,0.3) {\small{$-2$}};
			\draw (-4.1,0.8) -- (-3.9,2.2); 
			\node at (-4.3,1.5) {\small{$-1$}};
			\node at (-3.7,1.5) {\small{$E_{p_1}$}};
			\draw (-3.9,2) -- (-4.1,3.4);
			\node at (-4.3,2.8) {\small{$-3$}};
			\node at (-3.7,2.8) {\small{$L_{1}$}};	
			\draw (-2.6,-0.4) -- (-2.8,1);
			\node at (-3,0.3) {\small{$-2$}};
			\draw (-2.8,0.8) -- (-2.6,2.2); 
			\node at (-3,1.5) {\small{$-1$}};
			\draw (-2.6,2) -- (-2.8,3.4);
			\node at (-3.0,2.8) {\small{$-3$}};
			\node at (-2.45,2.8) {\small{$L_{q}$}};
			\draw (-2.8,3.2) -- (-2.6,4.6);
			\node at (-3.0,3.8) {\small{$-2$}};	
			\draw (-1.3,-0.4) -- (-1.5,1);
			\node at (-1.7,0.3) {\small{$-2$}};
			\draw (-1.5,0.8) -- (-1.35,1.7) to[out=80,in=180] (-1.1,2)--(-0.1,2); 
			\node at (-1.7,1.5) {\small{$-1$}};
			\node at (-1.1,1.5) {\small{$E_{s}$}};								
			\draw (-0.4,-0.3) -- (1,-0.1);
			\node at (0.2,0) {\small{$-2$}};
			\draw (0.8,-0.1) -- (2.2,-0.3);
			\node at (1.5,0) {\small{$-2$}};
			\draw (-0.2,-0.4) -- (-0.4,1);
			\node at (-0.6,0.3) {\small{$-2$}};
			\draw (-0.4,0.8) -- (0,3.6);
			\node at (-0.4,2.8) {\small{$-2$}};
			\node at (0.15,2.8) {\small{$C_{3}$}};
			\draw[dashed] (2,-0.4) -- (2.3,1.8) to[out=80,in=0] (2.1,2) -- (0.9,2);
			\node at (2,1.5) {\small{$-1$}};
			\node at (2.6,1.5) {\small{$E_{p_2}$}};
			\draw[dashed] (-2.8,2.4) -- (1.4,2.4);
			\node at (-1.7,2.2) {\small{$-1$}};
			\node at (-1.7,2.6) {\small{$E_{r}$}};
			\draw[dashed] (-4.3,3.2) -- (-3.9,3.2) to[out=0,in=180] (-2.8,4.4) -- (-2.4,4.4) to[out=0,in=180] (-0.2,3.4) -- (1.5,3.4);
			\node at (-1.7,3.8) {\small{$-1$}};
			\node at (-1.2,4.2) {\small{$E_{q}$}};
			\draw (0.9,0.8) -- (1.3,3.6);
			\node at (0.9,2.8) {\small{$-2$}};
			\node at (1.45,2.8) {\small{$C_{2}$}};			
			\draw (-4.4,1.2) -- (2.8,1.2);
			\node at (0.4,1.5) {\small{$C_1$}};
			\node at (0.3,1) {\small{$-3$}};
		\end{tikzpicture}
			\caption{\ref{def:A1A2_21}}
			\label{fig:A1A2_21}
		\end{subfigure}
	\end{tabular}
	\caption{Cases when $X_{\min}\cong \P(1,2,3)$, $n=1$, $a,b$ are cusps, $\delta_{a}=2$, $q\in\cc$.}
	\label{fig:A1A2_2n}
\end{figure}

Eventually, consider the case when $\delta_{a}=2$ and $q\in \qq$ is smooth, so $\qq$ has some other singular point $b$ with $\delta_{b}=1$, and the line $\cc$ passes through $q$. Assume that $b\in \qq$ is a cusp and put $D''=\tilde{D}''+L_{a}+L_{ab}$. Now $\qq$ meets $\cc\setminus \{q\}$ at one point (with multiplicity $3$) if $a\in\qq$ is a node and in two points (with multiplicities $2$, $1$), if $a\in \qq$ is a cusp. Thus $D''$, and hence $S$, is as in \ref{def:A1A2_2n1c} and \ref{def:A1A2_21}, respectively. Indeed, the curves 
$Q$, $L$, $C$, $U_{a}'$, $U_{b}$, $L_{a}$, $L_{ab}$, $E_{q}$
defined above correspond there to
$C_{1}$, $C_{2}$, $C_{3}$, $\hat{L}_{q}$, $L_{1}$, $E_{r}$, $E_{q}$, $E_{p_{2}}$, 
respectively; where $\hat{L}_{q}=L_{q}$ if $a\in\qq$ is a node and $\hat{L}_{q}=L_{qr}$ if $a\in \qq$ is a cusp.

		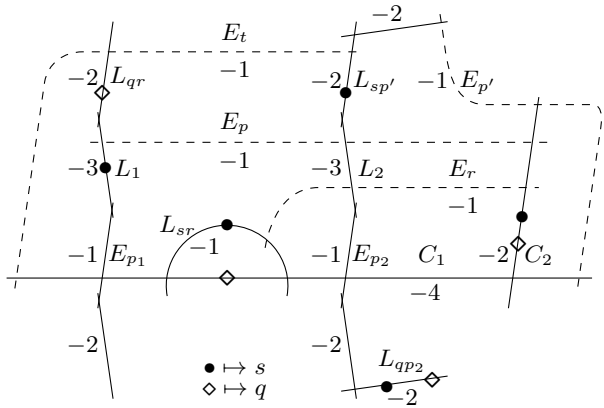
\begin{figure}[htbp]
				\begin{tikzpicture}
		\draw (-2.6,-0.4) -- (-2.8,1);
		\node at (-3,0.3) {\small{$-2$}};
		\draw (-2.8,0.8) -- (-2.6,2.2); 
		\node at (-3,1.5) {\small{$-1$}};
		\node at (-2.4,1.5) {\small{$E_{p_1}$}};
		\draw (-2.6,2) -- (-2.8,3.4);
		\node at (-3.0,2.65) {\small{$-3$}};
		\node at (-2.4,2.65) {\small{$L_{1}$}};
		\node at (-2.7,2.65) {\large{$\bullet$}};	
		\draw (-2.8,3.2) -- (-2.6,4.6);
		\node at (-3.0,3.85) {\small{$-2$}};
		\node at (-2.4,3.85) {\small{$L_{qr}$}};
		\node at (-2.74,3.65) {\large{$\boldsymbol{\diamond}$}};		
		\draw (-4,1.2) -- (3.7,1.2);
		\node at (1.6,1.5) {\small{$C_1$}};
		\node at (1.5,1) {\small{$-4$}};
		\node at (-1.1,1.2) {\large{$\boldsymbol{\diamond}$}};
		\draw (-1.1,1.1) [partial ellipse=190:-10 : 0.8 and 0.8];
		\node at (-1.75,1.9) {\small{$L_{sr}$}};
		\node at (-1.4,1.6) {\small{$-1$}};
		\node at (-1.1,1.9) {\large{$\bullet$}};
		\node at (-1,0) {\small{$\bullet\mapsto s$}};
		\node at (-1,-0.3) {\small{$\boldsymbol{\diamond}\mapsto q$}};
		\draw (0.6,-0.4) -- (0.4,1);
		\node at (0.2,0.3) {\small{$-2$}};
		\draw (0.4,0.8) -- (0.6,2.2); 
		\node at (0.2,1.5) {\small{$-1$}};
		\node at (0.8,1.5) {\small{$E_{p_2}$}};
		\draw (0.6,2) -- (0.4,3.4);
		\node at (0.2,2.65) {\small{$-3$}};
		\node at (0.8,2.65) {\small{$L_{2}$}};	
		\draw (0.4,3.2) -- (0.6,4.6);
		\node at (0.2,3.8) {\small{$-2$}};
		\node at (0.85,3.8) {\small{$L_{sp'}$}};
		\node at (0.46,3.65) {\large{$\bullet$}};
		\draw (0.4,-0.3) -- (1.8,-0.1);
		\node at (1.2,0.1) {\small{$L_{qp_2}$}};
		\node at (1,-0.25) {\large{$\bullet$}};
		\node at (1.6,-0.15) {\large{$\boldsymbol{\diamond}$}};
		\node at (1.2,-0.4) {\small{$-2$}};
		\draw (0.4,4.4) -- (1.8,4.6);
		\node at (1,4.7) {\small{$-2$}};
		\draw (2.6,0.8) -- (3,3.6); 
		\node at (2.4,1.5) {\small{$-2$}};
		\node at (3,1.5) {\small{$C_2$}};
		\node at (2.78,2) {\large{$\bullet$}};
		\node at (2.73,1.65) {\large{$\boldsymbol{\diamond}$}};
		\node at (1.6,3.8) {\small{$-1$}};
		\node at (2.2,3.8) {\small{$E_{p'}$}};
		\draw[dashed] (1.7,4.7) to[out=-80,in=180] (2.3,3.5)-- (3.7,3.5)
		to[out=0,in=80] (3.8,3.1) -- (3.5,1);
		\draw[dashed] (-2.9,3) -- (3.2,3);
		\node at (-1,3.25) {\small{$E_{p}$}};
		\node at (-1,2.75) {\small{$-1$}};
		\draw[dashed] (-0.6,1.6) to[out=80,in=180] (0.1,2.4) -- (3,2.4);
		\node at (2,2.65) {\small{$E_{r}$}};
		\node at (2,2.15) {\small{$-1$}};
		\draw[dashed] (0.6,4.2) -- (-3,4.2) to[out=180,in=80] (-3.5,3.8) -- (-3.9,1);
		\node at (-1,4.45) {\small{$E_t$}};
		\node at (-1,3.95) {\small{$-1$}};
		
	\end{tikzpicture}
	\caption{Case \ref{def:A1A2_2c1n}: $X_{\min}\cong \P(1,2,3)$, $n=1$, $a$ is a cusp, $\delta_{a}=2$, $q\not\in\cc$ and $b$ is a node.}
	\label{fig:A1A2_2c1n}
\end{figure}

Assume now that $b\in \qq$ is a node, so $a\in \qq$ is a cusp and $\qq$ meets $\cc\setminus \{q\}$ at one point, say $r$, with multiplicity $3$. The pencil of conics tangent to $\qq$ at $a$ and passing through $q,r$ contains a unique member $\cc_{b}$ passing through $b$. The singular members $\ll_{a}+\cc$ and $\ll_{aq}+\ll_{ar}$ do not contain $\{a,b\}\subseteq \cc_{b}$, so $\cc_{b}$ is smooth. Put $D''=\tilde{D}''+L_{a}+L_{ab}+L_{aq}+L_{ar}+C_{b}$. Then $D''-\tilde{D}''$ consists of disjoint $(-1)$-curves and
\begin{itemize}
	\item $L_{aq}\cdot\tilde{D}''=4$ and $L_{aq}$ meets  $\tilde{D}''$ in $Q$, $L$, $U_{a}$ and $U_{q}$,
	\item $L_{ar}\cdot\tilde{D}''=4$ and $L_{ar}$ meets  $\tilde{D}''$ in $Q$, $L$, $U_{a}$ and $U_{r}$,
	\item $C_{b}\cdot \tilde{D}''=5$ and $C_{b}$ meets  $\tilde{D}''$ in $L$, $U_{q}$, $U_{r}$, $U_{a}'$ and $U_{b}$.
\end{itemize}

Thus $D''$, and hence $S$, is as in  \ref{def:A1A2_2c1n}. Indeed, the curves 
$Q$, $L$, $C$, $U_{a}$, $U_{a}'$, $U_{b}$, $U_{q}$, $U_{r}$,
$L_{a}$, $L_{ab}$, $L_{aq}$, $L_{ar}$, $C_{b}$, $E_{q}$ 
defined above correspond there to 
$C_{1}$, $C_{2}$, $L_{2}$, $L_{rq}$, $L_{1}$, $L_{sr}$, $L_{sp'}$, $L_{qp_{2}}$,
$E_{p}$, $E_{r}$, $E_{t}$, $E_{q}$, $E_{s}$, $E_{p'}$, respectively.
\smallskip

At last, consider the case $n=0$. Let $\phi\colon X\am\to \F_2$ be the contraction of $U+T_1$. Then $\phi(T_2)$ is the negative section, and $R\de \phi(R\am)$ is as in Proposition \ref{prop:n0}\ref{item:n0_A1A2} below. Hence $S$ is as in  is as in \ref{def:A1A2_c=3}, \ref{def:A1A2_c=1}, \ref{def:A1A2_c=2_41} or  \ref{def:A1A2_c=2_32}.
	
\subsection{Case \texorpdfstring{$X_{\min}\cong \P(1,1,2)$}{Xmin=P(1,1,2)}}\label{sec:F2}
	In this case, $\alpha\colon X\am=\F_2\to \P(1,1,2)=X_{\min}$ is a blowup at the vertex, so $\Delta\am$ is the negative section, see Section \ref{sec:Fm}. Let $F$ be the fiber of $\F_2$, so $F\cdot \Delta\am=1$. Condition \eqref{eq:assumption_fibr} gives $F\cdot R\am=F\cdot D\am-1\geq 3$. By formula  \eqref{eq:ampleness}, 
	$0>F\cdot (2K_{X\am}+R\am+\tfrac{1}{2}\Delta\am)=F\cdot R\am-3-\tfrac{1}{2}\geq -\tfrac{1}{2}$, 
 	hence $F\cdot R\am=3$. 
	Recall that by formula \eqref{eq:R} $R\am$ has $2n$ nodes and $n+1$ components. By Lemma \ref{lem:expansions}\ref{item:D_rat-tree}, all components of $R\am$ are rational, and by Lemma  \ref{lem:MMP}\ref{item:mult_2} and all singularities of $R\am$ have multiplicity $2$. 
	
	Because $\Delta\am$ is a tip of $D\am=R\am+\Delta\am$, the divisor $R\am$ is of type $(1,3)$, see Section \ref{sec:Fm}. We claim that 
	\begin{equation}\label{eq:F2_delta}
	\mbox{if a singular point } p \mbox{ of } (R\am)\hor\mbox{ has }\delta_{p}\geq 3\mbox{ then }p\mbox{ lies on a positive section in }(R\am)\hor.
	\end{equation}
	
	Suppose the contrary. Then $p\not \in \Delta\am$ since $\Delta\am$ is a tip of $D\am$. Let $L_{p}$ be the tangent section to $R\am$ at $p$, see Section \ref{sec:Fm}. Then $L_{p}\cdot \Delta\am=0$ and $L_{p}\cdot R\am=7$ because $L_{p}$ and $R\am$ are of type $(1,3)$ and $(0,1)$, respectively. Moreover, the condition $\delta_{p}\geq 3$ implies that $(L_{p}\cdot R\am)_{p}\geq 6$. By assumption, $L_{p}\not\subseteq D\am$. By Lemma \ref{lem:MMP}\ref{item:Si_no-lines}, the surface $X\am\setminus D\am$ contains no contractible curves, so $L_{p}$ meets $D\am\setminus \{p\}$. Hence $L_{p}$ meets $D\am$ at $p$ with multiplicity $6$ and at some smooth point of $R\am$ with multiplicity $1$. It follows that the proper transform $L_{p}'$ of $L_{p}$ on $X'$ satisfies $(L_{p}')^2=L_{p}^{2}-3=-1$, $L_{p}'\cdot D'=2$ and $L_{p}'$ meets the proper transform of $R\am$ and an exceptional curve over $p$. Hence $L_{p}'$ is a bubble on $(X',D')$; a contradiction with \eqref{eq:no_bubbles}. This proves \eqref{eq:F2_delta}.

	\smallskip
	Assume first that $D\am$ contains a fiber, say $F_{0}$. By formula \eqref{eq:R} $n=\#(R\am-F_{0})$, so $n\leq (R\am-F_{0})\cdot F=3$. Suppose $n=3$. Then $R\am-F_{0}$ consists of three positive sections, say $C_{1},C_{2},C_{3}$. In particular, $C_{i}\cdot C_{j}=2$. Because the components of $R\am$ meet in $2n=6$ double points, three of them being $F_{0}\cap C_{i}$, the three remaining ones are the points where $C_{i}$ is tangent to $C_{j}$, $i\neq j$, with multiplicity $2$. Write $F_{0}\cap C_{3}=\{p\}$ and let $\sigma_{p}\colon \F_{2}\map \P^{2}$ be as in Section \ref{sec:Fm}. Then for $i\in \{1,2\}$ we have $((\sigma_{p})_{*}C_{i})^{2}=C_{i}^{2}+2=4$, so $(\sigma_{p})_{*}C_{i}$ are conics tangent to each other at two points with multiplicity $2$, and $((\sigma_{p})_{*}C_{3})^{2}=C_{3}^{2}-1=1$, so $(\sigma_{p})_{*}C_{3}$ is a line tangent to those conics outside their common points. This is a contradiction with Lemma \ref{lem:conics}\ref{item:conics_22}. 
	
	Consider the case $n=2$. Then $R\am-F_{0}$ consists of a positive section $H_{1}$ and a curve $H_{2}$ of type $(0,2)$. We have $H_{1}\cdot H_{2}=4$. By Lemma \ref{lem:F2-adjunction}, $H_{2}$ has a unique singular point $p$, which is ordinary. Because $R\am$ has $4$ self-intersection points, if $H_{2}$ is nodal then meets $H_{1}$ and $F_{0}$ in one point each; and if $H_{2}$ is cuspidal it meets $F_{0}+H_{1}$ in three points, either one, or two of them contained in $F_{0}$. Let $\sigma_{p}\colon \F_{2}\map \P^{2}$ be as in Section \ref{sec:Fm}. In case $\#(H_{1}\cap H_{2})=2$ put $\ll_{1}=\hat{\ll}_{p}\subseteq \P^{2}$, $\ll_{q}=(\sigma_{p})_{*}F_{0}$. Then for $i\in \{1,2\}$, $\cc_{i}\de (\sigma_{p})_{*}H_{3-i}$ are conics meeting with multiplicities $3$, $1$ off $\ll_{1}+\ll_{q}$, and $\cc_{1}+\cc_{2}+\ll_{1}+\ll_{q}$, and hence $S$, is as in  \ref{def:F2n2-cuspidal}. Note that if the expansions is at $(C_2,C_1;1)$ and $(L_q,C_2;1)$, then the proper transform of $|C_2+3E_{p_2}+2E_{p_2}'+E_{p_2}''|$ induces a $\P^1$-fibration which restricts to a $\C^{**}$-fibration of $S$, contrary to assumption \eqref{eq:assumption}.
	
	Consider the case $\#(H_{1}\cap H_{2})=1$. Put $p_{2}=\hat{p}\in \P^{2}$, $\ll_{2}=\hat{\ll}_{p}\subseteq \P^{2}$, $\ll_{2}'=(\sigma_{p})_{*}F_{0}$. Then for $i\in \{1,2\}$, $\cc_{i}\de (\sigma_{p})_{*}H_{i}$ are conics meeting with multiplicity $4$ off $\ll_{2}+\ll_{2}'$, and $\ll_{2}$, $\ll_{2}'$ are lines such that  $(\ll_{2}\cdot \cc_{1})_{p_{2}}=2$, $(\ll_{2}'\cdot \cc_{1})_{p_{2}}=1$.  By Lemma \ref{lem:conics}\ref{item:conics_4} $\ll_{2}$ is not tangent to $\cc_{2}$. This means that $H_{2}$ is nodal, so it is tangent to $F_{0}$, hence $(\ll_{2}'\cdot \cc_{2})_{r}=2$ for some $r\neq p_{2}$. Thus $\cc_{1}+\cc_{2}+\ll_{2}+\ll_{2}'$, and hence $S$, is as in  \ref{def:F2n2-nodal}.

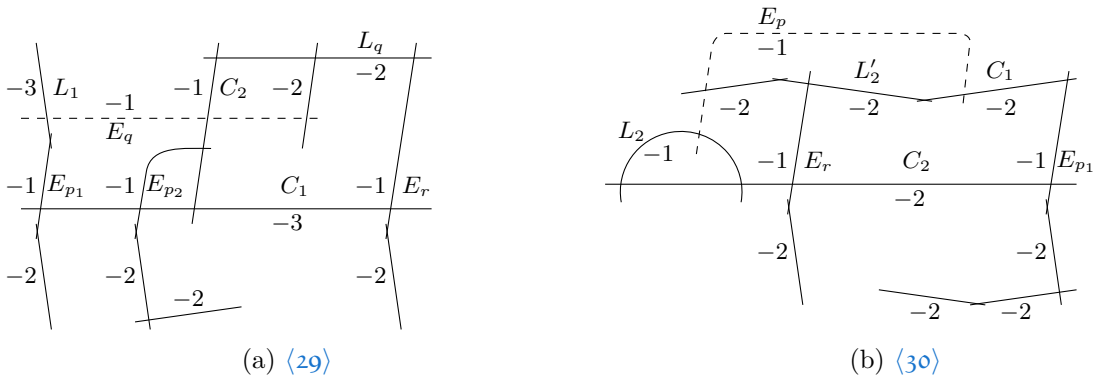
\begin{figure}[htbp]
	\begin{tabular}{cc}
		\begin{subfigure}[b]{0.45\textwidth}
					\begin{tikzpicture}
				\draw (-2.6,-0.4) -- (-2.8,1);
				\node at (-3,0.3) {\small{$-2$}};
				\draw (-2.8,0.8) -- (-2.6,2.2); 
				\node at (-3,1.5) {\small{$-1$}};
				\node at (-2.4,1.5) {\small{$E_{p_1}$}};
				\draw (-2.6,2) -- (-2.8,3.4);
				\node at (-3.0,2.8) {\small{$-3$}};
				\node at (-2.4,2.8) {\small{$L_{1}$}};	
				\draw (-1.3,-0.4) -- (-1.5,1);
				\node at (-1.7,0.3) {\small{$-2$}};
				\draw (-1.5,-0.3) -- (-0.1,-0.1);
				\node at (-0.8,0) {\small{$-2$}};
				\draw (-1.5,0.8) -- (-1.35,1.7) to[out=80,in=180] (-0.5,2); 
				\node at (-1.7,1.5) {\small{$-1$}};
				\node at (-1.1,1.5) {\small{$E_{p_2}$}};								
				\draw (-0.75,1) -- (-0.4,3.4);
				\node at (-0.8,2.8) {\small{$-1$}};
				\node at (-0.2,2.8) {\small{$C_{2}$}};
				\draw (0.7,2) -- (0.9,3.4);
				\node at (0.5,2.8) {\small{$-2$}};
				\draw (1.8,0.8) -- (2.2,3.4);
				\node at (1.6,1.5) {\small{$-1$}};
				\node at (2.2,1.5) {\small{$E_{r}$}};
				\draw (1.8,1) -- (2,-0.4);
				\node at (1.6,0.3) {\small{$-2$}};				
				\draw (-0.6,3.2) -- (2.4,3.2);
				\node at (1.6,3.4) {\small{$L_q$}};
				\node at (1.6,3) {\small{$-2$}};
				\draw[dashed] (-3,2.4) -- (0.9,2.4);
				\node at (-1.7,2.2) {\small{$E_q$}};
				\node at (-1.7,2.6) {\small{$-1$}};
				\draw (-3,1.2) -- (2.4,1.2);
				\node at (0.6,1.5) {\small{$C_1$}};
				\node at (0.5,1) {\small{$-3$}};
			\end{tikzpicture}
			\caption{\ref{def:F2n2-cuspidal}}
			\label{fig:F2n2-cuspidal}
		\end{subfigure}
		&
		\begin{subfigure}[b]{0.45\textwidth}
		\begin{tikzpicture}
			\draw (-2.2,1.1) [partial ellipse=190:-10 : 0.8 and 0.8];
			\node at (-2.85,1.9) {\small{$L_{2}$}};
			\node at (-2.5,1.6) {\small{$-1$}};
			\draw (-0.8,0.8) -- (-0.5,2.7);
			\node at (-1,1.5) {\small{$-1$}};
			\node at (-0.4,1.5) {\small{$E_{r}$}};
			\draw (-0.8,1) -- (-0.6,-0.4);
			\node at (-1,0.3) {\small{$-2$}};	
			\draw (2.6,0.8) -- (2.9,2.7);
			\node at (2.4,1.5) {\small{$-1$}};
			\node at (3,1.5) {\small{$E_{p_1}$}};
			\draw (2.6,1) -- (2.8,-0.4);
			\node at (2.4,0.3) {\small{$-2$}};
			\draw (3,-0.2) -- (1.6,-0.4);
			\node at (2.2,-0.5) {\small{$-2$}};
			\draw (1.8,-0.4) -- (0.4,-0.2);
			\node at (1,-0.5) {\small{$-2$}};				
			\draw (-0.8,2.6) -- (-2.2,2.4);
			\node at (-1.5,2.2) {\small{$-2$}};
			\draw (-1,2.6) -- (1.1,2.3);
			\node at (0.2,2.2) {\small{$-2$}};
			\node at (0.25,2.7) {\small{$L_{2}'$}};
			\draw (0.9,2.3) -- (3,2.6);
			\node at (2,2.2) {\small{$-2$}};
			\node at (2,2.7) {\small{$C_1$}};
			\draw[dashed] (-2,1.6) -- (-1.8,3) to[out=80,in=180] (-1.6,3.2) -- (1.4,3.2) to[out=0,in=100] (1.6,3.1) -- (1.5,2.2);
			\node at (-1,3.4) {\small{$E_p$}};
			\node at (-1,3) {\small{$-1$}};
			\draw (-3.2,1.2) -- (3,1.2);
			\node at (0.9,1.5) {\small{$C_2$}};
			\node at (0.8,1) {\small{$-2$}};
		\end{tikzpicture}
			\caption{\ref{def:F2n2-nodal}}
			\label{fig:F2n2-nodal}
		\end{subfigure}
	\end{tabular}
	\caption{Cases when $X\am\cong \F_{2}$, $D\am$ contains a fiber, and $n=2$.}
	\label{fig:F2n2_vert}
\end{figure}
		
	Consider now the case $n=1$. Then $H\de R\am-F_{0}$ is a curve of type $(0,3)$. Because $R\am$ has $2n=2$ nodes $H$ meets $F_{0}$ in $2-\nu$ points, where $\nu\leq 1$ is the number of nodes of $H$. The morphism $\pi_{\F_{2}}|_{H}\colon H\to \P^{1}$ has degree $3$ and is ramified at least at each singularity of $H$, with index at least $2$, and at the point of tangency of $F_{0}$ and $H$, with index $2+\nu$. By the Hurwitz formula, $2-\nu \leq 3\cdot 2-\#\Sing H-(1+\nu)$, so $\#\Sing H\leq 3$.
	
	Lemma \ref{lem:F2-adjunction} gives $\delta_{H}=4$, so by condition \eqref{eq:F2_delta} $H$ has either two singular points with $\delta=2$ or one with $\delta=2$ and two ordinary. Choose a point $p\in \Sing H$ with $\delta_{p}=2$, a node if possible. Let $\sigma_{p}\colon \F_{2}\map \P^{2}$ be as in Section \ref{sec:Fm}, $\ll_{0}\de (\sigma_{p})_{*}F_{0}$, $\hh\de (\sigma_{p})_{*}H$  and let $p'\in \hh$ be the point infinitely near to $p$. Then $\ll_{0}$ is a line meeting $\hat{\ll}_{p}$ at $\hat{p}$, and $\hh$ is a quartic such that $(\hat{\ll}_{p}\cdot \hh)_{\hat{p}}=(\hat{\ll}_{p}\cdot \hh)_{p'}=2$, $\hat{p}\in\hh$ is smooth and $\delta_{p'}=1$.
	
	Assume that $\#\Sing \hh=2$, so $\Sing \hh=\{p',q\}$, where $\delta_{q}=2$. The line $\ll_{qp'}$ joining $q$ and $p'$ meets $\hh$ only at $q$, $p'$, and the line $\ll_{q}$ tangent to $\hh$ at $q$ meets $\hh$ only at $q$, with multiplicity $4$, see Lemma \ref{lem:P2-curves}\ref{item:P2-tangent}. Consider the quadratic transformation centered at $p'$, $q$ and its infinitely near point on the proper transform of $\ll_{q}$, that is, blow up these points and contract the proper transforms of $\ll_{qp'}$, $\ll_{q}$ and the first exceptional curve over $q$. Let $\cc_{1}$, $\cc_{2}$, $\ll_{1}$, $\ll_{2}$, $\ll$ be the images of $\hh$, $\hat{\ll}_{p}$, $\ll_{0}$ and the exceptional curves over $q$ and $p'$, respectively.  Now  $\cc_{1}+\cc_{2}+\ll_{1}+\ll_{2}+\ll$, and hence $S$, is as in \ref{def:F2_n1-node} if $p'$ is a node and as in \ref{def:F2_n1-cusp} if $p'$ is a cusp; where the line $\ll$ correspond to $\ll_{r}$ in \ref{def:F2_n1-node} and to $\ll_{p'}$ in \ref{def:F2_n1-cusp}.
	
	Assume now that $\hh$ has three (ordinary) singular points $p'$, $q_{1}$, $q_{2}$. At most one of them is a node. We order them so that $q_{2}$ is a cusp. Consider the standard quadratic transformation centered at $p'$, $q_1$ and $q_2$. Let $\cc_{1}$, $\cc_{2}$, $\ll_{1}$, $\ll_{2}$, $\ll_{pq}$, $\ll_{r}'$ be the images of $\hh$, $\ll_{0}$, $\hat{\ll}_{p}$, and the exceptional curves over $p'$, $q_1$, $q_2$, respectively. The lines $\ll_{1}$, $\ll_{r}'$ and $\ll_{pq}$ pass through one point (the image of $\ll_{q_{1}q_{2}}$) and $\ll_{1}$, $\ll_{r}'$ are tangent to $\cc_{1}$, so $\ll_{pq}$ is not. This means that  $q_{1}\in\hh$ is a node, so $p'\in\hh$ is a cusp. It follows that $\cc_{1}+\cc_{2}+\ll_{1}+\ll_{2}+\ll_{pq}+\ll_{r}'$, and hence $S$, is as in  \ref{def:F2_n1-node-3}.

\begin{figure}[htbp]
	\begin{tabular}{cc}
		\begin{subfigure}[b]{0.4\textwidth}
			\begin{tikzpicture}
	\draw (-4,1.2) -- (1.4,1.2);
	\node at (0.2,1.5) {\small{$C_1$}};
	\node at (0.1,1) {\small{$-3$}};
	\draw (-2.1,1.1) [partial ellipse=190:-10 : 0.8 and 0.8];
	\node at (-1.62,1.7) {\large{$\bullet$}};
	\node at (-2.65,2) {\small{$L_r$}};
	\node at (-2.4,1.6) {\small{$-1$}};
	\draw (-2,1.7) -- (-2.2,3.4);
	\node at (-2.4,2.8) {\small{$-2$}};
	\node at (-1.8,2.8) {\small{$L_1$}};
	\draw (-0.8,-0.3) -- (0.6,-0.1);
	\node at (-0.1,0) {\small{$-2$}};				
	\draw (-0.6,-0.4) -- (-0.8,1);
	\node at (-1,0.3) {\small{$-2$}};
	\draw (-0.8,0.8) -- (-0.6,2.2); 
	\node at (-1,1.5) {\small{$-1$}};
	\node at (-0.4,1.5) {\small{$E_{p'}$}};			
	\draw (-0.6,2) -- (-0.8,3.4);
	\node at (-1.0,2.8) {\small{$-3$}};
	\node at (-0.4,2.8) {\small{$C_{2}$}};
	\node at (-0.7,2.8) {\large{$\bullet$}};
	\draw (-0.8,3.2) -- (-0.6,4.6);
	\node at (-1.0,3.8) {\small{$-2$}};
	\draw (1.2,-0.4) -- (1,1);
	\node at (0.8,0.3) {\small{$-2$}};
	\draw (1,0.8) -- (1.2,2.2); 
	\node at (0.8,1.5) {\small{$-1$}};
	\node at (1.4,1.5) {\small{$E_{p_2}$}};			
	\draw (1.2,2) -- (1,3.4);
	\node at (0.8,2.8) {\small{$-3$}};
	\node at (1.4,2.8) {\small{$L_{2}$}};
	\node at (1.1,4) {\large{$\bullet$}};
	\draw (1,3.2) -- (1.2,4.6);
	\node at (0.8,3.8) {\small{$-2$}};
	\draw[dashed] (-2.2,2.4) -- (1.2,2.4);
	\node at (0.2,2.6) {\small{$E_p$}}; 
	\node at (0.2,2.2) {\small{$-1$}};
	\draw[dashed] (-3.6,1) -- (-3.3,3.1) to[out=80,in=180] (-3.2,3.2) -- (-1.8,3.2) to[out=0,in=-100] (-1.7,3.3) -- (-1.5,4.3) to[out=80,in=180] (-1.4,4.4) -- (-0.4,4.4);
	\node at (-3.8,1.5) {\small{$-1$}};
	\node at (-3.2,1.5) {\small{$E_{p_1}$}};	
	\node at (-3,0) {\small{$\bullet\mapsto r$}};										
\end{tikzpicture}
			\caption{\ref{def:F2_n1-node}}
			\label{fig:F2_n1-node}
		\end{subfigure}
		&
		\begin{subfigure}[b]{0.4\textwidth}
			\begin{tikzpicture}
				\draw (0.2,-0.4) -- (0,1);
				\node at (-0.2,0.3) {\small{$-2$}};
				\draw (0,0.8) -- (0.2,2.2); 
				\node at (-0.2,1.5) {\small{$-1$}};
				\node at (0.4,1.5) {\small{$E_{p'}$}};			
				\draw (0.2,2) -- (0,3.4);
				\node at (-0.2,2.8) {\small{$-3$}};
				\node at (0.4,2.8) {\small{$L_{p'}$}};
				\node at (0.1,2.6) {\large{{$\boldsymbol{\circ}$}}};
				\draw (0,3.2) -- (0.2,4.6);
				\node at (-0.2,3.8) {\small{$-2$}};
				\node at (0.4,3.8) {\small{$L_{1}$}};
				\node at (0.1,4) {\large{$\bullet$}};
				\draw (1.6,-0.4) -- (1.4,1);
				\node at (1.2,0.3) {\small{$-2$}};
				\draw (1.4,0.8) -- (1.6,2.2); 
				\node at (1.2,1.5) {\small{$-1$}};
				\node at (1.8,1.5) {\small{$E_{p_2}$}};			
				\draw (1.6,2) -- (1.4,3.4);
				\node at (1.2,2.8) {\small{$-3$}};
				\node at (1.8,2.8) {\small{$L_{2}$}};
				\node at (1.5,2.6) {\large{$\bullet$}};
				\draw (1.4,3.2) -- (1.6,4.6);
				\node at (1.2,3.8) {\small{$-2$}};
				\node at (1.5,4) {\large{{$\boldsymbol{\circ}$}}};
				\draw (3,-0.4) -- (2.8,1);
				\node at (2.6,0.3) {\small{$-2$}};
				\draw (2.8,0.8) -- (3,2.2); 
				\node at (2.6,1.5) {\small{$-1$}};
				\node at (3.2,1.5) {\small{$E_{q}$}};			
				\draw (3,2) -- (2.8,3.4);
				\node at (2.6,2.8) {\small{$-2$}};
				\node at (3.2,2.8) {\small{$C_{2}$}};
				\node at (2.9,2.6) {\large{$\bullet$}};
				\node at (2.88,2.9) {\large{{$\boldsymbol{\circ}$}}};
				\draw (2.8,3.2) -- (3,4.6);
				\node at (2.6,3.8) {\small{$-2$}};
				\draw (-0.2,1.2) -- (3.8,1.2) to[out=0,in=0] (3.8,2.4) -- (2.6,2.4);
				\node at (3.8,1.8) {\small{$-4$}};
				\node at (4.4,1.8) {\small{$C_1$}};
				\node at (4,0.5) {\small{$\bullet\mapsto p$}};
				\node at (4,0.2) {\small{ $\boldsymbol{\circ}\mapsto r$}};
				\draw[dashed] (3.8,2.2) -- (3.5,4.2) to[out=100,in=0] (3.2,4.4) -- (2.2,4.4) to[out=180,in=0] (1.6,4.8) to[out=180,in=0] (1,4.4) -- (-0.2,4.4);
				\node at (3.35,3.4) {\small{$-1$}};
				\node at (3.95,3.4) {\small{$E_{p_1}$}};
			\end{tikzpicture}
			\caption{\ref{def:F2_n1-cusp}, see Example \ref{ex:7}.}
			\label{fig:F2_n1-cusp}
		\end{subfigure}
		\end{tabular}
		\begin{subfigure}[b]{0.4\textwidth}
					\begin{tikzpicture}
			\draw (-4,1.2) -- (2.6,1.2);
			\node at (0.2,1.5) {\small{$C_1$}};
			\node at (0.1,1) {\small{$-5$}};
			\draw (-2.1,1.1) [partial ellipse=190:-10 : 0.8 and 0.8];
			\node at (-2.4,1.45) {\small{$L_{pq}$}};
			\node at (-3.1,1.5) {\small{$-1$}};
			\draw (-0.8,-0.3) -- (0.6,-0.1);
			\node at (-0.1,0) {\small{$-2$}};				
			\draw (-0.6,-0.4) -- (-0.8,1);
			\node at (-1,0.6) {\small{$-2$}};
			\draw (-0.8,0.8) -- (-0.6,2.2); 
			\node at (-1,1.5) {\small{$-1$}};
			\node at (-0.4,1.5) {\small{$E_{p'}$}};			
			\draw (-0.6,2) -- (-0.85,3.75);
			\node at (-1.0,2.7) {\small{$-3$}};
			\node at (-0.4,2.7) {\small{$C_{2}$}};
			\draw (-0.85,3.55) -- (-0.65,4.95);
			\node at (-1.0,4.3) {\small{$-2$}};
			\draw (1.2,-0.4) -- (1,1);
			\node at (0.8,0.3) {\small{$-2$}};
			\draw (1,0.8) -- (1.2,2.2); 
			\node at (0.8,1.5) {\small{$-1$}};		
			\draw (1.2,2) -- (0.95,3.75);
			\node at (0.8,2.7) {\small{$-3$}};
			\node at (1.35,2.7) {\small{$L_{r}'$}};
			\draw (0.95,3.55) -- (1.2,4.95);
			\node at (0.8,4.3) {\small{$-2$}};
			\node at (1.35,4.3) {\small{$L_{1}$}};
			\draw (2.5,-0.4) -- (2.3,1);
			\node at (2.1,0.3) {\small{$-2$}};
			\draw (2.3,0.8) -- (2.5,2.2); 
			\node at (2.1,1.5) {\small{$-1$}};		
			\node at (2.7,1.5) {\small{$E_{p_2}$}};
			\draw (2.5,2) -- (2.25,3.75);
			\node at (2.1,2.7) {\small{$-3$}};
			\node at (2.65,2.7) {\small{$L_{2}$}};
			\draw[dashed] (-0.8,2.4) -- (2.8,2.4);
			\node at (0.2,2.6) {\small{$E_r$}}; 
			\node at (0.2,2.2) {\small{$-1$}};
			\draw[dashed] (-3.6,1) -- (-3.1,4.2) to[out=80,in=180] (-2.8,4.7) -- (1.4,4.7);
			\node at (-3.4,4) {\small{$-1$}};
			\node at (-2.8,4) {\small{$E_{p_1}$}};
			\draw[dashed] (-2.5,1.7) -- (-2.3,3.1) to[out=80,in=180] (-2.1,3.4) -- (-0.2,3.4) to[out=0,in=180] (0.8,4) -- (1.4,4) to[out=0,in=180] (2.2,3.4) -- (2.5,3.4);
			\node at (-2.6,2.85) {\small{$-1$}};
			\node at (-2.1,2.8) {\small{$E_{p}$}};	
			\draw[dashed] (-1.7,1.7) -- (-1.5,2.9) to[out=80,in=180] (-1.3,3.1) -- (1.3,3.1);
			\node at (-1.9,2.05) {\small{$-1$}};
			\node at (-1.4,2) {\small{$E_{q}$}};						
		\end{tikzpicture}
			\caption{\ref{def:F2_n1-node-3}}
			\label{fig:F2_n1-node-3}
		\end{subfigure}
	\caption{Cases when $X\am\cong \F_{2}$, $D\am$ contains a fiber, and $n=1$.}
	\label{fig:F2n1_vert}
\end{figure}
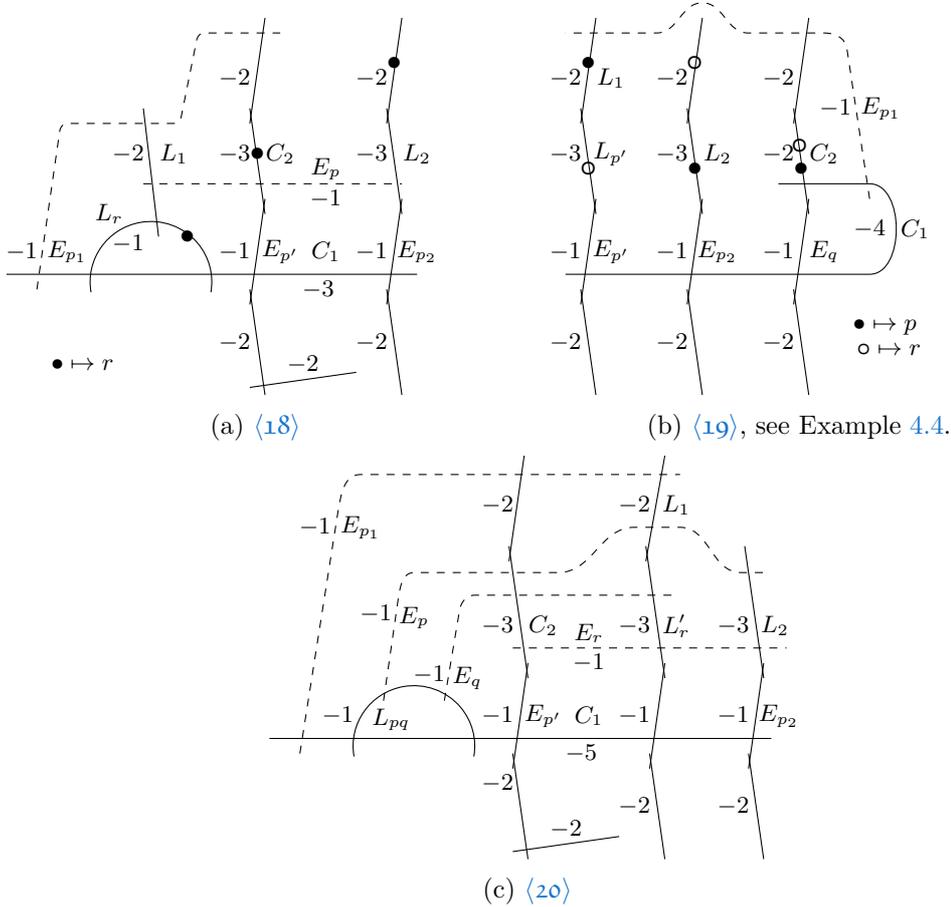

	\medskip
	
	Now consider the case when $R\am$ is horizontal. Like in the previous case, we have $n+1=\#R\am\leq R\am\cdot F=3$ for a fiber $F$, so $n\leq 2$.  Consider the case $n=2$. Then $R\am$ consists of three $1$-sections, say, $H$ of type $(1,1)$ and $H_{1}$, $H_{2}$ of type $(0,1)$. We have $H\cdot H_{1}=H\cdot H_{2}=3$ and $H_{1}\cdot H_{2}=2$. The divisor $R\am$ contains $2n=4$ self-intersection points, so by renaming $H_{1}$ with $H_{2}$ if necessary, we can assume that $H_{1}$ meets $H$ with multiplicity $3$, at, say, $p$, and $H_{2}$ meets $H$ with multiplicities $2$, $1$ if $H_{2}$ is tangent to $H_{1}$ and with multiplicity $3$ otherwise. Let $\sigma_{p}\colon\F_{2}\map\P^{2}$ be as in Section \ref{sec:Fm}. Note that $(\sigma_{p})_{*}H$, $(\sigma_{p})_{*}H_{2}$ are conics meeting at $p'$ with multiplicity $2$; $(\sigma_{p})_{*}H_{2}$ is tangent to $\hat{\ll}_{p}$ at $\hat{p}$, and $(\sigma_{p})_{*}H_{1}$ is a line tangent to $(\sigma_{p})_{*}H$ at $p'$. If $H_{1}$ is tangent to $H_{2}$ put $\cc_{2}\de (\sigma_{p})_{*}H$, $\cc_{1}\de (\sigma_{p})_{*}H_{2}$, $\ll_{1}\de (\sigma_{p})_{*}H_{1}$ and $\ll_{2}\de \hat{\ll}_{p}$. Then $\cc_{1}+\cc_{2}+\ll_{1}+\ll_{2}$, and hence $S$, is as in  \ref{def:F2_n2-tangent}. In the other case, put $\cc_{1}\de (\sigma_{p})_{*}H$, $\cc_{2}\de (\sigma_{p})_{*}H_{2}$, $\ll_{1}\de (\sigma_{p})_{*}H_{1}$, $\ll_{1}'\de \hat{\ll}_{p}$. Then $\cc_{1}+\cc_{2}+\ll_{1}+\ll_{1}'$, hence $S$, is as in  \ref{def:F2_n2-transversal}. In this case, we can assume that the weight at the preimage of $r\in \ll_1\cap \cc_2$ is not $1$: indeed, otherwise the proper transform of $\ll_{rp_2}$ on $X$ is a bubble, and contracting it instead of the one over $r$, we see that $S$ is as in \ref{def:nodal-cubic_P2}. Moreover, if both points of $\ll_1\cap \cc_2$ are centers of expansions with weights $v,w$, we can, by symmetry of the graph, take $v\geq w$, cf.\ Example \ref{ex:33}.
		\begin{figure}[htbp]
			\begin{tabular}{cc}
				\begin{subfigure}[b]{0.5\textwidth}
					\begin{tikzpicture}
						\draw (-1.8,0.8) -- (-1.4,3.6); 
						\node at (-1.9,2.6) {\small{$-2$}};
						\draw[dashed] (-1.6,3.4) -- (-1.4,3.4) to[out=0,in=180] (0,4.4) -- (0.4,4.4) to[out=0,in=180] (1.6,1.7) -- (2,1.7);
						\node at (0.75,2.6) {\small{$-1$}};
						\node at (1.35,2.6) {\small{$E_{p_2}$}};
						\draw (0.2,-0.4) -- (0,1);
						\node at (-0.2,0.3) {\small{$-3$}};
						\node at (0.35,0.3) {\small{$L_{1}$}};
						\draw (0,0.8) -- (0.2,2.2); 
						\node at (-0.2,1.5) {\small{$-1$}};
						\node at (0.35,1.5) {\small{$E_{p}$}};			
						\draw (0.2,2) -- (0,3.4);
						\node at (-0.2,2.6) {\small{$-2$}};
						\draw (0,3.2) -- (0.2,4.6);
						\node at (-0.2,3.8) {\small{$-2$}};
						\node at (0.35,3.8) {\small{$L_{2}$}};			
						\draw (-0.7,-0.4) -- (2.1,0);
						\node at (1.1,-0.4) {\small{$-1$}};
						\node at (1.1,0.1) {\small{$E_{p_1}$}};
						\draw (-0.5,-0.4) -- (-1.9,-0.2);
						\node at (-1.3,-0.1) {\small{$-2$}};
						\draw (2,-0.2) -- (1.7,1.9);
						\node at (1.6,0.3) {\small{$-2$}};
						\node at (2.2,0.3) {\small{$C_{1}$}};
						\draw (1.7,0.7) -- (2.6,0.7) to[out=0,in=-100] (2.7,0.8) -- (2.9,2.2);
						\node at (2.5,1.5) {\small{$-1$}};
						\node at (3.1,1.5) {\small{$E_{p'}$}};
						\draw (2.9,2) -- (2.7,3.4);
						\node at (2.5,2.6) {\small{$-2$}};	
						\draw (-2,1.2) -- (2.9,1.2);
						\node at (-1,1.45) {\small{$C_2$}};
						\node at (-1.1,0.95) {\small{$-1$}};
					\end{tikzpicture}
					\caption{\ref{def:F2_n2-tangent}}
					\label{fig:F2_n2-tangent}
				\end{subfigure}
				&
				\begin{subfigure}[b]{0.4\textwidth}
					\begin{tikzpicture}
						\draw (-1.8,0.8) -- (-1.4,3.6); 
						\node at (-1.9,2.6) {\small{$-2$}};
						\draw[dashed] (-1.6,3.4) -- (-1.4,3.4) to[out=0,in=180] (0,4.4) -- (0.4,4.4)  to[out=0,in=180] (1.4,4.6) to[out=0,in=100] (2.1,4.3) -- (2.6,0.4);
						\node at (2.05,2.6) {\small{$-1$}};
						\node at (2.65,2.6) {\small{$E_{p'}$}};
						\draw (-0.05,0.4) -- (0.2,2.2); 
						\node at (-0.2,1.5) {\small{$-1$}};
						\node at (0.4,1.5) {\small{$E_{p_1}$}};			
						\draw (0.2,2) -- (0,3.4);
						\node at (-0.2,2.6) {\small{$-2$}};
						\draw (0,3.2) -- (0.2,4.6);
						\node at (-0.2,3.8) {\small{$-2$}};
						\node at (0.35,3.8) {\small{$L_{1}'$}};			
						\draw (1.15,0.4) -- (1.4,2.2); 
						\node at (1,1.5) {\small{$-1$}};
						\node at (1.6,1.5) {\small{$E_{p_2}$}};			
						\draw (1.4,2) -- (1.2,3.4);
						\node at (1,2.6) {\small{$-2$}};
						\draw (1.2,3.2) -- (1.38,4.4);
						\node at (1,3.8) {\small{$-2$}};				
						\draw (-0.2,0.6) -- (0,0.6) to[out=0,in=90] (0.6,0) to[out=-90,in=0] (0,-0.6) -- (-0.4,-0.6);
						\node at (-0.2,-0.8) {\small{$-1$}};
						\node at (-0.2,-0.3) {\small{$L_{1}$}};
						\draw (2.8,0.6) -- (1,0.6) to[out=180,in=90] (0.4,0) to[out=-90,in=180] (1,-0.6) -- (1.4,-0.6);
						\node at (1.2,-0.8) {\small{$-1$}};
						\node at (1.2,-0.3) {\small{$C_{2}$}};
						\draw (-2,1.2) -- (1.6,1.2);
						\node at (-1,1.45) {\small{$C_1$}};
						\node at (-1.1,0.95) {\small{$-2$}};
					\end{tikzpicture}
					\caption{\ref{def:F2_n2-transversal}}
					\label{fig:F2_n2-transversal}
				\end{subfigure}
			\end{tabular}
			\caption{Cases when $X\am\cong \F_{2}$, $D\am$ is horizontal, and $n=2$.}
			\label{fig:F2n2_hor}
		\end{figure}
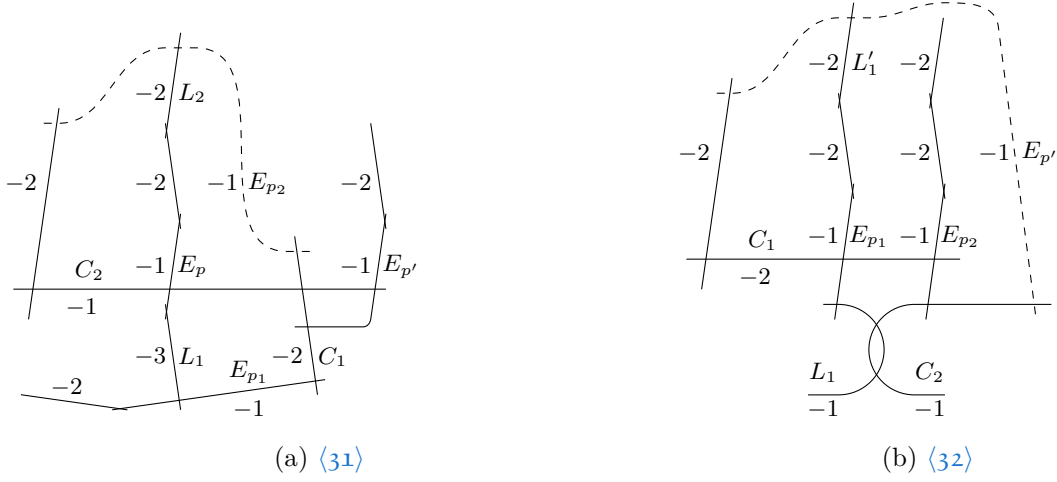

	Consider the case $n=1$. Then $R\am=H_{1}+H_{2}$, where $H_{1}$ is a $1$-section and $H_{2}$ is a $2$-section. Exactly one of the $H_{i}$'s meet $\Delta\am$ and $H_{1}$ meets $H_{2}$ in at most two points. In fact, $H_{1}$ is a positive section. Indeed,  otherwise $H_{1}$, $H_{2}$ are of type $(1,1)$ and $(0,2)$, respectively, so $H_{1}\cdot H_{2}=6$, and by condition \eqref{eq:no_bubbles} $H_{1}$ meets $H_{2}$ in at least three points, which is impossible. It follows that $H_{2}$ is of type $(1,2)$. In particular, $H_{2}\cdot \Delta\am=1$ and $H_{2}\cdot H_{1}=5$. By Lemma \ref{lem:F2-adjunction} we have $\delta_{H_{2}}=2$. Choose $p\in \Sing H_{2}$, a node if possible. Let $\sigma_{p}\colon \F_{2}\map\P^{2}$ be as in Section \ref{sec:Fm}. Then $\hh_{1}\de (\sigma_{p})_{*}H_{1}$ is a conic tangent to $\hat{\ll}_{p}$ at $\hat{p}$; $\hh_{2}\de (\sigma_{p})_{*}H_{2}$ is a cubic such that $(\hh_{2}\cdot \hat{\ll}_{p})_{p'}=2$, where $p'\in \hh_{2}$ is the point infinitely near to $p$; $(\hh_{2}\cdot \hat{\ll}_{p})_{\hat{p}}=1$, and for some $r,r'\in \P^{2}$ we have $(\hh_{2}\cdot \hh_{1})_{r}=5$ if $\hh_{2}$ is nodal and $((\hh_{2}\cdot \hh_{1})_{r},(\hh_{2}\cdot \hh_{1})_{r'})\in \{(4,1),(3,2)\}$ if $\hh_{2}$ is cuspidal. Let $s$ be the singular point of $\hh_{2}$ and let $\sigma\colon \P^{2}\map \P^{2}$ be the quadratic transformation centered at $\hat{p}$, $r$, $s$.	
			\begin{figure}[htbp]
		\begin{subfigure}[b]{0.4\textwidth}			
			\begin{tikzpicture}
				\draw (-2.7,1.1) -- (-2.4,3.2); 
				\node at (-2.9,2.6) {\small{$-2$}};
				\node at (-2.25,2.65) {\small{$L_{2}$}};		
				\draw (-3.8,1.2) -- (2.1,1.2);
				\node at (-3.2,1.45) {\small{$C_2$}};
				\node at (-3.2,1) {\small{$-1$}};
				\draw (-1.1,1.1) [partial ellipse=190:-10 : 0.8 and 0.8];
				\node at (-1.8,1.9) {\small{$L_{2}'$}};
				\node at (-1.4,1.6) {\small{$-1$}};
				\draw (-1.1,1.8) -- (-0.9,3.2);
				\node at (-1.3,2.6) {\small{$-2$}};
				\draw (0.4,2.4) -- (2.1,2.4);
				\node at (1.4,2.2) {\small{$-3$}};
				\node at (1.4,2.6) {\small{$C_{1}$}};
				\draw (1.8,1) -- (2.1,3.1); 
				\node at (1.6,1.5) {\small{$-1$}};
				\node at (2.2,1.5) {\small{$E_{p_1}$}};
				\draw (2.1,2.9) -- (1.85,4.75);
				\node at (1.7,3.6) {\small{$-2$}};	
				\draw (2.1,4.65) -- (0.7,4.45);
				\node at (1.4,4.35) {\small{$-2$}};
				\draw (0.9,4.45) -- (-0.5,4.65);
				\node at (0.1,4.35) {\small{$-2$}};
				\draw (-0.3,4.65) -- (-2.4,4.35);
				\node at (-1.3,4.35) {\small{$-2$}};
				\node at (-1.3,4.7) {\small{$L_{1}'$}};
				\draw[dashed] (-2.8,2.2) -- (0.4,2.2) to[out=0,in=-100] (0.6,2.6);
				\node at (-0.1,2) {\small{$-1$}};
				\node at (-0.1,2.4) {\small{$E_{p_2}$}};
				\draw[dashed] (-0.6,1.6) -- (0.6,1.6) to[out=0,in=-80] (0.8,1.8) -- (1,3.2) to[out=80,in=0] (0.8,3.4) -- (-0.6,3.4) to[out=180,in=-100] (-0.7,3.5) -- (-0.8,4.8);
				\node at (-0.1,3.2) {\small{$-1$}};
				\node at (-0.1,3.6) {\small{$E_{p'}$}};
				\draw[dashed] (-3.7,1) -- (-3.4,2.9) to[out=80,in=180] (-3.3,3)  -- (-2.1,3) to[out=0,in=-100] (-2,3.1) -- (-1.8,4.7);
				\node at (-2.2,3.6) {\small{$-1$}};
				\node at (-1.7,3.6) {\small{$E_{q}$}};
			\end{tikzpicture}
			\caption{\ref{def:F2-5}}
			\label{fig:F2-5}
		\end{subfigure}
			\begin{tabular}{cc}
		\begin{subfigure}[b]{0.6\textwidth}
			\begin{tikzpicture}
				\draw (-2.9,1.7) -- (-2.4,4.5); 
				\node at (-3,2.8) {\small{$-2$}};
				\node at (-2.5,2.8) {\small{$L_{2}$}};
				\node at (-2.54,3.6) {\large{$\bullet$}};			
				\draw (-3.8,1.9) -- (3,1.9);
				\node at (-3.2,2.15) {\small{$C_2$}};
				\node at (-3.2,1.7) {\small{$-3$}};
				\node at (2,5) {\small{$\bullet\mapsto p_2$}};									
				\draw (0.4,2.6) -- (3,2.6);
				\node at (1.7,2.4) {\small{$-3$}};
				\node at (1.7,2.8) {\small{$C_{1}$}};
				\node at (2.2,2.6) {\large{$\bullet$}};	
				\draw (2.7,1.7) -- (2.9,3.1); 
				\node at (2.5,2.2) {\small{$-1$}};
				\node at (3,2.2) {\small{$E_{s}$}};
				\draw (2.9,2.9) -- (2.7,4.3);
				\node at (2.5,3.6) {\small{$-2$}};
				\draw (0.4,1.7) -- (0.6,3.1); 
				\node at (0.2,2.2) {\small{$-1$}};
				\node at (0.8,2.2) {\small{$E_{p_1}$}};
				\draw (0.6,2.9) -- (0.4,4.3);
				\node at (0.2,3.6) {\small{$-2$}};
				\draw (0.4,4.1) -- (0.65,5.75);
				\node at (0.2,5) {\small{$-2$}};
				\node at (0.8,5) {\small{$L_{1}'$}};
				\draw (-1.5,1.7) -- (-1.3,3.1); 
				\node at (-1.7,2.2) {\small{$-1$}};
				\draw (-1.3,2.9) -- (-1.5,4.3);
				\node at (-1.7,3.6) {\small{$-3$}};
				\node at (-1.2,3.6) {\small{$L_{2}'$}};
				\draw (-1.5,4.1) -- (-1.15,5.75);
				\node at (-1.6,5) {\small{$-2$}};
				\node at (-1.22,5.45) {\large{$\bullet$}};	
				\draw (-1.5,2.6) -- (-0.5,2.6);
				\node at (-0.8,2.4) {\small{$-2$}};
				\draw[dashed] (1,2.4) to[out=80,in=-100] (1.1,3) -- (0.9,4.5) to[out=100,in=0] (0.8,4.6) -- (-0.4,4.6) to[out=180,in=0] (-0.8,3.9) -- (-1.6,3.9);
				\node at (-0.85,4.2) {\small{$-1$}};
				\node at (-0.3,4.2) {\small{$E_{p'}$}};
				\draw[dashed] (-3.7,1.7) -- (-3.2,4.3) to[out=80,in=180] (-3.1,4.4)  -- (-2.2,4.4) to[out=0,in=-100] (-2.1,4.5) -- (-1.8,5.9) to[out=-80,in=180] (-1.7,6) -- (-0.8,6) to[out=0,in=180] (-0.7,5.6) -- (0.8,5.6);
				\node at (-0.4,5.4) {\small{$-1$}};
				\node at (-0.4,5.8) {\small{$E_{q}$}};
			\end{tikzpicture}
			\caption{\ref{def:F2_n1-cusp-hor_32}}
			\label{fig:F2_n1-cusp-hor_32}
		\end{subfigure}
		&
		\begin{subfigure}[b]{0.35\textwidth}
			\begin{tikzpicture}
				\draw (-2.9,1.7) -- (-2.4,4.5); 
				\node at (-3,2.8) {\small{$-2$}};
				\node at (-2.5,2.8) {\small{$L_{2}$}};
				\node at (-2.54,3.6) {\large{$\bullet$}};			
				\draw (-3.8,1.9) -- (1.8,1.9);
				\node at (-3.2,2.15) {\small{$C_2$}};
				\node at (-3.2,1.7) {\small{$-2$}};
				\node at (2,5) {\small{$\bullet\mapsto p_2$}};									
				\draw (0.4,2.6) -- (1.6,2.6) to[out=0,in=100] (1.7,2.5) -- (1.6,1.7);
				\node at (1.4,2.2) {\small{$-3$}};
				\node at (2,2.2) {\small{$C_{1}$}};
				\node at (1.4,2.6) {\large{$\bullet$}};	
				\draw (0.4,1.7) -- (0.6,3.1); 
				\node at (0.2,2.2) {\small{$-1$}};
				\node at (0.8,2.2) {\small{$E_{p_1}$}};
				\draw (0.6,2.9) -- (0.4,4.3);
				\node at (0.2,3.6) {\small{$-2$}};
				\draw (0.4,4.1) -- (0.6,5.5);
				\node at (0.2,4.8) {\small{$-2$}};
				\draw (0.6,5.3) -- (0.4,6.7);
				\node at (0.2,6.2) {\small{$-2$}};
				\node at (0.7,6.2) {\small{$L_{1}'$}};				
				\draw (-1.5,1.7) -- (-1.3,3.1); 
				\node at (-1.7,2.2) {\small{$-1$}};
				\draw (-1.3,2.9) -- (-1.5,4.3);
				\node at (-1.7,3.6) {\small{$-3$}};
				\node at (-1.2,3.6) {\small{$L_{2}'$}};
				\draw (-1.5,4.1) -- (-1.3,5.6);
				\node at (-1.7,4.8) {\small{$-2$}};
				\node at (-1.3,5.4) {\large{$\bullet$}};	
				\draw (-1.5,2.6) -- (-0.5,2.6);
				\node at (-0.8,2.4) {\small{$-2$}};
				\draw[dashed] (1,2.4) to[out=80,in=-90] (1.1,3.1) -- (1.1,5.7) to[out=90,in=0] (1,5.8) -- (-0.3,5.8) to[out=180,in=80] (-0.4,5.7) -- (-0.65,4) to[out=-100,in=0] (-0.7,3.9) -- (-1.7,3.9);
				\node at (-0.9,4.25) {\small{$-1$}};
				\node at (-0.3,4.2) {\small{$E_{p'}$}};
				\draw[dashed] (-3.7,1.7) -- (-3.2,4.3) to[out=80,in=180] (-3.1,4.4)  -- (-2.2,4.4) to[out=0,in=-100] (-2.1,4.5) -- (-1.8,6.4) to[out=-80,in=180] (-1.7,6.5) -- (0.7,6.5);
				\node at (-2.1,6) {\small{$-1$}};
				\node at (-1.6,6) {\small{$E_{q}$}};
			\end{tikzpicture}
			\caption{\ref{def:F2_n1-cusp-hor_41}}
			\label{fig:F2_n1-cusp-hor_41}
		\end{subfigure}
	\end{tabular}
	\caption{Cases when $X\am\cong \F_{2}$, $D\am$ is horizontal, $n=1$ and $\delta_{p}=2$.}
	\label{fig:F2n1_hor_2}
\end{figure}
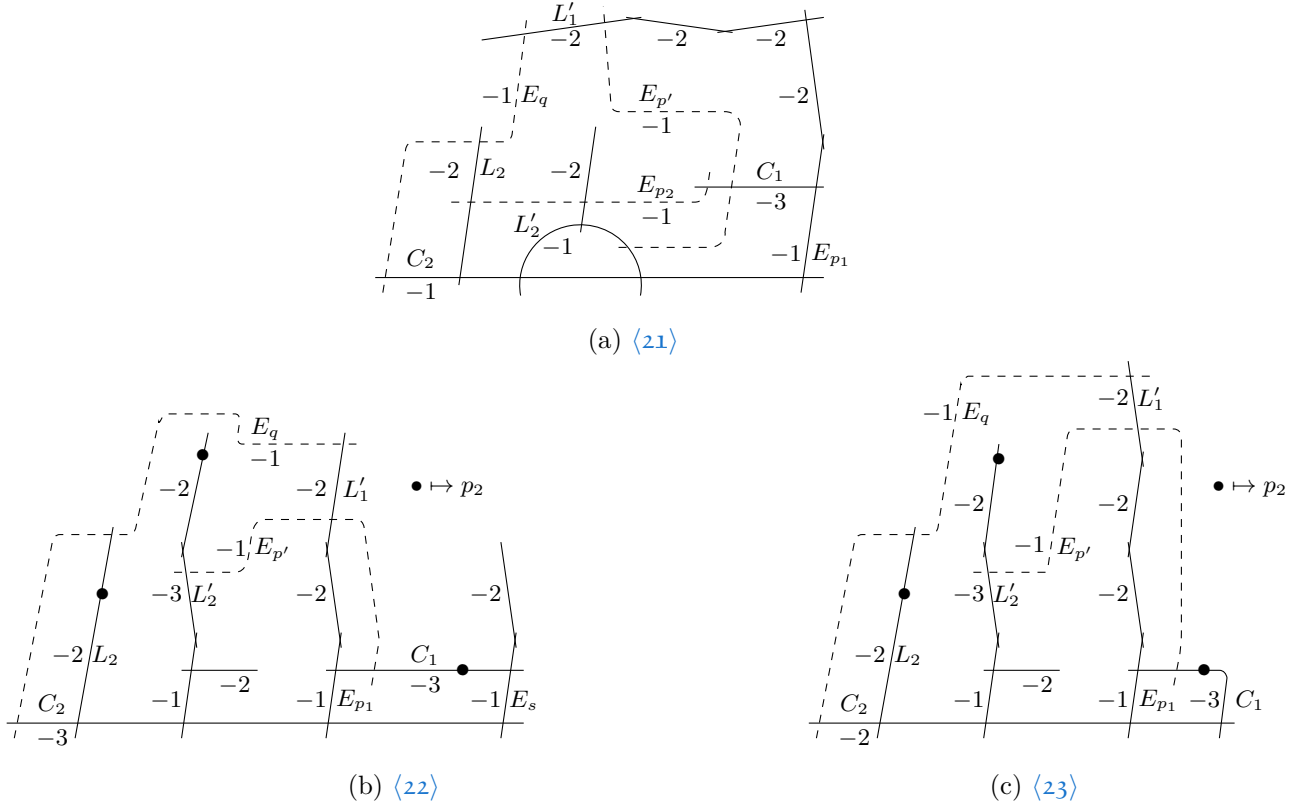
	Consider first the case $\delta_{p}=2$, so $s=p'\in \hat{\ll}_{p}$. Put $\cc_{1}=\sigma_{*}\hh_{1}$, $\cc_{2}=\sigma_{*}\hh_2$ and $\ll_{2}=(\sigma^{-1})^{*}(\hat{p})$, $\ll_{1}'=(\sigma^{-1})^{*}(r)$, $\ll_{2}'=(\sigma^{-1})^{*}(s)$, that is, $\ll_{2}$, $\ll_{1}'$, $\ll_{2}'$ are the images of the exceptional curves over $\hat{p}$, $r$, $s$, respectively. Then $\cc_{1}$, $\cc_{2}$ are conics meeting with multiplicity $(\hh_{1}\cdot\hh_{2})_{r}-1$ at some point $p_{1}\in \ll_{1}'$; $\cc_{1}$ is tangent to $\ll_{2}$ at $\ll_{2}'\cap \ll_{2}$ and passes through $\ll_{1}'\cap\ll_{2}'$; and $\cc_{2}$ passes through $\ll_{2}\cap \ll_{1}'$. If $p\in H_{2}$ was a node then $(\cc_{1}\cdot \cc_{2})_{p_{1}}=4$ and $S$ is as in  \ref{def:F2-5}. If $p\in H_{2}$ was a cusp then $\ll_{2}'$ is tangent to $\cc_{1}$ and $S$ is as in  \ref{def:F2_n1-cusp-hor_32} if $(\cc_{1}\cdot\cc_{2})_{p_{1}}=3$ and \ref{def:F2_n1-cusp-hor_41} otherwise.
	
			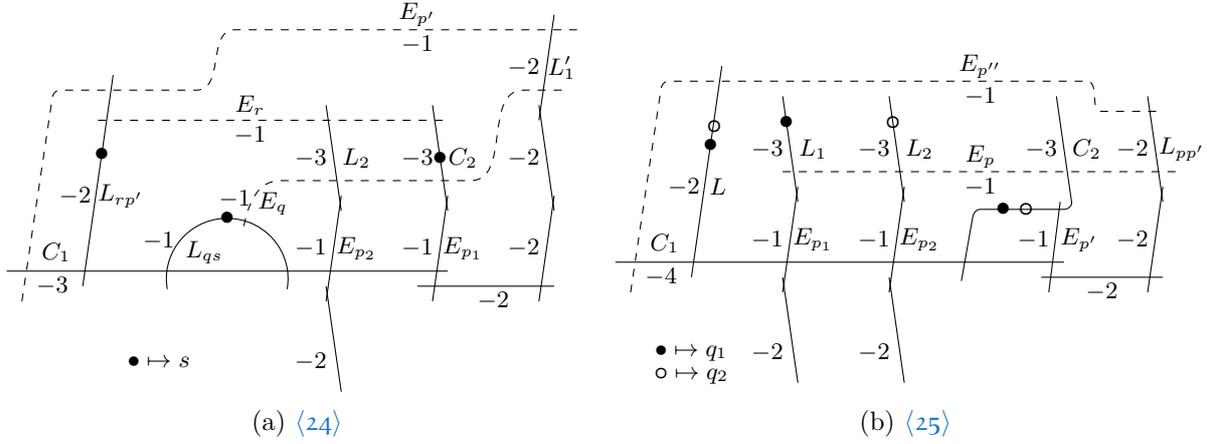
\begin{figure}[htbp]
	\begin{tabular}{cc}
		\begin{subfigure}[b]{0.45\textwidth}
			\begin{tikzpicture}
				\node at (-3,0) {\small{$\bullet\mapsto s$}};	
				\draw (-5,1.2) -- (0.8,1.2);
				\node at (-4.35,1.45) {\small{$C_1$}};
				\node at (-4.4,1) {\small{$-3$}};
				\draw (-2.1,1.1) [partial ellipse=190:-10 : 0.8 and 0.8];
				\node at (-2.4,1.45) {\small{$L_{qs}$}};
				\node at (-3,1.6) {\small{$-1$}};
				\node at (-2.1,1.9) {\large{$\bullet$}};
				\draw (-0.6,-0.4) -- (-0.8,1);
				\node at (-1,0) {\small{$-2$}};
				\draw (-0.8,0.8) -- (-0.6,2.2); 
				\node at (-1,1.5) {\small{$-1$}};
				\node at (-0.4,1.5) {\small{$E_{p_2}$}};			
				\draw (-0.6,2) -- (-0.8,3.4);
				\node at (-1.0,2.7) {\small{$-3$}};
				\node at (-0.4,2.7) {\small{$L_{2}$}};
				\draw (0.4,1) -- (2.2,1);
				\node at (1.4,0.8) {\small{$-2$}};
				\draw (0.6,0.8) -- (0.8,2.2); 
				\node at (0.4,1.5) {\small{$-1$}};
				\node at (1,1.5) {\small{$E_{p_1}$}};			
				\draw (0.8,2) -- (0.6,3.4);
				\node at (0.4,2.7) {\small{$-3$}};
				\node at (1,2.7) {\small{$C_{2}$}};
				\node at (0.7,2.7) {\large{$\bullet$}};
				\draw (2,0.8) -- (2.2,2.2); 
				\node at (1.8,1.5) {\small{$-2$}};			
				\draw (2.2,2) -- (2,3.4);
				\node at (1.8,2.7) {\small{$-2$}};
				\draw (2,3.2) -- (2.2,4.6);
				\node at (1.8,3.9) {\small{$-2$}};
				\node at (2.3,3.9) {\small{$L_{1}'$}};
				\draw (-4,1) -- (-3.6,3.8);
				\node at (-4.1,2.2) {\small{$-2$}};
				\node at (-3.5,2.2) {\small{$L_{rp'}$}};
				\node at (-3.75,2.75) {\large{$\bullet$}};
				\draw[dashed] (-4.8,0.8) -- (-4.4,3.5) to[out=80,in=180] (-4.3,3.6) -- (-2.5,3.6) to[out=0,in=180] (-2,4.4) -- (2.5,4.4);
				\node at (0.4,4.2) {\small{$-1$}};
				\node at (0.4,4.6) {\small{$E_{p'}$}};
				\draw[dashed] (-3.8,3.2) -- (0.8,3.2);
				\node at (-1.8,3) {\small{$-1$}};
				\node at (-1.8,3.4) {\small{$E_r$}};
				\draw[dashed] (2.3,3.6) -- (1.8,3.6) to[out=180,in=0] (1.1,2.4) -- (-1.4,2.4) to[out=180,in=80] (-1.9,1.7);	
				\node at (-2,2.15) {\small{$-1$}};
				\node at (-1.5,2.1) {\small{$E_q$}};
			\end{tikzpicture}
			\caption{\ref{def:F2-5_cn}}
			\label{fig:F2-5_cn}
		\end{subfigure}
		&
		\begin{subfigure}[b]{0.45\textwidth}
			\begin{tikzpicture}
				\node at (-4,0) {\small{$\bullet\mapsto q_1$}};
				\node at (-4,-0.3) {\small{$\boldsymbol{\circ}\mapsto q_2$}};	
				\draw (-5,1.2) -- (0.9,1.2);
				\node at (-4.35,1.45) {\small{$C_1$}};
				\node at (-4.4,1) {\small{$-4$}};
				\draw (-2.6,-0.4) -- (-2.8,1);
				\node at (-3,0) {\small{$-2$}};
				\draw (-2.8,0.8) -- (-2.6,2.2); 
				\node at (-3,1.5) {\small{$-1$}};
				\node at (-2.4,1.5) {\small{$E_{p_1}$}};			
				\draw (-2.6,2) -- (-2.8,3.4);
				\node at (-3,2.7) {\small{$-3$}};
				\node at (-2.4,2.7) {\small{$L_{1}$}};
				\node at (-2.75,3.05) {\large{$\bullet$}};
				\node at (-1.35,3.05) {$\boldsymbol{\circ}$};
				\draw (-1.2,-0.4) -- (-1.4,1);
				\node at (-1.6,0) {\small{$-2$}};
				\draw (-1.4,0.8) -- (-1.2,2.2); 
				\node at (-1.6,1.5) {\small{$-1$}};
				\node at (-1,1.5) {\small{$E_{p_2}$}};			
				\draw (-1.2,2) -- (-1.4,3.4);
				\node at (-1.6,2.7) {\small{$-3$}};
				\node at (-1,2.7) {\small{$L_{2}$}};
				\draw (0.6,1) -- (2.2,1);
				\node at (1.4,0.8) {\small{$-2$}};
				\draw (0.7,0.8) -- (0.85,2); 
				\node at (0.5,1.5) {\small{$-1$}};
				\node at (1.1,1.5) {\small{$E_{p'}$}};			
				\draw (0.8,3.4) -- (1,2) to[out=-80,in=0] (0.9,1.9) -- (-0.2,1.9) to[out=180,in=80] (-0.3,1.8) -- (-0.45,0.95);
				\node at (0.6,2.7) {\small{$-3$}};
				\node at (1.2,2.7) {\small{$C_{2}$}};
				\node at (0.1,1.9) {\large{$\bullet$}};
				\node at (0.4,1.9) {$\boldsymbol{\circ}$};
				\draw (2,0.8) -- (2.2,2.2); 
				\node at (1.8,1.5) {\small{$-2$}};			
				\draw (2.2,2) -- (2,3.4);
				\node at (1.8,2.7) {\small{$-2$}};
				\node at (2.45,2.7) {\small{$L_{pp'}$}};
				\draw (-4,1) -- (-3.6,3.8);
				\node at (-4.1,2.2) {\small{$-2$}};
				\node at (-3.65,2.2) {\small{$L$}};
				\node at (-3.75,2.75) {\large{$\bullet$}};
				\node at (-3.7,3) {$\boldsymbol{\circ}$};
				\draw[dashed] (-4.8,0.8) -- (-4.4,3.5) to[out=80,in=180] (-4.3,3.6) -- (1.2,3.6) to[out=0,in=180] (1.4,3.2) -- (2.2,3.2);
				\node at (-0.2,3.4) {\small{$-1$}};
				\node at (-0.2,3.8) {\small{$E_{p''}$}};
				\draw[dashed] (-2.8,2.4) -- (2.4,2.4);
				\node at (-0.2,2.2) {\small{$-1$}};
				\node at (-0.2,2.6) {\small{$E_{p}$}};
			\end{tikzpicture}
			\caption{\ref{def:F2-hor-ccc-41}}
			\label{fig:F2-hor-ccc-41}
		\end{subfigure}
	\end{tabular}
	\caption{Cases when $X\am\cong \F_{2}$, $D\am$ is horizontal, $n=1$ and $\delta_{p}=1$.}
	\label{fig:F2n1_hor_1}
\end{figure}
	Consider now the case $\delta_{p}=1$, so $s\in \hh_{2}\setminus \hat{\ll}_{p}$ is the image of $\Sing H_{2}\setminus \{p\}$, which is a cusp. Put $\cc_{1}= \sigma_{*}\hh_{2}$, $\cc_{2}= \sigma_{*}\hh_{1}$ and $\ll_{2}= (\sigma^{-1})^{*}(s)$. Then again $\cc_{1}$, $\cc_{2}$ are conics meeting with multiplicity $(\hh_{1}\cdot\hh_{2})_{r}-1$ at some point of the line $(\sigma^{-1})^{*}(r)$, call this point $\hat{r}$. Moreover, $\ll_{2}$ is a line tangent to $\cc_{1}$, $\sigma_{*}\hat{\ll}_{p}$ is a line meeting $\cc_{2}$ at $\hat{q}\de \ll_{2}\cap  (\sigma^{-1})^{*}(r)$, and the points $\cc_{1}\cap (\sigma^{-1})^{*}(r)\setminus \{\hat{r}\}$, $\cc_{2}\cap \sigma_{*}\hat{\ll}_{p}\setminus \{\hat{q}\}$ and $\cc_{2}\cap \ll_{2}\setminus \{\hat{q}\}$ lie on $(\sigma^{-1})^{*}(\hat{p})$. If $p\in H_{2}$ was a node then $(\cc_{1}\cdot \cc_{2})_{\hat{r}}=4$ and putting $\ll_{1}'\de (\sigma^{-1})^{*}(r)$, $\ll_{rp'}\de (\sigma^{-1})^{*}(\hat{p})$, $\ll_{qs}\de\sigma_{*}\hat{\ll}_{p}$ we see that $S$ is as in  \ref{def:F2-5_cn}. If $p\in H_{2}$ was a cusp then $\sigma_{*}\hat{\ll}_{p}$ is tangent to $\cc_{1}$, and  putting $\ll_{pp'}\de (\sigma^{-1})^{*}(r)$, $\ll\de (\sigma^{-1})^{*}(\hat{p})$ and $\ll_{1}\de\sigma_{*}\hat{\ll}_{p}$ we see that $S$ is as in  \ref{def:F2-hor-ccc-41} (in particular, $(\cc_{1}\cdot \cc_{2})_{\hat{r}}=3$).
\smallskip

	Eventually, consider the case $n=0$. Then $R\am\subseteq \F_2$ is a rational cuspidal curve, whose cusps have multiplicity two. By condition \eqref{eq:F2_delta}, $\delta$-invariants of those cusps are at most two. Hence $R\am$ is as in Proposition \ref{prop:n0}\ref{item:n0_F2} below, so $S$ is as in \ref{def:F2n0}.
	
\subsection{Case \texorpdfstring{$X_{\min}\cong \P^2$}{Xmin=P2}}\label{sec:P2} 

Recall that by formula \eqref{eq:R}, $D_{\min}$ has $2n$ self intersection points and $n+1$ components, all of which are rational by Lemma \ref{lem:expansions}\ref{item:D_rat-tree}; and all singularities of $D_{\min}$ are of multiplicity $2$ by Lemma \ref{lem:MMP}\ref{item:mult_2}. By Lemma \ref{lem:R}\ref{item:MFS_point}, $-(2K_{X_{\min}}+D_{\min})$ is ample, so $\deg D_{\min}\leq 5$. Suppose $\deg D_{\min}\leq 4$. If $\deg D_{\min}=4$ then since all components of $D_{\min}$ are rational, $D_{\min}$ has a singular point $q$; if $\deg D_{\min}\leq 3$ fix any $q\in D_{\min}$. The projection from $q$ restricts to a $\C^{1}$-, $\C^{*}$- or a $\C^{**}$-fibration of $\P^{2}\setminus D_{\min}$, contrary to condition \eqref{eq:assumption_fibr}. 

Thus $\deg D_{\min}=5$. We note that 
\begin{equation}\label{eq:P2_ordinary}
\mbox{if } p\in \Sing D_{\min} \mbox{ is not ordinary (i.e.\ $\delta_{p}>1$) then } p \mbox{ lies on a line in } D_{\min}.
\end{equation}
Indeed, otherwise by Lemma \ref{lem:no_lines} the line $\ll_{p}$ tangent to $D_{\min}$ at $p$ satisfies $(\ll_{p}\cdot  D_{\min})_{p}=4$, so its proper transform is a bubble on $(X',D')$, contrary to condition \eqref{eq:no_bubbles}. It follows that
\begin{equation}\label{eq:P2_lines}
\mbox{exactly one component of } D_{\min} \mbox{ is not a line.}
\end{equation}
Indeed, if $D_{\min}$ has two components $\cc$, $\cc'$ which are not lines then by condition  \eqref{eq:P2_ordinary} $\cc$ meets $\cc'$  in $\cc\cdot \cc'\geq 4$ points off $D_{\min}-\cc-\cc'$. Since $\#D_{\min}=n+1$ and  $D_{\min}$ has $2n$ nodes, we get $2n\geq \#D_{\min}-2+\cc\cdot \cc'\geq n+3$, hence $4\leq \#D_{\min}\leq \deg (D_{\min}-\cc-\cc')+2\leq 3$, which is impossible. On the other hand, if $D_{\min}$ is a sum of lines then $n+1=\#D_{\min}=\deg D_{\min}=5$, and  $D_{\min}$ has $10=2n$ nodes, which is false.

In particular, condition \eqref{eq:P2_lines} implies $n\leq 3$. Suppose $n=0$. Then $D_{\min}$ is a rational cuspidal  quintic, so $\#\Sing D_{\min}=6$ by condition \eqref{eq:P2_ordinary} and  Lemma  \ref{lem:P2-curves}\ref{item:P2-adjunction}. The projection from a cusp $p\in D_{\min}$ induces a map $(\bl{p}^{-1})_{*}D_{\min}\to \P^{1}$ of degree $3$, ramified at the preimages of the five remaining cusps; a contradiction with  the Hurwitz formula. Thus $n\in \{1,2,3\}$. 
We use Notation \ref{not:P2} for curves on $\P^{2}$ and their proper transforms on $X'$.

\begin{figure}[htbp]
	\begin{tikzpicture}[scale=0.7]
		\draw (0,0) circle (1);
		\draw (0, -0.8) -- (0,-2.6);
		\node at (-0.4,-1.6) {\small{$-1$}};
		\node at (0.4,-1.6) {\small{$E_{p_1}$}};
		\draw (-0.2,-2.4) -- (1.6,-2.4);
		\node at (0.7,-2.7) {\small{$-2$}};
		\draw (-3.7,-2) -- (3.7,-2);
		\node at (-2.2,-1.7) {\small{$L_{1}$}};
		\node at (-2.2,-2.3) {\small{$-1$}};
		\draw (-0.7,0.4) -- (-2.25,1.2);
		\node at (-1.2,0.9) {\small{$-1$}};
		\node at (-1.4,0.4) {\small{$E_{p_3}$}};
		\draw (-2,1.3) -- (-2.85,-0.15);
		\node at (-2.9,0.5) {\small{$-2$}};
		\draw (-3.5,-2.1) -- (0.1,4);
		\node at (-0.3,2.7) {\small{$L_3$}};
		\node at (-1.1,2.8) {\small{$-1$}};  
		\draw (0.7,0.4) -- (2.25,1.2);
		\node at (1.1,1) {\small{$E_{p_2}$}};
		\node at (1.4,0.4) {\small{$-1$}};
		\draw (2.2,0.9) -- (1.25,2.5);
		\node at (1.9,2) {\small{$-2$}};
		\draw (3.5,-2.1) -- (-0.1,4);
		\node at (2.5,-1) {\small{$L_2$}};
		\node at (3,-0.7) {\small{$-1$}};
		\node at (0,0.75) {\small{$-2$}};
		\node at (0,1.25) {\small{$C_1$}};
	\end{tikzpicture}
	\caption{Case \ref{def:P2n3}: $X_{\min}\cong \P^{2}$ and $n=3$. Then $D_{\min}$ is a conic inscribed in a triangle.}
	\label{fig:P2n3}
\end{figure}
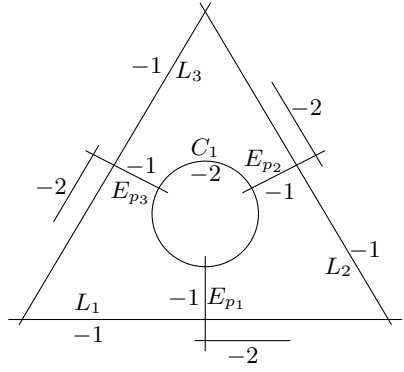
Consider the case $n=3$. Then $D_{\min}=\cc_1+\ll_{1}+\ll_{2}+\ll_{3}$, where $\cc_1$ is a conic inscribed in a triangle $\ll_{1}+\ll_{2}+\ll_{3}$ as in  \ref{def:P2n3}. Write $\{p_i\}=\cc_1\cap \ll_{i}$ for $i\in \{1,2,3\}$. If the expansion $\psi$ has a center at a preimage of $\ll_{i}\cap \ll_{j}$ for some $i,j\in \{1,2,3\}$, then the corresponding weight cannot be equal to one. Indeed, otherwise the pencil of lines through $\ll_{i}\cap \ll_{j}$  restricts to a $\C^{**}$-fibration of $S$, which is impossible by assumption \eqref{eq:assumption}.

We will now use the action of $\Aut(X',D')=\Aut(\P^2,D_{\min})\cong S_3$ to prove that the centers of $\psi$, say $p,q,r$, are as in Table \ref{table:P2n3}. For $\{i,j,k\}=\{1,2,3\}$ define $\tau_{ij}\in \Aut(X',D')$ by  $\tau_{ij}(L_i)=L_j$, $\tau_{ij}(L_{k})=L_{k}$.

Put $T=L_1+L_2+L_3$. Since $T$ is circular, one of the centers of $\psi$ is a node of $T$, say $\{p\}=L_1\cap L_2$. More precisely, the first center is $(L_1,L_2;u)$ for some $u\in \Q_{>0}\setminus \{1\}$. Assume that $q,r\not\in \Sing T$, so $q\in E_{p_i}$, $r\in E_{p_j}$ for some $\{i,j\}\subseteq \{1,2,3\}$, see Figure \ref{fig:P2n3}. Since $D$ is connected, we have $i\neq j$, so, say, $i<j$. Using the action of $\tau_{12}$, we can assume that $i=1$. Furthermore, if $j=2$ then using $\tau_{12}$ again we can assume that $u>1$ if $q,r\in C_1$ or $q,r\not\in C_1$; and $q\in C_1$, $r\not\in C_1$ otherwise. 

Assume now that $q$ is a node of $T$. Using $\tau_{12}$ we can assume that $\{q\}=L_2\cap L_3$. If $r\not\in \Sing T$ then since $D$ is connected, we have $r\in E_{p_1}\cup E_{p_3}$, and using $\tau_{13}$ we can assume $r\in E_{p_1}$. Assume $r\in \Sing T$, so $\{r\}=L_1\cap L_3$. Then the expansion $\psi$ is centered at  $(L_1,L_2;u)$, $(L_2,L_3;v)$, $(L_3,L_1;w)$ for some $u,v,w\in \Q_{>0}\setminus \{1\}$. The action of $\tau_{12}$ replaces the triple $(u,v,w)$ by $(\frac{1}{u},\frac{1}{w},\frac{1}{v})$, so we can assume that, say, $u,v>1$. Using the $3$-cycle in $\Aut(X',D')\cong S_3$, we can assume that $u\geq v,w$, as needed.
			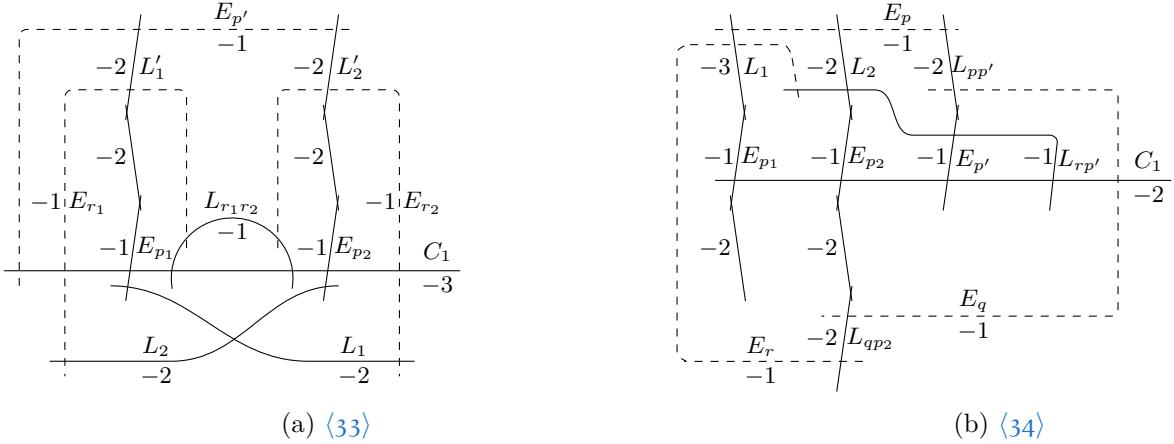
\begin{figure}[htbp]
	\begin{tabular}{cc}
		\begin{subfigure}[b]{0.5\textwidth}
			\begin{tikzpicture}	
				\draw (-5,1.2) -- (1,1.2);
				\node at (0.7,1.45) {\small{$C_1$}};
				\node at (0.7,1) {\small{$-3$}};
				\draw (-2,1.1) [partial ellipse=190:-10 : 0.8 and 0.8];
				\node at (-2,2.1) {\small{$L_{r_1 r_2}$}};
				\node at (-2,1.7) {\small{$-1$}};
				\draw (-0.8,0.8) -- (-0.6,2.2); 
				\node at (-0.95,1.5) {\small{$-1$}};
				\node at (-0.4,1.5) {\small{$E_{p_2}$}};			
				\draw (-0.6,2) -- (-0.8,3.4);
				\node at (-1,2.7) {\small{$-2$}};
				\draw (-0.8,3.2) -- (-0.6,4.6);
				\node at (-1,3.9) {\small{$-2$}};
				\node at (-0.45,3.9) {\small{$L_{2}'$}};
				\draw (-3.4,0.8) -- (-3.2,2.2); 
				\node at (-3.55,1.5) {\small{$-1$}};
				\node at (-3,1.5) {\small{$E_{p_1}$}};			
				\draw (-3.2,2) -- (-3.4,3.4);
				\node at (-3.6,2.7) {\small{$-2$}};
				\draw (-3.4,3.2) -- (-3.2,4.6);
				\node at (-3.6,3.9) {\small{$-2$}};
				\node at (-3.05,3.9) {\small{$L_{1}'$}};
				\draw (-3.6,1) to[out=0,in=180] (-1,0) -- (0.4,0);
				\node at (-3,0.2) {\small{$L_2$}};
				\node at (-3,-0.2) {\small{$-2$}};
				\draw (-0.6,1) to[out=180,in=0] (-2.8,0) -- (-4.4,0);
				\node at (-0.4,0.2) {\small{$L_1$}};
				\node at (-0.4,-0.2) {\small{$-2$}};
				\draw[dashed] (-4.8,1) -- (-4.8,4.3) to[out=90,in=180] (-4.7,4.4) -- (-0.4,4.4);
				\node at (-2,4.2) {\small{$-1$}};
				\node at (-2,4.6) {\small{$E_{p'}$}};
				\draw[dashed] (-2.6,1.5) -- (-2.6,3.5) to[out=90,in=0] (-2.7,3.6) -- (-4.1,3.6) to[out=180,in=90] (-4.2,3.5) -- (-4.2,-0.2);
				\node at (-4.45,2.1) {\small{$-1$}};
				\node at (-3.9,2.1) {\small{$E_{r_1}$}};
				\draw[dashed] (-1.4,1.5) -- (-1.4,3.5) to[out=90,in=180] (-1.3,3.6) -- (0.1,3.6) to[out=0,in=90] (0.2,3.5) -- (0.2,-0.2);
				\node at (-0.05,2.1) {\small{$-1$}};
				\node at (0.5,2.1) {\small{$E_{r_2}$}};
			\end{tikzpicture}
			\caption{\ref{def:nodal-cubic_P2}}
			\label{fig:nodal-cubic_P2}
		\end{subfigure}
		&
		\begin{subfigure}[b]{0.5\textwidth}
						\begin{tikzpicture}	
				\draw (-1,1.2) -- (5,1.2);
				\node at (4.7,1.45) {\small{$C_1$}};
				\node at (4.7,1) {\small{$-2$}};
				\draw (-0.6,-0.4) -- (-0.8,1);
				\node at (-1,0.3) {\small{$-2$}};
				\draw (-0.8,0.8) -- (-0.6,2.2); 
				\node at (-0.95,1.5) {\small{$-1$}};
				\node at (-0.4,1.5) {\small{$E_{p_1}$}};			
				\draw (-0.6,2) -- (-0.8,3.4);
				\node at (-1,2.7) {\small{$-3$}};
				\node at (-0.45,2.7) {\small{$L_1$}};
				\draw (0.6,-1.6) -- (0.8,-0.2);
				\node at (0.4,-0.9) {\small{$-2$}};
				\node at (1.05,-0.9) {\small{$L_{qp_2}$}};
				\draw (0.8,-0.4) -- (0.6,1);
				\node at (0.4,0.3) {\small{$-2$}};
				\draw (0.6,0.8) -- (0.8,2.2); 
				\node at (0.45,1.5) {\small{$-1$}};
				\node at (1,1.5) {\small{$E_{p_2}$}};			
				\draw (0.8,2) -- (0.6,3.4);
				\node at (0.4,2.7) {\small{$-2$}};
				\node at (0.95,2.7) {\small{$L_2$}};
				\draw (2,0.8) -- (2.2,2.2); 
				\node at (1.85,1.5) {\small{$-1$}};
				\node at (2.4,1.45) {\small{$E_{p'}$}};			
				\draw (2.2,2) -- (2,3.4);
				\node at (1.8,2.7) {\small{$-2$}};
				\node at (2.4,2.7) {\small{$L_{pp'}$}};
				\draw (-0.1,2.4) -- (1.1,2.4) to[out=0,in=180] (1.6,1.8) -- (3.4,1.8) to[out=0,in=80] (3.5,1.7) -- (3.4,0.8);
				\node at (3.25,1.5) {\small{$-1$}};
				\node at (3.8,1.45) {\small{$L_{rp'}$}};
				\draw[dashed] (0.1,2.3) -- (0,2.9) to[out=-80,in=0] (-0.1,3) -- (-1.4,3) to[out=0,in=90] (-1.5,2.9) -- (-1.5,-1.1) to[out=-90,in=0] (-1.4,-1.2) -- (1,-1.2);
				\node at (-0.4,-1) {\small{$E_r$}};
				\node at (-0.4,-1.4) {\small{$-1$}};
				\draw[dashed] (1.8,2.4) -- (4.2,2.4) to[out=0,in=90] (4.3,2.3) -- (4.3,-0.5) to[out=-90,in=0] (4.2,-0.6) -- (0.4,-0.6);
				\node at (2.4,-0.4) {\small{$E_q$}}; 
				\node at (2.4,-0.8) {\small{$-1$}};
				\draw[dashed] (-1,3.2) -- (2.2,3.2);
				\node at (1.4,3.4) {\small{$E_{p}$}};
				\node at (1.4,3) {\small{$-1$}};
			\end{tikzpicture}
			\caption{\ref{def:P2n2_cuspidal}}
			\label{fig:P2n2_cuspidal}
		\end{subfigure}
	\end{tabular}
	\caption{Cases when $X_{\min}\cong \P^{2}$ and $n=2$.}
	\label{fig:P2n2}
\end{figure}
		
Consider the case $n=2$. Then by condition \eqref{eq:P2_lines}, $D_{\min}$ consists of a rational cubic $\qq$ and two lines $\ll_{1}$, $\ll_{2}$. The cubic $\qq$ does not pass through $\ll_{1}\cap\ll_{2}$, so $\#\qq\cap\ll_{1}=1$ and $\qq$ meets $\ll_{2}$ in one point if $\qq$ is nodal and in two if $\qq$ is cuspidal. Let $\sigma\colon \P^{2}\map \P^{2}$ be a standard quadratic transformation centered at $p\de \Sing \qq$, $p_{1}\de \ll_{1}\cap \qq$ and the tangency point $p_{2}$ of $\ll_{2}$ and $\qq$. In other words, we blow up these points and contract the proper transforms of the lines joining them. Then $(\sigma^{-1})^{*}(\qq+\ll_{1}+\ll_{2})$ is as in  \ref{def:nodal-cubic_P2} if $\qq$ is nodal and \ref{def:P2n2_cuspidal} if $\qq$ is cuspidal. Indeed, $\sigma_{*}\qq$, $\sigma_{*}\ll_{1}$ and $\sigma_{*}\ll_{2}$ correspond there to $\cc_{1}$, $\ll_{2}$ and $\ll_{1}$ if $\qq$ is nodal or to $\ll_{rp'}$ if $\qq$ is cuspidal. The images of the exceptional curves over $p$, $p_{1}$, $p_{2}$ correspond there to $\ll_{r_{1}r_{2}}$, $\ll_{2}'$, $\ll_{1}'$ if $\qq$ is nodal and to $\ll_{1}$, $\ll_{qp_{2}}$, $\ll_{pp'}$ if $\qq$ is cuspidal. Note that in case \ref{def:nodal-cubic_P2} the weight $w$ corresponding to $(L_{1},L_{2})$ is different than $1$, for otherwise the pencil of conics through $p'$ $p$, $r_{1}$, $r_{2}$ would induce a $\C^{**}$-fibration of $S$. By symmetry we can assume $w>1$, and the centers are as in Table \ref{table:smooth_23}. In case \ref{def:P2n2_cuspidal}, the weight corresponding to $(C_{1},L_{rp'})$ or $(L_{2},L_{rp'})$ is different than $1$, since otherwise the pencil of conics through $p$, $q$, $r$ and $\cc_{1}\cap \ll_{rp'}$ or $\ll_{2}\cap \ll_{rp'}$ would induce a $\C^{**}$-fibration of $S$, contrary to assumption \eqref{eq:assumption}.
\smallskip

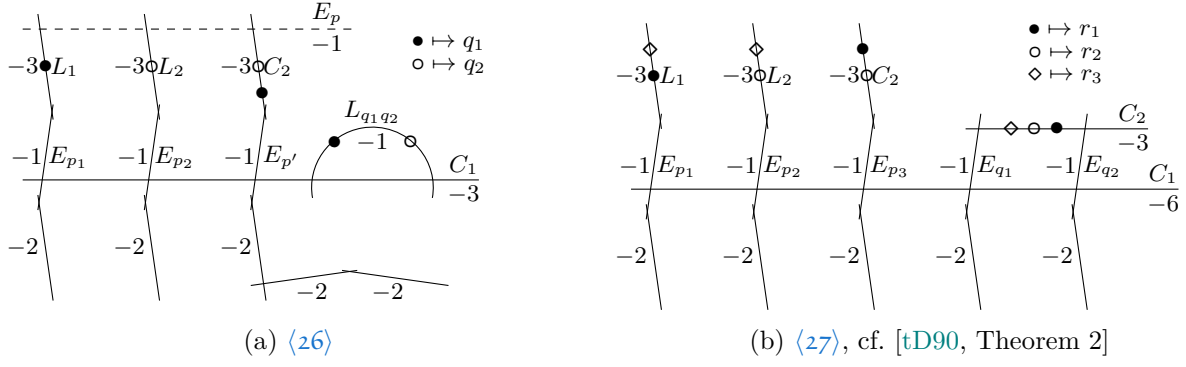
\begin{figure}[htbp]
	\begin{tabular}{cc}
		\begin{subfigure}[b]{0.45\textwidth}
					\begin{tikzpicture}	
			\node at (-1,3) {\small{$\bullet\mapsto q_1$}};
			\node at (-1,2.7) {\small{$\boldsymbol{\circ}\mapsto q_2$}};	
			\draw (-6.6,1.2) -- (-0.6,1.2);
			\node at (-0.8,1.4) {\small{$C_1$}};
			\node at (-0.8,1) {\small{$-3$}};
			\draw (-2,1.1) [partial ellipse=190:-10 : 0.8 and 0.8];
			\node at (-2,2.1) {\small{$L_{q_1 q_2}$}};
			\node at (-2,1.7) {\small{$-1$}};
			\node at (-2.5,1.7) {\large{$\bullet$}};
			\node at (-1.5,1.7) {$\boldsymbol{\circ}$};
			\draw (-3.4,-0.4) -- (-3.6,1);
			\node at (-3.8,0.3) {\small{$-2$}};
			\draw (-3.6,0.8) -- (-3.4,2.2); 
			\node at (-3.75,1.5) {\small{$-1$}};
			\node at (-3.2,1.5) {\small{$E_{p'}$}};			
			\draw (-3.4,2) -- (-3.6,3.4);
			\node at (-3.8,2.7) {\small{$-3$}};
			\node at (-3.25,2.7) {\small{$C_2$}};
			\node at (-3.5,2.7) {$\boldsymbol{\circ}$};
			\node at (-3.45,2.35) {\large{$\bullet$}};
			\draw (-3.6,-0.2) -- (-2.2,0);
			\node at (-2.8,-0.3) {\small{$-2$}};
			\draw (-2.4,0) -- (-1,-0.2);
			\node at (-1.8,-0.3) {\small{$-2$}};
			\draw (-4.8,-0.4) -- (-5,1);
			\node at (-5.2,0.3) {\small{$-2$}};
			\draw (-5,0.8) -- (-4.8,2.2); 
			\node at (-5.15,1.5) {\small{$-1$}};
			\node at (-4.6,1.5) {\small{$E_{p_2}$}};			
			\draw (-4.8,2) -- (-5,3.4);
			\node at (-5.2,2.7) {\small{$-3$}};
			\node at (-4.65,2.7) {\small{$L_2$}};
			\node at (-4.9,2.7) {$\boldsymbol{\circ}$};
			\draw (-6.2,-0.4) -- (-6.4,1);
			\node at (-6.6,0.3) {\small{$-2$}};
			\draw (-6.4,0.8) -- (-6.2,2.2); 
			\node at (-6.55,1.5) {\small{$-1$}};
			\node at (-6,1.5) {\small{$E_{p_1}$}};			
			\draw (-6.2,2) -- (-6.4,3.4);
			\node at (-6.6,2.7) {\small{$-3$}};
			\node at (-6.05,2.7) {\small{$L_1$}};
			\node at (-6.3,2.7) {\large{$\bullet$}};
			\draw[dashed] (-6.6,3.2) -- (-2.2,3.2);
			\node at (-2.6,3) {\small{$-1$}};
			\node at (-2.6,3.4) {\small{$E_p$}};
		\end{tikzpicture}
			\caption{\ref{def:P2n1-nodal}}
			\label{fig:P2n1-nodal}
		\end{subfigure}
		&
		\begin{subfigure}[b]{0.5\textwidth}
		\begin{tikzpicture}	
			\node[right] at (-1.5,3.3) {\small{$\large{\bullet}\mapsto r_1$}};
			\node[right] at (-1.5,3) {\small{$\boldsymbol{\circ}\mapsto r_2$}};
			\node[right] at (-1.5,2.7) {\small{$\boldsymbol{\diamond}\mapsto r_3$}};	
			\draw (-6.6,1.2) -- (0.6,1.2);
			\node at (0.4,1.4) {\small{$C_1$}};
			\node at (0.4,1) {\small{$-6$}};
			\draw (-3.4,-0.4) -- (-3.6,1);
			\node at (-3.8,0.3) {\small{$-2$}};
			\draw (-3.6,0.8) -- (-3.4,2.2); 
			\node at (-3.75,1.5) {\small{$-1$}};
			\node at (-3.2,1.5) {\small{$E_{p_3}$}};			
			\draw (-3.4,2) -- (-3.6,3.4);
			\node at (-3.8,2.7) {\small{$-3$}};
			\node at (-3.25,2.7) {\small{$C_2$}};
			\node at (-3.5,2.7) {$\boldsymbol{\circ}$};
			\node at (-3.55,3.05) {\large{$\bullet$}};
			\draw (-4.8,-0.4) -- (-5,1);
			\node at (-5.2,0.3) {\small{$-2$}};
			\draw (-5,0.8) -- (-4.8,2.2); 
			\node at (-5.15,1.5) {\small{$-1$}};
			\node at (-4.6,1.5) {\small{$E_{p_2}$}};			
			\draw (-4.8,2) -- (-5,3.4);
			\node at (-5.2,2.7) {\small{$-3$}};
			\node at (-4.65,2.7) {\small{$L_2$}};
			\node at (-4.9,2.7) {$\boldsymbol{\circ}$};
			\node at (-4.95,3.05) {$\boldsymbol{\diamond}$};
			\draw (-6.2,-0.4) -- (-6.4,1);
			\node at (-6.6,0.3) {\small{$-2$}};
			\draw (-6.4,0.8) -- (-6.2,2.2); 
			\node at (-6.55,1.5) {\small{$-1$}};
			\node at (-6,1.5) {\small{$E_{p_1}$}};			
			\draw (-6.2,2) -- (-6.4,3.4);
			\node at (-6.6,2.7) {\small{$-3$}};
			\node at (-6.05,2.7) {\small{$L_1$}};
			\node at (-6.3,2.7) {\large{$\bullet$}};
			\node at (-6.35,3.05) {$\boldsymbol{\diamond}$};
			\draw (-2,-0.4) -- (-2.2,1);
			\node at (-2.4,0.3) {\small{$-2$}};
			\draw (-2.2,0.8) -- (-2,2.2); 
			\node at (-2.35,1.5) {\small{$-1$}};
			\node at (-1.8,1.5) {\small{$E_{q_1}$}};
			\draw (-0.6,-0.4) -- (-0.8,1);
			\node at (-1,0.3) {\small{$-2$}};
			\draw (-0.8,0.8) -- (-0.6,2.2); 
			\node at (-0.95,1.5) {\small{$-1$}};
			\node at (-0.4,1.5) {\small{$E_{q_2}$}};
			\draw (-2.2,2) -- (0.2,2);
			\node at (0,2.2) {\small{$C_2$}};	
			\node at (0,1.8) {\small{$-3$}};
			\node at (-1.6,2) {$\boldsymbol{\diamond}$};
			\node at (-1.3,2) {$\boldsymbol{\circ}$};
			\node at (-1,2) {\large{$\bullet$}};
		\end{tikzpicture}
			\caption{\ref{def:P2n1-cuspidal}, cf.\ \cite[Theorem 2]{tDieck_symmetric_hp}}
			\label{fig:P2n1-cuspidal}
		\end{subfigure}
	\end{tabular}
	\caption{Cases when $X_{\min}\cong \P^{2}$ and $n=1$.}
	\label{fig:P2n1}
\end{figure}
Consider the case $n=1$. Then by condition \eqref{eq:P2_lines}, $D_{\min}$ consists of a line $\ll$ and a rational quartic $\qq$ with $2-\#(\ll\cap\qq)$ nodes. Lemma \ref{lem:P2-curves}\ref{item:P2-adjunction} and condition \eqref{eq:P2_ordinary} imply that $\qq$ has three singular points. Let $\sigma\colon \P^{2}\map \P^{2}$ be a standard quadratic transformation centered at $\Sing \qq$. The total transform of $\Sing \qq$ is a triangle inscribed in a conic $\sigma_{*}\ll$ and the conic $\sigma_{*}\qq$ is tangent to two of its sides if $\qq$ is nodal and three if $\qq$ is cuspidal. In the first case, $\#(\sigma_{*}\ll\cap\sigma_{*}\qq)=\#(\ll\cap \qq)=1$, so $(X',D')$ is as in  \ref{def:P2n1-nodal}. In the second case, write $\ll\cap \qq=\{p,q\}$ for some $p\neq q$. If $(\ll\cdot \qq)_{p}=1$, $(\ll\cdot \qq)_{q}=3$ then $|(\bl{p})^{-1}_{*}\ll|$ induces a $\P^{1}$-fibration which restricts to a morphism $(\bl{p})^{-1}_{*}\qq\to \P^{1}$ of degree $4$, ramified at thee cusps of $(\bl{p})^{-1}_{*}\qq$ and, with index $3$, at $\bl{p}^{-1}(q)$, which is impossible by the Hurwitz formula. Hence $(\ll\cdot \qq)_{p}=(\ll\cdot \qq)_{q}=2$ and $(X',D')$ is as in  \ref{def:P2n1-cuspidal}. Using $\Aut(\P^{2},\pp)$ as above we can assume that the center of $\psi'$ lies on $E_{q_{1}}$.

\subsection{Sporadic smooth $\Q$-homology planes}\label{sec:sporadic}

In this section we complete the proof of Theorem \ref{CLASS} in case $n=0$, omitted in Sections \ref{sec:A1A2} and \ref{sec:F2}.  This amounts to describe a very specific class of rational curves on the Hirzebruch surface $\F_2$.

Recall from Section \ref{sec:Fm} that we denote by $\Sec_2$ the negative section on $\F_2$, and we say that a divisor $C\subseteq \F_2$ is \emph{of type} $(a,b)$ if $C\cdot \Sec_2=a$ and $C\cdot F=b$ for a fiber $F$.

\begin{prop}\label{prop:n0}
Let $R\subseteq \F_2$ be a rational curve of type $(1,3)$, such that all cusps of $R$ have multiplicity $2$, and denote by $\boldsymbol{\delta}$ the sequence of their $\delta$-invariants, written in a non-increasing order.
	\begin{enumerate}
		\item\label{item:n0_F2} Assume that all singular points of $R$ are cusps with $\delta$-invariants at most $2$. Then such a curve $R$ exists, is unique up to an action of $\Aut(\F_2)$, and has $\boldsymbol{\delta}=(2,2,1,1)$. The surface $S\de \F_2\setminus (R+\Sec_2)$ is as in \ref{def:F2n0}.
		\item\label{item:n0_A1A2} Assume that $R$ has an ordinary node $q$, with one branch tangent to a fiber $F_q$; and all the singular points of $R\setminus \{q\}$ are cusps. Let $\phi\colon Y\to \F_2$ be blowup at $q$ and at its infinitely near point on the proper transform of $F_q$; let $E\subseteq \Exc \phi$ be the $(-1)$-curve, and let $S=Y\setminus (\pi^{*}(R+\Sec_2)\redd-E)$. Then $\boldsymbol{\delta}=(5),(4,1),(3,2)
		$ or $(3,1,1)$. The surface $S$ is as in \ref{def:A1A2_c=1}, \ref{def:A1A2_c=2_41}, \ref{def:A1A2_c=2_32} or \ref{def:A1A2_c=3}, respectively. Up to the action of  $\Aut(\F_2)$, there are exactly $3$ such curves $R$ if $\boldsymbol{\delta}=(5), (4,1)$ or $(3,2)$; and $1$ if $\boldsymbol{\delta}=(3,1,1)$.
	\end{enumerate}
In both cases, for each cusp $r\in R$ the fiber passing through $r$ satisfies $(F_r\cdot R)_{r}=2$ and $F_{r}\cap R\cap \Sec_2=\emptyset$.
\end{prop}
\begin{proof}
By Lemma \ref{lem:F2-adjunction}, we have  $\delta_{R}=6$. In case \ref{item:n0_F2} this implies that $R$ has at least three cusps. 

For $r\in \Sing R$ denote by $F_{r}$ the fiber passing through $r$. Since $F_{r}\cdot R=3$, we have  $F_{r}\neq F_{r'}$ for $r\neq r'$. Put $\eta_{r}=1$ if $F_{r}$ passes through $R\cap \Sec_2$ and $\eta_{r}=0$ otherwise. If $r\in R$ is a cusp and $F_{r}$ is tangent to $R$ at $r$, put $\epsilon_{r}=1$, otherwise put $\epsilon_{r}=0$. Note that since $R\cdot F_{r}=3$, if $\epsilon_{r}=1$ then $r\in R$ is ordinary. Put $\epsilon=\sum_{r}\epsilon_{r}$. Applying the Hurwitz formula to $\pi_{\F_{2}}|_{R}\colon R\to \P^{1}$ we infer that in case \ref{item:n0_F2} $\epsilon=0$ and $\boldsymbol{\delta}=(2,2,1,1)$ or $(2,2,2)$; and in case \ref{item:n0_A1A2} $R$ has at most $3-\epsilon$ cusps, so either $\epsilon=0$  or $\epsilon=1$ and $\boldsymbol{\delta}=(4,1)$.

Fix a cusp $p\in R$ with maximal $\delta$-invariant. Then $\delta_{p}\geq 2$, so $\epsilon_{p}=0$. Blow up over each singular point of $R$ until the proper transform of $R$ is smooth, and contract the exceptional $(-1)$-curve over $p$. This map is a composition of $\delta_{R}-1=5$ blowups over points of multiplicity $2$ on $R\setminus \Sec_2$. Next, contract the proper transforms of the fibers passing through $\Sing R$ and the $(-1)$-curves appearing in subsequent total transforms of $R$. Each contracted curve meets the proper transforms of $R$ and $\Sec_2$ once each. Denote the resulting map by $\theta$. We have $\rho(\theta(\F_{2}))=2$, so $\theta(\F_{2})$ is a Hirzebruch surface $\F_{m}$ for some $m\geq 0$. We have $(\theta_{*}R)^{2}=R^{2}-20+5=9$ and $(\theta_{*}\Sec_2)^{2}=\Sec_2^{2}+5=3$. Like on $\F_2$, the curves $\theta_{*}R$, $\theta_{*}\Sec_2$ are a  $3$-section and a $1$-section, respectively. Numerical properties of $\F_m$ imply that $m=1$ and $\theta_{*}R$, $\theta_{*}\Sec_2$ are of types $(0,3)$ and $(1,1)$, respectively.

Contract the negative section of $\F_{1}$ and denote its image by $u\in \P^{2}$. Let $\rr,\cc\subseteq \P^{2}$ the images of $\theta_{*}R$, $\theta_{*}\Sec_2$. For each $r\in \Sing R$ let $\ll_{r}$ be the image of the exceptional $(-1)$-curve over $r\in\Sing R$, and let $\hat{r},r'\in \ll_{r}$ be the image of $F_{r}$ and the point infinitely near to $r$, respectively. Note that $\hat{r}=r'$ if $\epsilon_{r}=1$ and $\hat{r}\neq r'$ otherwise. For a cusp $r\neq p$, we have $(\rr\cdot \ll_{r})_{r'}=2+\epsilon_{r}$, $(\rr\cdot\ll_{r})_{\hat{r}}=1$ and $(\rr\cdot \cc)_{\hat{r}}=\delta_{r}+\eta_{r}$. For the chosen cusp $p$, the point  $p'\in \rr$ is an ordinary cusp, and since $\epsilon_{p}=0$, we have $(\rr\cdot \ll_{p})_{p'}=2$,  $(\rr\cdot\ll_{r})_{\hat{p}}=1$ and $(\rr\cdot \cc)_{\hat{p}}=\delta_{p}+\eta_{p}-1$. Eventually, for the node $q\in R$ in case \ref{item:n0_A1A2}, we have $(\rr\cdot \ll_{q})_{q'}=1$, $(\rr\cdot\ll_{q})_{\hat{q}}=2$ and $(\rr\cdot \cc)_{\hat{q}}=1$.

The curve $\rr$ is a cuspidal cubic not passing through $u$. The curve $\cc$ is a conic such that for every $r\in \Sing R$, we have $\cc\cap \ll_{r}=\{\hat{r},u\}$ and $\hat{r}\neq u$. Moreover, $\rr\cdot (\cc\setminus \bigcup_{r}\ll_{r})\leq 1$, with equality if and only if $\eta_{r}=0$ for every $r$. 

\begin{figure}[htbp]
	\begin{tikzpicture}	
		\node[right] at (-1,4.5) {\small{$\large{\bullet}\mapsto q_1$}};
		\node[right] at (-1,4.1) {\small{$\boldsymbol{\circ}\mapsto q_2$}};
		\node[right] at (-1,3.7) {\small{$\boldsymbol{\times}\mapsto q_3$}};	
		\node[right] at (-1,3.3) {\small{$\spadesuit\mapsto q_1'$}};
		\node[right] at (-1,2.9) {\small{$\heartsuit\mapsto q_2'$}};
		\node[right] at (-1,2.5) {\small{$\diamondsuit\mapsto s_1$}};
		\node[right] at (-1,2.1) {\small{$\bigstar\mapsto r$}};
		\draw (-8,1.2) -- (0.2,1.2);
		\node at (0,1.4) {\small{$C_1$}};
		\node at (0,1) {\small{$-8$}};
		\node at (-1.4,1.2) {\small{$\spadesuit$}};
		\node at (-1.1,1.2) {\small{$\heartsuit$}};
		\node at (-0.8,1.2) {\small{$\diamondsuit$}};
		\node at (-0.5,1.2) {\small{$\bigstar$}};
		\draw (-3.4,-0.4) -- (-3.6,1);
		\node at (-3.8,0.3) {\small{$-2$}};
		\draw (-3.6,0.8) -- (-3.4,2.2); 
		\node at (-3.75,1.5) {\small{$-1$}};
		\node at (-3.2,1.5) {\small{$E_{p_3}$}};			
		\draw (-3.4,2) -- (-3.6,3.4);
		\node at (-3.8,2.7) {\small{$-3$}};
		\node at (-3.2,2.7) {\small{$L_3$}};
		\node at (-3.47,2.55) {\large{$\bullet$}};
		\node at (-3.52,2.9) {$\boldsymbol{\circ}$};
		\draw (-4.8,-0.4) -- (-5,1);
		\node at (-5.2,0.3) {\small{$-2$}};
		\draw (-5,0.8) -- (-4.8,2.2); 
		\node at (-5.15,1.5) {\small{$-1$}};
		\node at (-4.6,1.5) {\small{$E_{p_2}$}};			
		\draw (-4.8,2) -- (-5,3.4);
		\node at (-5.2,2.7) {\small{$-3$}};
		\node at (-4.6,2.7) {\small{$L_2$}};
		\node at (-4.87,2.55) {\large{$\bullet$}};
		\node at (-4.92,2.9) {$\boldsymbol{\times}$};
		\draw (-5,3.2) -- (-4.7,5.3);
		\node at (-5.2,3.9) {\small{$-2$}};
		\node at (-4.55,3.9) {\small{$L_{rq_2}$}};
		\node at (-4.75,4.95) {$\boldsymbol{\circ}$};
		\node at (-4.8,4.6) {\small{$\heartsuit$}};
		\node at (-4.85,4.25) {\small{$\bigstar$}};
		\draw (-6.2,-0.4) -- (-6.4,1);
		\node at (-6.6,0.3) {\small{$-2$}};
		\draw (-6.4,0.8) -- (-6.2,2.2); 
		\node at (-6.55,1.5) {\small{$-1$}};
		\node at (-6,1.5) {\small{$E_{p_1}$}};			
		\draw (-6.2,2) -- (-6.4,3.4);
		\node at (-6.6,2.7) {\small{$-3$}};
		\node at (-6,2.7) {\small{$L_1$}};
		\node at (-6.27,2.55) {$\boldsymbol{\circ}$};
		\node at (-6.32,2.9) {$\boldsymbol{\times}$};
		\draw (-6.4,3.2) -- (-6.1,5.3);
		\node at (-6.6,3.9) {\small{$-2$}};
		\node at (-5.95,3.9) {\small{$L_{rq_1}$}};
		\node at (-6.15,4.95) {\large{$\bullet$}};
		\node at (-6.2,4.6) {\small{$\spadesuit$}};
		\node at (-6.25,4.25) {\small{$\bigstar$}};
		\draw (-2,-0.4) -- (-2.2,1);
		\node at (-2.4,0.3) {\small{$-2$}};
		\draw (-2.2,0.8) -- (-2,2.2); 
		\node at (-2.35,1.5) {\small{$-1$}};
		\node at (-1.8,1.5) {\small{$E_{s_2}$}};
		\draw (-2,2) -- (-2.3,4.1);
		\node at (-2.4,2.7) {\small{$-3$}};
		\node at (-1.8,2.7) {\small{$C_3$}};
		\node at (-2.07,2.55) {\large{$\bullet$}};
		\node at (-2.12,2.9) {$\boldsymbol{\circ}$};
		\node at (-2.17,3.25) {$\boldsymbol{\times}$};
		\node at (-2.22,3.6) {\small{$\bigstar$}};
		\node at (-2.27,3.95) {\small{$\diamondsuit$}};
		\draw (-7.8,0.8) -- (-7.3,4.3); 
		\node at (-7.95,1.5) {\small{$-2$}};
		\node at (-7.45,1.5) {\small{$C_{2}$}};
		\node at (-7.35,3.95) {\large{$\bullet$}};
		\node at (-7.4,3.6) {$\boldsymbol{\circ}$};
		\node at (-7.45,3.25) {$\boldsymbol{\times}$};
		\node at (-7.5,2.9) {\small{$\spadesuit$}};
		\node at (-7.55,2.55) {\small{$\heartsuit$}};
		\node at (-7.6,2.2) {\small{$\diamondsuit$}};
	\end{tikzpicture}			
	\caption{Case \ref{def:F2n0}: $X\am\cong \F_2$ and $n=0$, see Proposition \ref{prop:n0}\ref{item:n0_F2} and Remark \ref{rem:Moe}}
	\label{fig:F2n0}
\end{figure}
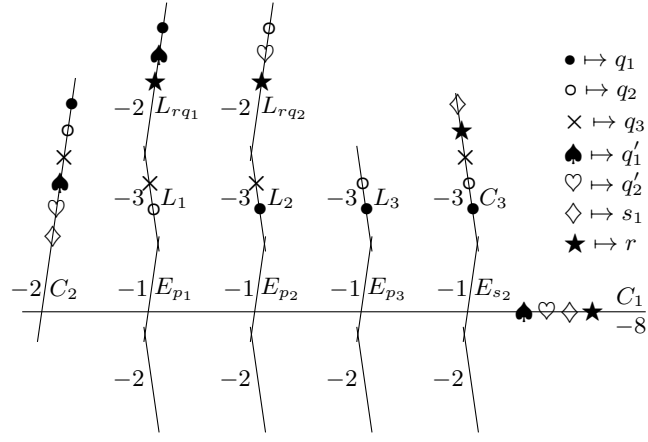

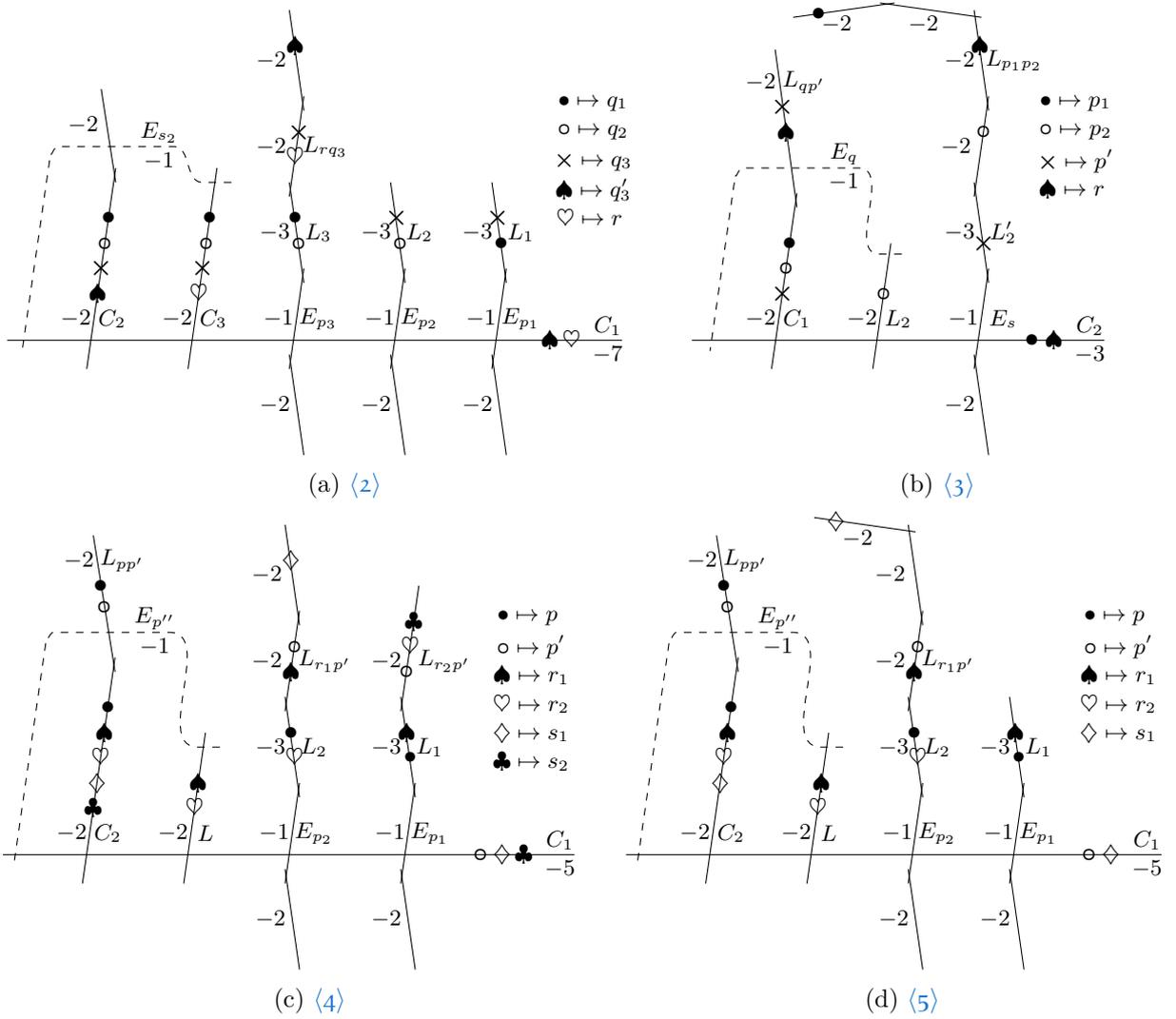
\begin{figure}[htbp]
		\begin{subfigure}[b]{0.55\textwidth}
			\begin{tikzpicture}	
				\node at (-2.2,4.5) {\small{$\large{\bullet}\mapsto q_1$}};
				\node at (-2.2,4.1) {\small{$\boldsymbol{\circ}\mapsto q_2$}};
				\node at (-2.2,3.7) {\small{$\boldsymbol{\times}\mapsto q_3$}};	
				\node at (-2.2,3.3) {\small{$\spadesuit\mapsto q_3'$}};
				\node at (-2.25,2.9) {\small{$\heartsuit\mapsto r$}};
				\draw (-10.3,1.2) -- (-1.8,1.2);
				\node at (-2,1.4) {\small{$C_1$}};
				\node at (-2,1) {\small{$-7$}};
				\node at (-2.8,1.2) {\small{$\spadesuit$}};
				\node at (-2.5,1.2) {\small{$\heartsuit$}};
				\draw (-3.4,-0.4) -- (-3.6,1);
				\node at (-3.8,0.3) {\small{$-2$}};
				\draw (-3.6,0.8) -- (-3.4,2.2); 
				\node at (-3.75,1.5) {\small{$-1$}};
				\node at (-3.2,1.5) {\small{$E_{p_1}$}};			
				\draw (-3.4,2) -- (-3.6,3.4);
				\node at (-3.8,2.7) {\small{$-3$}};
				\node at (-3.2,2.7) {\small{$L_1$}};
				\node at (-3.47,2.55) {\large{$\bullet$}};
				\node at (-3.52,2.9) {$\boldsymbol{\times}$};
				\draw (-4.8,-0.4) -- (-5,1);
				\node at (-5.2,0.3) {\small{$-2$}};
				\draw (-5,0.8) -- (-4.8,2.2); 
				\node at (-5.15,1.5) {\small{$-1$}};
				\node at (-4.6,1.5) {\small{$E_{p_2}$}};			
				\draw (-4.8,2) -- (-5,3.4);
				\node at (-5.2,2.7) {\small{$-3$}};
				\node at (-4.6,2.7) {\small{$L_2$}};
				\node at (-4.87,2.55) {$\boldsymbol{\circ}$};
				\node at (-4.92,2.9) {$\boldsymbol{\times}$};
				\draw (-6.2,-0.4) -- (-6.4,1);
				\node at (-6.6,0.3) {\small{$-2$}};
				\draw (-6.4,0.8) -- (-6.2,2.2); 
				\node at (-6.55,1.5) {\small{$-1$}};
				\node at (-6,1.5) {\small{$E_{p_3}$}};			
				\draw (-6.2,2) -- (-6.4,3.4);
				\node at (-6.6,2.7) {\small{$-3$}};
				\node at (-6,2.7) {\small{$L_3$}};
				\node at (-6.27,2.55) {$\boldsymbol{\circ}$};
				\node at (-6.32,2.9) {\large{$\bullet$}};
				\draw (-6.4,3.2) -- (-6.2,4.6);
				\node at (-6.65,3.9) {\small{$-2$}};
				\node at (-5.9,3.9) {\small{$L_{rq_3}$}};
				\node at (-6.27,4.1) {$\boldsymbol{\times}$};
				\node at (-6.32,3.75) {\small{$\heartsuit$}};
				\draw (-6.2,4.4) -- (-6.4,5.8);
				\node at (-6.65,5.1) {\small{$-2$}};
				\node at (-6.32,5.3) {\small{$\spadesuit$}};
				\draw (-7.8,0.8) -- (-7.4,3.6); 
				\node at (-7.95,1.5) {\small{$-2$}};
				\node at (-7.45,1.5) {\small{$C_{3}$}};
				\node at (-7.5,2.9) {\large{$\bullet$}};
				\node at (-7.55,2.55) {$\boldsymbol{\circ}$};
				\node at (-7.6,2.2) {$\boldsymbol{\times}$};
				\node at (-7.65,1.85) {\small{$\heartsuit$}};
				\draw (-9.2,0.8) -- (-8.8,3.6); 
				\node at (-9.35,1.5) {\small{$-2$}};
				\node at (-8.85,1.5) {\small{$C_{2}$}};
				\node at (-8.9,2.9) {\large{$\bullet$}};
				\node at (-8.95,2.55) {$\boldsymbol{\circ}$};
				\node at (-9,2.2) {$\boldsymbol{\times}$};
				\node at (-9.05,1.85) {\small{$\spadesuit$}};
				\draw (-8.8,3.4) -- (-9,4.7);
				\node at (-9.25,4.2) {\small{$-2$}};
				\draw[dashed] (-7.2,3.4) -- (-7.6,3.4) to[out=180,in=0] (-8,3.9) -- (-9.6,3.9) to[out=180,in=-80] (-9.7,3.8) -- (-10.1,1);
				\node at (-8.2,4.1) {\small{$E_{s_2}$}};
				\node at (-8.2,3.7) {\small{$-1$}};				   		
			\end{tikzpicture}	
			\caption{\ref{def:A1A2_c=3}}
			\label{fig:A1A2_c=3}
		\end{subfigure}		
		\begin{subfigure}[b]{0.4\textwidth}
			\begin{tikzpicture}	
				\node at (-5,4.5) {\small{$\large{\bullet}\mapsto p_1$}};
				\node at (-5,4.1) {\small{$\boldsymbol{\circ}\mapsto p_2$}};
				\node at (-5,3.7) {\small{$\boldsymbol{\times}\mapsto p'$}};	
				\node at (-5.05,3.3) {\small{$\spadesuit\mapsto r$}};
				\draw (-10.3,1.2) -- (-4.6,1.2);
				\node at (-4.8,1.4) {\small{$C_2$}};
				\node at (-4.8,1) {\small{$-3$}};
				\node at (-5.6,1.2) {\large{$\bullet$}};
				\node at (-5.3,1.2) {\small{$\spadesuit$}};
				\draw (-6.2,-0.4) -- (-6.4,1);
				\node at (-6.6,0.3) {\small{$-2$}};
				\draw (-6.4,0.8) -- (-6.2,2.2); 
				\node at (-6.55,1.5) {\small{$-1$}};
				\node at (-6,1.5) {\small{$E_{s}$}};			
				\draw (-6.2,2) -- (-6.4,3.4);
				\node at (-6.6,2.7) {\small{$-3$}};
				\node at (-6,2.7) {\small{$L_2'$}};
				\node at (-6.27,2.55) {$\boldsymbol{\times}$};
				\draw (-6.4,3.2) -- (-6.2,4.6);
				\node at (-6.65,3.9) {\small{$-2$}};
				\node at (-6.27,4.1) {$\boldsymbol{\circ}$};
				\draw (-6.2,4.4) -- (-6.4,5.8);
				\node at (-6.6,5.1) {\small{$-2$}};
				\node at (-5.85,5.1) {\small{$L_{p_1 p_2}$}};
				\node at (-6.32,5.3) {\small{$\spadesuit$}};
				\draw (-6.3,5.7) -- (-7.7,5.9);
				\node at (-7.1,5.6) {\small{$-2$}};
				\draw (-7.5,5.9) -- (-8.9,5.7);
				\node at (-8.3,5.6) {\small{$-2$}};
				\node at (-8.55,5.75) {\large{$\bullet$}};
				\draw (-7.8,0.8) -- (-7.55,2.55); 
				\node at (-7.95,1.5) {\small{$-2$}};
				\node at (-7.45,1.5) {\small{$L_{2}$}};
				\node at (-7.65,1.85) {$\boldsymbol{\circ}$};
				\draw (-9.2,0.8) -- (-8.85,3.25); 
				\node at (-9.35,1.5) {\small{$-2$}};
				\node at (-8.85,1.5) {\small{$C_{1}$}};
				\node at (-8.95,2.55) {\large{$\bullet$}};
				\node at (-9,2.2) {$\boldsymbol{\circ}$};
				\node at (-9.05,1.85) {$\boldsymbol{\times}$};
				\draw (-8.85,3.05) -- (-9.15,5.25);
				\node at (-9,4.1) {\small{$\spadesuit$}};
				\node at (-9.05,4.45) {$\boldsymbol{\times}$};
				\node at (-9.35,4.8) {\small{$-2$}};
				\node at (-8.75,4.8) {\small{$L_{qp'}$}};
				\draw[dashed] (-7.4,2.4) -- (-7.7,2.4) to[out=180,in=0] (-8,3.6) -- (-9.6,3.6) to[out=180,in=-80] (-9.7,3.5) -- (-10.05,1.05);
				\node at (-8.2,3.8) {\small{$E_{q}$}};
				\node at (-8.2,3.4) {\small{$-1$}};				   		
			\end{tikzpicture}
			\caption{\ref{def:A1A2_c=1}}
			\label{fig:A1A2_c=1}
		\end{subfigure}	
		\begin{subfigure}[b]{0.5\textwidth}
			\begin{tikzpicture}	
				\node at (-3.05,4.5) {\small{$\large{\bullet}\mapsto p$}};
				\node at (-3,4.1) {\small{$\boldsymbol{\circ}\mapsto p'$}};
				\node at (-3,3.7) {\small{$\spadesuit\mapsto r_1$}};	
				\node at (-3,3.3) {\small{$\heartsuit\mapsto r_2$}};
				\node at (-3,2.9) {\small{$\diamondsuit\mapsto s_1$}};
				\node at (-3,2.5) {\small{$\clubsuit\mapsto s_2$}};
				\draw (-10.3,1.2) -- (-2.4,1.2);
				\node at (-2.6,1.4) {\small{$C_1$}};
				\node at (-2.6,1) {\small{$-5$}};
				\node at (-3.7,1.2) {$\boldsymbol{\circ}$};
				\node at (-3.4,1.2) {\small{$\diamondsuit$}};
				\node at (-3.1,1.2) {\small{$\clubsuit$}};
				\draw (-4.6,-0.4) -- (-4.8,1);
				\node at (-5,0.3) {\small{$-2$}};
				\draw (-4.8,0.8) -- (-4.6,2.2); 
				\node at (-4.95,1.5) {\small{$-1$}};
				\node at (-4.4,1.5) {\small{$E_{p_1}$}};			
				\draw (-4.6,2) -- (-4.8,3.4);
				\node at (-5,2.7) {\small{$-3$}};
				\node at (-4.4,2.7) {\small{$L_1$}};
				\node at (-4.67,2.55) {$\large{\bullet}$};
				\node at (-4.72,2.9) {\small{$\spadesuit$}};
				\draw (-4.8,3.2) -- (-4.55,4.95);
				\node at (-5,3.9) {\small{$-2$}};
				\node at (-4.2,3.9) {\small{$L_{r_2 p'}$}};
				\node at (-4.72,3.75) {$\boldsymbol{\circ}$};
				\node at (-4.67,4.1) {\small{$\heartsuit$}};
				\node at (-4.62,4.45) {\small{$\clubsuit$}};
				\draw (-6.2,-0.4) -- (-6.4,1);
				\node at (-6.6,0.3) {\small{$-2$}};
				\draw (-6.4,0.8) -- (-6.2,2.2); 
				\node at (-6.55,1.5) {\small{$-1$}};
				\node at (-6,1.5) {\small{$E_{p_2}$}};			
				\draw (-6.2,2) -- (-6.4,3.4);
				\node at (-6.6,2.7) {\small{$-3$}};
				\node at (-6,2.7) {\small{$L_2$}};
				\node at (-6.27,2.55) {\small{$\heartsuit$}};
				\node at (-6.32,2.9) {\large{$\bullet$}};
				\draw (-6.4,3.2) -- (-6.2,4.6);
				\node at (-6.65,3.9) {\small{$-2$}};
				\node at (-5.85,3.9) {\small{$L_{r_1 p'}$}};
				\node at (-6.27,4.1) {$\boldsymbol{\circ}$};
				\node at (-6.32,3.75) {\small{$\spadesuit$}};
				\draw (-6.2,4.4) -- (-6.4,5.8);
				\node at (-6.65,5.1) {\small{$-2$}};
				\node at (-6.32,5.3) {\small{$\diamondsuit$}};
				\draw (-7.8,0.8) -- (-7.5,2.9); 
				\node at (-7.95,1.5) {\small{$-2$}};
				\node at (-7.5,1.5) {\small{$L$}};
				\node at (-7.6,2.2) {\small{$\spadesuit$}};
				\node at (-7.65,1.85) {\small{$\heartsuit$}};
				\draw (-9.2,0.8) -- (-8.75,3.95); 
				\node at (-9.35,1.5) {\small{$-2$}};
				\node at (-8.85,1.5) {\small{$C_{2}$}};
				\node at (-8.85,3.25) {\large{$\bullet$}};
				\node at (-8.9,2.9) {\small{$\spadesuit$}};
				\node at (-8.95,2.55) {\small{$\heartsuit$}};
				\node at (-9,2.2) {\small{$\diamondsuit$}};
				\node at (-9.05,1.85) {\small{$\clubsuit$}};
				\draw (-8.75,3.75) -- (-9.05,5.65);
				\node at (-9.25,5.3) {\small{$-2$}};
				\node at (-8.65,5.3) {\small{$L_{pp'}$}};
				\node at (-8.95,4.95) {\large{$\bullet$}};
				\node at (-8.9,4.65) {$\boldsymbol{\circ}$};
				\draw[dashed] (-7.3,2.7) -- (-7.6,2.7) to[out=180,in=0] (-8,4.3) -- (-9.6,4.3) to[out=180,in=-80] (-9.7,4.2) -- (-10.15,1.05);
				\node at (-8.2,4.5) {\small{$E_{p''}$}};
				\node at (-8.2,4.1) {\small{$-1$}};			   		
			\end{tikzpicture}
			\caption{\ref{def:A1A2_c=2_32}}
			\label{fig:A1A2_c=2_32}
		\end{subfigure}				
		\begin{subfigure}[b]{0.45\textwidth}
			\begin{tikzpicture}	
				\node at (-3.55,4.5) {\small{$\large{\bullet}\mapsto p$}};
				\node at (-3.5,4.1) {\small{$\boldsymbol{\circ}\mapsto p'$}};
				\node at (-3.5,3.7) {\small{$\spadesuit\mapsto r_1$}};	
				\node at (-3.5,3.3) {\small{$\heartsuit\mapsto r_2$}};
				\node at (-3.5,2.9) {\small{$\diamondsuit\mapsto s_1$}};
				\draw (-10.3,1.2) -- (-2.9,1.2);
				\node at (-3.1,1.4) {\small{$C_1$}};
				\node at (-3.1,1) {\small{$-5$}};
				\node at (-3.9,1.2) {$\boldsymbol{\circ}$};
				\node at (-3.6,1.2) {\small{$\diamondsuit$}};
				\draw (-4.8,-0.4) -- (-5,1);
				\node at (-5.2,0.3) {\small{$-2$}};
				\draw (-5,0.8) -- (-4.8,2.2); 
				\node at (-5.15,1.5) {\small{$-1$}};
				\node at (-4.6,1.5) {\small{$E_{p_1}$}};			
				\draw (-4.8,2) -- (-5,3.4);
				\node at (-5.2,2.7) {\small{$-3$}};
				\node at (-4.6,2.7) {\small{$L_1$}};
				\node at (-4.87,2.55) {$\large{\bullet}$};
				\node at (-4.92,2.9) {\small{$\spadesuit$}};
				\draw (-6.2,-0.4) -- (-6.4,1);
				\node at (-6.6,0.3) {\small{$-2$}};
				\draw (-6.4,0.8) -- (-6.2,2.2); 
				\node at (-6.55,1.5) {\small{$-1$}};
				\node at (-6,1.5) {\small{$E_{p_2}$}};			
				\draw (-6.2,2) -- (-6.4,3.4);
				\node at (-6.6,2.7) {\small{$-3$}};
				\node at (-6,2.7) {\small{$L_2$}};
				\node at (-6.27,2.55) {\small{$\heartsuit$}};
				\node at (-6.32,2.9) {\large{$\bullet$}};
				\draw (-6.4,3.2) -- (-6.2,4.6);
				\node at (-6.65,3.9) {\small{$-2$}};
				\node at (-5.85,3.9) {\small{$L_{r_1 p'}$}};
				\node at (-6.27,4.1) {$\boldsymbol{\circ}$};
				\node at (-6.32,3.75) {\small{$\spadesuit$}};
				\draw (-6.2,4.4) -- (-6.4,5.8);
				\node at (-6.65,5.1) {\small{$-2$}};
				\draw (-6.3,5.7) -- (-7.7,5.9);
				\node at (-7.1,5.6) {\small{$-2$}};
				\node at (-7.4,5.85) {\small{$\diamondsuit$}};
				\draw (-7.8,0.8) -- (-7.5,2.9); 
				\node at (-7.95,1.5) {\small{$-2$}};
				\node at (-7.5,1.5) {\small{$L$}};
				\node at (-7.6,2.2) {\small{$\spadesuit$}};
				\node at (-7.65,1.85) {\small{$\heartsuit$}};
				\draw (-9.2,0.8) -- (-8.75,3.95); 
				\node at (-9.35,1.5) {\small{$-2$}};
				\node at (-8.85,1.5) {\small{$C_{2}$}};
				\node at (-8.85,3.25) {\large{$\bullet$}};
				\node at (-8.9,2.9) {\small{$\spadesuit$}};
				\node at (-8.95,2.55) {\small{$\heartsuit$}};
				\node at (-9,2.2) {\small{$\diamondsuit$}};
				\draw (-8.75,3.75) -- (-9.05,5.65);
				\node at (-9.25,5.3) {\small{$-2$}};
				\node at (-8.65,5.3) {\small{$L_{pp'}$}};
				\node at (-8.95,4.95) {\large{$\bullet$}};
				\node at (-8.9,4.65) {$\boldsymbol{\circ}$};
				\draw[dashed] (-7.3,2.7) -- (-7.6,2.7) to[out=180,in=0] (-8,4.3) -- (-9.6,4.3) to[out=180,in=-80] (-9.7,4.2) -- (-10.15,1.05);
				\node at (-8.2,4.5) {\small{$E_{p''}$}};
				\node at (-8.2,4.1) {\small{$-1$}};		
			\end{tikzpicture}
			\caption{\ref{def:A1A2_c=2_41}}
			\label{fig:A1A2_c=2_41}
		\end{subfigure}	
\vspace{-.5em}
	\caption{Cases with $X_{\min}=\P(1,2,3)$ and $n=0$, see Proposition \ref{prop:n0}\ref{item:n0_A1A2}.}
	\label{fig:sporadic_F2}
\end{figure}

Consider the case \ref{item:n0_A1A2}, $\boldsymbol{\delta}=(5)$, so $\delta_{p}=5$. Perform a standard quadratic transformation centered at $p'$, $\hat{p}$ and $\hat{q}$. Let 
$\cc_{1}$, $\cc_{2}$, $\ll_{2}$ and 
$\ll_{2}'$, $\ll_{p_{1}p_{2}}$, $\ll_{rp'}$ be the images of 
$\cc$, $\rr$, $\ll_{q}$ and of the exceptional curves over 
$p'$, $\hat{p}$, $\hat{q}$. Then $\cc_{1}+\cc_{2}+\ll_{2}+\ll_{2}'+\ll_{rp'}+\ll_{p_{1}p_{2}}$ is as in Configuration \ref{conf:A1A2_c=1}, so $S$ is as in \ref{def:A1A2_c=1}. We have $\eta_{p}=(\rr\cdot \cc)_{\hat{p}}-4=(\cc_2\cdot \cc_1)_{p_1}-3=0$.

Consider the case \ref{item:n0_A1A2}, $\boldsymbol{\delta}=(3,2)$ or $(4,1)$. Let $r$ be the cusp of $R\setminus \{p\}$, respectively. Perform a standard quadratic transformation centered at $p'$, $\hat{r}$, $\hat{q}$. Let 
$\cc_{1}$, $\cc_{2}$, $\ll_{r_1p'}$, $\ll_{1}$, $\ll$ and 
$\ll_{2}$, $\ll_{r_{2}p'}$, $\ll_{pp'}$ be the images of  
$\rr$, $\cc$, $\ll_{p}$, $\ll_{r}$, $\ll_{q}$ and of the exceptional curves over 
$p'$, $\hat{r}$, $\hat{q}$. This way, we obtain Configuration \ref{conf:A1A2_c=2}\ref{item:32} if $\boldsymbol{\delta}=(3,2)$ and \ref{conf:A1A2_c=2}\ref{item:41} if $\boldsymbol{\delta}=(4,1)$. Hence the surface $S$ is as in \ref{def:A1A2_c=2_32} and \ref{def:A1A2_c=2_41}, respectively. Note that $\eta_{p}=(\cc_1\cdot \cc_2)_{s_1}-\delta_{p}+1=0$, $\eta_{r}=(\cc_1\cdot\cc_2)_{s_2}-\delta_{r}+1=0$. Moreover,  $\epsilon_{r}=(\cc_1\cdot \ll_1)_{s_1}=0$, so $\epsilon=0$.
\smallskip

In the remaining cases, let $p,r_{1},r_{2}\in R$ be cusps with $\delta_{r_1}\geq \delta_{r_2}$, and let $r$ be the remaining singular point of $R$ (if exists). As before, we perform a standard quadratic transformation centered at $p'$, $\hat{r}_{1}$, $\hat{r}_{2}$. 

Consider case \ref{item:n0_F2}. Let $\cc_1$, $\cc_2$, $\ll_{rq_1}$, $\ll_2$, $\ll_3$ and $\ll_1$, $\ll_{rq_2}$, $\ll_{rq_3}$ be the images of $\rr$, $\cc$, $\ll_{p}$, $\ll_{r_1}$, $\ll_{r_2}$ and of the exceptional curves over $p'$, $\hat{r}_{1}$, $\hat{r}_2$. Then $\cc_1+\cc_2+\ll_1+\ll_2+\ll_3+\ll_{rq_1}+\ll_{rq_2}$ is as in Configuration \ref{conf:F2n0}. 

Suppose $\boldsymbol{\delta}=(2,2,2)$. Then $(\rr\cdot \cc)_{\hat{r}_2}=2$, so $\cc_1$, $\cc_2$ and $\ll_{rq_3}$ meet at one point. In the notation of Configuration \ref{conf:F2n0}, this means that $q_{3}'\in \cc_2$, which is impossible; a contradiction.

Thus $\boldsymbol{\delta}=(2,2,1,1)$. Denoting by $\cc_3$ the image of $\ll_{r}$, we get exactly Configuration \ref{conf:F2n0}. Thus $S$ is as in \ref{def:F2n0}. The condition $q_{3}'\not\in \cc_2$ in Configuration \ref{conf:F2n0} means that $\eta_{r_2}=0$. In fact, $\eta_{p}=\eta_{r_1}=\eta_{r}=0$ because in Configuration \ref{conf:F2n0}, the conic $\cc_1$ is not tangent to $\cc_2$ or to $\cc_3\setminus \{s_2\}$.

Consider now the case \ref{item:n0_A1A2}. Let $\cc_1$, $\cc_2$, $\ll_{rq_3}$, $\ll_1$, $\ll_2$, $\cc_3$ and $\ll_3$, $\ll_{rq_1}$, $\ll_{rq_2}$ 
be the images of $\rr$, $\cc$, $\ll_{p}$, $\ll_{r_1}$, $\ll_{r_2}$, $\ll_{r}$ and of the exceptional curves over $p'$, $\hat{r}_1$, $\hat{r}_2$. Then $\cc_1+\cc_2+\cc_3+\ll_1+\ll_2+\ll_3+\ll_{rq_1}+\ll_{rq_2}+\ll_{rq_3}$ is as in Configuration \ref{conf:311}. Using the notation introduced there, we have $q_{1}'\in \cc_2$ in case $\boldsymbol{\delta}=(2,2,1)$; and $(\cc_1\cdot \cc_2)_{q_3'}=2+\eta_{r_1}$ in case $\boldsymbol{\delta}=(3,1,1)$. Thus the latter case holds; and $\eta_{r_2}=1$. As in the previous case, we get $\eta_{p}=\eta_{r_2}=0$, too.
\end{proof}

\begin{rem}
	\label{rem:Moe}
	The curve $R$ from Proposition \ref{prop:n0}\ref{item:n0_F2} is a new example of a rational cuspidal curve on $\F_2$ with four cusps. Some other examples were found in \cite{Moe_4-cusps}. Conjecture 1.2 loc.\ cit.\ asserts that a rational cuspidal curve on $\F_2$ cannot have more than four cusps; the same holds for curves on $\P^2$ by \cite{KoPa_four-cusps}.
	
	The above construction allows one to compute 
	the explicit equation of $R\subseteq \F_{2}$, or equivalently, of its image $\bar{R}\subseteq \P(1,1,2)$. In some coordinates $[x:y:z]$ of weights $1,1,2$, this curve is given by
\begin{equation*}
\begin{split}
\{0= &\ 2304\  x^{5}y^{2}+324\ x^{4}y^{3}-7680\ x^{4}yz+1980\ x^{3}y^{2}z+360\ x^{2}y^{3}z+\\
+&\ 6400\ x^{3}z^{2}-6900\ x^{2}yz^{2}
+300\ xy^{2}z^{2}+100\ y^{3}z^{2}+4500\ xz^{3}-1125\ yz^{3}\},
\end{split}
\end{equation*}
has cusps with $\delta=2$ at $[0:1:0]$ and $[1:0:0]$; and ordinary ones at
\begin{equation*}
\left[-\tfrac{61}{256}+\tfrac{5}{256}\sqrt{-15}:1:-\tfrac{79}{1280}+\tfrac{53}{3840}\sqrt{-15}\right] \mbox{ and }
\left[-\tfrac{61}{256}-\tfrac{5}{256}\sqrt{-15}:1:-\tfrac{79}{1280}-\tfrac{53}{3840}\sqrt{-15}\right].
\end{equation*}
This computation was kindly provided to us by Torguun K.\ Moe. 
\smallskip

We remark that the above equation of $\bar{R}\subseteq \P(1,1,2)$ has rational coefficients, so the complement $\P(1,1,2)\setminus \bar{R}$, i.e.\ the surface \ref{def:F2n0}, is defined over $\Q$. This fact can also be seen from an alternative construction of the curve $R$, given in Example \ref{ex:F2n0_Aut}, see Remark \ref{rem:F2n0_Aut-on-F2}. There, we construct an involution $\hat{\iota}\in \Aut(\F_2,R)$ such that $\hat{\iota}(p_1)=p_2$, $\hat{\iota}(p_1')=p_2'$, where $p_i,p_i'\in R$ are cusps with $\delta_{p_i}=2$, $\delta_{p_i'}=1$, $i\in \{1,2\}$.
\end{rem}

\section{$\Q$-homology planes with $\C^{**}$-fibrations} \label{sec:Cstst}

In Section \ref{sec:classification} we have classified \QHPs of log general type which satisfy the Negativity Conjecture \ref{conj:negativity} and admit no $\C^{**}$-fibration, see assumption \eqref{eq:assumption}. Here we go back to $\Q$HPs with $\C^{**}$-fibrations. They are much easier to understand. In fact, as we have remarked in the introduction, they are classified in \cite[Theorem 1]{MiySu-Cstst_fibrations_on_Qhp} (to be precise, loc.\ cit.\ is formulated for $\Z$HPs, but the  proof works for $\Q$HPs as well, the only difference is that one has to include types $(UC_{2-1})'$ and $(T3C_{2})$ in the statement).  It was observed in \cite{Su-hp_having_rational_pencils} that the proof can be much simplified by means of Lemma \ref{lem:no_lines}, which was not available at that time. In order to deduce Theorem \ref{CLASS} from loc.\ cit, we need to slightly adjust this result:  for example, we are interested in the minimal log smooth completions, while loc.\ cit.\ considers those completions for which some $\C^{**}$-fibration extends to a $\P^{1}$-fibration with disjoint sections in the boundary. Hence for the convenience of the reader, we give a self-contained argument below; of course relying on Lemma \ref{lem:no_lines}.
\smallskip

The proof is organized as follows. In Lemma \ref{lem:C**_fibers} we complete some $\C^{**}$-fibration of $S$ to a $\P^1$-fibration of $X$ and describe degenerate fibers of the latter. While doing it, we introduce some notation which will be used throughout in the proof. In Lemma \ref{lem:C**_formulas} we gather some useful formulas relating e.g.\ the number of fibers of each type with the shape of the horizontal part $D\hor$ of $D$. Eventually, in Proposition \ref{prop:C**} we study each possible $D\hor$ and classify the corresponding $\Q$-homology planes.

\begin{lem}\label{lem:C**_fibers}
	Let $S$ be a smooth \QHP of log general type, and let $(X,D)$ be its minimal log smooth completion. Assume that $S$ admits a $\C^{**}$-fibration. Then the following holds. 
	\begin{enumerate}
		\item\label{item:C**_extends} There is a $\P^1$-fibration $\bar{p}\colon X\to \P^1$ which restricts to a $\C^{**}$-fibration of $S$.
		\item\label{item:C**_expansion} Contract a vertical bubble in $(X,D)$, and all vertical $(-1)$-curves which are superfluous in the subsequent images of $D$. Repeat this process until the resulting pair $(X',D')$ has no vertical bubbles; and $D'$ has no vertical superfluous $(-1)$-curves. Then the resulting morphism $\psi'\colon (X,D)\to (X',D')$ is an expansion.
		\item\label{item:C**_descends} There is a $\P^1$-fibration $\bar{p}'\colon X'\to \P^1$ such that $\bar{p}=\bar{p}'\circ \psi'$. We have $D'\hor=\psi'(D\hor)$ and $\#D'\hor=\#D\hor$.
		\item\label{item:C**_F0} Let $F^{\circ}$ be a degenerate fiber of $\bar{p}'|_{X'\setminus D'}$. Then $F^{\circ}\cong \C^{*}$. The closure $\bar{F}^{\circ}$ is a nondegenerate fiber of $\bar{p}'$.
		\item\label{item:C**_fibers} Let $F$ be a degenerate fiber of $\bar{p}'$. Then $F\subseteq D'$ and one of the following holds:
		\begin{enumerate}
			\item\label{item:C**_F2} $F\redd=[2,1,2]$ and $D'\hor$ meets the first and the second component of  $F\redd$,
			\item\label{item:C**_F3} $F\redd=[3,1,2,2]$ and $D'\hor$ is a $3$-section meeting $F\redd$ only in the $(-1)$-curve.
		\end{enumerate}
	\end{enumerate}
\end{lem}
\begin{proof}
	Part \ref{item:C**_extends} is proved in \cite[Proposition 3.3]{PaPe_Cstst-fibrations_singularities}. Part \ref{item:C**_expansion} follows from Lemma \ref{lem:sequence_of_expansions}. Part \ref{item:C**_descends} is a direct consequence of the definition of $\psi'$. Part \ref{item:C**_F0} follows from \ref{item:C**_fibers}: indeed, since $\bar{F}^{\circ}\not\subseteq D'$, part \ref{item:C**_fibers} implies that $\bar{F}^{\circ}$ is nondegenerate. In particular $F^{\circ}$ is irreducible, so $F^{\circ}\cong \C^{*}$ by \cite[Corollary 2.8(a)]{PaPe_Cstst-fibrations_singularities}.
	
	It remains to prove part \ref{item:C**_fibers}. Let $L\subseteq F$ be a $(-1)$-curve. We claim that  $\beta_{(D'+L)\redd}(L)\geq 3$. Suppose the contrary. If $L\subseteq D'$ then $L$ is a superfluous, vertical $(-1)$-curve in $D$, which is impossible. Hence $L\not\subseteq D'$ and $L\cdot D'\leq 2$. Lemma \ref{lem:no_lines} gives $L\cdot D'=2$. Because $L$ is not a bubble, it meets exactly one component of $D'$, say $H$. Since the fiber $F$ is a rational tree, we have $H\not\subseteq F$, so $H\subseteq D'\hor$. It follows that $L$ meets a component $L'$ of $F$ such that $L'\not\subseteq D'$. Since $D'$ is connected, each point of $L'\cap D'$ lies in $D'\hor$ or in a connected component of $D'\vert$ meeting $D'\hor$. Hence $\#(L'\cap D')\leq (F-L)\cdot D'\hor=1$; a contradiction with Lemma \ref{lem:no_lines}.
	
	Thus $\beta_{(D'+L)\redd}(L)\geq 3$. Since $\beta_{F\redd}(L)\leq 2$, it follows that $L$ meets $D'\hor$. Let $\mu$ be the multiplicity of $L$ in $F$. If $\mu=1$ then $L$ is a tip of $F\redd$, so $L\cdot D'\hor\geq 2$. In any case, we have $2\leq \mu L\cdot D'\hor\leq F\cdot D'\hor=3$. It follows that $L$ is a unique $(-1)$-curve in $F$. In particular, $\mu>1$. Thus the above inequality gives $\mu \in \{2,3\}$ and $L\cdot D'\hor=1$, so $\beta_{F\redd}(L)\geq \beta_{(D'+L)\redd}(L)-L\cdot D'\hor \geq 2$. It follows that $\beta_{F\redd}(L)=2$ and both components of $F\redd$ meeting $L$ are contained in $D'$. Because $(F-\mu L)\cdot D'\hor\leq 1$, at most one of the connected components of $F\redd-L$ meets $D'\hor$. Since $D'$ is connected, it follows that $L\subseteq D'$. Now if $\mu=2$ then $F\redd=[2,1,2]$, so \ref{item:C**_F2} holds. If $\mu=3$ then $F\redd=[3,1,2,2]$, $F$ meets $D'\hor$ only in $L$, and all components of $F$ except possibly the $(-2)$-tip are contained in $D'$. Lemma \ref{lem:no_lines} implies that $F\subseteq D'$, so \ref{item:C**_F3} holds, as claimed.
\end{proof}

\begin{lem}\label{lem:C**_formulas}
	We keep notation from Lemma \ref{lem:C**_fibers}. Furthermore, let $n$ be the number of centers of the expansion $\psi'$. Let $n_h$ be the number of nodes of $D'\hor$. Let $\nu'$ be the number of degenerate fibers of $p'\de \bar{p}'|_{X'\setminus D'}$, see Lemma \ref{lem:C**_fibers}\ref{item:C**_F0}.  Let $\nu_2$, $\nu_3$ be the number of fibers as in Lemma \ref{lem:C**_fibers}\ref{item:C**_F2}, \ref{item:C**_F3}, respectively; and let $\nu_1$ be the number of nondegenerate fibers contained in $D'$. Then the following holds.
	\begin{enumerate}
		\item \label{item:C**_tangent} Exactly $\nu'-n_h$ fibers of $p'$ are tangent to some multi-section $H\subseteq D'\hor$. Every such fiber is tangent to $H$ in a unique point, with multiplicity $2$.
		\item \label{item:C**_Hurwitz} $0\leq \nu'-n_{h}=6-2\#D\hor-2\nu_{3}-\nu_{2}$.
		\item \label{item:C**_n} $n=3\#D\hor+2\nu_{3}+\nu_{2}-n_{h}-5$.
		\item \label{item:C**_nodes} $\nu_{1}-\nu_{3}+n_{h}-2\#D\hor+3=0$.
		\item \label{item:C**_nu3} $\nu_3\leq 1$.
	\end{enumerate}
\end{lem}
\begin{proof}
	\ref{item:C**_tangent} Let $F$ be a closure of a degenerate fiber of $p'$. By Lemma \ref{lem:C**_fibers}\ref{item:C**_F0}, $F$ is a $0$-curve not contained in $D'$. We have $\#(F\cap D')\leq F\cdot D'-1=2$, so $\#(F\cap D')=2$ by Lemma \ref{lem:no_lines}. Thus $F$ meets $D'$ in two points, with multiplicity $1$ and $2$. In particular, $F$ passes through at most one node of $D'\hor$. Since $D'$ is snc,  every node lies on a fiber of this type, so there are $n_h$ such fibers. The remaining $\nu'-n_h$ degenerate fibers of $p'$ are therefore tangent to some multi-section in $D\hor$; as needed.
	
	\ref{item:C**_Hurwitz} The inequality $\nu'-n_h\geq 0$ follows from \ref{item:C**_tangent}. If $D'\hor$ consists of $1$-sections then $\#D\hor=\#D'\hor=3$, $\nu_2=\nu_3=0$, and $\nu'-n_h=0$ by \ref{item:C**_tangent}, so \ref{item:C**_Hurwitz} holds. Assume that $D'\hor$ contains a $2$-section $H$. Then $\# D\hor=2$ and $\nu_3=0$. The restriction $\bar{p}|_{H}\colon H\to \P^1$ is a morphism of degree $2$, ramified with index $2$ on each fiber of $p'$ tangent to $H$; and on each fiber of type $[2,1,2]$. Thus by \ref{item:C**_tangent}, $\bar{p}|_{H}$ has  $(\nu'-n_h)+\nu_2$ ramification points. By Hurwitz formula, $\nu'-n_h+\nu_2=2$, which proves \ref{item:C**_Hurwitz} in this case. Eventually, assume that $D'\hor$ is a $3$-section. Then $\#D\hor=1$ and $\nu_2=2$. The restriction $\bar{p}|_{D'\hor}\colon D'\hor\to \P^1$ is a morphism of degree $3$. It is ramified with index $2$ at $\nu'-n_h$ points on fibers as in \ref{item:C**_tangent}; and with index $3$ at $\nu_3$ points on fibers of type $[3,1,2,2]$. By Hurwitz formula, we get  $\nu'-n_h+2\nu_3=4$, as needed.
	
	\ref{item:C**_n}  By \cite[Corollary 2.8]{PaPe_Cstst-fibrations_singularities}, the sum of all degenerate fibers of $p$ is a sum of $\#D\hor+1$ curves isomorphic to $\C^{*}$. If such a curve is contracted by $\psi'$ then it is a vertical bubble; otherwise it is a proper transform of a degenerate fiber of $p'$. Thus $n+\nu'=\#D\hor+1$. Substituting the formula for $\nu'$ from  \ref{item:C**_Hurwitz}, we get \ref{item:C**_n}.
	
	\ref{item:C**_nodes} Choose a contraction $(X',D')\to (\F_{m},D'')$ of degenerate fibers which does not introduce new nodes in $D''\hor$. Then $D''$ has $3\nu_{1}+2\nu_{2}+\nu_{3}+n_{h}$ nodes. On the other hand, since $D$ is a tree, the properties of expansions imply that the number of nodes in $D''$ equals $\#D''-1+n$. We have $\#D''=\nu_1+\nu_2+\nu_3+\#D\hor$. Hence $2\nu_{1}+\nu_2+n_{h}=\#D\hor-1+n$. Substituting the formula for $n$ from \ref{item:C**_n}, we get the required equality.
	
	\ref{item:C**_nu3} A fiber $F\subseteq D'$ as in Lemma \ref{lem:C**_fibers}\ref{item:C**_F3} meets $D'\hor$ in one point. Connectedness of $D$ implies that $F$ does not contain a center of $\psi'$. Therefore, the preimage of $F$ is a fiber contained in $D$. Because components of $D$ are numerically independent, $D$ contains at most one fiber, so $\nu_3\leq 1$, as claimed. 
\end{proof}

\begin{prop}[{cf.\ \cite{MiySu-Cstst_fibrations_on_Qhp}}]\label{prop:C**}
 Let $S$ be a smooth \QHP of log general type. If $S$ admits a $\C^{**}$ fibration then $S$ is obtained by tom Dieck--Petrie algorithm from data listed in Table \ref{table:C**}, see Figure \ref{fig:Cstst}.
\end{prop}
\begin{proof}
	We keep the notation introduced in Lemmas \ref{lem:C**_fibers}, \ref{lem:C**_formulas}.
	
	Assume first that $D'\hor$ consists of three $1$-sections. Then $X'$ is a Hirzebruch surface $\F_{m}$ for some $m\geq 0$; and $\nu_2=\nu_3=0$. By Lemma \ref{lem:C**_formulas}\ref{item:C**_nodes}, we have $n_{h}=3-\nu_1$. In particular, $n_{h}\leq 3$. If $n_{h}=3$ then $\nu_1=0$, but since the centers of expansion lie on the fibers in $D'$, we get $n=0$, contrary to Lemma \ref{lem:C**_formulas}\ref{item:C**_n}.
	
	Consider the case  $n_{h}=2$. Numerical properties of Hirzebruch surfaces, together with the fact that $D'\hor$ is snc, imply that $m\in \{0,2\}$. Suppose $m=0$, i.e.\ $X'\cong \P^1\times \P^1$. Then $D'\hor$ consists of a diagonal and two horizontal lines, i.e.\ one curve of type $(1,1)$ and two of type $(0,1)$. Since $\nu_1=3-n_{h}=1$, each horizontal line meets $D'$ twice. Hence the second projection $\P^1\times\P^1\to \P^1$ restricts to a $\C^{*}$-fibration of $X'\setminus D'$; a contradiction with the Iitaka easy addition theorem. 
	
	Thus $m=2$. Now, $D'\hor$ consists of the negative section and two positive ones. As before, $\nu_1=1$. By Lemma \ref{lem:C**_formulas}\ref{item:C**_n}, we have $n=2$. The centers of the expansion $\psi'$ are either two nodes of $D'\hor$, or one node of $D'\hor$ and one common point of $D'\hor$ and the fiber in $D'$. In each case, the pencil of positive sections passing through those centers restricts to a $\C^{**}$-fibration of $X'\setminus D'$, and pulls back to a $\P^1$-fibration $\tilde{p}\colon X\to \P^1$. By the Iitaka easy addition theorem, the restriction $\tilde{p}|_{S}$ is a $\C^{**}$-fibration, too. The horizontal part of $D$ with respect to $\tilde{p}$ consists of three $1$-sections, namely: the first exceptional curves over each center of $\psi'$; and the proper transform of the fiber $D'\vert$ or of one of the positive sections in $D'$. We now replace $\bar{p}$ with $\tilde{p}$, and choose $\psi'$ in such a way that the number of nodes in $D\hor$ and $D'\hor$ is the same. This way, we reduce the proof to the case $n_{h}\in \{0,1\}$.
	\smallskip
	
	Consider the case $n_{h}=0$. Then $m=0$ and $D'$ is a union of three vertical and three horizontal lines. This case covers types \cite[{$(UP_{3-1})$, $(UC_{2-1})$ $(UC_{2-1})'$}]{MiySu-Cstst_fibrations_on_Qhp}, see pp.\ 13-18 loc.\ cit. We claim that $S$ is as in \ref{def:C**_1}. 
	
	Let $p$, $q$ be two centers of $\psi'$ which do not lie on the same component of $D'$. Then the weight corresponding to, say, $p$, is different than $1$. Indeed, otherwise a proper transform on $X$ of the pencil of curves of type $(1,1)$ passing through $p$ and $q$ would induce a $\C^{*}$-fibration of an open subset of $S$, which is impossible since $\kappa(S)=2$. Let $\bl{p}\colon \tilde{X}'\to \P^{1}\times \P^{1}$ be a blowup at $p$ and let $\tilde{D}'\de (\bl{p}^{*}D')\redd$. Then $(\tilde{X}',\tilde{D}')\to (\P^{2},\pp)$ is a minimal log resolution of $\pp\subseteq \P^{2}$ as in  \ref{def:C**_1}, and $\psi'=\bl{p}\circ \psi$ for some expansion $\psi\colon(X,D)\to (\tilde{X}',\tilde{D}')$. Using the action of $\Aut(\P^{2},\pp)\cong \Z_{2}\times \Z_{2}$ we can arrange the centers of $\psi$ to be as in  \ref{def:C**_1}. Note that if one of the centers is $(L_{j}',L_{j})$ for some $j\in \{1,2\}$, then the corresponding weight cannot be $1$, because otherwise the proper transform of $L_{j}'$ would be a superfluous $(-1)$-curve in $D$, which is impossible. 
	The weights $(1,1,2,1)$ and $(1,1,2,2)$ are explicitly excluded in Table \ref{table:C**_1} because they correspond to the Fujita surfaces  $Y\{2,4,4\}$ and $Y\{2,3,6\}$, respectively, which have Kodaira dimension zero  \cite[8.64]{Fujita-noncomplete_surfaces}.

	Consider now the case $n_{h}=1$. Then $m=1$ and $D'\hor$ consists of the negative section $\Sec_1$ and two sections of type $(0,1)$. Let $\varsigma\colon \F_1\to \P^2$ be the contraction of $\Sec_1$. Then $\varsigma_{*}D'$ is a sum of four lines, as in  \ref{def:C**_1a}. If some weight of the expansion $\psi\de \varsigma\circ \psi'$ equals one then the pencil of lines through the corresponding center induces a $\C^{*}$-fibration of some open subset of $S$, which is impossible. Similarly, if two weights are $(2,\tfrac{1}{2})$ then so does a pencil of conics tangent to the corresponding lines. 
	
	By symmetry, we can assume that the centers of $\psi$ are $(\ll_{1},\ll_{2})$, $(\ll_{1}',\ll_{1})$ and either $(\ll_{2},\ll_{1}')$ or $(\ll_{2},\ll_{2}')$, see \cite[3.16]{DiPe-hp_and_alg_curves}; with weights $(u,v_{1},v_{2})$ such that $u>1$. In the former case, $S$ is as in  \ref{def:C**_1a}: the excluded weights $(2,2,2)$ correspond to the remaining Fujita surface $Y\{3,3,3\}$ of Kodaira dimension zero. In the latter case, $S$ is obtained in  \ref{def:C**_1} by expansion at $(L_{1},L_{2};1)$, $(E_{q_{1}},L_{q_1 q_2};u-1)$ and, for $j\in\{1,2\}$, at $(L_{j},E_{q_{j}};v_{j}-1)$ if $v_{j}>1$ or at $(L_{j}',L_{j};v_{j}^{-1}-1)$ otherwise, see \cite[p.\ 117]{tDieck_optimal-curves}.
	\smallskip

	Assume now that $D'$ contains a multi-section. Let $(X',D')\to (\F_{m},D'')$ be some contraction of degenerate fibers of $\bar{p}'$ such that the number of nodes in $D''\hor$ and $D'\hor$ is the same.
	
	Consider the case when $D''\hor$ consists of a $2$-section $H_2$ and a $1$-section $H_1$. Then $\nu_3=0$. By Lemma \ref{lem:C**_formulas}\ref{item:C**_Hurwitz} we have $\nu_2\leq 2$. Moreover, $H_1\cdot H_2=n_h=1-\nu_1$ by Lemma \ref{lem:C**_formulas}\ref{item:C**_nodes}, so $H_{1}\cdot H_{2}\leq 1$. Since $H_2\cong \P^1$,  adjunction formula implies that $H_2$ is of type $(1-m,2)$. In particular, $m\leq 1$. 
	
	Assume $m=0$. Then $H_2$ is of type $(1,2)$, so it meets $H_1$. Hence $n_h=H_1\cdot H_2=1$ and $\nu_1=0$. Now $D''\vert+H_2$ is the horizontal part of $D''$ for the other projection from $\P^{1}\times \P^{1}$. It consists of $\nu_2+1\leq 3$ $1$-sections. If $\nu_2\leq 1$, we get this way a $\C^{1}$- or a $\C^{*}$-fibration of an open subset of $S$, contrary to the assumption $\kappa(S)=2$. Thus  $\nu_2=2$, and we get a $\P^1$-fibration with three $1$-sections in $D''$, i.e.\ the previous case.
	
	Therefore, we can assume $m=1$. If $H_1$ is not the negative section then it is of type $(\frac{1}{2}n_{h}-1,1)$, which is impossible since $n_{h}=1-\nu_1\leq 1$. Thus $H_1=\Sec_1$, so $n_{h}=0$ and $\nu_1=1$. Let $\varsigma\colon \F_1\to\P^2$ be the contraction of $H_1$. Then $\cc_{1}\de \varsigma_{*}D''\hor$ is a conic, and $\varsigma_{*}D''\vert$ is a sum of $\nu_1+\nu_2$ lines meeting outside $\cc_{1}$. Out of these lines, $\nu_{2}\leq 2$ is tangent to $\cc_{1}$ and $\nu_{1}=1$ is not. If $\nu_2\leq 1$ then a pencil of lines through some point of $\cc_{1}\cap \varsigma(D''\vert)$ gives a $\P^{1}$-fibration of $X$ such that $D\hor$ is a sum of at most three $1$-sections, as in the previous case. Hence we can assume $\nu_{2}=2$. Then $S$ is as in  \ref{def:C**_2}. We can assume that the weight of the expansion at  $(L_{pp'}, C_1)$ is not one: indeed, otherwise the pencil of lines through that point pulls back to a $\P^1$-fibration of $X$ such that the horizontal part of $D$ consists of three one sections, namely $C_1,L_1$ and $L_2$, which is the case considered before. Now using symmetries $(p_1,p_2,p,p')\mapsto (p_2,p_1,p,p')$ and $(p_1,p_2,p,p')\mapsto (p_1,p_2,p,p'')$, we can take the expansions as in Table \ref{table:C**_2}.  In the notation of \cite{MiySu-Cstst_fibrations_on_Qhp}, we get cases $(TP_2)$ and  $(TC_{2-1})$, see pp.\ 18-20 loc.\ cit.
	\begin{figure}[ht]
	\begin{subfigure}[b]{.24\textwidth}
		\centering
		\begin{tikzpicture}[scale=0.8]
			\path[use as bounding box] (-0.4,-0.4) rectangle (4.6,4.4);
			\draw (-0.1,4) -- (4.1,4);
			\node[above] at (1,4) {\small{$L_2$}};
			\node[below] at (1,4) {\small{$0$}};
			\draw (-0.1,2) -- (4.1,2);
			\node[above] at (1,2) {\small{$L_1'$}};
			\node[below] at (1,2) {\small{$0$}};
			\draw (-0.1,0) -- (3.1,0);
			\node[above] at (1,0) {\small{$E_{r_1}$}};
			\node[below] at (1,0) {\small{$-1$}};
			\draw (0,4.1) -- (0,-0.1);
			\node[right] at (0,3) {\small{$L_1$}};
			\node[left] at (0,3) {\small{$0$}};
			\draw (2,4.1) -- (2,-0.1);
			\node[right] at (2,3) {\small{$L_2'$}};
			\node[left] at (2,3) {\small{$0$}};
			\draw (4,4.1) -- (4,0.9);
			\node[right] at (4,3) {\small{$E_{r_2}$}};
			\node[left] at (4,3) {\small{$-1$}};
			\draw (2.9,-0.1) -- (4.1,1.1);
			\node at (3.2,0.6) {\small{$-1$}};
			\node at (4.1,0.4) {\small{$L_{r_1 r_2}$}};
		\end{tikzpicture}
		\caption{\ref{def:C**_1}}
		\label{fig:C**_1}
	\end{subfigure} 
	\begin{subfigure}[b]{.2\textwidth}
		\centering
		\begin{tikzpicture}[scale=0.8]
			\path[use as bounding box] (-0.4,-0.4) rectangle (4.4,4.4);
			\draw (-0.4,0) -- (4.2,0);
			\node at (1,0.3) {\small{$\ell_{2}'$}};
			\node at (1,-0.3) {\small{$1$}};
			\draw (-0.2,-0.2) -- (1.1,4);
			\node at (-0.2,0.4) {\small{$1$}};
			\node at (0.3,0.4) {\small{$\ell_1$}};
			\draw (2.2,-0.2) -- (0.9,4);
			\node at (1.8,0.4) {\small{$1$}};
			\node at (2.3,0.4) {\small{$\ell_2$}};
			\draw (0.2,2) -- (4.2,-0.2);
			\node at (1,1.9) {\small{$\ell_{1}'$}};
			\node at (1,1.3) {\small{$1$}};
		\end{tikzpicture}
		\caption{\ref{def:C**_1a}}
		\label{fig:C**_1a}
	\end{subfigure} 
	\begin{subfigure}[b]{.26\textwidth}
	\centering
			\begin{tikzpicture}
			\path[use as bounding box] (-0.4,-0.4) rectangle (4.4,4.4);
			\draw (0.2,-0.4) -- (0,1);
			\node at (-0.2,0.3) {\small{$-2$}};
			\draw (0,0.8) -- (0.2,2.2); 
			\node at (-0.2,1.5) {\small{$-1$}};
			\node at (0.4,1.5) {\small{$E_{p_1}$}};			
			\draw (0.2,2) -- (0,3.4);
			\node at (-0.2,2.8) {\small{$-2$}};
			\node at (0.4,2.8) {\small{$L_{1}$}};
			\draw (1.6,-0.4) -- (1.4,1);
			\node at (1.2,0.3) {\small{$-2$}};
			\draw (1.4,0.8) -- (1.6,2.2); 
			\node at (1.2,1.5) {\small{$-1$}};
			\node at (1.8,1.5) {\small{$E_{p_2}$}};			
			\draw (1.6,2) -- (1.4,3.4);
			\node at (1.2,2.8) {\small{$-2$}};
			\node at (1.8,2.8) {\small{$L_{2}$}};
			\draw (2.8,3.4) -- (2.8,-0.4);
			\node at (2.7,0.3) {\small{$0$}};
			\node at (3.2,0.3) {\small{$L_{pp'}$}};
			\draw (-0.2,3.2) -- (3.6,3.2);
			\node at (3.4,3.4) {\small{$E_{p}$}};
			\node at (3.4,3) {\small{$-1$}};			
			\draw (-0.2,1.2) -- (3.4,1.2) to[out=0,in=0] (3.4,2.4) -- (2.4,2.4);
			\node at (3.6,1.8) {\small{$0$}};
			\node at (4,1.8) {\small{$C_1$}};
		\end{tikzpicture}
		\caption{\ref{def:C**_2}}
		\label{fig:C**_2}
	\end{subfigure}
	\begin{subfigure}[b]{.28\textwidth}
	\centering
		\begin{tikzpicture}
		\path[use as bounding box] (-0.4,-0.4) rectangle (4.4,4.4);
			\draw (0.2,-0.4) -- (0,1);
			\node at (-0.2,0.3) {\small{$-3$}};
			\node at (0.35,0.3) {\small{$L_1$}};	
			\node at (0.15,-0.05) {\large{$\bullet$}};
			\node at (0.1,0.3) {$\boldsymbol{\circ}$};
			\draw (0,0.8) -- (0.2,2.2); 
			\node at (-0.2,1.5) {\small{$-1$}};
			\node at (0.3,1.5) {\small{$L_{2}'$}};			
			\draw (0.2,2) -- (0,3.4);
			\node at (-0.2,2.8) {\small{$-2$}};
			\draw (0,3.2) -- (0.2,4.6);
			\node at (-0.2,4) {\small{$-2$}};
			\draw (1.6,-0.4) -- (1.4,1);
			\node at (1.2,0.3) {\small{$-2$}};
			\node at (1.55,-0.05) {\large{$\bullet$}};
			\draw (1.4,0.8) -- (1.6,2.2); 
			\node at (1.2,1.5) {\small{$-1$}};
			\node at (1.85,1.5) {\small{$L_{qp'}$}};			
			\draw (1.6,2) -- (1.35,3.75);
			\node at (1.2,2.8) {\small{$-2$}};
			\node at (1.75,2.8) {\small{$C_1$}};
			\node at (1.45,3.05) {$\boldsymbol{\circ}$};
			\draw (3,-0.4) -- (2.8,1);
			\node at (2.6,0.3) {\small{$-2$}};
			\node at (2.95,0.05) {$\boldsymbol{\circ}$};
			\draw (2.8,0.8) -- (3,2.2); 
			\node at (2.6,1.5) {\small{$-1$}};
			\node at (3.1,1.5) {\small{$L_{1}'$}};			
			\draw (3,2) -- (2.75,3.75);
			\node at (2.6,2.8) {\small{$-2$}};
			\node at (3.15,2.8) {\small{$C_{2}$}};
			\node at (2.85,3.05) {\large{$\bullet$}};			
			\draw (-0.2,1.2) -- (3.6,1.2) to[out=0,in=0] (3.6,2.4) -- (1.2,2.4);
			\node at (3.7,1.8) {\small{$-1$}};
			\node at (4.2,1.8) {\small{$E_{p'}$}};
			\draw[dashed] (-0.2,4.4) -- (0.4,4.4) to[out=0,in=180] (1.2,3.6) -- (4,3.6);
			\node at (3.7,3.8) {\small{$E_{p_2}$}};
			\node at (3.7,3.4) {\small{$-1$}};
			\node at (3.8,0.6) {\small{$\boldsymbol{\circ}\mapsto p_1$}};
			\node at (3.7,0.2) {\small{$\bullet\mapsto q$}};
		\end{tikzpicture}
		\caption{\ref{def:C**_3}}
		\label{fig:C**_3}
	\end{subfigure}
	\vspace{-.5em}
	\caption{The $\C^{**}$-fibered case, see \cite{MiySu-Cstst_fibrations_on_Qhp}.}
	\label{fig:Cstst}
\end{figure}
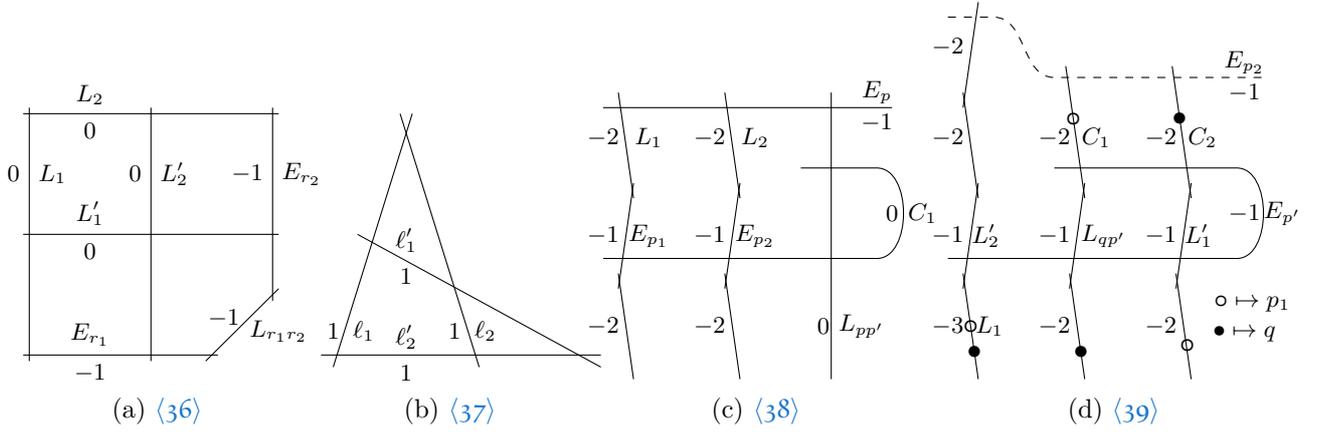	
	
	We are left with the case when $D\hor$ is a $3$-section. Like before, by adjunction we get that $m=0$ and $D''\hor$ is of type $(1,3)$. We have $\nu_3\leq 1$ by Lemma \ref{lem:C**_formulas}\ref{item:C**_nu3}, so $\nu_3=1$, $\nu_1=n_h=0$ by Lemma  \ref{lem:C**_formulas}\ref{item:C**_nodes}, and $\nu_2\leq 2$ by Lemma \ref{lem:C**_formulas}\ref{item:C**_Hurwitz}. Assume $\nu_2\leq 1$. Then as before, the horizontal part of $D''$ for the other projection from $\P^{1}\times \P^{1}$ consists of at most three $1$-sections, so we are reduced to the previous case. Hence we can assume that $(\nu_{1},\nu_{2},\nu_{3})=(0,2,1)$. This is the case $(T3C_2)$ from \cite[p.\ 21]{MiySu-Cstst_fibrations_on_Qhp}. We claim that $S$ is as in   \ref{def:C**_3}. 
	
	Let $F_{0}$ and $F_{i}$, for $i\in\{1,2\}$ be the proper transforms on $X'$ of the fibers meeting $D''\hor$ in one and two points, respectively; and let $H_{j}$ for $j\in \{0,1,2\}$ be the proper transform of the horizontal lines through the points of tangency of $F_{j}$ and $D''\hor$. Let $U_{j},E_{j}\subseteq X'$ be the first and the last exceptional curve over $H_{j}\cap F_{j}$, respectively, see Notation \ref{not:bl}. Now $D'+H_{0}+H_{2}+H_{3}$, and hence $S$, is as in  \ref{def:C**_3}. Indeed, the curves $F_{0}$, $F_{1}$, $F_{2}$, $E_{0}$, $E_{1}$, $E_{2}$, and $D''\hor$, $U_{0}$, $H_{1}$, $H_{2}$ defined above correspond there to the proper transforms of $\ll_{1}$, $\cc_{1}$, $\cc_{2}$, $\ll_{2}'$, $\ll_{qp'}$, $\ll_{1}'$ and the exceptional $(-1)$-curves over $p'$ $p_{2}$, $q$, $p_1$, respectively. Using the symmetry $(p_1,p_2,p',q) \mapsto(q,p_2,p',p_1)$, we can assume that the expansions are as in Table \ref{table:C**_3}.
\end{proof}

\section{Construction of planar divisors}\label{sec:constructions}

The results obtained so far allow us to conclude that all \QHPs of log general type satisfying the Negativity Conjecture \ref{conj:negativity} are obtained by the tom Dieck--Petrie algorithm from certain arrangements of lines and conics. Combinatorial types of the latter are listed in \tables. In this section, we construct arrangements realizing each combinatorial type, and count their number. The latter computation, summarized in Proposition \ref{prop:conf_uniq} below, will be important in the proof of Corollary \ref{cor:uniq}\ref{item:n}. Indeed, as we will see in Section \ref{sec:uniqueness}, different (marked) arrangements of the same type will yield different $\Q$HPs with the same weighted boundary graph (with the exception of tower \ref{def:F2_n1-cusp}, discussed in Examples \ref{ex:7} and \ref{ex:7_uniq}). 
\smallskip

To state Proposition \ref{prop:conf_uniq}, we recall technical definitions of a combinatorial type and a marking. Let $\pp\subseteq \P^2$ be an arrangement (also called: a configuration) of lines and conics. The \emph{combinatorial type} of $\pp$, see \cite[Remark 3]{ACT_Zariski-pairs}, is a pair $(\cC,\cS)$, where $\cC$ is the set of components of $\pp$, each decorated by its degree (either $1$ or $2$); and $\cS$ is the set of singular points of $\pp$, each decorated with the following data: a list of components meeting at that point, and for each pair of such components, their local intersection number. Note that since all components of $\pp$ are smooth, a decoration of a point $q\in \cS$ encodes the topological type of $\pp$ at $q$. We identify two combinatorial types $(\cC,\cS)$ and $(\cC',\cS')$ if there are bijections $\cC\cong \cC'$ and $\cS\cong \cS'$ which preserve decorations. This way, projectively equivalent configurations have the same combinatorial type. 

A \emph{marking} on a configuration $\pp$ is the choice of an order on each maximal subset of $\cS$  whose elements have the same decorations. We say that \emph{marked} configurations $\pp,\pp'$ are projectively equivalent if there is an automorphism of $\P^2$ which maps $\pp$ to $\pp'$ and preserves the marking. The reader might have in mind the following simple example, which will often appear as a sub-configuration of the ones constructed below.

\begin{ex}\label{ex:triangle}
Let $\ll_1$, $\ll_2$, $\ll_3$ be non-concurrent lines, let $\cc$ be a conic tangent to each of them, and let $\pp_0=\cc+\ll_1+\ll_2+\ll_3$. The configuration $\pp_0$ admits a unique (trivial) marking; and is uniquely determined, up to a projective equivalence, by its combinatorial type. Let now $\ll$ be a line meeting $\pp_0$ normally. The configuration $\pp\de \pp_0+\ll$ admits two markings, corresponding to the two possible orders on the set $\cc\cap \ll$. The two marked configurations are projectively equivalent if and only if those two points lie in the same orbit of $\Aut(\P^2,\pp)$: this happens if and only if the pair $(\P^2,\pp_0)$ admits an involution fixing $\ll$. 
\end{ex}

We say that a configuration $\pp$ is \emph{realized} over a number field $\kk$ if in some coordinates of $\P^2$, the equation defining $\pp$ has coefficients in $\kk$; and all singular points of $\pp$ have coordinates in $\kk$. In Example \ref{ex:triangle}, the configuration $\pp_0$ is realized over $\Q$: to see this, one can use  e.g.\ coordinates shown in Figure \ref{fig:coordinates} below. If in these coordinates the equation of $\ll$ has rational coefficients, too, then the configuration $\pp_0+\ll$ is realized over a minimal extension $\kk$ of $\Q$ containing coordinates of the points in $\cc\cap \ll$. In case $\kk\neq \Q$, the field $\kk$ is a quadratic extension of $\Q$, and the points in $\cc\cap \ll$ are conjugate by $\Aut_{\Q}(\kk)$.
\smallskip

Each Configuration \ref{conf:F2n0}--\ref{conf:P2n2_cuspidal} below specifies a combinatorial type of an arrangement of lines and conics. These types are then referred to in \tables (note that some of them, e.g.\ Configuration \ref{conf:4}, are quoted multiple times, but each time we refer to a type of a different sub-arrangement). We introduce the following notation.

\begin{notation}\label{not:P}
	For a row $\rst$ of \tables we denote by $\tst{\cP}$ the set of projective equivalence classes of marked configurations whose combinatorial type is specified in $\rst$. Moreover, we denote by $\tst{\ngr}$ the number in the column \enquote{$\ngr$}, and by $\tst{\kk}$ the number field in the column \enquote{$\kk$}.
\end{notation}

Recall that Corollaries \ref{cor:uniq}\ref{item:n} and \ref{cor:k} assert that $\tst{\ngr}$ is the number of isomorphism classes of $\Q$HPs in tower $\rst$ sharing the same boundary graph; and $\tst{\kk}$ is a field over which all $\Q$HPs in $\rst$ are defined. To prove this assertion, we first relate $\tst{\ngr}$ and $\tst{\kk}$ with the corresponding planar configurations, as follows.

\begin{prop}\label{prop:conf_uniq}
	Let $\rst$ be a row of \tables, and let $\tst{\cP}$, $\tst{\ngr}$, $\tst{\kk}$ be as in Notation \ref{not:P}.
	\begin{enumerate}
		\item\label{item:conf_number} We have $\#\tst{\cP}=\tst{\ngr}$ if $\rst\neq$ \ref{def:F2_n1-cusp}; and $\#\tref{\cP}{def:F2_n1-cusp}=2$.
		\item\label{item:conf_field} Every configuration in $\tst{\cP}$ is realized over $\tst{\kk}$ if  $\rst\neq$\ref{def:F2n0}. The configuration $\tref{\cP}{def:F2n0}$ is realized over $\Q(\sqrt{-15})$. 
		\item\label{item:conf_conj} Every two configurations in $\tst{\cP}$ are conjugate by some element of the Galois group of $\kk_{\mathfrak{\langle * \rangle}}$.
		\item\label{item:conf_diffeo} Assume that  $\#\tst{\cP}\geq 2$ and $\rst\neq$ \ref{def:A1A2_q-cn_22}, \ref{def:F2_n1-cusp}, \ref{def:F2_n1-node-3}, \ref{def:F2-5}. Then $\tst{\cP}$ contains a pair of complex-conjugate configurations. If $\#\tst{\cP}=4$ then $\tst{\cP}$ consists of two pairs of complex-conjugate configurations.
	\end{enumerate}
\end{prop}

\begin{proof}[Proof of Corollary \ref{cor:k}]
	By Definition \ref{def:TDP} of the tom Dieck--Petrie algorithm, each surface $S$ in Theorem \ref{CLASS} is obtained by taking a minimal log resolution of some $(\P^2,\pp)$, followed by an expansion. Thus if a configuration $\pp$ is realized over $\kk$, then the surface $S$ is defined over $\kk$. Hence for every $\rst\neq$ \ref{def:F2n0}, Corollary \ref{cor:k} follows from Proposition \ref{prop:conf_uniq}\ref{item:conf_field}. It remains to see that the surface \ref{def:F2n0} is defined over $\Q$: this is proved in Remark \ref{rem:Moe}. An alternative argument (relying on fewer computations) is given in Remark \ref{rem:F2n0_Q}.
\end{proof}

\begin{rem}[candidate Zariski pairs, cf.\ Remark \ref{rem:Marco}]\label{rem:candidate-Zariski}
	Consider tower \ref{def:A1A2_C2C3-node}, constructed from Configuration \ref{conf:4}, see Figure \ref{fig:4_conf}. The set $\tref{\cP}{def:A1A2_C2C3-node}$ consists of four classes of marked configurations. Viewing them as unmarked ones, we get two non-projectively equivalent configurations of the same combinatorial type. As we will see below, the field automorphism from Proposition \ref{prop:conf_uniq} which conjugates them is the generator of $\Aut_{\Q}(\Q(\sqrt{5}))$, so it does not extend to a complex conjugation.
	
	It remains an open question whether there exists a homeomorphism $\P^2\to \P^2$ which maps one of those configurations to the other. This makes the set $\tref{\cP}{def:A1A2_C2C3-node}$  a candidate for an \emph{algebraic Zariski pair}, see \cite[Definition 8.2]{ACC_algebraic-Zariski} for a definition and \cite{ACT_Zariski-pairs} for a survey. Note that the complements of both configurations in $\tref{\cP}{def:A1A2_q-cn_31}$ have the same algebraic fundamental groups, see \cite[Corollary 12.11]{AM_etale}. 
	
	A similar phenomenon occurs in triples $\tref{\cP}{def:A1A2_c=1}$, $\tref{\cP}{def:A1A2_c=2_32}$ and $\tref{\cP}{def:A1A2_c=2_41}$, see Figures \ref{fig:A1A2_c=1_conf} and \ref{fig:A1A2_c=2_conf}. In each of those cases, the three configurations are not projectively equivalent even viewed as unmarked ones (in fact, their markings are trivial). Each of those triples consists of one configuration realized over $\R$, and a pair of complex-conjugate ones. Modding out by complex conjugation, we get candidates for algebraic Zariski pairs.

	Some examples of Zariski pairs of degree $7$ conic-line arrangements were recently given in \cite{Tokunaga_Zar-deg-7,AS_Zariski-pairs-of-conics}. 
	Like in our setting, one can use them to construct towers of affine surfaces. The main difference is that those affine surfaces will have topological Euler characteristic at least $3$: this can be seen from the combinatorics of those arrangements by the same argument as in Lemma \ref{lem:expansions}\ref{item:expansions_det}.
\end{rem}

We will now construct all our configurations, proving Proposition \ref{prop:conf_uniq} along the way. 
\smallskip

From now on, by a \emph{conic} $\cc\subseteq \P^2$ we mean a \emph{smooth} curve of degree two. By a \emph{pencil of conics} we mean a line in the linear system $|\cc|$, so its degenerate members are not conics in our sense. In all configurations, we keep Notation \ref{not:C1}. As in Notation \ref{not:P2}, for two points $a,b\in \P^2$ we denote by $\ll_{ab}$ the line joining them. Moreover, for each configuration we will introduce additional notation, which will \emph{not} be consistent throughout the section.

\begin{notation}[Figure \ref{fig:coordinates}]\label{not:C1}
	Fix four points $p, p_{1}, p_{2}, p'$ in a general position. 
	For $i\in \{1,2\}$ let $\ll_{i}$ and $\ll_{i}'$ be the lines joining $p_{i}$ with $p$ and $p'$, respectively. Let $\cc_{1}$ be the unique conic tangent to $\ll_{i}$ at $p_{i}$ and passing through $p'$. Let $\ll_{p'}$ be the line tangent to $\cc_{1}$ at $p'$, and let $\tau\in \Aut(\P^{2})$ be the involution $(p,p',p_1,p_2)\mapsto (p,p',p_2,p_1)$, so $\langle \tau \rangle =\Aut(\P^2,\cc_1+\ll_1+\ll_2+\ll_1'+\ll_2')$. We fix coordinates $[x:y:z]$ on $\P^{2}$ such that 
	\begin{equation}\label{eq:coordinates}
		p=[0:1:0],\quad p_{1}=[1:0:0], \quad p_{2}=[0:0:1], \quad p'=[1:1:1].
	\end{equation}
\end{notation}
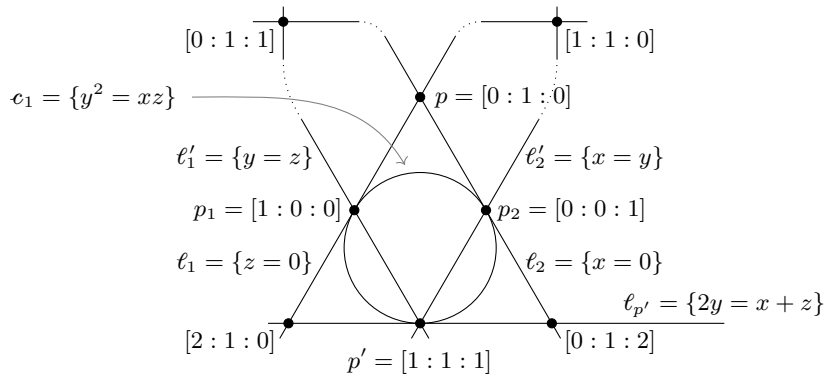
\begin{figure}[htbp]
	\begin{tikzpicture}
		\draw (0,1) circle (1);
		\filldraw (-1.732,0) circle (0.06);
		\node at (-2.5,-0.25) {\small{$[2:1:0]$}};
		\filldraw (-0.866,1.5) circle (0.06);
		\node at (-2,1.5) {\small{$p_1=[1:0:0]$}};
		\filldraw (1.732,0) circle (0.06);
		\node at (2.5,-0.25) {\small{$[0:1:2]$}};
		\filldraw (0.866,1.5) circle (0.06);
		\node at (2,1.5) {\small{$p_2=[0:0:1]$}};
		\filldraw (0,3) circle (0.06);
		\node at (1.1,3)  {\small{$p=[0:1:0]$}};
		\draw (-2,0)--(4,0);
		\node at (4,0.25) {\small{$\ll_{p'}=\{2y=x+z\}$}};
		\draw (-1.847,-0.2) -- (0.46,3.8);
		\draw[dotted] (0.46,3.8) to[out=60,in=180] (0.8,4);
		\draw (0.8,4) -- (2.2,4);
		\node at (-2.3,0.8) {\small{$\ll_{1}=\{z=0\}$}};
		\draw (1.847,-0.2) -- (-0.46,3.8);
		\draw[dotted] (-0.46,3.8) to[out=120,in=0] (-0.8,4);
		\draw (-0.8,4) -- (-2.2,4);
		\node at (2.3,0.8) {\small{$\ll_{2}=\{x=0\}$}};
		\filldraw (0,0) circle (0.06);
		\node at (0,-0.5) {\small{$p'=[1:1:1]$}}; 
		\draw (-0.115,-0.2) -- (1.558,2.7);
		\node at (2.3,2.2) {\small{$\ll_{2}'=\{x=y\}$}};
		\draw[dotted] (1.558,2.7) to[out=60,in=-90] (1.8,3.5);
		\draw (1.8,3.5) -- (1.8,4.2);
		\filldraw (1.8,4) circle (0.06);
		\node at (2.5,3.75) {\small{$[1:1:0]$}};
		\draw (0.115,-0.2) -- (-1.558,2.7);
		\node at (-2.3,2.2) {\small{$\ll_{1}'=\{y=z\}$}};
		\draw[dotted] (-1.558,2.7) to[out=120,in=-90] (-1.8,3.5);
		\draw (-1.8,3.5) -- (-1.8,4.2);
		\filldraw (-1.8,4) circle (0.06);
		\node at (-2.5,3.75) {\small{$[0:1:1]$}};
		\node at (-4.3,3)  {\small{$\cc_{1}=\{y^2=xz\}$}};
		\draw [->, black!50] (-3,3) to[out=0,in=120] (-0.2,2.1);
	\end{tikzpicture}
	\caption{Notation \ref{not:C1} in coordinates given by \eqref{eq:coordinates}.}
	\label{fig:coordinates}
\end{figure}	

To make the statements shorter, we will often abuse the notation and say that the set $\tst{\cP}$ \emph{is realized over } some field $\kk$. This means that all configurations in $\tst{\cP}$ are realized over $\kk$, and any two of them are conjugate by some element of the Galois group of $\kk$, see Proposition \ref{prop:conf_uniq}\ref{item:conf_field},\ref{item:conf_conj}. 

We denote by $\omega$ a primitive third root of unity.

\begin{conf}[Figure \ref{fig:F2n0_conf}]\label{conf:F2n0}
	Put $p_{3}=p'$, $\ll_{3}=\ll_{p'}$. 
	Write $\ll_{3}\cap \ll_{1}=\{q_{2}\}$, $\ll_{3}\cap \ll_{2}=\{q_{1}\}$, $q_3= p$. Fix a point $r\in \cc_{1}\setminus \bigcup_{i}\ll_{p_i q_i}$ and write $\cc_{1}\cap \ll_{rq_{i}}=\{r,q_{i}'\}$, $i\in \{1,2,3\}$. For a fixed $r$, define conics $\cc_2$, $\cc_3$ as follows.
	\begin{itemize}
		\item $\cc_{2}$ is the unique conic passing through $q_1,q_2,q_3,q_1',q_2'$. 
		\item $\cc_{3}$ is the unique conic passing through $q_1,q_2,q_3,r$ and tangent to $\cc_1$ at some point $s_2$. 
	\end{itemize}
	 We have  $q_3'\not \in  \cc_2$ and $(\cc_3\cdot \cc_2)_{s_2}=2$. Write $\cc_3\cap \cc_1=\{r,s_2,s_1\}$. Up to the action of $\tau$, there is a unique point $r\in \cc_1$ such that $s_1\in \cc_2$. For such $r$, we have $(\cc_{2}\cdot \cc_{1})_{s_1}=1$.
	 
	 As a consequence, $\#\tref{\cP}{def:F2n0}=1$. The configuration $\tref{\cP}{def:F2n0}$ is realized over $\Q(\sqrt{-15})$.
\end{conf}
\begin{proof}
	 Parametrize $\cc_{1}$ by $\nu\colon  \P^{1}\ni [u:v]\mapsto [u^{2}:uv:v^{2}]\in \P^{2}$. We have $\ll_{p_1 q_1}=\{2y=z\}$, so $\ll_{p_1 q_1}$ meets $\cc_{1}$ at $p_{1}=\nu[1:0]$ and at $\nu[1:2]$. Similarly, $\ll_{p_2 q_2}=\{2y=x\}$ meets $\cc_{1}$ at $p_2=\nu[0:1]$ and at $\nu[2:1]$. Eventually, $\ll_{p_3 q_3}=\{x=z\}$ meets $\cc_1$ at $p_3=\nu[1:1]$ and $\nu[1:-1]$. Thus $r=\nu[t:1]=[t:1:t^{-1}]$ for some $t\in \C\setminus \{0,1,-1,2,\frac{1}{2}\}$. We compute
	\begin{align*}
		&&\ll_{rq_{1}}&=\{2y=z+(2-t^{-1})t^{-1}x\},\qquad & \ll_{rq_{2}}&=\{2y=x+(2-t)tz\}, &\ll_{rq_{3}}&=\{t^{-1}x=tz\}, \\
		&& q_{1}'&=[(2-t^{-1})^{-1}:1:2-t^{-1}], & q_{2}'&=[2-t:1:(2-t)^{-1}],& q_{3}'&=[t:-1:t^{-1}].
	\end{align*}
	Therefore, $\ll_{q_1' q_2'}=\{2(t^2-3t+1)y=(1-2t)x+(t-2)tz\}$. The pencil of conics passing through $q_1,q_2,q_{1}',q_{2}'$ is generated by $\ll_{rq_1}+\ll_{rq_{2}}$ and $\ll_{q_1 q_2}+\ll_{q_1' q_{2}'}$. It has a unique member passing through $q_{3}=[0:1:0]$, namely:
	\begin{equation*}
	\begin{split}
		\cc_{2}=\{&(t^2-3t+1)(z+(2-t^{-1})t^{-1}x-2y)(x+(2-t)tz-2y)=\\
		&=(x+z-2y)((1-2t)x+(t-2)tz-2(t^2-3t+1)y)\}.
	\end{split}
	\end{equation*}
	Now if $q_3'=[t:-1:t^{-1}]$ lies on $\cc_2$ then $t=1$, which is impossible. Hence $q_3'\not\in \cc_2$.
	\smallskip
	
	The pencil of conics passing through $q_{1},q_{2},q_{3},r$ is generated by $\ll_{2}+\ll_{rq_2}$ and $ \ll_{1}+\ll_{rq_1}$. It induces a morphism $\cc_{1}\to \P^{1}$ of degree $3$, ramified at points which corresponds to degenerate members: $\ll_{i}+\ll_{rq_{i}}$, $i=1,2,3$. By Hurwitz formula, this morphism is ramified at exactly one other point, which corresponds to a conic $\cc_{3}$ as in the statement. Thus $\cc_3$ is unique. We compute that 
	\begin{equation}\label{eq:c3}
		\cc_3=\{(1-2t)^{3} \cdot x(x-2y+t(2-t)z)=t(t-2)^{3}\cdot  z (t^2(z-2y)+(2t-1)x)\}.
	\end{equation}
	Thus we have $\cc_1\cap \cc_{3}=\{s_1,s_2,r\}$ and $(\cc_{3}\cdot \cc_{1})_{s_1}=1$, $(\cc_{3}\cdot \cc_{1})_{s_2}=2$, where:
	\begin{equation}\label{eq:s1}
		s_{1}=[(t-2)^2:(t-2)(2t-1):(2t-1)^2],\ s_{2}=[t^{2}(t-2)^{2}:t(t-2)(1-2t):(1-2t)^{2}].
	\end{equation}
	Substituting $s_1$ to the equation of $\cc_1$, we see that $s_1\in \cc_1$ if and only if $t^{\pm 1}=\frac{1}{4}(1-\sqrt{-15})$. The two solutions are mapped to each other by the action of $\tau$; as claimed.
\end{proof}

\begin{conf}[Figure \ref{fig:311_conf}]\label{conf:311} 
	Let $r,\ll_{i},q_{i},q_{i}'$ for  $i\in \{1,2,3\}$ and $\cc_{3}$ be as in Configuration \ref{conf:F2n0}. Define $\cc_{2}$ as the unique conic passing through $q_1,q_2,q_3,q_3',s_2$. Then $q_{1}',q_{2}'\not\in\cc_{2}$. Moreover, up to the action of $\tau$ there is a unique point $r\in \cc_1$ such that $\cc_{2}$ is tangent to $\cc_{1}$ at $q_{3}'$. In this case, $(\cc_{2}\cdot \cc_{1})_{q_{3}'}=2$. 
	
	As a consequence, $\#\tref{\cP}{def:A1A2_c=3}=1$. The configuration $\tref{\cP}{def:A1A2_c=3}$ is realized over $\Q(\sqrt{21})$.
\end{conf}
\begin{proof}
	We  keep the notation from the previous proof; except for $\cc_2$ which will now be a different conic, constructed below. In particular, we have formulas \eqref{eq:c3} and \eqref{eq:s1} for $\cc_3$ and $s_2$, respectively, depending on the chosen point $r=[t:1:t^{-1}]\in \cc_1$, where $t\in \C\setminus \{0,1,-1,2,\frac{1}{2}\}$.
	
	The pencil of conics passing through $q_1,q_2,q_3,q_3'$ is generated by $\ll_{q_1 q_2}+\ll_{q_3 q_3'}$ and $\ll_{q_1 q_3}+\ll_{q_2 q_3'}$. By definition, $\cc_2$ is the unique member of this pencil which passes through $s_2$. Recall that $\ll_{q_1 q_2}= \ll_{3}=\{2y=x+z\}$, $\ll_{q_3 q_3'}=\ll_{r q_3}=\{t^{-1}x=tz\}$, $\ll_{q_1 q_3}=\ll_2=\{x=0\}$. Moreover, $q_2=[2:1:0]$, $q_{3}'=[t:-1:t^{-1}]$, so $\ll_{q_2 q_3'}=\{x-t(t+2)z=2y\}$. Substituting these equations of lines to the above pencil and evaluating at the point $s_2$ given by the formula \eqref{eq:s1}, we conclude that 
	\begin{equation*}
		\cc_{2}=\{(t-2)^2 t^2 (3t-2) (x+z-2y)(t^{-1}x-t z)=3(t-1)^3(t+1)x(x-t(t+2)z-2y)\}.
	\end{equation*}	
	Recall that $q_{1}'=[(2-t^{-1})^{-1}:1:2-t^{-1}]$, $q_{2}'=[2-t:1:(2-t)^{-1}]$. Substituting each of these points to the above formula we get $t\in \{0,1\}$, which is impossible. Hence $q_{1}',q_{2}'\not\in \cc_{2}$, as claimed. 
	
	The line tangent to $\cc_{1}$ at $q_{3}'=[t:-1:t^{-1}]$ is given by $\{x+2ty+t^2z=0\}$. This line is tangent to $\cc_2$ at $q_{3}'$ if and only if $t^{\pm 1}=\frac{1}{10}(11-\sqrt{21})$. As before, these two solutions are equivalent under the action of $\tau$. Thus we get a unique $\cc_2$, as claimed. A direct computation shows that this $\cc_2$ satisfies $(\cc_2\cdot \cc_1)_{q_3'}=2$.
\end{proof}

\begin{figure}[htbp]
	\centering
	\captionsetup{width=.32\linewidth}
	\begin{minipage}[b]{.32\textwidth}
		\centering
		\begin{tikzpicture}
			\path[use as bounding box] (-2.5,-1.5) rectangle (2.5,2.5);
			\draw[name path=c1] (0,0) circle (1);
			\coordinate (Q3) at (0,2); 
			\node at ($(Q3)+(0.3,-0.1)$) {\small{$q_3$}}; \filldraw (Q3) circle (0.06);
			\coordinate (Q2) at (-1.732,-1); 
			\node at ($(Q2)+(0.2,-0.2)$) {\small{$q_2$}}; \filldraw (Q2) circle (0.06);
			\coordinate (Q1) at (1.732,-1);
			\node at ($(Q1)+(0.3,-0.2)$) {\small{$q_1$}}; \filldraw (Q1) circle (0.06);
			\coordinate (P3) at (0,-1);
			\node at ($(P3)-(0,0.2)$) {\small{$p_3$}}; \filldraw (P3) circle (0.06);
			\coordinate (P1) at (-0.866,0.5);
			\node at ($(P1)+(-0.2,0.2)$) {\small{$p_1$}}; \filldraw (P1) circle (0.06);
			\coordinate (P2) at (0.866,0.5);
			\node at ($(P2)+(0.2,0.2)$) {\small{$p_2$}}; \filldraw (P2) circle (0.06);
			\node at (-0.7,1.2) {\small{$\ll_{1}$}};
			\node at (0.7,1.2) {\small{$\ll_{2}$}};
			\draw[add= 0.05 and 0.05, name path = L1] (Q3) to (Q1);
			\draw[add= 0.05 and 0.05, name path = L2] (Q3) to (Q2);
			\draw[add= 0.05 and 0.05, name path = L3] (Q1) to (Q2);
			\node at (0.9,-1.2) {\small{$\ll_{3}$}};
			\coordinate (R) at (-0.8,-0.6);
			\node at ($(R)+(0.1,0.3)$) {\small{$r$}}; \filldraw (R) circle (0.06);
			\draw[add= 0.1 and 0.3, name path = LR1] (Q1) to (R);
			\path [name intersections={of=c1 and LR1}] (intersection-2) coordinate (Q1');
			\node at ($(Q1')+(-0.1,0.3)$) {\small{$q_1'$}}; \filldraw (Q1') circle (0.06);
			\draw[add= 0.2 and 2.2, name path = LR2] (Q2) to (R);
			\path [name intersections={of=c1 and LR2}] (intersection-1) coordinate (Q2');
			\node at ($(Q2')+(-0.3,0.1)$) {\small{$q_2'$}}; \filldraw (Q2') circle (0.06);
			\coordinate (S1) at (0,1);
			\node at ($(S1)-(0,0.3)$) {\small{$s_1$}}; \filldraw (S1) circle (0.06);
			\coordinate (S2) at (-1,0);
			\node at ($(S2)+(0.3,0)$) {\small{$s_2$}}; \filldraw (S2) circle (0.06);
			\coordinate (X) at ($(Q1')!0.5!(Q2')$);
			\node at (X) {\small{$\cc_{1}$}};
			\coordinate (Y) at ($(X)+(0.4,-0.2)$); 
			\draw[blue] ($(Q3)+(0.1,0.4)$) to[out=-105,in=75] (Q3) to[out=-105,in=120] (S1) to[out=-60,in=120] (Q2') to[out=-60,in=60] (Y) to[out=-120,in=0] (Q1') to[out=180,in=10] (Q2) to[out=-170,in=10] ($(Q2)-(1,0.2)$);
			\draw[red] ($(Q3)+(-0.1,0.4)$) to[out=-75,in=105] (Q3) to[out=-75,in=60] (S1) to[out=-120,in=90] (S2) to[out=-90,in=60] (R) to[out=-120,in=20] (Q2) to[out=-160,in=20] ($(Q2)-(1,0.3)$);
			\node at (-0.3,2.3) {\small{\red{$\cc_3$}}};
			\node at (0.3,2.3) {\small{\blue{$\cc_2$}}};
			\draw[blue]  ($(Q1)+(0.5,-0.4)$) [partial ellipse=110:180: 0.6 and 0.8];
			\draw[red]  ($(Q1)+(0.4,-0.5)$) [partial ellipse=90:160: 0.8 and 0.6];
			\node at ($(Q1)+(0.5,0.4)$) {\small{\blue{$\cc_2$}}};
			\node at ($(Q1)+(0.6,0.1)$) {\small{\red{$\cc_3$}}};
		\end{tikzpicture}
		\caption{Configuration \ref{conf:F2n0} yields surface \ref{def:F2n0}, cf.\ Figure \ref{fig:F2n0}}
		\label{fig:F2n0_conf}
	\end{minipage}
	\begin{minipage}[b]{.32\textwidth}
		\centering
		\begin{tikzpicture}
			\path[use as bounding box] (-2.5,-1.5) rectangle (2.5,2.5);
			\draw[name path=c1] (0,0) circle (1);
			\node at (0,1.1) {\small{$\cc_{1}$}};
			\coordinate (Q3) at (0,2); 
			\node at ($(Q3)+(-0.2,0.2)$) {\small{$q_3$}}; \filldraw (Q3) circle (0.06);
			\coordinate (Q2) at (-1.732,-1); 
			\node at ($(Q2)+(0.2,-0.2)$) {\small{$q_2$}}; \filldraw (Q2) circle (0.06);
			\coordinate (Q1) at (1.732,-1);
			\node at ($(Q1)+(0.3,-0.2)$) {\small{$q_1$}}; \filldraw (Q1) circle (0.06);
			\coordinate (P3) at (0,-1);
			\node at ($(P3)-(0,0.2)$) {\small{$p_3$}}; \filldraw (P3) circle (0.06);
			\coordinate (P1) at (-0.866,0.5);
			\node at ($(P1)+(-0.2,0.2)$) {\small{$p_1$}}; \filldraw (P1) circle (0.06);
			\coordinate (P2) at (0.866,0.5);
			\node at ($(P2)+(0.2,0.2)$) {\small{$p_2$}}; \filldraw (P2) circle (0.06);
			\node at (-0.7,1.2) {\small{$\ll_{1}$}};
			\node at (0.7,1.2) {\small{$\ll_{2}$}};
			\draw[add= 0.05 and 0.05, name path = L1] (Q3) to (Q1);
			\draw[add= 0.05 and 0.05, name path = L2] (Q3) to (Q2);
			\draw[add= 0.05 and 0.05, name path = L3] (Q1) to (Q2);
			\node at (0.9,-1.2) {\small{$\ll_{3}$}};
			\coordinate (R) at (-0.8,-0.6);
			\node at ($(R)+(0.2,0.2)$) {\small{$r$}}; \filldraw (R) circle (0.06);
			\draw[add= 0.1 and 0.3, name path = LR3] (Q3) to (R);
			\path [name intersections={of=c1 and LR3}] (intersection-1) coordinate (Q3');
			\node at ($(Q3')+(0.1,-0.3)$) {\small{$q_3'$}}; \filldraw (Q3') circle (0.06);
			\coordinate (S2) at (0.866,-0.5);
			\node at ($(S2)+(0.3,0)$) {\small{$s_2$}}; \filldraw (S2) circle (0.06);
			\draw[blue] ($(Q2)-(0.3,0.3)$) to[out=45,in=-140] (Q2) to[out=40,in=-90] ($(P1)+(0.2,0)$) to[out=90,in=-165] (Q3') to[out=15,in=150] (S2) to[out=-30,in=150] (Q1) to[out=-30,in=150] ($(Q1)+(0.2,-0.1)$);
			\node at (0.2,0.1) {\small{\blue{$\cc_2$}}};
			\draw[red] ($(Q2)-(0.4,0.1)$) to[out=15,in=-160] (Q2) to[out=20,in=180] (R) to[out=0,in=-120] (S2) to[out=60,in=-60] ($(P2)-(0.3,0)$) to[out=120,in=-90] (Q3) to[out=90,in=90] ($(Q3)+(0,0.2)$);
			\node at (0,-0.6) {\small{\red{$\cc_3$}}};
			\draw[red]  ($(Q1)+(0.4,-0.5)$) [partial ellipse=90:160: 0.8 and 0.6]; 
			\node at ($(Q1)+(0.6,0.1)$) {\small{\red{$\cc_3$}}};
			\draw[blue] ($(Q3)+(0,0.6)$) [partial ellipse=-130:-50: 0.8 and 0.6]; 
			\node at ($(Q3)+(0.8,0.1)$) {\small{\blue{$\cc_2$}}};
		\end{tikzpicture}
		\caption{Configuration \ref{conf:311} yields surface \ref{def:A1A2_c=3}, cf.\ Figure \ref{fig:A1A2_c=3}}
		\label{fig:311_conf}
	\end{minipage}
	\begin{minipage}[b]{.33\textwidth}
		\centering
		\begin{tikzpicture}
			\path[use as bounding box] (-1.5,-1) rectangle (3.2,1.8);
			\draw[name path=c1] (0,0) circle (1);
			\coordinate (P) at (0,2); 
			\coordinate (P1) at (-0.866,0.5);
			\node at ($(P1)+(-0.15,0.2)$) {\small{$p_1$}}; \filldraw (P1) circle (0.06);
			\coordinate (P2) at (0.866,0.5);
			\node at ($(P2)+(0.1,0.3)$) {\small{$p_2$}}; \filldraw (P2) circle (0.06);
			\node at (0.7,1.2) {\small{$\ll_{2}$}};
			\node at (0.9,-0.8) {\small{$\cc_{1}$}};
			\draw[add= -0.2 and 0.5, name path = L2] (P) to (P2);
			\coordinate (P') at (-0.866,-0.5); 
			\node at ($(P')+(0,-0.2)$) {\small{$p'$}}; \filldraw (P') circle (0.06);
			\draw[add=0.1 and 0.6, name path=L2'] (P') to (P2);
			\node at (1.7,1.2) {\small{$\ll_{2}'$}};
			\draw (-1,0.5) -- (3.2,0.5);
			\coordinate (R) at (3,0.5);
			\node at ($(R)+(0.1,0.2)$) {\small{$r$}}; \filldraw (R) circle (0.06);
			\draw[add=0.05 and 0.05,name path=L] (P') to (R);
			\path [name intersections={of=L2 and L}] (intersection-1) coordinate (Q);
			\node at ($(Q)+(0,0.2)$) {\small{$q$}}; \filldraw (Q) circle (0.06);
			\coordinate (S) at (0,0);
			\node at ($(S)+(0,0.2)$) {\small{$s$}}; \filldraw (S) circle (0.06);
			\draw[blue] (P1) to[out=-120,in=-150] (0,0) to[out=30,in=150] (Q) to[out=-30,in=-90] (R) to[out=90,in=0] ($(P2)+(0,1.2)$) to[out=180,in=60] (P1);
			\node at (2,-0.4) {\small{\blue{$\cc_{2}$}}}; 
		\end{tikzpicture}
		\caption{Configuration \ref{conf:A1A2_c=1} yields surface \ref{def:A1A2_c=1}, cf.\ Figure \ref{fig:A1A2_c=1}}
		\label{fig:A1A2_c=1_conf}
	\end{minipage}
\end{figure}

\begin{conf}[Figure \ref{fig:A1A2_c=1_conf}]\label{conf:A1A2_c=1}
	Fix $r\in \ll_{p_{1}p_{2}}\setminus \{p_{1},p_{2}\}$. Write $\ll_{2}\cap \ll_{rp'}=\{q\}$ and let $\cc_{2}$ be the unique conic passing through $q$, $r$ and tangent to $\cc_{1}$ at $p_{1}$ with multiplicity at least $3$. There exist exactly three points $r$ such that $(\cc_{2}\cdot \ll_{2}')_{s}=2$ for some $s\neq p_{2},p'$. For these points, $(\cc_{2}\cdot \cc_{1})_{p_{1}}=3$.
	
	We have $\#\tref{\cP}{def:A1A2_c=1}=3$, and $\tref{\cP}{def:A1A2_c=1}$ is realized over the splitting field of the polynomial $\tref{g}{def:A1A2_c=1}(t)\de 4t^3-8t^2+4t-1\in \Q[t]$. Moreover, for every root $t_0$ of $\tref{g}{def:A1A2_c=1}$, some configuration in $\tref{\cP}{def:A1A2_c=1}$ is realized over $\Q(t_0)$. In particular, $\tref{\cP}{def:A1A2_c=1}$ contains a configuration realized over $\R$, and a pair of complex-conjugate ones.
\end{conf}
\begin{proof}
	 We have $\ll_{p_1p_2}=\{y=0\}$, so $r=[1:0:t]$ for some $t\in \C^{*}$. Then $\ll_{rp'}=\{z=tx+(1-t)y\}$, so $q=[0:1:1-t]$. The pencil of conics passing through $q,r$ and tangent to $\cc_1$ at $p_1$ is generated by $\ll_{1}+\ll_{qr}$ and $\ll_{rp_1}+\ll_{qp_1}$. Recall that $\ll_1=\{z=0\}$, $\ll_{qr}=\ll_{rp'}=\{z=tx+(1-t)y\}$, $\ll_{rp_1}=\ll_{p_1p_2}=\{y=0\}$ and $\ll_{qp_1}=\{z=(1-t)y\}$. Hence the unique member $\cc_2$ of the above pencil satisfying $(\cc_2\cdot \cc_1)_{p_1}\geq 3$ is 
	\begin{equation*}
		\cc_2=\{(1-t)\cdot z \cdot (tx+(1-t)y-z)=t\cdot y\cdot ((1-t)y-z)\}.
	\end{equation*}
	For any $t\in \C^{*}$, we have $(\cc_2\cdot \cc_1)_{p_1}=3$. Now $\cc_2$ is tangent to $\ll_{2}'=\{y=x\}$ if and only if $4t(t-1)^2=1$, i.e.\ $4t^3-8t^2+4t-1=0$. This equation has three different solutions, which correspond to three different conics $\cc_2$. We have $s=[1:1:-\frac{1}{2}(t-1)]$, so $s\neq p_2,p'$, as claimed.
\end{proof}

\begin{conf}[Figure \ref{fig:A1A2_c=2_conf}]\label{conf:A1A2_c=2}
	Write $\ll_{pp'}\cap\cc_{1}=\{p',p''\}$. Fix a line $\ll\neq \ll_{pp'},\ll_{p_{1}p''},\ll_{p_{2}p''}$ passing through $p''$, not tangent to $\cc_{1}$. For $i\in \{1,2\}$ write $\ll\cap\ll_{i}=\{r_{i}\}$ and $\ll_{r_{i}p'}\cap\cc_{1}=\{p',s_{i}\}$. 
	 Assume $s_{1}\neq p_{2}$. 
	Let $\cc_{2}$ be the unique conic passing through $p$, $r_{1}$, $r_{2}$ and tangent to $\cc_{1}$ at $s_{1}$. Then
	\begin{enumerate}
		\item\label{item:32} There are exactly three lines $\ll$ such that $s_{2}\in \cc_{2}$. In this case, $(\cc_1\cdot \cc_2)_{s_1}=2$, $(\cc_1\cdot \cc_2)_{s_2}=1$.
		\item\label{item:41} There are exactly three lines $\ll$ such that $(\cc_{2}\cdot \cc_{1})_{s_{1}}\geq 3$. In this case, $(\cc_{2}\cdot \cc_{1})_{s_{1}}=3$, $s_2\not\in \cc_2$.
	\end{enumerate}
	For $\rst=$\ref{def:A1A2_c=2_32},\ref{def:A1A2_c=2_41} we have  $\#\tst{\cP}=3$, and $\tst{\cP}$ is realized over the splitting field of $\tst{g}$ where $\tref{g}{def:A1A2_c=2_32}(t)= t^3 + 7 t^2 + 15 t + 5$ and $\tref{g}{def:A1A2_c=2_41}(t)= t^3 + 3 t^2 + 3 t + 5$. Each configuration in $\tst{\cP}$ is realized over an extension of $\Q$ by a different root of $\tst{g}$. In particular, $\tst{\cP}$ contains a configuration realized over $\R$, and a pair of complex-conjugate ones.
\end{conf}
\begin{proof}
	 We have $p''=[1:-1:1]$, $\ll_{pp'}=\{x=z\}$, $\ll_{p_{1}p''}=\{y=-z\}$, $\ll_{p_{2}p''}=\{y=-x\}$ and the line tangent to $\cc_{1}$ at $p''$ is given by $\{-2y=x+z\}$. Thus $\ll=\{x+(1+t)y+tz=0\}$ for some $t\in \C\setminus \{0,1,-1\}$. We compute
	\begin{align*}
	& r_{1}=[t+1:-1:0],& & 
	\ll_{r_{1}p'}=\{(t+1)y+x=(t+2)z\}, &  &
	s_{1}=[(t+2)^{2}:-(t+2):1],\\
	& r_{2}=[0:-1:t^{-1}+1], & &
	\ll_{r_{2}p'}=\{(t^{-1}+1)y+z=(t^{-1}+2)x\}, &&
	s_{2}=[1:-(t^{-1}+2):(t^{-1}+2)^2].
	\end{align*}
	Note that $t\neq -2$ because $s_{1}\neq p_{2}$. We have $\ll_{r_1s_1}=\ll_{r_1 p'}=\{(t+1)y+x=(t+2)z\}$ and
	\begin{equation*}
	\ll_{r_2 s_1}=\{(t+2)^2(tz+(1+t)y)+(t^2+2t+2)x=0\}.	
	\end{equation*}
	The pencil of conics passing through $p,r_1,r_2,s_1$ is generated by $\ll_{2}+\ll_{r_1 s_1}$ and $\ll_{1}+\ll_{r_2s_1}$. Substituting a parametrization $[u:v]\mapsto [u^2:uv:v^2]$ of $\cc_1$ to the above equations, we compute that the unique member of this pencil which is tangent to $\cc_1$ at $s_1$ is given by  
	\begin{equation*}
	\cc_{2}=\{(t^2+t+2)x(x+(1+t)y-(2+t)z)=(t+2)(t+3)z((t+2)^2(tz+(1+t)y)+(t^2+2t+2)x)\}.
	\end{equation*}
	Once again, we substitute the parametrization $[u:v]\mapsto [u^2:uv:v^2]$ of $\cc_1$ to the equation of $\cc_2$, and denote the resulting polynomial by $f_{t}\in \C[u,v]$. Since $(\cc_2\cdot \cc_1)_{s_1}\geq 2$, we have $(u+(t+2)v)^2|f_{t}$. Put $g_{t}(u)=f_{t}(u,1)\cdot (u+t+2)^{-2}\in \C[u]$. We have 
	\begin{equation*}
	\begin{split}
	s_{2}\in \cc_{2} &\iff g_{t}(-t^{-1}-2)=0 \iff t^3 + 7 t^2 + 15 t + 5=0,\\
	(\cc_{2}\cdot\cc_{1})_{s_{1}}\geq 3 &\iff g_{t}(t+2)=0 \iff  t^3 + 3 t^2 + 3 t + 5 =0.
	\end{split}
	\end{equation*}
	Each of these polynomials has three different roots, which proves the first statement of \ref{item:32} and \ref{item:41}.	Moreover, the above polynomials have no roots in common, so $(\cc_1\cdot \cc_2)_{s_1}=2$ in case \ref{item:32} and $s_2\not\in \cc_2$ in case \ref{item:41}. Eventually, neither of those polynomials has a common root with the discriminant of $g_{t}$, which implies that $(\cc_{1}\cdot \cc_{2})_{s_2}=2$ in case \ref{item:32} and $(\cc_{1}\cdot \cc_{2})_{s_3}=3$ in case \ref{item:41}, as claimed.
\end{proof}

\begin{figure}[htbp]
	\captionsetup{width=.3\linewidth}
	\centering
	\begin{minipage}[tb]{.65\textwidth}
		\centering
		\begin{subfigure}{.48\textwidth}
			\centering
			\begin{tikzpicture}
				\path[use as bounding box] (-2.6,-2) rectangle (1.8,2.5);
				\draw[name path=c1] (0,0) circle (1);
				\coordinate (P) at (0,2); 
				\node at ($(P)+(0,0.3)$) {\small{$p$}}; \filldraw (P) circle (0.06);
				\coordinate (P1) at (-0.866,0.5);
				\node at ($(P1)+(-0.15,0.2)$) {\small{$p_1$}}; \filldraw (P1) circle (0.06);
				\coordinate (P2) at (0.866,0.5);
				\node at ($(P2)+(0.2,0.2)$) {\small{$p_2$}}; \filldraw (P2) circle (0.06);
				\node at (-0.7,1.2) {\small{$\ll_{1}$}};
				\node at (0.7,1.2) {\small{$\ll_{2}$}};
				\node at (0.65,-0.5) {\small{$\cc_{1}$}};
				\node at (-2,-0.5) {\small{\blue{$\cc_{2}$}}};
				\draw[add= 0.1 and 1.6, name path = L1] (P) to (P1);
				\draw[add= 0.1 and 0.8, name path = L2] (P) to (P2);
				\coordinate (P') at (0,1); \node at ($(P')+(-0.15,0.2)$) {\small{$p'$}}; \filldraw (P') circle (0.06);
				\coordinate (P'') at (0,-1); \node at ($(P'')+(0.2,-0.2)$) {\small{$p''$}}; \filldraw (P'') circle (0.06);
				\draw[add= 0.05 and 0.1, name path = LPP'] (P) to (P'');
				\coordinate (R1) at ($(P)!2.5!(P1)$); 
				\node at ($(R1)+(-0.2,0.2)$) {\small{$r_1$}}; \filldraw (R1) circle (0.06);
				\draw[add = 0.1 and 0.8, name path=L] (R1) to (P'');
				\node at ($(R1)+(1.1,0.2)$) {\small{$\ll$}};
				\path [name intersections={of=L and L2}] (intersection-1) coordinate (R2); 
				\node at ($(R2)+(0.1,0.2)$) {\small{$r_2$}}; \filldraw (R2) circle (0.06);
				\draw[add = 0.1 and 0.22, name path=LR1] (R1) to (P');
				\path [name intersections={of=c1 and LR1}] (intersection-2) coordinate (S1); 
				\node at ($(S1)+(0.3,0)$) {\small{$s_1$}}; \filldraw (S1) circle (0.06);
				\draw[add = 0.12 and 0.3, name path=LR2] (R2) to (P');
				\path [name intersections={of=c1 and LR2}] (intersection-2) coordinate (S2); 
				\node at ($(S2)-(0.2,0.1)$) {\small{$s_2$}}; \filldraw (S2) circle (0.06);
				\coordinate (P1') at ($(P')!0.6!(S1)$); 
				\coordinate (P2') at ($(P')!0.6!(S2)$); 
				\draw[blue] ($(P2')+(-0.1,0)$) to[out=180,in=0] ($(P1')+(0.1,0)$);
				\draw[blue] ($(P1')+(-0.1,0)$) to[out=180,in=105] (S1) to[out=-75,in=90] (-0.5,-0.7) to[out=-90,in=0] (-0.7,-0.9)  to[out=180,in=40] (R1) to[out=-140,in=180] (P) to[out=0,in=-30] (R2) to[out=160,in=-70] (S2) to[out=110,in=0] ($(P2')+(0.1,0)$);
			\end{tikzpicture}
			\caption{yields surface \ref{def:A1A2_c=2_32}, cf.\ Figure \ref{fig:A1A2_c=2_32}}
			\label{fig:A1A2_c=2_32_conf}
		\end{subfigure}
		\begin{subfigure}{.48\textwidth}
			\centering
			\begin{tikzpicture}
				\path[use as bounding box] (-2.6,-2) rectangle (1.8,2.5);
				\draw[name path=c1] (0,0) circle (1);
				\coordinate (P) at (0,2); 
				\node at ($(P)+(0,0.3)$) {\small{$p$}}; \filldraw (P) circle (0.06);
				\coordinate (P1) at (-0.866,0.5);
				\node at ($(P1)+(-0.15,0.2)$) {\small{$p_1$}}; \filldraw (P1) circle (0.06);
				\coordinate (P2) at (0.866,0.5);
				\node at ($(P2)+(0.2,0.2)$) {\small{$p_2$}}; \filldraw (P2) circle (0.06);
				\node at (-0.7,1.2) {\small{$\ll_{1}$}};
				\node at (0.7,1.2) {\small{$\ll_{2}$}};
				\node at (0.65,-0.5) {\small{$\cc_{1}$}};
				\node at (-2,-0.5) {\small{\blue{$\cc_{2}$}}};
				\draw[add= 0.1 and 1.6, name path = L1] (P) to (P1);
				\draw[add= 0.1 and 0.8, name path = L2] (P) to (P2);
				\coordinate (P') at (0,1); \node at ($(P')+(-0.15,0.2)$) {\small{$p'$}}; \filldraw (P') circle (0.06);
				\coordinate (P'') at (0,-1); \node at ($(P'')+(0.2,-0.2)$) {\small{$p''$}}; \filldraw (P'') circle (0.06);
				\draw[add= 0.05 and 0.1, name path = LPP'] (P) to (P'');
				\coordinate (R1) at ($(P)!2.5!(P1)$); 
				\node at ($(R1)+(-0.2,0.2)$) {\small{$r_1$}}; \filldraw (R1) circle (0.06);
				\draw[add = 0.1 and 0.8, name path=L] (R1) to (P'');
				\node at ($(R1)+(1.1,0.2)$) {\small{$\ll$}};
				\path [name intersections={of=L and L2}] (intersection-1) coordinate (R2); 
				\node at ($(R2)+(0.1,0.2)$) {\small{$r_2$}}; \filldraw (R2) circle (0.06);
				\draw[add = 0.1 and 0.22, name path=LR1] (R1) to (P');
				\path [name intersections={of=c1 and LR1}] (intersection-2) coordinate (S1); 
				\node at ($(S1)+(0.3,0)$) {\small{$s_1$}}; \filldraw (S1) circle (0.06);
				\coordinate (P1') at ($(P')!0.6!(S1)$); 
				\draw[blue] ($(P1')+(-0.1,0)$) to[out=180,in=105] (S1) to[out=-75,in=40] (R1) to[out=-140,in=180] (P) to[out=0,in=-30] (R2) to[out=160,in=0] ($(P1')+(-0.1,0)$);
			\end{tikzpicture}
			\caption{yields \ref{def:A1A2_c=2_41}, cf.\ Figure \ref{fig:A1A2_c=2_41}}
			\label{fig:A1A2_c=2_41_conf}
		\end{subfigure}
	\vspace{-.5em}
		\caption{Configuration \ref{conf:A1A2_c=2}}
		\label{fig:A1A2_c=2_conf}
	\end{minipage}
	\begin{minipage}[tb]{.33\textwidth}
		\centering
		\begin{tikzpicture}
			\path[use as bounding box] (-2.5,-1.1) rectangle (2.4,2.6);
			\draw[name path=c1] (0,0) circle (1); \node at (-0.7,0.2) {\small{$\cc_{1}$}};
			\draw[blue, name path=c2] (0,0.5) circle (1.5); \node at (-1.7,0.2) {\small{\blue{$\cc_{2}$}}};
			\draw[name path = L2] (-2,1) -- (2.4,1);
			\node at (2.2,0.8) {\small{$\ll_{2}$}};
			\coordinate (P1) at (0,-1); \node at ($(P1)+(-0.1,0.2)$) {\small{$p_1$}}; \filldraw (P1) circle (0.06);
			\coordinate (P2) at (0,1); \node at ($(P2)+(0.1,0.2)$) {\small{$p_2$}}; \filldraw (P2) circle (0.06);
			\coordinate (R) at (-1.414,1); \node at ($(R)+(-0.1,0.2)$) {\small{$r$}}; \filldraw (R) circle (0.06);
			\coordinate (Q) at (1.414,1); \node at ($(Q)+(0.1,0.3)$) {\small{$q$}}; \filldraw (Q) circle (0.06);
			\draw[add= 0.1 and 0.1, name path = L1'] (Q) to (P1);
			\node at (0.2,-0.3) {\small{$\ll_{1}'$}};
			\path [name intersections={of=L1' and c1}] (intersection-1) coordinate (P');
			\node at ($(P')-(0.3,0.1)$) {\small{$p'$}}; \filldraw (P') circle (0.06);
			\draw[add= 0.4 and 0.4, name path = LRS] (R) to (P');
			\path [name intersections={of=LRS and c2}] (intersection-2) coordinate (S);
			\node at ($(S)+(0.2,0.1)$) {\small{$s$}}; \filldraw (S) circle (0.06);
			\draw[add= 2 and 1, name path = LQS] (Q) to (S);
			\draw[add= 1.2 and 1.4, name path = L2'] (P2) to (P');
			\node at (2.2,-0.3) {\small{$\ll_{2}'$}};
			\path [name intersections={of=L2' and c2}] (intersection-1) coordinate (T);
			\node at ($(T)+(0,0.2)$) {\small{$t$}}; \filldraw (T) circle (0.06);
			\draw[add= 0.8 and 0.1] (T) to (Q);
			\coordinate (R') at ($(R)!-0.4!(P')$); \coordinate (R'') at ($(R')+(-0.4,0.6)$); 
			\draw[dotted] (R') to[out=160,in=-90] (R'');
			\draw (R'') -- ($(R'')+(0,0.6)$);
			\draw ($(R'')+(0.8,0.3)$) [partial ellipse=-190:-30: 0.5 and 0.5];
			\node at ($(R'')+(0.6,-0.3)$) {\small{$\cc_1$}};
			\draw (1.1,2) [partial ellipse=80:-70: 0.5 and 0.5];
			\node at (1.8,2) {\small{$\cc_1$}};
		\end{tikzpicture}
		\caption{Configuration \ref{conf:4} yields towers \ref{def:A1A2_C2C3-node}, \ref{def:F2-5}, \ref{def:F2-5_cn},  cf.\ Figures \ref{fig:A1A2_C2C3-node}, \ref{fig:F2-5}, \ref{fig:F2-5_cn}.}
		\label{fig:4_conf}
	\end{minipage}
\end{figure}

\begin{conf}[Figure \ref{fig:4_conf}]\label{conf:4}
	Write $\ll_{2}\cap\ll_{1}'=\{q\}$. Let $\cc_{2}$ be the unique conic tangent to $\cc_{1}$ with multiplicity $4$ at $p_{1}$ and passing through $q$. Then the following holds.
	\begin{enumerate}
		\item\label{item:4-l_2} The line $\ll_{2}$ meets $\cc_{2}$ in some point $r\neq q$.	
		\item\label{item:4-l_2'} The line $\ll_{2}'$ is not tangent to $\cc_{1}$, $\cc_{2}$. The set  $\tref{\cP}{def:F2-5}$ has two elements, and is realized over $\Q(\sqrt{5})$. 
		\item\label{item:4-l_rp'} The line $\ll_{rp}'$ is not tangent to $\cc_{1}$, $\cc_{2}$.  Write $\ll_{rp'}\cap\cc_{2}=\{r,s\}$. The line $\ll_{qs}$ meets $\cc_1$ in two point. The set  $\tref{\cP}{def:F2-5_cn}$ has two elements, and is realized over $\Q(\sqrt{-15})$. 
		\item\label{item:4-l_qt} Choose $t\in \cc_{2}\cap\ll_{2}'$. The line $\ll_{qt}$ meets $\cc_1$ in two complex-conjugate   points. The set $\tref{\cP}{def:A1A2_C2C3-node}$ consists of four configurations, two of them are realized over $\Q(\sqrt{5},\alpha_{+})$, and two over $\Q(\sqrt{5},\alpha_{-})$, where  $\alpha_{\pm}=(-5\pm 2\sqrt{5})^{1/2}$. All four configurations are conjugate by $\Aut_{\Q}(\Q(\sqrt{5},\alpha_{-},\alpha_{+}))$. 
	\end{enumerate}
\end{conf}
\begin{proof}
	\ref{item:4-l_2}  We have $q=[0:1:1]$, see Figure \ref{fig:coordinates}. The pencil of conics tangent to $\cc_1$ in $p_1$ with multiplicity $4$ is generated by $2\ll_{1}=\{z^2=0\}$ and $\cc_{1}=\{y^2=xz\}$. Its unique member passing through $q=[0:1:1]$ is  $\cc_{2}=\{y^{2}-xz=z^2\}$. Substituting the equation $\ll_{2}=\{x=0\}$ to the equation of $\cc_2$, we get that $\ll_{2}\cap \cc_{2}=\{q,r\}$, where $r=[0:1:-1]\neq q$. 
	
	\ref{item:4-l_2'} Recall that $\ll_{2}'\cap \cc_{1}=\{p_{2},p'\}$, see Figure \ref{fig:coordinates}, so $\ll_{2}'$ is not tangent to $\cc_1$. Substituting the equation $\ll_{2}'=\{x=y\}$ to the equation of $\cc_2$, we get $\ll_{2}'\cap \cc_{2}=\{[2:2:\sqrt{5}-1],[2:2:-\sqrt{5}-1]\}$, so $\ll_{2}'$ is not tangent to $\cc_2$. The resulting configuration $\tref{\pp}{def:F2-5}\de \cc_1+\cc_2+\ll_2+\ll_1+\ll_2'$ is realized over $\Q(\sqrt{5})$. Without a marking, $\tref{\pp}{def:F2-5}$ is unique up to a projective equivalence, and $\Aut(\P^2,\tref{\pp}{def:F2-5})$ is trivial: indeed, any automorphism of $(\P^2,\tref{\pp}{def:F2-5})$ sends a quadruple $(p,p_1,p_2,p')$ of points in general position to itself, hence is the identity. Therefore, the two possible orders of the set $\ll_2'\cap \cc_2$ give two non-equivalent markings of $\tref{\pp}{def:F2-5}$. Thus $\#\tref{\cP}{def:F2-5}=2$, as claimed.
	
	\ref{item:4-l_rp'} Since $r=[0:1:-1]$, $p'=[1:1:1]$, we have  $\ll_{rp'}=\{2x=y+z\}$, so $\ll_{rp'}\cap \cc_{1}= \{p', [1:-2:4]\}$ and $\ell_{rp'}\cap \cc_2=\{r,s\}$, where $s=[5:6:4]$. Now   $\ll_{qs}=\{2x=5(y-z)\}$, so $\ll_{qs}\cap\cc_{1}=\{[\alpha^2:\alpha:1],[\bar{\alpha}^2:\alpha:1]\}$, where $\alpha=\frac{1}{4}(5-\sqrt{-15})$. Like before, we see that the configuration $\tref{\pp}{def:F2-5_cn}\de \cc_1+\cc_2+\ll_{2}+\ll_{1}'+\ll_{rp'}+\ll_{qs}$ is realized over $\Q(\sqrt{-15})$, and admits two non-equivalent markings corresponding to the order of $\ll_{qs}\cap \cc_1$. 
	
	\ref{item:4-l_qt} We have seen in \ref{item:4-l_2'} that two choices of $t$ lead to two unmarked configurations 
	$\tref{\pp}{def:A1A2_C2C3-node}\de \cc_1+\cc_2+\ll_{2}+\ll_{1}'+\ll_{2}'+\ll_{qt}$, which are not projectively equivalent, but conjugate by $\Aut_{\Q}(\Q(\sqrt{5}))$. Take $t=[2:2:\sqrt{5}-1]$. Then $\ll_{qt}=\{(3-\sqrt{5})x=2(y-z)\}$. Substituting this to the equation of $\cc_1$ we see that $\ll_{qt}$ meets $\cc_1$ in two points, which are conjugate by $\Aut_{\Q(\sqrt{5})}(\Q(\alpha_{-}))$. Choosing the order of those two points gives two non-equivalent markings on $\tref{\pp}{def:A1A2_C2C3-node}$. Thus $\#\tref{\cP}{def:A1A2_C2C3-node}=4$, and $\tref{\cP}{def:A1A2_C2C3-node}$ is realized over $\Q(\sqrt{5},\alpha_{-},\alpha_{+})$. Note that since $\alpha_{\pm}\not \in \R$, the two markings are complex-conjugate.
\end{proof}

\begin{conf}[Figure \ref{fig:31_conf}]\label{conf:31}
	There is a unique conic $\cc_{2}\neq \cc_{1}$ meeting $\cc_{1}$ at $p_{2}$, $p'$ with multiplicity $3$, $1$ and tangent to $\ll_{1}$. Write $\cc_{2}\cap\ll_{1}=\{q\}$. We have the following configurations.
	\begin{enumerate}
		\item \label{item:C**_3} We have $\#\tref{\cP}{def:C**_3}=1$. The configuration $\tref{\cP}{def:C**_3}$ is realized over $\Q$ and admits an involution.
		\item \label{item:A1A2_C2C3-cusp-41} Write $\cc_{1}\cap\ll_{qp'}=\{p',r\}$. There is a unique conic $\cc_{3}$ passing through $p_2,q,r$, which meets $\cc_1$ and $\cc_2$ at three points each, is tangent to $\cc_2$ at $q$ and to $\cc_1$ at $r$. We have $\#\tref{\cP}{def:A1A2_C2C3-cusp-41}=1$, and $\tref{\cP}{def:A1A2_C2C3-cusp-41}$ is realized over $\Q$.
		\item \label{item:A1A2_C2C3-cusp-32} Like in \ref{item:A1A2_C2C3-cusp-41}, write $\cc_{1}\cap\ll_{qp'}=\{p',r\}$. There is a unique conic $\cc_{3}$ passing through $p_2,q,r$, which meets $\cc_1$ and $\cc_2$ at three points each, is tangent to $\cc_2$ at $q$ and to $\cc_1$ at some point  $r'\neq r,p_2$. Thus $\#\tref{\cP}{def:A1A2_C2C3-cusp-32}=1$, and $\tref{\cP}{def:A1A2_C2C3-cusp-32}$ is realized over $\Q$.
		\item\label{item:A1A2_3-cusp} There are exactly two conics $\cc_{3}$ such that $(\cc_3\cdot \cc_2)_{q}=3$, $(\cc_3\cdot \cc_2)_{p_2}=1$ and $(\cc_3\cdot \cc_1)_{r}=2$ for some $r\not\in\{ p_{1},p'\}$. 
		We have $\#\tref{\cP}{def:A1A2_3-cusp}=2$, and $\tref{\cP}{def:A1A2_3-cusp}$ is realized over $\Q(\sqrt{-7})$. 
		\item\label{item:A1A2_2n1c} There is a unique conic $\cc_{3}$ such that $(\cc_{3}\cdot \cc_{1})_{p'}=3$ and $(\cc_{3}\cdot\cc_{2})_{q}=2$. Write $\cc_{2}\cap\cc_{3}=\{p',q,r\}$. Then $r\not\in \cc_{1}$, and the line $\ll_{qr}$ is not tangent to $\cc_1$. 
		We have $\#\tref{\cP}{def:A1A2_2n1c}=2$, and $\tref{\cP}{def:A1A2_2n1c}$ is realized over $\Q(\sqrt{-7})$.
		\item\label{item:A1A2_21} Let $\ll_{q}\neq \ll_{1}$ be the other line tangent to $\cc_{1}$ passing through $q$. Write $\cc_{2}\cap\ll_{q}=\{q,r\}$. Then there is a unique conic $\cc_{3}$ passing through $r$, $p_{2}$, such that $(\cc_3\cdot\cc_2)_{q}=2$ and $(\cc_3\cdot \cc_1)_{s}=2$ for some  $s\not\in\{p',p_{1},p_{2}\}$. We have $\#\tref{\cP}{def:A1A2_21}=\#\tref{\cP}{def:F2n2-cuspidal}=1$, and $\tref{\cP}{def:A1A2_21}$, $\tref{\cP}{def:F2n2-cuspidal}$ are realized over $\Q$.
	\end{enumerate}
\end{conf}
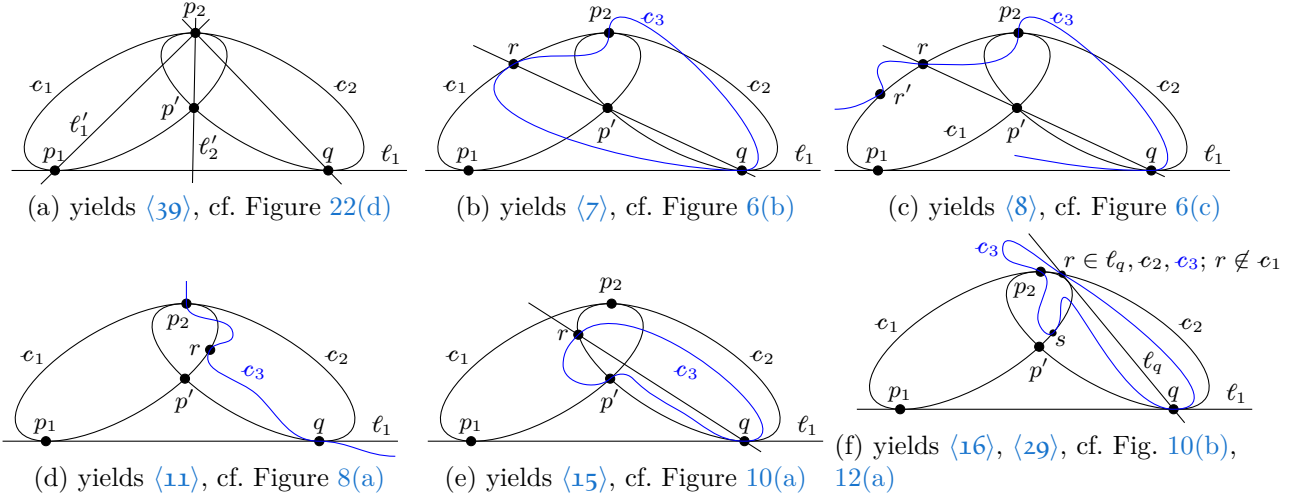
\begin{figure}[htbp]
	\begin{subfigure}[b]{.32\linewidth}
		\centering
		\begin{tikzpicture}
			\path[use as bounding box] (-2.4,-1) rectangle (2.8,1.4);
			\draw[rotate around={-60:(-0.9,0)}, name path=c1] (-0.9,0) ellipse (0.6 and 1.5);
			\node at (-2,0.2) {\small{$\cc_1$}};
			\draw[rotate around={60:(0.9,0)}, name path=c2] (0.9,0) ellipse (0.6 and 1.5);
			\node at (2,0.2) {\small{$\cc_2$}};
			\draw[name path = L1] (-2.4,-0.912) to (2.8,-0.912);
			\node at (2.6,-0.7) {\small{$\ll_1$}};
			\path [name intersections={of=L1 and c1}] (intersection-1) coordinate (P1);
			\node at ($(P1)+(0,0.2)$) {\small{$p_1$}}; \filldraw (P1) circle (0.06);
			\path [name intersections={of=L1 and c2}] (intersection-1) coordinate (Q);
			\node at ($(Q)+(0,0.2)$) {\small{$q$}}; \filldraw (Q) circle (0.06);
			\path [name intersections={of=c1 and c2}] 
			(intersection-1) coordinate (P') 
			(intersection-2) coordinate (P2') 
			(intersection-3) coordinate (P2'');
			\coordinate (P2) at ($(P2')!0.5!(P2'')$);
			\node at ($(P2)+(0,0.3)$) {\small{$p_2$}}; \filldraw (P2) circle (0.06);
			\node at ($(P')-(0.3,0)$) {\small{$p'$}}; \filldraw (P') circle (0.06);
			\draw[add= 0.2 and 1, name path = L2'] (P2) to (P');
			\node at (0.2,-0.6) {\small{$\ll_{2}'$}};
			\draw[add= 0.1 and 0.1, name path = L1'] (P1) to (P2);
			\node at (-1.5,-0.3) {\small{$\ll_{1}'$}};
			\draw[add= 0.1 and 0.1] (Q) to (P2);
		\end{tikzpicture}
		\caption{yields \ref{def:C**_3}, cf.\ Figure \ref{fig:C**_3}}
		\label{fig:C**_3_conf}	
	\end{subfigure}
	\begin{subfigure}[b]{.31\linewidth}
		\begin{tikzpicture}
			\path[use as bounding box] (-2.4,-1) rectangle (2.8,1.3);
			\draw[rotate around={-60:(-0.9,0)}, name path=c1] (-0.9,0) ellipse (0.6 and 1.5);
			\node at (-2,0.2) {\small{$\cc_1$}};
			\draw[rotate around={60:(0.9,0)}, name path=c2] (0.9,0) ellipse (0.6 and 1.5);
			\node at (2,0.2) {\small{$\cc_2$}};
			\draw[name path = L1] (-2.4,-0.912) to (2.8,-0.912);
			\node at (2.6,-0.7) {\small{$\ll_1$}};
			\path [name intersections={of=L1 and c1}] (intersection-1) coordinate (P1);
			\node at ($(P1)+(0,0.2)$) {\small{$p_1$}}; \filldraw (P1) circle (0.06);
			\path [name intersections={of=L1 and c2}] (intersection-1) coordinate (Q);
			\node at ($(Q)+(0,0.2)$) {\small{$q$}}; \filldraw (Q) circle (0.06);
			\path [name intersections={of=c1 and c2}] 
			(intersection-1) coordinate (P') 
			(intersection-2) coordinate (P2') 
			(intersection-3) coordinate (P2'');
			\coordinate (P2) at ($(P2')!0.5!(P2'')$);
			\node at ($(P2)+(-0.1,0.25)$) {\small{$p_2$}}; \filldraw (P2) circle (0.06);
			\node at ($(P')-(0,0.3)$) {\small{$p'$}}; \filldraw (P') circle (0.06);
			\draw[add=0.1 and 1, name path=L] (Q) to (P');
			\path [name intersections={of=c1 and L}] (intersection-2) coordinate (R); 
			\node at ($(R)+(0,0.2)$) {\small{$r$}}; \filldraw (R) circle (0.06);
			\draw[blue] (Q) to[out=0,in=90] (P2) to[out=-90,in=30] (R) to[out=-150,in=180] (Q); 
			\node at (0.6,1.1) {\blue{\small{$\cc_3$}}}; 
		\end{tikzpicture}
		\caption{yields \ref{def:A1A2_C2C3-cusp-41}, cf.\ Figure \ref{fig:A1A2_C2C3-cusp-41}}
		\label{fig:A1A2_C2C3-cusp-41_conf}
	\end{subfigure}
	\begin{subfigure}[b]{.34\linewidth}
		\begin{tikzpicture}
			\path[use as bounding box] (-2.4,-1) rectangle (2.8,1.3);
			\draw[rotate around={-60:(-0.9,0)}, name path=c1] (-0.9,0) ellipse (0.6 and 1.5);
			\node at (-0.8,-0.4) {\small{$\cc_1$}};
			\draw[rotate around={60:(0.9,0)}, name path=c2] (0.9,0) ellipse (0.6 and 1.5);
			\node at (2,0.2) {\small{$\cc_2$}};
			\draw[name path = L1] (-2.4,-0.912) to (2.8,-0.912);
			\node at (2.6,-0.7) {\small{$\ll_1$}};
			\path [name intersections={of=L1 and c1}] (intersection-1) coordinate (P1);
			\node at ($(P1)+(0,0.2)$) {\small{$p_1$}}; \filldraw (P1) circle (0.06);
			\path [name intersections={of=L1 and c2}] (intersection-1) coordinate (Q);
			\node at ($(Q)+(0,0.2)$) {\small{$q$}}; \filldraw (Q) circle (0.06);
			\path [name intersections={of=c1 and c2}] 
			(intersection-1) coordinate (P') 
			(intersection-2) coordinate (P2') 
			(intersection-3) coordinate (P2'');
			\coordinate (P2) at ($(P2')!0.5!(P2'')$);
			\node at ($(P2)+(-0.1,0.25)$) {\small{$p_2$}}; \filldraw (P2) circle (0.06);
			\node at ($(P')-(0,0.3)$) {\small{$p'$}}; \filldraw (P') circle (0.06);
			\draw[add=0.1 and 1, name path=L] (Q) to (P');
			\path [name intersections={of=c1 and L}] (intersection-2) coordinate (R); 
			\node at ($(R)+(0,0.2)$) {\small{$r$}}; \filldraw (R) circle (0.06);
			\coordinate (R') at (-1.8,0.1);
			\node at ($(R')+(0.3,0)$) {\small{$r'$}}; \filldraw (R') circle (0.06); 
			\draw [blue] ($(Q)+(-1.8,0.2)$) to[out=-10,in=180] (Q) to[out=0,in=90] (P2) to[out=-90,in=0] (R) to[out=180,in=90] ($(R)-(0.6,0.1)$) to[out=-90,in=30] (R') to[out=-150,in=0] ($(R')-(0.6,0.2)$);
			\node at (0.6,1.1) {\blue{\small{$\cc_3$}}}; 
		\end{tikzpicture}
		\caption{yields \ref{def:A1A2_C2C3-cusp-32}, cf.\ Figure \ref{fig:A1A2_C2C3-cusp-32}}
		\label{fig:A1A2_C2C3-cusp-32_conf}
	\end{subfigure}

	\begin{subfigure}[b]{.32\linewidth}
		\begin{tikzpicture}
			\path[use as bounding box] (-2.4,-1) rectangle (2.8,1.3);
			\draw[rotate around={-60:(-0.9,0)}, name path=c1] (-0.9,0) ellipse (0.6 and 1.5);
			\node at (-2,0.2) {\small{$\cc_1$}};
			\draw[rotate around={60:(0.9,0)}, name path=c2] (0.9,0) ellipse (0.6 and 1.5);
			\node at (2,0.2) {\small{$\cc_2$}};
			\draw[name path = L1] (-2.4,-0.912) to (2.8,-0.912);
			\node at (2.6,-0.7) {\small{$\ll_1$}};
			\path [name intersections={of=L1 and c1}] (intersection-1) coordinate (P1);
			\node at ($(P1)+(0,0.2)$) {\small{$p_1$}}; \filldraw (P1) circle (0.06);
			\path [name intersections={of=L1 and c2}] (intersection-1) coordinate (Q);
			\node at ($(Q)+(0,0.2)$) {\small{$q$}}; \filldraw (Q) circle (0.06);
			\path [name intersections={of=c1 and c2}] 
			(intersection-1) coordinate (P') 
			(intersection-2) coordinate (P2') 
			(intersection-3) coordinate (P2'');
			\coordinate (P2) at ($(P2')!0.5!(P2'')$);
			\node at ($(P2)-(0.1,0.25)$) {\small{$p_2$}}; \filldraw (P2) circle (0.06);
			\node at ($(P')-(0,0.3)$) {\small{$p'$}}; \filldraw (P') circle (0.06);
			\coordinate (R) at (0.33,0.3);
			\node at ($(R)-(0.2,0)$) {\small{$r$}}; \filldraw (R) circle (0.06);
			\draw[blue] ($(Q)+(1,-0.2)$) to[out=180,in=0] (Q) to[out=180,in=-45] ($(Q)+(-0.8,0.6)$) to[out=135,in=-120] (R) to[out=60,in=-90] ($(R)+(0.3,0.2)$) to[out=90,in=-90] (P2) to[out=90,in=-90] ($(P2)+(0,0.3)$);
			\node at (0.9,0) {\small{\blue{$\cc_3$}}};
		\end{tikzpicture}
		\caption{yields \ref{def:A1A2_3-cusp}, cf.\ Figure \ref{fig:A1A2_3-cusp}}
		\label{fig:A1A2_3-cusp_conf}
	\end{subfigure}
	\begin{subfigure}[b]{.31\linewidth}
	\centering
		\begin{tikzpicture}
			\path[use as bounding box] (-2.4,-1) rectangle (2.8,1.3);
			\draw[rotate around={-60:(-0.9,0)}, name path=c1] (-0.9,0) ellipse (0.6 and 1.5);
			\node at (-2,0.2) {\small{$\cc_1$}};
			\draw[rotate around={60:(0.9,0)}, name path=c2] (0.9,0) ellipse (0.6 and 1.5);
			\node at (2,0.2) {\small{$\cc_2$}};
			\draw[name path = L1] (-2.4,-0.912) to (2.8,-0.912);
			\node at (2.6,-0.7) {\small{$\ll_1$}};
			\path [name intersections={of=L1 and c1}] (intersection-1) coordinate (P1);
			\node at ($(P1)+(0,0.2)$) {\small{$p_1$}}; \filldraw (P1) circle (0.06);
			\path [name intersections={of=L1 and c2}] (intersection-1) coordinate (Q);
			\node at ($(Q)+(0,0.2)$) {\small{$q$}}; \filldraw (Q) circle (0.06);
			\path [name intersections={of=c1 and c2}] 
			(intersection-1) coordinate (P') 
			(intersection-2) coordinate (P2') 
			(intersection-3) coordinate (P2'');
			\coordinate (P2) at ($(P2')!0.5!(P2'')$);
			\node at ($(P2)+(0,0.25)$) {\small{$p_2$}}; \filldraw (P2) circle (0.06);
			\node at ($(P')-(0,0.3)$) {\small{$p'$}}; \filldraw (P') circle (0.06);
			\coordinate (R) at (-0.42,0.5);
			\node at ($(R)-(0.2,0)$) {\small{$r$}}; \filldraw (R) circle (0.06);
			\draw[add = 0.3 and 0.1] (R) to (Q);
			\coordinate (P'') at ($(P')+(0.3,0)$); 
			\coordinate (Q'') at ($(Q)+(-0.85,0.4)$); 
			\draw[blue] (Q) to[out=180, in=-30] (Q'') to[out=150,in=-45] (P'') to[out=135,in=42] (P') to[out=-138,in=-90] ($(R)-(0.2,0.4)$) to[out=90,in=-135] (R) to[out=45,in=0] (Q);
			\node at (1,0) {\small{\blue{$\cc_3$}}};
		\end{tikzpicture}
		\caption{yields \ref{def:A1A2_2n1c}, cf.\ Figure \ref{fig:A1A2_2n1c}}
		\label{fig:A1A2_2n1c_conf}
	\end{subfigure}
	\begin{subfigure}[b]{.34\linewidth}
	\centering
		\begin{tikzpicture}
			\path[use as bounding box] (-2.4,-1) rectangle (2.8,1.5);
			\draw[rotate around={-60:(-0.9,0)}, name path=c1] (-0.9,0) ellipse (0.6 and 1.5);
			\node at (-2,0.2) {\small{$\cc_1$}};
			\draw[rotate around={60:(0.9,0)}, name path=c2] (0.9,0) ellipse (0.6 and 1.5);
			\node at (2,0.2) {\small{$\cc_2$}};
			\draw[name path = L1] (-2.4,-0.912) to (2.8,-0.912);
			\node at (2.6,-0.7) {\small{$\ll_1$}};
			\path [name intersections={of=L1 and c1}] (intersection-1) coordinate (P1);
			\node at ($(P1)+(0,0.2)$) {\small{$p_1$}}; \filldraw (P1) circle (0.06);
			\path [name intersections={of=L1 and c2}] (intersection-1) coordinate (Q);
			\node at ($(Q)+(0,0.2)$) {\small{$q$}}; \filldraw (Q) circle (0.06);
			\path [name intersections={of=c1 and c2}] 
			(intersection-1) coordinate (P') 
			(intersection-2) coordinate (P2') 
			(intersection-3) coordinate (P2'');
			\coordinate (P2) at ($(P2')!0.5!(P2'')$);
			\node at ($(P2)+(-0.2,-0.25)$) {\small{$p_2$}}; \filldraw (P2) circle (0.06);
			\node at ($(P')-(0,0.3)$) {\small{$p'$}}; \filldraw (P') circle (0.06);
			\coordinate (R) at (0.3,0.88);
			\node[above right] at ($(R)-(0.1,0.1)$) {\small{$r\in \ll_q,\cc_2,\blue{\cc_3}$; $r\not\in \cc_1$}}; \filldraw (R) circle (0.04);
			\draw[add = 0.3 and 0.1] (R) to (Q);
			\coordinate (S) at (0.18,0.1);
			\node at ($(S)+(0.1,-0.1)$) {\small{$s$}}; \filldraw (S) circle (0.04);
			\draw[blue] (Q) to[out=180,in=90] ($(S)+(0.1,0.3)$) to[out=-90,in=45] (S) to[out=-135,in=-90] ($(S)+(-0.2,0.2)$) to[out=90,in=-30] (P2) to[out=150,in=-90] ($(P2)+(-0.5,0.3)$) to[out=90,in=150] (R) to[out=-30,in=0] (Q);
			\node at (1.5,-0.3) {\small{$\ll_{q}$}};
			\node at ($(P2)+(-0.7,0.3)$) {\small{\blue{$\cc_3$}}};
		\end{tikzpicture}
		\caption{yields \ref{def:A1A2_21}, \ref{def:F2n2-cuspidal}, cf.\ Fig.\ \ref{fig:A1A2_21}, \ref{fig:F2n2-cuspidal}}
		\label{fig:A1A2_21_conf}	
	\end{subfigure}
	\caption{Configuration \ref{conf:31}}
	\label{fig:31_conf}
\end{figure}
\begin{proof}
	The pencil of conics passing through $p'$ and tangent to $\cc_1$ at $p_2$ with multiplicity $3$ is generated by $\ll_{2}+\ll_{2}'=\{x(x-y)\}$ and $\cc_{1}=\{y^2=xz\}$, see Figure \ref{fig:coordinates}. It has two members tangent to $\ll_{1}=\{z=0\}$, namely $\cc_1$ and $\cc_{2}=\{4x(x-y)=xz-y^2\}=\{(2x-y)^{2}=xz\}$. Now, $q=[1:2:0]$.
	
	\ref{item:C**_3} We have $\ll_{qp_2}=\{2x=y\}$. The required involution is given by $(p',q,p_1,p_2,q)\mapsto (p',p_1,q_,p_2)$.
	
	\ref{item:A1A2_C2C3-cusp-41}, \ref{item:A1A2_C2C3-cusp-32} Since $p'=[1:1:1]$, we have $\ll_{qp'}=\{2x=y+z\}$, so $r=[1:-2:4]$. The pencil of conics  tangent to $\ll_{1}$ at $q$ and passing through $p_{2}$, $r$ is generated by $\ll_{1}+\ll_{rp_{2}}$ and $\ll_{qr}+\ll_{qp_2}$, so it is given by $\{\lambda z (2x+y)=\mu(2x-y-z)(2x-y)\}_{[\lambda:\mu]\in \P^1}$. Substituting a parametrization $u\mapsto [1:u:u^2]$ of $\cc_1\setminus \{p_2\}$ to the above equation, we conclude that this pencil has exactly two members tangent to $\cc_1$ away from $p_2=[0:0:1]$. These members correspond to $[\lambda:\mu]=[3:1]$ and $[1:-8]$, and are tangent to $\cc_1$ and $\cc_2$ and at $r'\de [9:12:16]\neq r$, respectively. Both of these members meet $\cc_1$ and $\cc_2$ in three points each (including $p_2$), as required. Indeed, the above pencil induces a morphism $\cc_2\to \P^1$ of degree $1$; and $\cc_1\to \P^1$ of degree $2$, ramified at $r$, $r'$. By Hurwitz formula, neither of these morphisms can have further ramification points.
	
	\ref{item:A1A2_3-cusp} The pencil of conics meeting $\cc_{2}$ at $q$, $p_{2}$ with multiplicity $3$, $1$, respectively, is generated by $\ll_{1}+\ll_{qp_2}$ and $\cc_2$, so it is given by 
	is given by $\{\lambda z(2x-y)=\mu((2x-y)^{2}-xz)\}_{[\lambda:\mu]\in \P^{1}}$. Apart from the above generators, this pencil has two members tangent to $\cc_1$. They correspond to $[\lambda:\mu]=[13+\sqrt{-7}:-16]$ and $[13-\sqrt{-7}:-16]$. They tangent to $\cc_1$ with multiplicity two. Indeed, the above pencil induces a morphism $\cc_{1}\to \P^{1}$ of degree $3$, ramified at $p_1$, $p_2$ and the above tangency points, so no further ramification is possible by Hurwitz formula.
	
	\ref{item:A1A2_2n1c} The pencil of conics tangent to $\cc_1$ at $p'$ and to $\cc_2$ at $q$ is generated by $2\ll_{qp'}=\{(2x-y-z)^2=0\}$ and $\ll_{1}+\ll_{p'}=\{z(2y-x-z)=0\}$, see Figure \ref{fig:coordinates}. Its unique member $\cc_3$ satisfying $(\cc_{3}\cdot \cc_{1})_{p'}\geq 3$ is given by $\{9z(2y-x-z)+(2x-y-z)^2=0\}$. The induced map $\cc_1\to \P^1$ of degree $2$ is ramified at the common points of $\cc_1$ with $\ll_{1}$ and $\ll_{qp'}$, so by Hurwitz formula, it is not ramified at $p'$. This means that $(\cc_3\cdot \cc_1)_{p_1}=3$. We compute that $\cc_3$ meets $\cc_{2}\setminus \{p',q\}$ at $r\de [4:22:49]\neq p_2$, so $r\not\in \cc_1$. The line $\ll_{qr}=\{7y=14x+2z\}$ meets $\cc_1=\{y^2=xz\}$ at two points $[1:\alpha:\alpha^2]$ and $[1:\bar{\alpha}:\bar{\alpha}^2]$, where $\alpha=\frac{1}{4}(7-3\sqrt{-7})$. The two possible orderings of those points give two non-equivalent marked configurations in $\tref{\cP}{def:A1A2_2n1c}$, as claimed.
	
	\ref{item:A1A2_21} The pencil of conics tangent to $\ll_{1}$ at $q$ and passing through $p_{2}$, $r$ gives a morphism $\cc_{1}\to \P^{1}$ of degree three, ramified at the points corresponding to $\cc_{2}$ and to the degenerate members $\ll_{q}+\ll_{qp_2}$, $\ll_{1}+\ll_{rp_2}$. By Hurwitz formula, there is exactly one ramification point more, corresponding to the required $\cc_{3}$.
\end{proof}

\begin{conf}[Figure \ref{fig:P2n1-nodal_conf}]\label{conf:P2n1-nodal}
	There is a unique conic $\cc_{2}$ such that $(\cc_1\cdot \cc_2)_{p'}=4$ and $p\in \cc_2$. Write $\ll_{1}\cap \cc_{2}=\{p,q_{1}\}$, $\ll_{2}\cap \cc_{2}=\{p,q_{2}\}$. The following holds.
	\begin{enumerate}
		\item \label{item:P2n1-nodal} The line $\ll_{q_1 q_2}$ is not tangent to $\cc_1$. We have $\#\tref{\cP}{def:P2n1-nodal}=1$, and $\tref{\cP}{def:P2n1-nodal}$ is realized over $\Q(\sqrt{3})$. 		
		\item \label{item:A1A2_q-nnc} Write  $\ll_{q_{2}p_{1}}\cap \cc_{2}=\{q_{2},r\}$. The line $\ll_{pr}$ is not tangent to $\cc_1$. We have  $\#\tref{\cP}{def:A1A2_q-nnc}=2$ and $\tref{\cP}{def:A1A2_q-nnc}$ is realized over $\Q(\omega)$, where $\omega$ is a primitive third root of unity.
	\end{enumerate}
\end{conf}
\begin{proof}
	 The pencil of conics tangent to $\cc_1$ at $p'$ with multiplicity $4$ is generated by $\cc_1$ and $2\ll_{p'}$, see Figure \ref{fig:coordinates}. Its unique member passing through $p$ is $\cc_2=\{4(y^2-xz)=(2y-x-z)^2\}$. Hence $q_1=[4:1:0]$, $q_2=[0:1:4]$. 
	 
	 \ref{item:P2n1-nodal} The line $\ll_{q_1 q_2}=\{4y=x+z\}$ meets $\cc_1$ in $[\alpha^{-1}:1:\alpha]$ and $[\alpha:1:\alpha^{-1}]$, where $\alpha=2\pm \sqrt{3}$. The two possible markings on $\cc_1+\cc_2+\ll_1+\ll_2+\ll_{q_1 q_2}$, corresponding to the orders of $\ll_{q_1 q_2}\cap \cc_1$, are equivalent by the action of $\tau\colon [x:y:z]\mapsto [z:y:x]$. Thus $\#\tref{\cP}{def:P2n1-nodal}=1$, as claimed.
	 
	 \ref{item:A1A2_q-nnc} We have $\ll_{q_2 p_1}=\{4y=z\}$, so $r=[-12:1:4]$. Hence $\ll_{pr}=\{x+3z=0\}$ meets $\cc_1$ in $[1:\sqrt{-3}:-3]$ and $[1:-\sqrt{-3}:-3]$. Ordering those two points gives two non-equivalent markings on $\cc_1+\ll_2+\ll_1+\ll_2+\ll_{q_2 p_1}+\ll_{pr}$. Hence $\#\tref{\cP}{def:A1A2_q-nnc}=2$, and $\tref{\cP}{def:A1A2_q-nnc}$ is realized over $\Q(\sqrt{-3})=\Q(\omega)$, as claimed.
\end{proof}

\begin{conf}[Figure \ref{fig:A1A2_q-nc_conf}]\label{conf:A1A2_q-nc}
	There is a unique conic $\cc_{2}$ tangent to $\ll_{2}'$ at some point $r$, such that $(\cc_{2}\cdot\cc_{1})_{p_{1}}=4$.  
	\begin{enumerate}
		\item\label{item:F2n2-nodal}  The line $\ll_2$ is not tangent to $\cc_2$. We have $\#\tref{\cP}{def:F2n2-nodal}=2$, and $\tref{\cP}{def:F2n2-nodal}$ is realized over $\Q(\imath)$.
		\item\label{item:A1A2_q-nc} Choose $q\in \ll_2\cap \cc_2$. Then $\ll_{qr}$ is not tangent to $\cc_1$. We have $\tref{\cP}{def:A1A2_q-nc}=4$, and $\tref{\cP}{def:A1A2_q-nc}$ is realized over $\Q(\imath,\beta_{+},\beta_{-})$, where $\beta_{\pm}=\sqrt{1\pm 2\imath}$. In fact, two configurations in $\tref{\cP}{def:A1A2_q-nc}$ are realized over $\Q(\imath,\beta_{+})$, and the other two are realized over $\Q(\imath,\beta_{-})$. These two pairs are complex-conjugate.
	\end{enumerate}
\end{conf}
\begin{proof}
	  The pencil of conics tangent to $\cc_1$ at $p_1$ with multiplicity $4$ is generated by $2\ll_{1}$ and $\cc_1$. It has two members tangent to $\ll_{1}'=\{x=y\}$, namely $2\ll_{1}$ and $\cc_2\de \{z^2=4(xz-y^2)\}$. 
	  
	  \ref{item:F2n2-nodal} The line $\ll_{2}=\{x=0\}$ meets $\cc_2$ at $[0:1:\pm 2\imath]$. Ordering $\ll_2\cap \cc_2$ gives two marked configurations in $\tref{\cP}{def:F2n2-nodal}$. 
	  
	  \ref{item:A1A2_q-nc} Take $q=[0:1:2\imath]$. Now $\ll_{qr}=\{(2\imath-2)x-2\imath y+z=0\}$ meets $\cc_1$ at $[1:\beta:\beta^2]$, where $\beta=\imath\pm\beta_{-}$. Ordering those two points gives -- for each of the two choices of $q$ -- two marked configurations in $\tref{\cP}{def:A1A2_q-nc}$. 
\end{proof}

\begin{figure}[htbp]
	\captionsetup{width=.3\linewidth}
	\begin{minipage}[b]{.33\textwidth}
		\centering
		\begin{tikzpicture}
			\path[use as bounding box] (-2.7,-1) rectangle (2,2.1);	
			\draw[name path=c1] (0,0) circle (1);
			\coordinate (P) at (0,2); 
			\node at ($(P)-(0,0.3)$) {\small{$p$}}; \filldraw (P) circle (0.06);
			\coordinate (P') at (0,-1);
			\node at ($(P')+(0,0.25)$) {\small{$p'$}}; \filldraw (P') circle (0.06);
			\coordinate (P1) at (-0.866,0.5);
			\node at ($(P1)+(0.15,-0.25)$) {\small{$p_1$}}; \filldraw (P1) circle (0.06);
			\coordinate (P2) at (0.866,0.5);
			\node at ($(P2)+(-0.2,-0.2)$) {\small{$p_2$}}; \filldraw (P2) circle (0.06);
			\node at (-1.85,-0.8) {\small{$\ll_{1}$}};
			\node at (1.85,-0.8) {\small{$\ll_{2}$}};
			\node at (0,0.8) {\small{$\cc_{1}$}};
			\node at (1.7,0.6) {\small{\blue{$\cc_{2}$}}};
			\draw[add= 0.1 and 1, name path = L1] (P) to (P1);
			\draw[add= 0.1 and 1, name path = L2] (P) to (P2);
			\draw[blue, name path=c2] (0,0.5) circle (1.5);
			\path [name intersections={of=L1 and c2}] (intersection-2) coordinate (Q1);
			\node at ($(Q1)+(0,-0.35)$) {\small{$q_1$}}; \filldraw (Q1) circle (0.06);
			\path [name intersections={of=L2 and c2}] (intersection-2) coordinate (Q2);
			\node at ($(Q2)+(0,-0.35)$) {\small{$q_2$}}; \filldraw (Q2) circle (0.06);
			\draw[add= 0.55 and 0.1, name path = L] (Q1) to (Q2);
			\draw[add= 0.4 and 0.1, name path = L'] (P1) to (Q2);
			\path [name intersections={of=L' and c2}] (intersection-1) coordinate (R);
			\node at ($(R)+(-0.2,0.2)$) {\small{$r$}}; \filldraw (R) circle (0.06);
			\draw[add= 0.1 and 0.8, name path = L'] (P) to (R);
			\draw (-2.3,0.3) [partial ellipse=30:-120: 0.4 and 0.4];
			\node at (-2.1,0.5) {\small{$\cc_1$}};
		\end{tikzpicture}
		\caption{Configuration  \ref{conf:P2n1-nodal} yields \ref{def:A1A2_q-nnc}, \ref{def:P2n1-nodal} cf.\ Figures \ref{fig:A1A2_q-nnc}, \ref{fig:P2n1-nodal}}
		\label{fig:P2n1-nodal_conf}
	\end{minipage}
	\begin{minipage}[b]{.32\textwidth}
		\centering
		\begin{tikzpicture}
			\path[use as bounding box] (-2.3,-0.9) rectangle (2.1,2.8);
			\draw[blue, name path=c2] (0,0) circle (1); \node at (-0.7,0.2) {\small{\blue{$\cc_{2}$}}};
			\draw[name path=c1] (0,0.5) circle (1.5); \node at (1.7,0.2) {\small{$\cc_{1}$}};
			\draw[name path = L2'] (-2,1) -- (2,1);
			\coordinate (P1) at (0,-1); \node at ($(P1)+(0,0.2)$) {\small{$p_1$}}; \filldraw (P1) circle (0.06);
			\coordinate (R) at (0,1); \node at ($(R)+(0,0.2)$) {\small{$r$}}; \filldraw (R) circle (0.06);
			\coordinate (P2) at (-1.414,1);
			\coordinate (P') at (1.414,1);
			\node at ($(P2)+(-0.1,0.2)$) {\small{$p_2$}}; \filldraw (P2) circle (0.06);
			\node at ($(P')+(0.2,0.2)$) {\small{$p'$}}; \filldraw (P') circle (0.06);
			\coordinate (Q) at (-1.06,2); \node at ($(Q)+(-0.05,0.25)$) {\small{$q$}}; \filldraw (Q) circle (0.06);
			\draw[add= 1.8 and 0.9, name path = L2] (P2) to (Q);
			\node at (-2.1,-0.4) {\small{$\ll_{2}$}};
			\draw[add= 0.1 and 0.8] (Q) to (R);
			\draw[add= -2 and 1.8] (Q) to (R);
			\node at (1.8,-0.4) {\small{$\ll_{2}'$}};
			\draw[blue] (-1.2,2.47) [partial ellipse=50:-90: 0.5 and 0.5];
			\node at ($(Q)+(0.6,0.5)$) {\small{\blue{$\cc_{2}$}}};
		\end{tikzpicture}
		\caption{Configuration \ref{conf:A1A2_q-nc} yields \ref{def:A1A2_q-nc}, \ref{def:F2n2-nodal} cf.\ Figures \ref{fig:A1A2_q-nc}, \ref{fig:F2n2-nodal}}
		\label{fig:A1A2_q-nc_conf}
	\end{minipage}
	\begin{minipage}[b]{.33\textwidth}
		\centering
		\begin{tikzpicture}
			\path[use as bounding box] (-2.2,0.5) rectangle (1.3,3.4);
			\draw (0,2) circle (1); \node at (0.5,1.3) {\small{$\cc_{1}$}};
			\draw[blue, rotate around={-15:(0.5,1.5)}] (0.5,1.5) ellipse (0.44 and 1);
			\node at (0.95,0.8) {\small{\blue{$\cc_2$}}};
			\coordinate (P') at (0,1); \filldraw (P') circle (0.06); \node at ($(P')-(0.2,0.2)$) {\small{$p'$}};
			\coordinate (P1) at (0,3); \filldraw (P1) circle (0.06); \node at ($(P1)+(-0.2,0.2)$) {\small{$p_1$}};
			\coordinate (P2) at (1,2); \filldraw (P2) circle (0.06); \node at ($(P2)+(0.2,-0.2)$) {\small{$p_2$}};
			\coordinate (R) at (-1,2); \filldraw (R) circle (0.06); \node at ($(R)+(0.2,0)$) {\small{$r$}};
			\draw[name path =L1] (-2.2,3) -- (1.3,3); \node at (1.2,3.2) {\small{$\ll_1$}}; 
			\draw[add= 0.1 and 0.2, name path = L1'] (P1) to (P'); \node at ($(P')+(-0.2,1)$) {\small{$\ll_{1}'$}};
			\draw[add= 0.2 and 1.2, name path = L] (P') to (R);
			\draw[add= 0.2 and 1.2, name path = L'] (P1) to (R);
			\coordinate (Q) at (-2,3); \filldraw (Q) circle (0.06); \node at ($(Q)+(0,0.2)$) {\small{$q$}};
			\draw[blue] (-1.48,2.7) [partial ellipse=160:20: 0.6 and 0.6];
			\node at (-0.9,3.2) {\small{\blue{$\cc_2$}}};
			\draw[blue] (-1.5,1.1) [partial ellipse=210:60: 0.6 and 0.6];
			\node at (-2,1.7) {\small{\blue{$\cc_2$}}};
		\end{tikzpicture}
		\caption{Configuration \ref{conf:A1A2_3-node} yields \ref{def:A1A2_3-node}, \ref{def:F2_n2-transversal}, cf.\ Figures \ref{fig:A1A2_3-node}, \ref{fig:F2_n2-transversal}}
		\label{fig:A1A2_3-node_conf}
	\end{minipage}
\end{figure}

\begin{conf}[Figure \ref{fig:A1A2_3-node_conf}]\label{conf:A1A2_3-node}
	Let $\cc_{2}$ be the unique conic such that $(\cc_2\cdot \cc_1)_{p_2}=3$, $(\cc_2\cdot \ll_{1}')_{p'}=2$. 
	\begin{enumerate}
		\item \label{item:F2_n2-transversal} The line $\ll_{1}$ is not tangent to $\cc_2$. We have $\#\tref{\cP}{def:F2_n2-transversal}=2$, and $\tref{\cP}{def:F2_n2-transversal}$ is realized over $\Q(\omega)$.
		\item \label{item:A1A2_3-node} Choose $q\in \ll_{1}\cap \cc_2$ and write  $\ll_{qp'}\cap \cc_1=\{p',r\}$. Then $r\neq p'$, and the line $\ll_{rp_1}$ is not tangent to $\cc_2$. We have $\#\tref{\cP}{def:A1A2_3-node}=4$, and $\tref{\cP}{def:A1A2_3-node}$ is realized over $\Q(\omega,\zeta_{+},\zeta_{-})$, where $\zeta_{\pm}=(\omega^{\pm 1}-4)^{1/2}$. Moreover,  $\tref{\cP}{def:A1A2_3-node}$ consists of two complex-conjugate pairs of configurations, one realized over $\Q(\omega,\zeta_{+})$ and the other over $\Q(\omega,\zeta_{-})$.
	\end{enumerate}
\end{conf}
\begin{proof}
	 The pencil of conics tangent to $\cc_1$ at $p_2$ and to $\ll_{1}'$ at $p'$ is generated by $\ll_{2}+\ll_{1}'=\{x(z-y)=0\}$ and $2\ll_{2}'=\{(x-y)^2=0\}$. Its unique member $\cc_2$ satisfying $(\cc_2\cdot \cc_1)_{p_2}\geq 3$ is given by $\{x(z-y)=(x-y)^2\}$. 
	 
	 \ref{item:F2_n2-transversal} We have $\cc_{2}\cap \ll_{1}=\{[1:-\omega:0],[1:-\omega^2:0]\}$, where $\omega$ is a primitive third root of unity. Since the group $\Aut(\P^2,\cc_1+\cc_2+\ll_{1}+\ll_{1}')$ is trivial, those two points give two non-equivalent marked configurations in $\tref{\cP}{def:F2_n2-transversal}$.
	 
	 \ref{item:A1A2_3-node} Take $q=[1:-\omega:0]$. Then $\ll_{qp'}=\{x+\omega^2 y+\omega z=0\}$ and $r=[\omega^2:\omega:1]$, hence $\ll_{rp_{1}}=\{y=\omega z\}$. This line meets $\cc_2$ in points $[\xi:1:\omega^2]$, where $\xi$ satisfies $\xi^2+\omega\xi+1=0$, so $\xi\in \Q(\omega,\sqrt{\omega^2-4})$. Ordering those points gives, for each of the two choices of $q$, two marked configurations in $\tref{\cP}{def:A1A2_3-node}$.
\end{proof}

\begin{conf}[Figure \ref{fig:A1A2_q-cn_conf}]\label{conf:A1A2_q-cn}
	Fix a point $r\in \ll_{2}$ not lying on $\cc_{1},\ll_{1}$ or $\ll_{p'}$. Define $\cc_{2}$ as the unique conic passing through $p$ and tangent to $\ll_{rp_1}$ and $\cc_1$ at $r$ and $p'$, respectively. Then
	\begin{enumerate}
		\item\label{item:A1A2_q-cn_31} There are exactly two points $r$ such that $(\cc_{2}\cdot\cc_{1})_{p'}=3$;  $\#\tref{\cP}{def:A1A2_q-cn_31}=2$, and $\tref{\cP}{def:A1A2_q-cn_31}$ is realized over $\Q(\sqrt{-15})$. 
		\item\label{item:A1A2_q-cn_22} There are exactly two points $r$ such that $(\cc_{2}\cdot\cc_{1})_{p''}=2$ for some $p''\neq p'$. In this case, the two unmarked configurations $\pp\de \ll_1+\ll_2+\cc_1+\cc_2+\ll_{rp_1}$ corresponding to two points $r$ are projectively equivalent. 
		We have $\#\tref{\cP}{def:A1A2_q-cn_22}=2$ and $\tref{\cP}{def:A1A2_q-cn_22}$ is realized over $\Q(\sqrt{5})$.
	\end{enumerate} 
\end{conf}
\begin{proof}
	We have $r=[0:1:\alpha]$ for some $\alpha\in \C\setminus \{0,2\}$, see Figure \ref{fig:coordinates}. The pencil of conics tangent to $\ll_{rp_1}$ at $r$ and to $\cc_1$ at $p'$ is generated by $\ll_{rp_1}+\ll_{p'}=\{(\alpha y-z)(2y-x-z)=0\}$ and by $2\ll_{rp'}=\{((1-\alpha)x+\alpha y-z)^2=0\}$. Its unique member passing through $p=[0:1:0]$ is given by $\cc_{2}\de \{\alpha(\alpha y- z)(2y-x-z)=2((1-\alpha)x+\alpha y-z)^{2}\}$. 
	
	\ref{item:A1A2_q-cn_31} Substituting a parametrization $u\mapsto [1:u:u^2]$ of $\cc_1\setminus \{p_2\}$ to the equation of $\cc_2$, we get that $(\cc_{2}\cdot \cc_{1})_{p'}=3$ if and only if $\alpha=\tfrac{1}{6}(9\pm\sqrt{-15})$. Thus we get two possible configurations in $\tref{\cP}{def:A1A2_q-cn_31}$, as claimed.
	
	\ref{item:A1A2_q-cn_22} As before, substituting a parametrization of $\cc_1$ to the equation of $\cc_2$, we get that $\cc_2$ is tangent to $\cc_1$ at some $p''\neq p'$ if and only if $\alpha\in\{\alpha_{-},\alpha_{+}\}$, where $\alpha_{\pm}=-1\pm\sqrt{5}$. 
	
	Denote by $r_{\pm}$, $p''_{\pm}$ the corresponding points, and by $\pp_{\pm}$ the corresponding configurations $\cc_{1}+\cc_{2}+\ll_{1}+\ll_{2}+\ll_{rp'}$, used in \ref{def:A1A2_q-cn_22}. Then $r_{\pm}=[0:1:\alpha_{\pm}]$ and $p_{\pm}''=[1:\beta_{\pm}:\beta_{\pm}^2]$, where $\beta_{\pm}=\tfrac{1}{2}\cdot (-3\pm\sqrt{5})$. We note that $\beta_{+}\cdot \alpha_{-}=\alpha_{+}$. 
	
	Put $h[x:y:z]=[x:\beta_{+}\cdot y:\beta_{+}^2\cdot z]$. Then $h$ maps $(p,p_1,p_2)$ to itself, $h(p')=p''_{+}$ and  $h(r_{-})=[0:\beta_{+}:\beta_{+}^2\alpha_{-}]=[0:1:\beta_{+}\alpha_{-}]=[0:1:\alpha_{+}]=r_{+}$. Thus $h(\pp_{-})=\pp_{+}$.
	
	We have shown that the unmarked configuration $\pp_{-}\subseteq \P^2$ is unique up to a projective equivalence, and is realized over $\Q(\sqrt{5})$. We claim that the two possible markings on $\pp_{-}$, corresponding to the orders of $\{p',p''_{-}\}$, are not projectively equivalent. Suppose there is $g\in \Aut(\P^2,\pp_{-})$ such that $g(p')=p''_{-}$. Then $g$ maps $(p,p_1,p_2,p')$ to $(p,p_1,p_2,p''_{-})$, so $g[x:y:z]=[x:\beta_{-} y:\beta_{-}^2 z]$. But then $g(r_{-})=[0:1:\alpha_{-} \beta_{-}]\neq r_{-}$, a contradiction.
	
	Hence $\#\tref{\cP}{def:A1A2_q-cn_22}=2$. It remains to prove that the two marked configurations in $\tref{\cP}{def:A1A2_q-cn_22}$, are conjugate over $\Q(\sqrt{5})$. Consider a diagonal coordinate change $\phi\in \Aut(\P^2)$ given by $\phi[x:y:z]=[-\alpha_{-} x:y:\alpha_{+} z]$. Then $\phi$ maps $(p,p_1,p_2)$ to itself and $\phi(r_{-})=[0:1:1]$. Moreover, $\phi(p')=[-\alpha_{-}:1:\alpha_{+}]$, $\phi(p''_{-})=[-\alpha_{-}:\beta_{-}:\beta_{-}^2\alpha_{+}]=[-\alpha_{-}\beta_{-}^{-1}:1:\beta_{-}\alpha_{+}]$. We have $\beta_{-}^{-1}=\beta_{+}$, so $\alpha_{-}\beta_{-}^{-1}=\alpha_{-}\beta_{+}=\alpha_{+}$, and similarly $\alpha_{+}\beta_{-}=\alpha_{-}$. Thus $\phi(p''_{-})=[-\alpha_{+}:1:\alpha_{-}]$ is the conjugate of $\phi(p')$, and the claim follows.
\end{proof}

\begin{figure}[htbp]
	\captionsetup{width=.3\linewidth}
	\centering
	\begin{minipage}[tb]{0.7\textwidth}
		\begin{subfigure}{.44\textwidth}
			\centering
			\begin{tikzpicture}
				\path[use as bounding box] (-1.8,-1) rectangle (2.1,2.2);
				\draw[name path=c1] (0,0) circle (1);
				\coordinate (P) at (0,2); 
				\node at ($(P)-(0,0.3)$) {\small{$p$}}; \filldraw (P) circle (0.06);
				\coordinate (P') at (0,1);
				\node at ($(P')-(0,0.3)$) {\small{$p'$}}; \filldraw (P') circle (0.06);
				\coordinate (P1) at (-0.866,0.5);
				\node at ($(P1)+(-0.1,0.25)$) {\small{$p_1$}}; \filldraw (P1) circle (0.06);
				\coordinate (P2) at (0.866,0.5);
				\node at ($(P2)+(0.1,0.25)$) {\small{$p_2$}}; \filldraw (P2) circle (0.06);
				\node at (-1.5,-0.2) {\small{$\ll_{1}$}};
				\node at (1.5,-0.2) {\small{$\ll_{2}$}};
				\node at (-0.7,-0.2) {\small{$\cc_{1}$}};
				\draw[add= 0.1 and 0.8, name path = L1] (P) to (P1);
				\draw[add= 0.1 and 1, name path = L2] (P) to (P2);
				\coordinate (R) at ($(P)!1.8!(P2)$);
				\node at ($(R)+(0.15,0.1)$) {\small{$r$}}; \filldraw (R) circle (0.06);
				\draw[add= 0.2 and 0.1, name path = L] (R) to (P1);
				\draw[blue] (P) to[out=0,in=-20] (R) to[out=160,in=0] (P') to[out=180,in=-90] (-0.5,1.5) to[out=90,in=180] (P);
				\node at ($(P2)+(0.9,0.25)$) {\small{\blue{$\cc_2$}}};
			\end{tikzpicture}
			\caption{yields \ref{def:A1A2_q-cn_31}, cf.\ Figure \ref{fig:A1A2_q-cn_31}}
			\label{fig:A1A2_q-cn_31_conf}
		\end{subfigure}
		\begin{subfigure}{.44\textwidth}
			\centering
			\begin{tikzpicture}
				\path[use as bounding box] (-2.2,-1) rectangle (2.1,2.2);
				\draw[name path=c1] (0,0) ellipse (1.5 and 1);
				\coordinate (P) at (0,2); 
				\node at ($(P)-(0,0.3)$) {\small{$p$}}; \filldraw (P) circle (0.06);
				\coordinate (P') at (0,1);
				\node at ($(P')-(0,0.25)$) {\small{$p'$}}; \filldraw (P') circle (0.06);
				\coordinate (P'') at (-0.5,0.943);
				\node at ($(P'')-(0,0.25)$) {\small{$p''$}}; \filldraw (P'') circle (0.06);
				\coordinate (P1) at (-1.2,0.6);
				\node at ($(P1)-(-0.2,0.2)$) {\small{$p_1$}}; \filldraw (P1) circle (0.06);
				\coordinate (P2) at (1.2,0.6);
				\node at ($(P2)-(0.15,0.2)$) {\small{$p_2$}}; \filldraw (P2) circle (0.06);
				\node at (-2,0) {\small{$\ll_{1}$}};
				\node at (2,-0) {\small{$\ll_{2}$}};
				\node at (1.4,-0.7) {\small{$\cc_{1}$}};
				\draw[add= 0.1 and 0.7, name path = L1] (P) to (P1);
				\draw[add= 0.1 and 0.7, name path = L2] (P) to (P2);
				\coordinate (R) at ($(P)!0.7!(P2)$);
				\node at ($(R)+(0.3,-0.1)$) {\small{$r$}}; \filldraw (R) circle (0.06);
				\draw[add= 0.3 and 0.3, name path = L] (R) to (P1);
				\draw[blue] (P) to[out=0,in=10] (R) to[out=-170,in=0] (0.2,1.2) to[out=180,in=0] (P') to[out=180,in=0] (-0.2,1.2) to[out=180,in=10] (P'') to[out=-170,in=-90] (-1,1.2) to[out=90,in=180] (P);
				\node at (1.1,1.5) {\small{\blue{$\cc_2$}}};
			\end{tikzpicture}
			\caption{yields \ref{def:A1A2_q-cn_22}, cf.\ Figure \ref{fig:A1A2_q-cn_22}}
			\label{fig:A1A2_q-cn_22_conf}		
		\end{subfigure}
	\vspace{-.5em}
		\caption{Configuration \ref{conf:A1A2_q-cn}}
		\label{fig:A1A2_q-cn_conf}
	\end{minipage}
	\begin{minipage}[tb]{.28\textwidth}
		\centering
		\begin{tikzpicture}
			\path[use as bounding box] (-2,-1.5) rectangle (3,2.2);
			\draw[name path=c1] (0,0) circle (1);
			\coordinate (P) at (0,2); 
			\node at ($(P)-(0.2,0)$) {\small{$p$}}; \filldraw (P) circle (0.06);
			\coordinate (P1) at (-0.866,0.5);
			\node at ($(P1)+(-0.25,-0.1)$) {\small{$p_1$}}; \filldraw (P1) circle (0.06);
			\coordinate (P2) at (0.866,0.5);
			\node at ($(P2)+(0.3,-0.1)$) {\small{$p_2$}}; \filldraw (P2) circle (0.06);
			\node at (-0.7,1.2) {\small{$\ll_{1}$}};
			\node at (0.7,1.2) {\small{$\ll_{2}$}};
			\node at (1.1,-0.5) {\small{$\cc_{1}$}};
			\draw[add= 0.1 and 1.1, name path = L1] (P) to (P1);
			\draw[add= 0.1 and 1.4, name path = L2] (P) to (P2);
			\draw[add= 0.6 and 0.6, name path = L] ($(P1)-(0,1.4)$) to ($(P2)-(0,1.4)$);
			\path [name intersections={of=L and L1}] (intersection-1) coordinate (S);
			\node at ($(S)+(0,0.2)$) {\small{$s$}}; \filldraw (S) circle (0.06);
			\path [name intersections={of=L and c1}] (intersection-1) coordinate (P')
			(intersection-2) coordinate (T);
			\node at ($(P')+(0.1,0.25)$) {\small{$p'$}}; \filldraw (P') circle (0.06);
			\node at ($(T)+(0,0.2)$) {\small{$t$}}; \filldraw (T) circle (0.06);
			\draw[add= 0.1 and 0.1, name path = LQ] (S) to (P2);
			\path [name intersections={of=LQ and c1}] (intersection-2) coordinate (Q);
			\node at ($(Q)+(0.05,0.25)$) {\small{$q$}}; \filldraw (Q) circle (0.06);
			\draw[add= 0.5 and 1.4, name path = LR] (Q) to (T);
			\path [name intersections={of=LR and L2}] (intersection-1) coordinate (R);
			\node at ($(R)+(0.25,-0.25)$) {\small{$r$}}; \filldraw (R) circle (0.06);
			\draw[blue] ($(P)+(0,0.2)$) -- (P) to[out=-90,in=60] (Q) to[out=-120,in=90] ($(Q)-(0.1,0.1)$) to[out=-90,in=150] (P'.center) to[out=-30,in=0] ($(P')-(0.4,0.4)$) to[out=180,in=-45] (S) to[out=135,in=0] ($(S)+(-0.4,0.2)$);
			\draw[blue] ($(R)+(0.5,-0.35)$) [partial ellipse=160:70: 1 and 0.4];
			\node at ($(R)+(0.7,0.2)$) {\small{\blue{$\cc_2$}}};
			\node at (-0.1,0.5) {\small{\blue{$\cc_2$}}};
			\draw[add= 0.15 and 0.2, name path = LSR] (S) to (R);
			\draw ($(T)-(0,0.6)$) [partial ellipse=160:0: 0.8 and 0.4];
			\node at ($(T)-(0.55,0.6)$) {\small{$\cc_1$}};
		\end{tikzpicture}
		\caption{Configuration \ref{conf:A1A2_2c1n} yields \ref{def:A1A2_2c1n}, cf.\ Figure \ref{fig:A1A2_2c1n}}
		\label{fig:A1A2_2c1n_conf}
	\end{minipage}
\end{figure}

\begin{conf}[Figure \ref{fig:A1A2_2c1n_conf}]\label{conf:A1A2_2c1n}
	Fix a point $s\in \ll_{1}\setminus \{p,p_{1}\}$ and write $\ll_{sp_{2}}\cap\cc_{1}=\{p_{2},q\}$. Let $\cc_{2}$ be the unique conic passing through $p,q,s$ and tangent to $\cc_{1}$ at $p'$. Write $\cc_{2}\cap \ll_{2}=\{p,r\}$ and $\cc_{1}\cap\ll_{rq}=\{t,q\}$. Then there exists a unique point $s$ such that the points $p'$, $s$, $t$ are colinear. For such $s$, the line $\ll_{sr}$ meets $\cc_1$ in two points.
	
	We have $\#\tref{\cP}{def:A1A2_2c1n}=2$, and $\tref{\cP}{def:A1A2_2c1n}$ is realized over $\Q(\sqrt{-7})$.
\end{conf}
\begin{proof}
	We have $s=[\alpha:1:0]$ for some $\alpha\in \C^{*}$. We have $\ll_{sp_2}=\{x=\alpha y\}$, hence $q=[\alpha^{2}:\alpha:1]$. The pencil of conics passing through $p,s$ and tangent to $\cc_1$ at $p'$ is generated by $\ll_{1}+\ll_{p'}=\{z(2y-x-z)=0\}$ and $\ll_{pp'}+\ll_{sp'}=\{(z-x)(x-\alpha y-(1-\alpha)z)=0\}$, see Figure \ref{fig:coordinates}. Its unique member passing through $q$ is $\cc_{2}=\{(\alpha+1)z(2y-x-z)=(z-x)(x-\alpha y-(1-\alpha)z)\}$. Hence $r=[0:2\alpha:3\alpha+2]$, and therefore $\ll_{rq}=\{(3\alpha+2)y=3x+2\alpha z\}$, $t=[4:6:9]$. Thus $p'$, $s$ and $t$ are colinear if and only if $\alpha=\tfrac{5}{3}$.
	
	For $\alpha=\frac{5}{3}$ we have $s=[5:3:0]$, $r=[0:10:21]$, so $\ll_{sr}=\{105y=63x+50z\}$. This line meets $\cc_1=\{y^2-xz\}$ in two points $[1:\beta:\beta^2]$, where $\beta=\frac{3}{20}(7\pm \sqrt{-7})$.
\end{proof}

\begin{conf}[Figure \ref{fig:F2_n1-node_conf}]\label{conf:F2_n1-cusp}
	Let $\ll_{p'}$ be the line tangent to $\cc_{1}$ at $p'$ and write $\ll_{p'}\cap \ll_{2}=\{r\}$. There are exactly two conics $\cc_{2}$ passing through $p$, $p_{1}$, tangent to $\ll_{p'}$ at $r$ and to $\cc_1$ at some $q\neq p_1$. They both satisfy $(\cc_{2}\cdot\cc_{1})_{q}=2$. We have $\#\tref{\cP}{def:F2_n1-cusp}=2$, and $\tref{\cP}{def:F2_n1-cusp}$ is realized over $\Q(\sqrt{5})$.
\end{conf}
\begin{proof}
	We have $r=[0:1:2]$, see Figure \ref{fig:coordinates}. The pencil of conics tangent to $\ll_{p'}$ at $r$ and passing through $p_{1}$, $p$ is generated by $\ll_{1}+\ll_{p'}$ and $\ll_{2}+\ll_{rp_1}$, so it is given by $\{\lambda z(2y-x-z)=\mu x(2y-z)\}_{[\lambda:\mu]\in \P^1}$. Substituting a parametrization $u\mapsto[u^2:u:1]$ of $\cc_1\setminus \{p_1\}$ to the above equation, we conclude that this pencil has two nondegenerate members tangent to $\cc_{1}\setminus \{p_1\}$. They correspond to $[\lambda:\mu]=[-(11\pm 5\sqrt{5}):2]$, and meet $\cc_1\setminus \{p_1\}$ at two points each, with multiplicities $2$, $1$.
\end{proof}

\begin{figure}[htbp]
	\centering
	\captionsetup{width=.4\linewidth}
	\begin{minipage}[b]{.45\textwidth}
		\centering
		\begin{tikzpicture}
			\path[use as bounding box] (-2,-1.2) rectangle (3,2.5);
			\draw[name path=c1] (0,0) circle (1);
			\coordinate (P) at (0,2); 
			\node at ($(P)+(0,0.3)$) {\small{$p$}}; \filldraw (P) circle (0.06);
			\coordinate (R) at (1.732,-1);
			\coordinate (R') at (-1.732,-1);
			\node at ($(R)+(0.1,0.2)$) {\small{$r$}}; \filldraw (R) circle (0.06);
			\coordinate (P') at (0,-1);
			\node at ($(P')+(0,0.3)$) {\small{$p'$}}; \filldraw (P') circle (0.06);
			\coordinate (Q) at (0,1);
			\node at ($(Q)+(0,0.2)$) {\small{$q$}}; \filldraw (Q) circle (0.06);
			\coordinate (P1) at (-0.866,0.5);
			\node at ($(P1)+(-0.3,0)$) {\small{$p_1$}}; \filldraw (P1) circle (0.06);
			\coordinate (P2) at (0.866,0.5);
			\node at ($(P2)+(0.3,0)$) {\small{$p_2$}}; \filldraw (P2) circle (0.06);
			\node at (-1.5,-0.2) {\small{$\ll_{1}$}};
			\node at (1.5,-0.2) {\small{$\ll_{2}$}};
			\node at (-0.7,-0.2) {\small{$\cc_{1}$}};
			\draw[add= 0.05 and 0.05, name path = L1] (P) to (R');
			\draw[add= 0.05 and 0.05, name path = L2] (P) to (R);
			\draw[add= 0.05 and 0.3, name path = LP'] (R') to (R);
			\node at (2.6,-0.8) {\small{$\ll_{p'}$}};
			\draw[blue] (P) to[out=180,in=135] (P1) to[out=-45,in=-120]
			($(Q)+(-0.4,-0.4)$)
			to[out=60,in=180]
			(Q) 
			to[out=0,in=100]
			($(Q)+(0.4,-0.4)$) 
			to[out=-80,in=180] (R) to[out=0,in=0] (P);
			\node at ($(P2)+(1.3,0)$) {\small\blue{{$\cc_{2}$}}};
		\end{tikzpicture}
		\caption{Configuration \ref{conf:F2_n1-cusp} yields \ref{def:F2_n1-cusp}, cf.\ Figure \ref{fig:F2_n1-cusp} and Example \ref{ex:7}}
		\label{fig:F2_n1-node_conf}
	\end{minipage}
	\begin{minipage}[b]{.45\textwidth}
		\centering
		\begin{tikzpicture}
			\path[use as bounding box] (-2,-1.5) rectangle (4,2.5);
			\draw[name path=c1] (0,0) circle (1);
			\coordinate (P) at (0,2); 
			\node at ($(P)+(0,0.3)$) {\small{$p$}}; \filldraw (P) circle (0.06);
			\coordinate (R) at (1.732,-1);
			\coordinate (R') at (-1.732,-1);
			\node at ($(R)+(0.1,0.2)$) {\small{$r$}}; \filldraw (R) circle (0.06);
			\coordinate (P') at (0,1);
			\node at ($(P')+(0,0.2)$) {\small{$p'$}}; \filldraw (P') circle (0.06);
			\coordinate (P1) at (-0.866,0.5);
			\node at ($(P1)+(-0.3,0)$) {\small{$p_1$}}; \filldraw (P1) circle (0.06);
			\coordinate (P2) at (0.866,0.5);
			\node at ($(P2)+(0.25,0)$) {\small{$p_2$}}; \filldraw (P2) circle (0.06);
			\node at (-1.6,-0.4) {\small{$\ll_{1}$}};
			\node at (1.5,-0.2) {\small{$\ll_{2}$}};
			\node at (-0.8,0) {\small{$\cc_{1}$}};
			\draw[add= 0.05 and 0.05, name path = L1] (P) to (R');
			\draw[add= 0.05 and 0.05, name path = L2] (P) to (R);
			\draw[add= 0.05 and 0.6, name path = LR'] (R') to (R);
			\node at (3.8,-0.8) {\small{$\ll_{r}'$}};
			\coordinate (Q) at (2.8,-1);
			\node at ($(Q)+(0,0.2)$) {\small{$q$}}; \filldraw (Q) circle (0.06);
			\draw[add= 0.05 and 0.15, name path =L] (P) to (Q);
			\draw[blue] (P) to[out=180,in=135] (P1) to[out=-45,in=-120]
			($(P')+(-0.4,-0.4)$)
			to[out=60,in=180]
			(P') 
			to[out=0,in=120]
			($(0.96,-0.28)+(-0.05,0.1)$);
			\draw[blue] ($(0.96,-0.28)+(0.05,-0.1)$)
			to[out=-60,in=165]
			(R) to[out=-15,in=-165] (Q) to[out=15,in=0] (P);
			\node at ($(P2)+(2.1,0)$)  {\small\blue{{$\cc_{2}$}}};
			\draw[add= 0.45 and -0.1, name path =LR] (R) to ($(R)+(-3.732,1)$);
			\node at (0,-0.3) {\small{$\ll_{r}$}};
			\draw ($(P2)+(0.5,0)$) [partial ellipse=140:-60: 0.5 and 0.5];
			\node at ($(P2)+(1.25,0)$) {\small{$\cc_1$}};
		\end{tikzpicture}
		\caption{Configuration \ref{conf:F2_n1} yields \ref{def:F2_n1-node}, \ref{def:F2_n1-node-3}, cf.\ Figures \ref{fig:F2_n1-node}, \ref{fig:F2_n1-node-3}}
		\label{fig:F2_n1_conf}
	\end{minipage}
\end{figure}

\begin{conf}[Figure \ref{fig:F2_n1_conf}]\label{conf:F2_n1}
	Let $\cc_{2}$ be the unique conic meeting $\cc_{1}$ at $p'$, $p_{1}$ with multiplicity $3$, $1$ and passing through $p$. Write $\ll_{2}\cap\cc_{2}=\{p,r\}$. Let $\ll_{r}$ be the line tangent to $\cc_{2}$ at $r$, and let $\ll_{r}'\neq\ll_{2}$ be the other line tangent to $\cc_{1}$ and passing through $r$.  Write $\ll_{r}'\cap \cc_{2}=\{q,r\}$.
	\begin{enumerate}
		\item \label{item:node} The line $\ll_{r}$ meets $\cc_1$ in two points. We have $\#\tref{\cP}{def:F2_n1-node}=2$, and $\tref{\cP}{def:F2_n1-node}$ is realized over $\Q(\sqrt{-15})$. 
		\item \label{item:node-3} The line $\ll_{pq}$ meets $\cc_1$ in two points. We have $\#\tref{\cP}{def:F2_n1-node-3}=2$, and $\tref{\cP}{def:F2_n1-node-3}$ is realized over $\Q(\sqrt{5})$.
	\end{enumerate}
\end{conf}
\begin{proof}
	 The pencil of conics passing through $p,p_1$ and tangent to $\cc_1$ at $p'$ is generated by $\ll_{1}+\ll_{p'}=\{z(x+z-2y)=0\}$ and $\ll_{1}'+\ll_{pp'}=\{(y-z)(x-z)=0\}$, see Figure \ref{fig:coordinates}. Its unique member $\cc_2$ satisfying $(\cc_2\cdot \cc_1)_{p'}=3$ is given by $\{2z(x+z-2y)=(y-z)(x-z)\}$. Now $\ll_{2}$ meets $\cc_2$ in $p$ and $r=[0:1:3]$. 
	
	\ref{item:node} The line $\ll_{r}=\{8x=9y-3z\}$ meets $\cc_{1}$ at two points $[\alpha^2:\alpha:1]$, where $\alpha=\frac{1}{16}(9\pm\sqrt{-15})$. 
	
	\ref{item:node-3} We have $\ll_{r}'=\{12y=9x+4z\}$, $q=[20:27:36]$. The line  $\ll_{pq}=\{9x=5z\}$ meets $\cc_1$ at two points $[5:\pm 3\sqrt{5}:9]$.
\end{proof}

\begin{conf}[Figure \ref{fig:F2_n1-cusp-hor_conf}]\label{conf:F2_n1-cusp-hor}
	Write $\ll_{2}\cap \ll_{1}'=\{q\}$. 
	\begin{enumerate}
		\item\label{item:F2_32} There is a unique conic $\cc_{2}$ tangent to $\ll_{2}'$, such that $(\ll_2\cdot \cc_2)_{q}=1$, $(\cc_{1}\cdot\cc_{2})_{p_{1}}=2$ and $(\cc_{1}\cdot \cc_{2})_{s}=2$ for some $s\neq p_{1}$. We have $\#\tref{\cP}{def:F2_n1-cusp-hor_32}=1$, and $\tref{\cP}{def:F2_n1-cusp-hor_32}$ is realized over $\Q$.
		\item\label{item:F2_41} There are exactly two conics $\cc_{2}$ tangent to $\ll_{2}'$, such that $(\ll_2\cdot \cc_2)_{q}=1$ and $(\cc_{1}\cdot\cc_{2})_{p_{1}}=3$. We have $\#\tref{\cP}{def:F2_n1-cusp-hor_41}=2$, and $\tref{\cP}{def:F2_n1-cusp-hor_41}$ is realized over $\Q(\imath)$.
	\end{enumerate}
\end{conf}
\begin{proof}
	 We have $q=[0:1:1]$, see Figure \ref{fig:coordinates}. Choose a point $r\in \ll_{2}'\setminus (\ll_1\cup\cc_1)$, so $r=[1:1:\alpha]$ for some $\alpha\in \C\setminus \{0,1\}$. The pencil of conics tangent to $\ll_{1}$ at $p_{1}$ and to $\ll_{2}'$ at $r$ is generated by $\ll_{1}+\ll_{2}'=\{z(y-x)=0\}$ and $2\ll_{rp_{1}}=\{(\alpha y-z)^2=0\}$. Its unique member passing through $q=[0:1:1]$ is $\cc_2=\{(\alpha-1)^{2} z(y-x)=(\alpha y-z)^{2}\}$. We have $\cc_2\cap \ll_2=\{q,[0:1:\alpha^2]\}$, so $(\cc_2\cdot \ll_{2})_{q}=1$. 
	
	Substituting a parametrization $u\mapsto [u^2:u:1]$ of $\cc_1\setminus \{p_1\}$ to the equation of $\cc_2$, we see that $\cc_2$ is tangent to $\cc_{1}\setminus \{p_{1}\}$ if and only if $\alpha=-3$, which proves \ref{item:F2_32}; and $\cc_2$ satisfies $(\cc_2\cdot \cc_1)_{p_1}=3$ if and only if $\alpha=\tfrac{1}{2}(1\pm \imath)$, which proves \ref{item:F2_41}.
\end{proof}

\begin{figure}[htbp]
	\centering
	\begin{subfigure}{.45\textwidth}
		\centering
		\begin{tikzpicture}
			\path[use as bounding box] (-1,-1.2) rectangle (3.2,1.6);
			\draw (0,0) circle (1);
			\coordinate (P1) at (0,1);
			\node at (-0.8,0) {\small{$\cc_1$}};
			\node at ($(P1)-(0,0.3)$) {\small{$p_1$}}; \filldraw (P1) circle (0.06);
			\coordinate (P2) at (0,-1);
			\node at ($(P2)+(0,0.25)$) {\small{$p_2$}}; \filldraw (P2) circle (0.06);
			\coordinate (P') at (1,0);
			\node at ($(P')+(0.3,0)$) {\small{$p'$}}; \filldraw (P') circle (0.06);
			\coordinate (Q) at (2,-1);
			\node at ($(Q)+(0,0.25)$) {\small{$q$}}; \filldraw (Q) circle (0.06);
			\draw (-0.2,1.2) -- (2.2,-1.2);
			\node at (1.7,-0.4) {\small{$\ll_{1}'$}};
			\draw (-0.2,-1.2) -- (2.2,1.2);
			\node at (1.7,0.4) {\small{$\ll_{2}'$}};
			\draw (-1,-1) -- (3,-1);
			\node at (2.8,-0.8) {\small{$\ll_{2}$}};
			\draw[blue] (-0.6,1.4) to[out=-30,in=180] (P1) to[out=0,in=180] (0.3,1.2) to[out=0,in=120] (0.64,0.8) to[out=-60,in=-135] (1.8,0.8) to[out=45,in=180] (2.6,1.6) to[out=0,in=45] (Q) to[out=-135,in=-45] (1.4,-1) to[out=135,in=0] (1,-0.8);
			\node at (3,0) {\small{\blue{$\cc_2$}}};
			\filldraw (0.64,0.8) circle (0.06);
			\node at (0.8,0.9) {\small{$s$}};
		\end{tikzpicture}
		\caption{yields \ref{def:F2_n1-cusp-hor_32}, cf.\ Figure \ref{fig:F2_n1-cusp-hor_32}}
		\label{fig:F2_n1-cusp-hor_32_conf}
	\end{subfigure}
	\begin{subfigure}{.45\textwidth}
		\centering
		\begin{tikzpicture}
			\path[use as bounding box] (-1,-1.2) rectangle (3.2,1.6);
			\draw (0,0) circle (1);
			\coordinate (P1) at (0,1);
			\node at (-0.8,0) {\small{$\cc_1$}};
			\node at ($(P1)-(0,0.3)$) {\small{$p_1$}}; \filldraw (P1) circle (0.06);
			\coordinate (P2) at (0,-1);
			\node at ($(P2)+(0,0.25)$) {\small{$p_2$}}; \filldraw (P2) circle (0.06);
			\coordinate (P') at (1,0);
			\node at ($(P')+(0.3,0)$) {\small{$p'$}}; \filldraw (P') circle (0.06);
			\coordinate (Q) at (2,-1);
			\node at ($(Q)+(0,0.25)$) {\small{$q$}}; \filldraw (Q) circle (0.06);
			\draw (-0.2,1.2) -- (2.2,-1.2);
			\node at (1.7,-0.4) {\small{$\ll_{1}'$}};
			\draw (-0.2,-1.2) -- (2.2,1.2);
			\node at (1.7,0.4) {\small{$\ll_{2}'$}};
			\draw (-1,-1) -- (3,-1);
			\node at (2.8,-0.8) {\small{$\ll_{2}$}};
			\draw[blue] (-0.6,1.2) to[out=0,in=180] (P1) to[out=0,in=180] (0.8,0.64) to[out=0,in=-135] (1.9,0.9) to[out=45,in=180] (2.6,1.6) to[out=0,in=45] (Q) to[out=-135,in=-45] (1.4,-1) to[out=135,in=0] (1,-0.8);
			\node at (3,0) {\small{\blue{$\cc_2$}}};
		\end{tikzpicture} 
		\caption{yields \ref{def:F2_n1-cusp-hor_41}, cf.\ Figure \ref{fig:F2_n1-cusp-hor_41}}
		\label{fig:F2_n1-cusp-hor_41_conf}
	\end{subfigure}
\vspace{-.5em}
	\caption{Configuration \ref{conf:F2_n1-cusp-hor}}
	\label{fig:F2_n1-cusp-hor_conf}
\end{figure}

\begin{conf}[Figure \ref{fig:F2-hor-ccc-41_conf}]\label{conf:F2-hor-ccc-41}
	Write $\cc_{1}\cap\ell_{pp'}=\{p',p''\}$. Fix a line $\ll\neq \ll_{pp'}$ passing through $p''$, which is not tangent to $\cc_1$ and does not pass through $p_1$ or $p_2$. Write $\ll\cap \ll_{1}=\{q_1\}$, $\ll\cap \ll_{2}=\{q_2\}$. Let $\cc_2$ be the unique conic passing through $p$, $q_1$, $q_2$ and tangent to $\cc_{1}$ at $p'$. Then, up to the action of $\tau\in \Aut(\P^2)$, there is exactly one line $\ll$ such that $\#(\cc_1\cap \cc_2)=2$. In this case, $(\cc_1\cdot \cc_{2})_{p'}=3$.
	
	We have $\#\tref{\cP}{def:F2-hor-ccc-41}=1$, and $\tref{\cP}{def:F2-hor-ccc-41}$ is realized over $\Q(\sqrt{3})$.
\end{conf}
\begin{proof}
	 We have $\ll_{pp'}=\{x=z\}$, so $p''=[1:-1:1]$ and $\ll=\{(\alpha+1)x+2y+(1-\alpha)z=0\}$ for some $\alpha\in \C\setminus \{0,1,-1\}$. We have $q_1=[-2:\alpha+1:0]$, $q_2=[0:1-\alpha:-2]$. The pencil of conics passing through $p,p',q_1,q_2$ is generated by $\ll_{1}+\ll_{q_2p'}=\{z((\alpha-3)x+2y+(1-\alpha)z)=0\}$ and $\ll_{2}+\ll_{q_1p'}=\{x((\alpha+1)x+2y-(\alpha+3)z)=0\}$. Its unique member tangent to $\cc_1$ at $p'$ is
	\begin{equation*}
		\cc_2=
		\{
		(\alpha+2)\cdot z((\alpha-3)x+2y+(1-\alpha)z)
		=
		(\alpha-2)\cdot x((\alpha+1)x+2y-(\alpha+3)z)
		\}
	\end{equation*} 
	We have $\#(\cc_1\cap \cc_2)=2$ if and only if $\alpha=\pm 2$ or $\pm \sqrt{3}$. In the first case, $\cc_2$ is reducible. In the second case, we have $(\cc_2\cdot \cc_1)_{p'}=3$, as needed. The two choices of $\alpha$ are equivalent by the action of $\tau$.
\end{proof}

\begin{figure}[htbp]
	\captionsetup{width=.4\linewidth}
	\begin{minipage}[b]{.4\textwidth}
		\centering
		\begin{tikzpicture}
			\path[use as bounding box] (-1.8,-1.2) rectangle (2.2,2.2);
			\draw[name path=c1] (0,0) circle (1);
			\coordinate (P) at (0,2); 
			\node at ($(P)-(0.3,0)$) {\small{$p$}}; \filldraw (P) circle (0.06);
			\coordinate (P') at (0,-1);
			\node at ($(P')+(0.2,0.2)$) {\small{$p'$}}; \filldraw (P') circle (0.06);
			\coordinate (P1) at (-0.866,0.5);
			\node at ($(P1)+(-0.2,0.2)$) {\small{$p_1$}}; \filldraw (P1) circle (0.06);
			\coordinate (P2) at (0.866,0.5);
			\node at ($(P2)+(0.2,0.2)$) {\small{$p_2$}}; \filldraw (P2) circle (0.06);
			\node at (-1.4,-0.9) {\small{$\ll_{1}$}};
			\node at (1.5,-0.9) {\small{$\ll_{2}$}};
			\node at (2,-0.6) {\small{$\ll$}};
			\node at (0.6,-0.5) {\small{$\cc_{1}$}};
			\draw[add= 0.1 and 1, name path = L1] (P) to (P1);
			\draw[add= 0.1 and 1, name path = L2] (P) to (P2);
			\draw[add= 0.05 and 0.05, name path = L''] (P) to (P');
			\path [name intersections={of=L'' and c1}] (intersection-1) coordinate (P'');
			\node at ($(P'')+(0.2,0.2)$) {\small{$p''$}}; \filldraw (P'') circle (0.06);
			\coordinate (Q1) at ($(P1)!0.57!(P)$); 
			\node at ($(Q1)+(-0.2,-0.1)$) {\small{$q_1$}}; \filldraw (Q1) circle (0.06);
			\draw[add=  1 and 5.5, name path = L] (Q1) to (P'');
			\path [name intersections={of=L and L2}] (intersection-1) coordinate (Q2);
			\node at ($(Q2)+(0.3,0)$) {\small{$q_2$}}; \filldraw (Q2) circle (0.06);
			\draw[blue] (P') to[out=180,in=-110] (Q1) to[out=70,in=200] (P) to[out=20,in=70] (Q2) to[out=-110,in=0] (P');
			\node at ($(Q2)+(0.1,1.5)$) {\small{\blue{$\cc_{2}$}}};
		\end{tikzpicture}
		\caption{Configuration \ref{conf:F2-hor-ccc-41} yields \ref{def:F2-hor-ccc-41}, cf.\ Figure \ref{fig:F2-hor-ccc-41}}
		\label{fig:F2-hor-ccc-41_conf}
	\end{minipage}
	\begin{minipage}[b]{.5\textwidth}
		\centering
		\begin{tikzpicture}
			\path[use as bounding box] (-2.2,-2) rectangle (4.5,2.2);
			\draw[name path=c1] (0,0) circle (1);
			\draw[name path=c2, blue] (0,0) circle (2);
			\coordinate (P) at (0,2); 
			\node at ($(P)-(0,0.3)$) {\small{$p$}}; \filldraw (P) circle (0.06);
			\coordinate (R1) at (-1.732,-1); 
			\node at ($(R1)+(-0.3,-0.2)$) {\small{$r_1$}}; \filldraw (R1) circle (0.06);
			\coordinate (R2) at (1.732,-1);
			\node at ($(R2)+(0.3,-0.2)$) {\small{$r_2$}}; \filldraw (R2) circle (0.06);
			\coordinate (P') at (0,-1);
			\node at ($(P')-(0,0.3)$) {\small{$p'$}}; \filldraw (P') circle (0.06);
			\coordinate (P1) at (-0.866,0.5);
			\node at ($(P1)+(-0.2,0.2)$) {\small{$p_1$}}; \filldraw (P1) circle (0.06);
			\coordinate (P2) at (0.866,0.5);
			\node at ($(P2)+(0.2,0.2)$) {\small{$p_2$}}; \filldraw (P2) circle (0.06);
			\node at (-1.5,-0.2) {\small{$\ll_{1}$}};
			\node at (1.5,-0.2) {\small{$\ll_{2}$}};
			\node at (2.2,-0.2) {\small{\blue{$\cc_{2}$}}};
			\node at (1.1,-0.6) {\small{$\cc_{1}$}};
			\draw[add= 0.05 and 0.05, name path = L1] (P) to (R1);
			\draw[add= 0.05 and 0.05, name path = L2] (P) to (R2);
			\draw[add= 0.05 and 0.05, name path = LP'] (R1) to (R2);
			\node at (0.9,-1.2) {\small{$\ll_{p'}$}};
			\draw (4,2.2) -- (4,-2);
			\draw (4,1.3) [partial ellipse=140:40: 0.5 and 0.5];
			\draw[blue] (4,2.3) [partial ellipse=-140:-40: 0.5 and 0.5];
			\node at (3.5,1.4) {\small{$\cc_1$}};
			\node at (3.5,2.1) {\small{\blue{$\cc_2$}}};
			\draw (4,0) [partial ellipse=140:40: 0.5 and 0.5];
			\draw[blue] (4,1) [partial ellipse=-140:-40: 0.5 and 0.5];
			\node at (3.5,0.2) {\small{$\cc_1$}};
			\node at (3.5,0.8) {\small{\blue{$\cc_2$}}};
			\draw (3.7,-0.7) -- (4.5,-0.7);
			\node at (3.5,-0.7) {\small{$\ll_{1}$}};
			\draw (3.7,-1.2) -- (4.5,-1.2);
			\node at (3.5,-1.2) {\small{$\ll_{2}$}};
			\draw (3.7,-1.7) -- (4.5,-1.7);
			\node at (3.5,-1.7) {\small{$\ll_{p'}$}};
			\node at (4.2,-0.2) {\small{$\ll$}};
			\filldraw (4,1.8) circle (0.06); \node at (4.2,1.6) {\small{$q_1$}};
			\filldraw (4,0.5) circle (0.06); \node at (4.2,0.2) {\small{$q_2$}};
		\end{tikzpicture}
		\caption{Configuration \ref{conf:P2n1-cuspidal} yields \ref{def:P2n1-cuspidal}, \ref{def:P2n3}, cf.\ Figures \ref{fig:P2n1-cuspidal}, \ref{fig:P2n3}.}
		\label{fig:P2n1-cuspidal_conf}
	\end{minipage}
\end{figure}

\begin{conf}[Figure \ref{fig:P2n1-cuspidal_conf}] \label{conf:P2n1-cuspidal}
	Write $p_3=p'$, $\ll_{3}=\ll_{p'}$ and $\{r_i\}=\ll_{j}\cap \ll_{k}$ for $\{i,j,k\}=\{1,2,3\}$. There is a unique conic $\cc_{2}$ passing through $r_1,r_2,r_3$ and meeting $\cc_{1}$ in exactly two points, say $q_{1}, q_{2}$. 
	
	Moreover, $\tau(\cc_{2})=\cc_{2}$ and $\tau(q_{1})=q_{2}$, so  $(\cc_{1}\cdot\cc_{2})_{q_{1}}=(\cc_{1}\cdot\cc_{2})_{q_{2}}=2$ and $\#\tref{\cP}{def:P2n1-cuspidal}=1$. The configuration $\tref{\cP}{def:P2n1-cuspidal}$ is realized over $\Q(\omega)$. Furthermore, $\#\tref{\cP}{def:P2n3}=1$ and $\tref{\cP}{def:P2n3}$ is realized over $\Q$. 
\end{conf}
\begin{proof}
	Write $\ll_{p_{1}r_{2}}\cap\ll_{p_{2}r_{1}}=\{q\}$ and let $\sigma\in \Aut\P^{2}$ be the automorphism mapping $(r_{1},r_{2},r_3,q)$ to $(r_3,r_{1},r_{2},q)$. Then the group $\langle \tau, \sigma\rangle\cong S_{3}$ permutes $(r_{1},r_{2},r_3)$. We have $\sigma(\cc_{1})=\cc_{1}$, so since $\sigma^{3}=\id$, $\sigma$ has two fixed points on $\cc_{1}$, say  $q_{1}$, $q_{2}$. Since $\tau$ acts non trivially on $\cc_{1}$, we have $\tau(q_{1})=q_{2}$. Let $\cc_{2}$ be the unique conic passing through  $q_{1}$, $q_{2}$, $r_{1}$, $r_{2}$, $r_3$. Then $\sigma(\cc_{2})=\cc_{2}$, so $\cc_{1}\cap \cc_{2}=\{q_{1},q_{2}\}$, hence $\cc_{2}$ is as in the statement.
	
	To see the uniqueness, suppose that $\cc_{2}'$ is another conic as in the statement. The pencil of conics generated by $\cc_{2}$ and $\cc_{2}'$ induces a morphism $\cc_{1}\to \P^{1}$ of degree $4$, ramified with index $2$ at $p_1,p_2,p_3,q_1,q_2$, and with indices $e_1'$, $e_2'$ satisfying $e_1'+e_2'=4$ at $\cc_1\cap \cc_2'$. This is a contradiction with the Hurwitz formula.
	
	The configuration $\tref{\pp}{def:P2n1-cuspidal}=\cc_1+\cc_2+\ll_1+\ll_2+\ll_3+\ll_{q_1q_2}$ is realized over $\Q(\omega)$ because the eigenvalues of $\sigma$ are $1,\omega,\omega^2$. All possible markings on $\tref{\pp}{def:P2n1-cuspidal}$ are equivalent by the action of $\langle \tau,\sigma \rangle$, so $\#\tref{\cP}{def:P2n1-cuspidal}=1$. The configuration $\tref{\pp}{def:P2n3}=\cc_1+\ll_1+\ll_2+\ll_3$ is a conic inscribed in a triangle; it is unique and realized over $\Q$.
\end{proof}

\begin{conf}[Figure \ref{fig:F2_n2-tangent_conf}]\label{conf:F2_n2-tangent}
	Let $\cc_2$ be the unique conic passing through $p_2$, which is tangent to $\ll_1$ at $p$ and to $\cc_1$ at $p'$. Then $(\cc_2\cdot \cc_{1})_{p'}=2$. We have $\#\tref{\cP}{def:F2_n2-tangent}=1$, and $\tref{\cP}{def:F2_n2-tangent}$ is realized over $\Q$. 
\end{conf}
\begin{proof}
	Clearly, $\cc_2$ is unique, and $\ll_1+\ll_2+\cc_1+\cc_2$ is realized over $\Q$. We need to show that $(\cc_2\cdot \cc_{1})_{p'}=2$. 
	The pencil of conics tangent to $\ll_{1}$ at $p$ and passing through $p_{2}$, $p'$ induces a map $\cc_{1}\to \P^{1}$ of degree two, which by Hurwitz formula is ramified only at the points corresponding to the degenerate members $\ll_{1}+\ll_{p_{2}p'}$ and $\ll_{2}+\ll_{pp'}$, so it is not ramified at the point corresponding to $\cc_{2}$.
\end{proof}

\begin{conf}[Figure \ref{fig:nodal-cubic_P2_conf}]\label{conf:nodal-cubic_P2}
	Write $\ll_{1}\cap \ll_{2}'=\{r_1\}$, $\ll_{2}\cap \ll_{1}'=\{r_2\}$.  
	Let $\ll\neq \ll_{1}$ be the line tangent to $\cc_{1}$ passing through $r_{1}$.
	The line $\ll_{r_1 r_2}$ meets $\cc_1$ in two points, which are interchanged by $\tau$.
	
	We have $\#\tref{\cP}{def:nodal-cubic-P2_un}=2$ and $\#\tref{\cP}{def:nodal-cubic_P2}=\#\tref{\cP}{def:C**_1}=\#\tref{\cP}{def:C**_1a}=1$. The configurations in  $\tref{\cP}{def:nodal-cubic-P2_un},\tref{\cP}{def:nodal-cubic_P2}$ are realized over $\Q(\omega)$, and $\tref{\cP}{def:C**_1},\tref{\cP}{def:C**_1a}$ are realized over $\Q$.
\end{conf}
\begin{proof}
	The configuration in $\tref{\cP}{def:C**_1a}$ consists of four lines in a general position, so it is clearly unique and realized over $\Q$. The configuration in $\tref{\cP}{def:C**_1}$ is obtained from it by adding a line through two singular points, so it is again unique and realized over $\Q$. It remains to prove the assertions about $\tref{\cP}{def:nodal-cubic-P2_un}$ and $\tref{\cP}{def:nodal-cubic_P2}$. 
	
	We have $r_1=[1:1:0]$, $r_2=[0:1:1]$, so $\ll_{r_1 r_2}=\{y=x+z\}$. This line intersects the conic $\cc_1=\{y^2=xz\}$ in two points $[\omega^{-1}:-1:\omega]$ and $[\omega:-1:\omega^{-1}]$, which are interchanged by $\tau\colon [x:y:z]\mapsto [z:y:x]$. Thus the marked configuration $\tref{\pp}{def:nodal-cubic_P2}\de \cc_1+\ll_1+\ll_2+\ll_1'+\ll_2'+\ll_{r_1r_2}$ is unique up to a projective equivalence, and realized over $\Q(\omega)$. We have $\ll=\{z=4(y-x)\}$ and $\ll\cap \cc_1=\{[1:2:4]\}$, so $\tref{\pp}{def:nodal-cubic-P2_un}\de \tref{\pp}{def:nodal-cubic_P2}+\ll$ is realized over $\Q(\omega)$, too, but since $\tau(\ll)\neq \ll$, the two markings on $\tref{\pp}{def:nodal-cubic-P2_un}$ are not projectively equivalent. Hence $\#\tref{\cP}{def:nodal-cubic-P2_un}=2$.
\end{proof}

\begin{conf}[Figure \ref{fig:P2n2_cuspidal_conf}]\label{conf:P2n2_cuspidal}
	Write $\ll_{pp'}\cap\cc_{1}=\{p',q\}$ and $\ll_{qp_{2}}\cap\ll_{1}=\{r\}$. We have $\#\tref{\cP}{def:P2n2_cuspidal}=\#\tref{\cP}{def:C**_2}=1$, and $\tref{\cP}{def:P2n2_cuspidal},\tref{\cP}{def:C**_2}$ are realized over $\Q$. \qed
\end{conf}

\begin{figure}[htbp]
	\captionsetup{width=.3\linewidth}
\begin{minipage}[b]{.32\textwidth}
	\centering
	\begin{tikzpicture}
		\path[use as bounding box] (-1.7,-1) rectangle (1.7,2.6);
		\draw[name path=c1] (0,0) circle (1);
		\coordinate (P) at (0,2); 
		\node at ($(P)-(0.2,0)$) {\small{$p$}}; \filldraw (P) circle (0.06);
		\coordinate (P') at (0,-1);
		\node at ($(P')+(0,0.3)$) {\small{$p'$}}; \filldraw (P') circle (0.06);
		\coordinate (P1) at (-0.866,0.5);
		\node at ($(P1)+(-0.2,0.1)$) {\small{$p_1$}}; \filldraw (P1) circle (0.06);
		\coordinate (P2) at (0.866,0.5);
		\node at ($(P2)+(0.3,0.1)$) {\small{$p_2$}}; \filldraw (P2) circle (0.06);
		\node at (-1.5,-0.2) {\small{$\ll_{1}$}};
		\node at (1.5,-0.2) {\small{$\ll_{2}$}};
		\node at (0.9,-0.8) {\small{$\cc_{1}$}};
		\draw[add= 0.4 and 0.8, name path = L1] (P) to (P1);
		\draw[add= 0.4 and 0.8, name path = L2] (P) to (P2);
		\draw[rotate around={-10:(0.3,0.8)}, blue] (0.3,0.8) ellipse (0.63 and 1.81);
		\node at ($(P)+(1.2,0)$) {\small{\blue{$\cc_2$}}};
	\end{tikzpicture}
	\caption{Configuration \ref{conf:F2_n2-tangent}, yields \ref{def:F2_n2-tangent}, cf.\ Figure \ref{fig:F2_n2-tangent}.}
	\label{fig:F2_n2-tangent_conf}
\end{minipage}
\begin{minipage}[b]{.33\textwidth}
	\centering
	\begin{tikzpicture}
		\path[use as bounding box] (-2.2,-1.3) rectangle (2.2,3.7);
		\draw[name path=c1] (0,0) circle (1);
		\coordinate (P) at (0,2); 
		\node at ($(P)+(0,0.3)$) {\small{$p$}}; \filldraw (P) circle (0.06);
		\coordinate (R1) at (-1.732,-1); 
		\node at ($(R1)+(-0.2,0.2)$) {\small{$r_1$}}; \filldraw (R1) circle (0.06);
		\coordinate (R2) at (1.732,-1);
		\coordinate (P1) at (-0.866,0.5);
		\node at ($(P1)+(-0.2,0.1)$) {\small{$p_1$}}; \filldraw (P1) circle (0.06);
		\coordinate (P2) at (0.866,0.5);
		\node at ($(P2)+(0.2,0.3)$) {\small{$p_2$}}; \filldraw (P2) circle (0.06);
		\node at (-0.5,1.5) {\small{$\ll_{1}$}};
		\node at (0.5,1.5) {\small{$\ll_{2}$}};
		\node at (1.1,-0.6) {\small{$\cc_{1}$}};
		\node at (0,0.3) {\small{$\ll_{2}'$}};
		\draw[add= 0.05 and 0.05, name path = L1] (P) to (R1);
		\draw[add= 0.55 and 0.05, name path = L2] (P) to (R2);
		\draw[add= 0.05 and 0.05, name path = L] (R1) to (R2);
		\draw[add= 0.05 and 0.05, name path=LP'] (R1) to (P2);
		\path [name intersections={of=c1 and LP'}] (intersection-2) coordinate (P'); 
		\node at ($(P')+(0.2,0.3)$) {\small{$p'$}}; \filldraw (P') circle (0.06);
		\draw[add= 0.7 and 3.2, name path=L1'] (P') to (P1);
		\node at (-0.6,2.3) {\small{$\ll_{1}'$}};
		\path [name intersections={of=L2 and L1'}] (intersection-1) coordinate (R2');
		\node at ($(R2')+(0.3,0)$) {\small{$r_2$}}; \filldraw (R2') circle (0.06); 
		\node at (0.9,-1.2) {\small{$\ll$}};
		\draw[add= 0.05 and 0.05, name path=L12] (R1) to (R2');
		\draw[name path=c1'] (-1.6,1.5) [partial ellipse=60:-120: 0.5 and 0.5];
		\node at (-1.6,1.2) {\small{$\cc_1$}};
	\end{tikzpicture}
	\caption{Configuration \ref{conf:nodal-cubic_P2} yields \ref{def:nodal-cubic-P2_un},  \ref{def:nodal-cubic_P2}, \ref{def:C**_1}, \ref{def:C**_1a}, cf.\ Figures \ref{fig:nodal-cubic-P2_un}, \ref{fig:nodal-cubic_P2}, \ref{fig:C**_1}, \ref{fig:C**_1a}.}
	\label{fig:nodal-cubic_P2_conf}
\end{minipage}
\begin{minipage}[b]{.32\textwidth}
	\centering
	\begin{tikzpicture}
		\path[use as bounding box] (-1.8,-1.4) rectangle (1.8,3.8);
		\draw[name path=c1] (0,0) circle (1);
		\coordinate (P) at (0,2); 
		\node at ($(P)+(0.2,0)$) {\small{$p$}}; \filldraw (P) circle (0.06);
		\coordinate (P') at (0,-1);
		\node at ($(P')+(0.3,-0.2)$) {\small{$p'$}}; \filldraw (P') circle (0.06);
		\coordinate (P1) at (-0.866,0.5);
		\node at ($(P1)+(-0.2,0.2)$) {\small{$p_1$}}; \filldraw (P1) circle (0.06);
		\coordinate (P2) at (0.866,0.5);
		\node at ($(P2)+(0.2,0.2)$) {\small{$p_2$}}; \filldraw (P2) circle (0.06);
		\node at (-1.5,-0.2) {\small{$\ll_{1}$}};
		\node at (1.5,-0.2) {\small{$\ll_{2}$}};
		\node at (0.9,-0.8) {\small{$\cc_{1}$}};
		\draw[add= 0.1 and 0.8, name path = L1] (P) to (P1);
		\draw[add= 1.1 and 0.8, name path = L2] (P) to (P2);
		\draw[add= 0.1 and 0.1, name path = L] (P) to (P');
		\path [name intersections={of=L and c1}] (intersection-1) coordinate (Q);
		\node at ($(Q)+(0.2,0.2)$) {\small{$q$}}; \filldraw (Q) circle (0.06);
		\draw[add= 0.9 and 0.35, name path = L'] (Q) to (P2);
		\path [name intersections={of=L' and L1}] (intersection-1) coordinate (R);
		\node at ($(R)+(-0.2,0.3)$) {\small{$r$}}; \filldraw (R) circle (0.06);
		\draw[add= 0.1 and 1.1] (P') to (R);
	\end{tikzpicture}
	\caption{Configuration \ref{conf:P2n2_cuspidal} yields \ref{def:P2n2_cuspidal}, \ref{def:C**_2} cf.\ Figures  \ref{fig:P2n2_cuspidal}, \ref{fig:C**_2}.}
	\label{fig:P2n2_cuspidal_conf}
\end{minipage}
\end{figure}

\section{Uniqueness and automorphisms}\label{sec:uniqueness}

 In this section we prove Corollaries \ref{cor:uniq}\ref{item:n} and \ref{cor:aut}. The first result counts the numbers of non-isomorphic \QHPs in Theorem \ref{CLASS} whose minimal log smooth completions $(X,D)$ share the same weighted graph of $D$. 
 Its proof relies on a simple observation that each surface in a tom Dieck--Petrie tower is obtained by a blowup whose center is uniquely determined by the geometry of the preceding surface. 

Nonetheless, the correspondence between isomorphism classes of $\Q$HPs and the combinatorial data 
they are constructed from, is not exactly one to one. Roughly speaking, at each step of the tom Dieck--Petrie algorithm, see Definition \ref{def:TDP}, we get the following ambiguities:
\begin{itemize}
	\item We choose $\pp\subseteq \P^{2}$ and $P\subseteq \Sing \pp$, which may not be unique up to a projective equivalence.
	\item We choose centers of expansions: they may depend not only on the geometry of $\pp$, but also on its marking, e.g.\ when two components of $\pp$ meet transversally in two points, and we choose one but not the other.
	\item Nonetheless, some of the above choices might be covered by automorphisms of $S$,  see Example \ref{ex:7_uniq}. 
\end{itemize}

In order to make the above discussion precise, we need to introduce some notation. We will illustrate it by explicit computations in Examples \ref{ex:7_bis}, \ref{ex:33} below. 

We keep Notation \ref{not:P}. In particular, $\tst{\cP}$ denotes the set of projective equivalence classes of marked configurations of lines and conics, whose combinatorial type is specified by the row $\rst$ of \tables. 

Fix a representative $\pp$ of a class in $\tst{\cP}$. As in Section \ref{sec:strategy}, let $\pi\colon X'\to \P^{2}$ be a minimal resolution of $\pp$; let $E'$ be the sum of $(-1)$-curves in $\pi^{-1}(P)$; and $D'=(\pi^{*}\pp)\redd-E'$. The combinatorial data in row in the last columns of $\rst$ specify centers and weights of an expansion $\psi\colon (X,D)\to(X',D')$. By Remark \ref{rem:expansion}\ref{item:1/v}, they uniquely determine the isomorphism class of $(X,D)$, hence of $S\de X\setminus D$. Thus we get a well defined map 
\begin{equation}\label{eq:Phi}
\Phi\colon \tst{\cP} \ni \pp \mapsto [S] \in \{\mbox{$\Q$-homology planes}\}/[\mbox{isomorphism}].
\end{equation}
To prove Corollary \ref{cor:uniq}\ref{item:n}, we will show that $\#\Phi(\tst{\cP})=\tst{\ngr}$.
\smallskip

Fix $[S]\in \Phi(\tst{\cP})$, so $S$ is a \QHP whose (unique) log smooth completion is constructed by $\theta=\pi\circ\psi\colon (X,D)\to (\P^2,\pp)$ as above. By our construction, $E\de (\theta^{*}\pp)\redd-D$ is a sum of $(-1)$-curves.

Let $\GAut(D)$ be the automorphism group of the weighted graph of $D$. Since the components of $D$ form a basis of $\NS_{\Q}(X)$, every $\sigma\in \GAut(D)$ induces an automorphism of $\NS_{\Q}(X)$ preserving the intersection form. We denote this automorphism by $\sigma$, too. We define a subset $\cG(S)$ of $\GAut(D)$ as
\begin{equation*}
	\cG(S)=\{\sigma\in \GAut(D): \mbox{for every component $L$ of $E$, the class $\sigma[L]$ is represented by a $(-1)$-curve}\}
\end{equation*}

By Lemma \ref{lem:unique_completion} and Proposition \ref{prop:k=2}, the log smooth completion $(X,D)$ of $S$ is unique, hence $\Aut(S)=\Aut(X,D)\subseteq \GAut(D)$. This group acts on $\cG(S)$ by $\Aut(S)\times \cG(S)\ni (\tau,\sigma)\mapsto \tau\circ \sigma \in \cG(S)$.

\begin{lem}\label{lem:fiber_formula}
	We have 
	$\#\Phi^{-1}[S]=\# \cG(S)/\Aut(S)$. Moreover, the action of $\Aut(S)$ on $\cG(S)$ is free.
\end{lem}
\begin{proof}
	Assume that $[S]=\Phi(\pp_{i})$ for $i\in \{1,2\}$. Write $\theta_{i}\colon (X,D)\to(\P^{2},\pp_{i})$, and  $E_{i}=(\theta_{i}^{*}\pp_i)\redd-D$: it is a sum of $(-1)$-curves, one in each connected component $G$ of $\Exc\theta_{i}$. The marking on $\pp_i$ gives an order on the set of components of $E_i$. Moreover, writing $\theta_{i}$ as a sequence of blowups naturally orders the components of each $G$. Thus we get a unique $\sigma\in \cG(S)$ mapping $\Exc\theta_{1}$ to $\Exc\theta_{2}$ which preserves the above order. 
	Applying inductively the universal property of blowing up, we see that the marked configurations $\pp_1$ and $\pp_2$ are projectively equivalent if and only if $\sigma$ comes from an automorphism of $(X,D)$, hence of $S$. 
	
	Thus every marked configuration $\pp\in \Phi^{-1}[S]$ can be obtained from a fixed element of $\Phi^{-1}[S]$ by an action of some $\sigma_{\pp}\in \cG(S)$, and $\pp$ is equivalent to $\pp'$ if and only if $\sigma_{\pp'}=\tau\circ \sigma_{\pp}$ for some $\tau\in \Aut(S)$, as needed.
	
	This proves the first statement. To prove the second one, assume that $\tau\circ \sigma=\sigma$ for some $\tau\in \Aut(S)$, $\sigma\in \cG(S)$. Then $\tau$ induces the trivial automorphism of the weighted graph of $D$, hence the identity on $\NS_{\Q}(X)$. In particular, $\tau$ maps $D+E$ to itself componentwise. It follows that $\tau$ descends to an automorphism of $\P^2$ preserving the marked configuration $\pp$. Thus it fixes a quadruple of points in general position, namely $(p,p_{1},p_{2},p')$ as in Notation \ref{not:C1}. Therefore, $\tau=\id$.
\end{proof}

\begin{ex}
	\label{ex:7_bis}
	Let $S$ be as in \ref{def:F2_n1-cusp}. By Proposition \ref{prop:conf_uniq}\ref{item:conf_number}, we have $\#\tref{\cP}{def:F2_n1-cusp}=2$. We will compute $\Aut(S)$ and show that  $\#\Phi(\tref{\cP}{def:F2_n1-cusp})=1$, i.e.\ for each center, $S$ is unique up to an isomorphism.
	
	 Assume first that our center is $(C_1,C_2;1)$, so $S$ is as in Example \ref{ex:7}, see Figure \ref{fig:7}. In Example \ref{ex:reduction}, we have seen that there is a morphism $\tau\colon (X,D)\to (\P^2,\Qb)$, where $\Sing \Qb=\{q_1,q_2,q_3\}$ are cusps of multiplicity sequences $(2,2)$. The group $\GAut(D)\cong S_3$ permutes $\{q_1,q_2,q_3\}$.
	 
	 We claim that $\cG(S)=\GAut(D)$. Figure \ref{fig:7} shows that  $E=\Exc\psi+E_{p}+E_{p_1}+E_{r}$, and $\tau$ maps these curves to $\ll'_{3}$, $\cc'_{1}$, $\ll'_{13}$ and $\cc'_{2}$; respectively. Here $\ll_{j}'$ is the line tangent to $\Qb$ at $q_j$, $\cc_{j}'$ is the conic passing through $q_j$ and tangent to $\Qb$ at the remaining cusps, and $\ll_{ij}'$ is the line joining $q_i$ with $q_j$. Hence for any component $E_0$ of $E$ and any $\sigma\in \GAut(D)$, the class $\sigma[E_0]\in \NS_{\Q}(X)$ is represented by some $\ll_{j}'$, $\cc_{j}'$ or $\ll_{ij}'$, which is a $(-1)$-curve. Thus $\cG(S)=\GAut(D)=S_3$, as claimed.
	 
	 We claim that $\Aut(S)=\Z_3$. The inclusion $\Aut(S)\supseteq \Z_3$ was shown in \cite[Proposition 2.9]{tDieck_optimal-curves}. We give an independent argument, similar to \cite[Remark 4.8]{PaPe_delPezzo}, but not relying on an explicit parametrization of $\Qb$.
	 
	 By Lemma \ref{lem:fiber_formula}, $\#\Aut(S)\geq \#\cG(S)/\#\Phi^{-1}[S]\geq \#\cG(S)/\#\tref{\cP}{def:F2_n1-cusp}=6/2=3$, so $\Aut(S)=\Z_3$ or $S_{3}$. Suppose $\Aut(S)=S_{3}$. Then some $\sigma\in \Aut(S)$ descends to $\bar{\sigma}\in \Aut(\P^2,\Qb)$ such that $\bar{\sigma}\colon (q_1,q_2,q_3)\mapsto(q_2,q_1,q_3)$. Then $\bar{\sigma}$ fixes three points of $\Qb$, namely $q_3$, $\ll_3'\cap \Qb\reg$, and $\ll_{12}'\cap \Qb\reg$. These points are pairwise distinct: indeed, if $\ll_{12}'\cap\ll_{3}'\cap \Qb\reg=\{p\}$ then we infer from Figure \ref{fig:7} that $(\pi\circ \psi)_{*}(\tau^{-1}_{*}\ll_{12}')$ is a line passing through $r$ and tangent to $\cc_1$ (at $\cc_1\cap \cc_2\setminus \{q,p_1\}$), other than $\ll_{2}$ and $\ll_{p'}$, which is impossible. Let $\nu \colon \P^1\to \Qb$ be the normalization, and let $\sigma^{\nu}\in \Aut(\P^1)$ be the lift of $\bar{\sigma}|_{\Qb}$. Since $\nu$ is a bijection, $\sigma^{\nu}\in \Aut(\P^1)$ fixes three points, hence $\sigma^{\nu}=\id_{\P^1}$. Thus $\bar{\sigma}|_{\Qb}=\id_{\Qb}$, but $\bar{\sigma}(q_1)=q_2$; a contradiction. 
	 \smallskip

	We conclude that, for $S$ obtained by an expansion at $(C_1,C_2;1)$, we have 
	\begin{equation*}
		\cG(S)=\GAut(D)=S_{3},\quad \Aut(S)=\Z_3, \quad\mbox{so by Lemma \ref{lem:fiber_formula}}\quad \#\Phi^{-1}[S]=\# \cG(S)/\Aut(S)=2.
	\end{equation*}

	Assume now that our center is not $(C_1,C_2;1)$. Then $\GAut(D)=\Z_2$ is generated by the involution interchanging the total transforms $[2,3,1,2]$ of $\tau^{-1}(q_j)$, $j\in \{1,2\}$. Hence, as in the previous case, we get $\GAut(D)=\cG(S)$. Moreover, every automorphism of $S$ descends to $\Aut(\P^2,\pp)=\{\id\}$, so
	\begin{equation*}
	\cG(S)=\GAut(D)=\Z_2,\quad \Aut(S)=\{\id\}, \quad\mbox{so by Lemma \ref{lem:fiber_formula}}\quad \#\Phi^{-1}[S]=\# \cG(S)/\Aut(S)=2.
	\end{equation*}
	Thus we have shown that $\#\Phi^{-1}[S]=2$ for every $[S]\in \Phi(\tref{\cP}{def:F2_n1-cusp})$. It follows that 
	\begin{equation*}
	\# \Phi(\tref{\cP}{def:F2_n1-cusp})=\tfrac{1}{2}\#\tref{\cP}{def:F2_n1-cusp}=1,
	\end{equation*}
	that is, each combinatorial datum in \ref{def:F2_n1-cusp} leads to a unique isomorphism class of a \QHP.
\end{ex}

\begin{ex}\label{ex:33}	
	Let $S$ be as in \ref{def:F2_n2-transversal}, see Figure \ref{fig:F2_n2-transversal}. Then $\pp=\cc_1+\cc_1'+\ll_1+\ll_1'\subseteq \P^2$ is as in Configuration \ref{conf:A1A2_3-node}. It admits two possible markings, corresponding to the orders of the pair $\{q,q'\}=\ll_1\cap \cc_2$. Those two marked configurations are not projectively equivalent, in fact, $\#\tref{\cP}{def:F2_n2-transversal}=2$ by Proposition \ref{prop:conf_uniq}\ref{item:conf_number}.
	
	We have $\GAut(D)=\Z_2$ if $S$ is obtained by expansions at $(C_2,L_1,v)$, $(C_2,L_1,v^{-1})$ for some $v>1$; and $\GAut(D)=\{\id\}$ otherwise. Example \ref{ex:33_Aut} below shows that in either case, $\GAut(D)$ is generated by $\Aut(S)$. Conversely, if $\sigma\in \Aut(S)$ fixes $D$ componentwise, then $\sigma$ descends to $\Aut(\P^2,\pp)=\{\id\}$. 
	
	Hence $\Aut(S)=\GAut(D)$. In particular, $\cG(S)=\Aut(S)$, so by Lemma \ref{lem:fiber_formula}, $\#\Phi^{-1}[S]=1$ for every $[S]\in  \Phi(\tref{\cP}{def:F2_n2-transversal})$. It follows that 
	\begin{equation*}
	\# \Phi(\tref{\cP}{def:F2_n2-transversal})=\#\tref{\cP}{def:F2_n2-transversal}=2,
	\end{equation*}	
	that is, each combinatorial datum in \ref{def:F2_n2-transversal} leads to exactly two isomorphism classes of \QHPs.
\end{ex}

\begin{ex}\label{ex:33_Aut}	
	Let again $S$ be as in \ref{def:F2_n2-transversal}, with expansions at $(C_2,L_1,v)$, $(C_2,L_1,v^{-1})$ for some $v>1$. A particular case $v=2$ is shown in Figure \ref{fig:33_Aut}. We will now construct an involution $\iota\in \Aut(S)$ which realizes the generator of $\GAut(D)\cong \Z_2$.
	
	Recall that the pair $(X,D)$ is constructed by expansion from a pair $(X',D')$. The graph of $D'$ is shown in Figure \ref{fig:F2_n2-transversal}. The \enquote{symmetric} choice of weights guarantees that  $\GAut(D)=\GAut(D')$. Thus to construct $\iota$, it is enough to construct an involution $\iota'\in \Aut(X',D')$ realizing the generator of $\GAut(D')$.
	
	To this end, we will construct a birational map $\bar{\pi}\colon (X',D')\map (\P^2,\bar{\pp})$, where $\bar{\pp}\subseteq \P^2$ is a configuration of lines and conics with an involution $\tau\in \Aut(\P^2,\bar{\pp})$, such that $\bar{\pi}^{-1}\circ \tau \circ \bar{\pi}$ extends to the required involution $\iota'$.
	
	The map $\bar{\pi}$ will be a composition $\sigma\circ \pi$, where $\sigma\colon (\P^2,\bar{\pp})\map (\P^2,\bar{\pp})$ is a Cremona map chosen in such a way that $\bar{\pi}=\tilde{\pi}\circ \beta^{-1}$, where $\tilde{\pi}$ is a minimal log resolution of $(\P^2,\bar{\pp})$, and $\beta$ is a blowup at the common point of the branching $(-2)$-curve and the twig $[2]$ of $D'$, see Figure \ref{fig:33_Aut}. The maps $\tilde{\pi}$ and $\beta$ will be $\Z_2$-equivariant.
	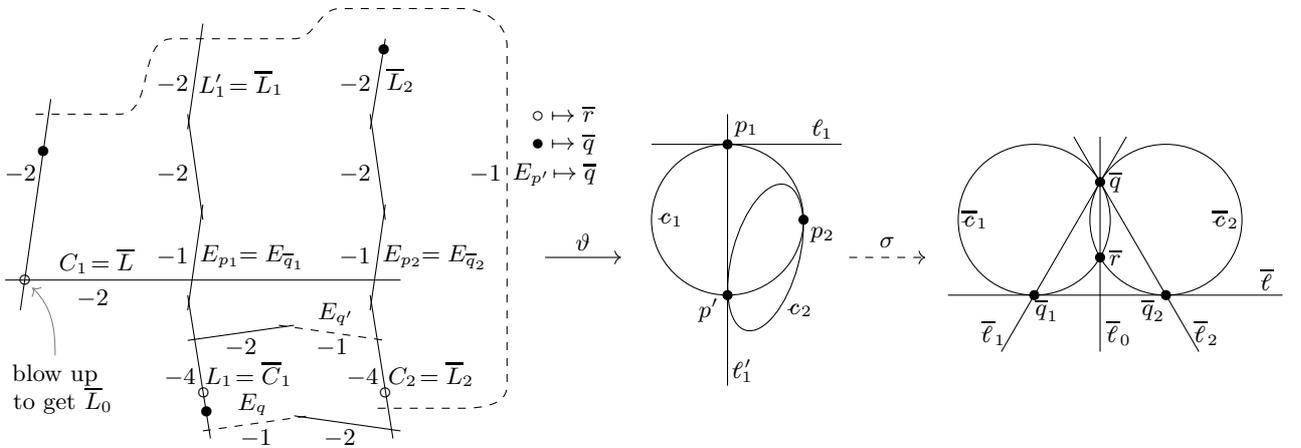
\begin{figure}[htbp]
			\begin{tikzpicture}
			\begin{scope}
				\draw[name path=B] (-2.2,0.8) -- (-1.8,3.6); 
				\node at (-2.2,2.6) {\small{$-2$}};
				\node at (-1.9,2.9) {$\bullet$};
				\draw[dashed] (-2,3.4) -- (-0.9,3.4) to[out=0,in=180] (-0.2,4.4) -- (1.4,4.4)  to[out=0,in=180] (2,4.8) --  (3.6,4.8) to[out=0,in=90] (4.2,4.3) -- (4.2, -0.2) to[out=-90,in=0] (3.6,-0.5) -- (2.4,-0.5);
				\node[left] at (4.25,2.6) {\small{$-1$}};
				\node[right] at (4.15,2.6) {\small{$E_{p'}\! \mapsto \bar{q}$}};
				\node[right] at (4.4,3) {\small{$\bullet\mapsto \bar{q}$}};
				\node[right] at (4.4,3.4) {\small{$\circ\mapsto \bar{r}$}};
				\draw (0.3,-0.9) -- (0,1);
				\node at (-0.1,-0.1) {\small{$-4$}};
				\node at (0.8,-0.05) {\small{$L_1\! =\bar{C}_1$}};
				\node at (0.21,-0.3) {$\circ$};%
				\node at (0.25,-0.55) {$\bullet$};					
				\draw (0,0.8) -- (0.2,2.2); 
				\node at (-0.2,1.5) {\small{$-1$}};
				\node at (0.85,1.5) {\small{$E_{p_1}\!\! =E_{\bar{q}_1}$}};			
				\draw (0.2,2) -- (0,3.4);
				\node at (-0.2,2.6) {\small{$-2$}};
				\draw (0,3.2) -- (0.2,4.6);
				\node at (-0.2,3.8) {\small{$-2$}};
				\node at (0.7,3.8) {\small{$L_{1}'\! =\bar{L}_1$}};	
				\draw (0,0.4) -- (1.4,0.6);
				\node at (0.7,0.3) {\small{$-2$}};
				\draw[dashed] (1.2,0.6) -- (2.6,0.4);
				\node at (1.9,0.3) {\small{$-1$}};
				\node at (1.95,0.75) {\small{$E_{q'}$}};				
				\draw[dashed] (0.2,-0.8) -- (1.6,-0.6);
				\node at (0.9,-0.9) {\small{$-1$}};
				\node at (0.85,-0.45) {\small{$E_q$}};
				\draw (1.4,-0.6) -- (2.8,-0.8);
				\node at (2,-0.9) {\small{$-2$}};			
				\draw (2.7,-0.9) -- (2.4,1);
				\node at (2.3,-0.1) {\small{$-4$}};
				\node at (3.2,-0.05) {\small{$C_2\! =\bar{L}_2$}};
				\node at (2.6,-0.3) {$\circ$};%
				\draw (2.4,0.8) -- (2.6,2.2);  
				\node at (2.2,1.5) {\small{$-1$}};
				\node at (3.25,1.5) {\small{$E_{p_2}\!\! =E_{\bar{q}_2}$}};			
				\draw (2.6,2) -- (2.4,3.4);
				\node at (2.2,2.6) {\small{$-2$}};
				\draw (2.4,3.2) -- (2.6,4.4);
				\node at (2.2,3.8) {\small{$-2$}};
				\node at (2.8,3.85) {\small{$\bar{L}_2$}};
				\node at (2.58,4.25) {$\bullet$};				
				\draw[name path=A] (-2.4,1.2) -- (2.8,1.2);
				\node at (-1.2,1.45) {\small{$C_1\! =\bar{L}$}};
				\node at (-1.25,0.95) {\small{$-2$}};
				\draw[->] (4.7,1.5) -- (5.7,1.5);
				\node at (5.2,1.7) {\small{$\theta$}};
				\path [name intersections={of=A and B}] (intersection-1) coordinate (U);
				\node at (U) {$\circ$};
				\draw[->, gray] ($(U)+(0.4,-1)$) to[out=90,in=-45] ($(U)+(0.1,-0.1)$);
				\node[below] at ($(U)+(0.4,-1)$) {\small{blow up}};
				\node[below] at ($(U)+(0.5,-1.3)$) {\small{to get $\bar{L}_0$}};
			\end{scope}
			\begin{scope}[shift={(7.1,0)}]			
				\draw (-1,3) -- (1.5,3);
				\node at (1.25,3.2) {\small{$\ll_1$}}; 
				\draw (0,3.4) -- (0,-0.2);
				\node at (0.2,0) {\small{$\ll_{1}'$}};
				\draw (0,2) circle (1);
				\node at (-0.75,2) {\small{$\cc_{1}$}};
				\draw[rotate around={-15:(0.5,1.5)}] (0.5,1.5) ellipse (0.44 and 1);
				\node at (0.95,0.8) {\small{$\cc_2$}};
				\filldraw (0,1) circle (0.06);
				\node at (-0.25,0.8) {\small{$p'$}};
				\filldraw (1,2) circle (0.06);
				\node at (1.25,1.8) {\small{$p_2$}};
				\filldraw (0,3) circle (0.06);
				\node at (0.25,3.2) {\small{$p_1$}};
				\draw[dashed, ->] (1.6,1.5) -- (2.6,1.5);
				\node at (2.1,1.7) {\small{$\sigma$}};
			\end{scope}
			\begin{scope}[shift={(12,1)}]
				\draw (-2,0) -- (2.4,0);
				\node at (2.2,0.2) {\small{$\bar{\ll}$}};
				\draw (-0.866,1) circle (1);
				\node at (-1.65,1) {\small{$\bar{\cc}_1$}};
				\draw (0.866,1) circle (1);
				\node at (1.65,1) {\small{$\bar{\cc}_2$}};
				\draw (-1.299,-0.75) -- (0.3464,2.1);
				\node at (-1.4,-0.5) {\small{$\bar{\ll}_1$}};
				\draw (1.299,-0.75) -- (-0.3464,2.1);
				\node at (1.4,-0.5) {\small{$\bar{\ll}_2$}};
				\draw (0,-0.75) -- (0,2.1);
				\node at (0.25,-0.5) {\small{$\bar{\ll}_0$}};
				\filldraw (-0.866,0) circle (0.06);
				\node at (-0.7,-0.2) {\small{$\bar{q}_1$}};
				\filldraw (0.866,0) circle (0.06);
				\node at (0.7,-0.2) {\small{$\bar{q}_2$}};
				\filldraw (0,0.5) circle (0.06);
				\node at (0.2,0.5) {\small{$\bar{r}$}};
				\filldraw (0,1.5) circle (0.06);
				\node at (0.2,1.5) {\small{$\bar{q}$}};
			\end{scope}
			\end{tikzpicture}
	\caption{Construction of $\iota\in \Aut(S)$ in Example \ref{ex:33_Aut} via a symmetric planar divisor. Points $q,q'\in \ll_1\cap \cc_2$, and their images in $\bar{\cc}_1\cap \bar{\cc}_2$ are not visible in the real picture.}
	\label{fig:33_Aut}
\end{figure}
	
	The arrangement $\pp=\cc_1+\cc_2+\ll_1+\ll_1'$ is constructed in Configuration \ref{conf:A1A2_3-node}, and shown in the middle of Figure \ref{fig:33_Aut}. We define $\sigma\colon \P^2\map \P^2$ as a blowup at $p_2$, $p'$ and its infinitely near point $\hat{p}'$ on the proper transform of $\cc_1$, followed by the contraction of the proper transforms of $\ll_{2}'$, $\ll_{p'}$ and the first exceptional curve over $p'$. In other words, $\sigma$ is a quadratic transformation centered at $p_2$, $p'$ and $\hat{p}'$. 
	Let $\bar{\cc}_1$, $\bar{\cc}_2$, $\bar{\ll}$, $\bar{\ll}_1$, $\bar{q}$, $\bar{r}$, 
	and 
	$\bar{\ll}_2$, $\bar{\ll}_0$
	be the images of 
	$\ll_1$, $\cc_2$, $\cc_1$, $\ll_1$, $\ll_{2}'$, $\ll_{p'}$, $p_1$, 
	and of the exceptional curves over
	$p_2$, $p'$. 
	Furthermore, let $\bar{q}_1$ be the image of $p_1$, and let $\bar{q}_2$ be the point infinitely near to $p_2$ on the proper transform of $\cc_2$. Now $\bar{\cc}_1$, $\bar{\cc}_2$ are conics meeting at $\bar{q}$, $\bar{r}$, and at the image of $\ll_1\cap \cc_1$. For $i\in \{1,2\}$, the conic $\bar{\cc}_i$ is tangent to the line $\bar{\ll}$ at $\bar{q}_i$ and to the line $\bar{\ll}_{3-i}$ at $\bar{q}$. The line $\bar{\ll}_0$ passes through points $\bar{q}$, $\bar{r}$. The involution $\tau\in \Aut(\P^2)$ given by $(\bar{q}_1,\bar{q}_2,\bar{q},\bar{r})\mapsto (\bar{q}_2,\bar{q}_1,\bar{q},\bar{r})$ fixes the line $\bar{\ll}_0$ and the whole configuration $\bar{\pp}\de \bar{\cc}_1+\bar{\cc}_2+\bar{\ll}+\bar{\ll}_0+\bar{\ll}_1+\bar{\ll}_2$, shown in the right of Figure \ref{fig:33_Aut}. 
	
	Let $\tilde{\pi}\colon (\tilde{X},\tilde{D}'')\to (\P^2,\bar{\pp})$ be the minimal log resolution. Define $\tilde{D}'$ as $\tilde{D}''$ minus the sum of $(-1)$-curves over $\bar{r}$ and $\bar{q}$. The proper transform of $\bar{\ll}_0$, say $\bar{L}_0$, is  a superfluous $(-1)$-curve in $\tilde{D}'$, which is fixed by the involution $\tilde{\pi}^{-1}\circ\tau\circ \tilde{\pi}$. Let $\beta\colon (\tilde{X},\tilde{D}')\to (\bar{X}',\bar{D}')$ be the contraction of $\bar{L}_0$, and put $\bar{\pi}=\tilde{\pi}\circ \beta^{-1}$. Then $\bar{\pi}^{-1}\circ \sigma \circ\pi$ extends to an isomorphism $(X',D')\to (\bar{X}',\bar{D}')$, and $\bar{\pi}^{-1}\circ \tau\circ \bar{\pi}$ gives the required involution $\iota'$.
	\smallskip
	
	The divisor $D$ (for weight $v=2$) is shown in the left of Figure \ref{fig:33_Aut}, where we use Notation \ref{not:P2} for the proper transforms of components $\cc_1,\ll_1\dots$ of $\pp$ and $\bar{\ll}_1,\bar{\cc}_1,\dots$ of $\bar{\pp}$.  
	Note that over the point $\bar{q}$, there are two $(-1)$-curves: one meets $\bar{L}_1$, and the other meets $\bar{L}_2$. The first one is the same as $E_{p'}$ in Figure \ref{fig:F2_n2-transversal}. The second one is represented by a \enquote{$\bullet$}, see Remark \ref{rem:notation-graphs}. Over the point $\bar{r}$ there is only one $(-1)$-curve, marked by \enquote{$\circ$}: it is the proper transform of the line $\ll_{2}'$ joining the points $p_2,p'\in\pp$. 
	
	We remark that by composing $\bar{\pi}$ with expansions as in \ref{def:F2_n2-transversal}, we get an alternative construction of the tower \ref{def:F2_n2-transversal}. This construction is not exactly as in tom Dieck--Petrie algorithm, because the map $\bar{\pi}$ is not regular.
\end{ex}

\begin{rem}\label{rem:33}
	Let $S$ be as in Example \ref{ex:33_Aut}. It might be tempting to believe that, by symmetry, the choice of the marking on $\pp$, i.e.\ the order of $\ll_{1}\cap \cc_2=\{q,q'\}$, does not matter, hence $S$ is unique up to an isomorphism. This is not true:   
	to get an isomorphism between surfaces obtained from two different markings, we need to find an element of $\GAut(D)$ whose extension to $\NS_{\Q}(X)$ not only maps $(E_{q},E_{q'})$ to $(E_{q'},E_{q})$, but also fixes $E_{p'}$, and there is no such. 
%
%
\end{rem}

\begin{ex}\label{ex:F2n0_Aut}
	Let $S$ be the surface \ref{def:F2n0}, and let $(X,D)$ be its minimal log smooth completion. The graph of $D$ is shown in Figure \ref{fig:F2n0_Aut} below, cf.\ Figure \ref{fig:F2n0}. We will now construct an involution $\iota\in \Aut(X,D)$ such that $\iota(T_1)=T_2$, $\iota(T_2')=T_2'$, where $T_{i}$ and  $T_{i}'$ for $i\in \{1,2\}$ are the subchains of $D$ of type $[2,3,1,2]$ and $[3,1,2]$, respectively, meeting the remaining part of $D$ in the $(-1)$-curve. In Proposition \ref{prop:aut} we will see that the involution $\iota|_{S}$ generates the group $\Aut(S)\cong \Z_2$.
	
	Like in Example  \ref{ex:33_Aut}, we will obtain $\iota$ by constructing a birational map $\bar{\pi}\colon (X,D)\map (\P^2,\bar{\pp})$, where $\bar{\pp}$ is a configuration of lines and conics admitting an involution $\tau\in \Aut(\P^2,\bar{\pp})$; such that $\iota\de \bar{\pi}^{-1}\circ \tau \circ\bar{\pi}$ is regular, see Figure \ref{fig:F2n0_Aut}. This gives an alternative construction of the surface \ref{def:F2n0}. However, since the map $\bar{\pi}$ will not be a morphism, this construction cannot be used to infer tom Dieck--Petrie Conjecture \ref{conj:classical}\ref{item:tDP} in this case.
	\smallskip
	
\begin{figure}[htbp]
	\begin{tikzpicture}	
		\node[right] at (-1,4.9) {\small{$\large{\bullet}\mapsto \bar{q}_1$}};
		\node[right] at (-1,4.5) {\small{$\boldsymbol{\circ}\mapsto \bar{q}_2$}};
		\node[right] at (-1,4.1) {\small{$\spadesuit\mapsto \bar{q}_1'$}};	
		\node[right] at (-1,3.7) {\small{$\heartsuit\mapsto \bar{q}_2'$}};
		\node[right] at (-1,3.3) {\small{$\diamondsuit\mapsto \bar{s}_1$}};
		\node[right] at (-1,2.9) {\small{$\clubsuit\mapsto \bar{s}_2$}};
		\node[right] at (-1,2.5) {\small{$\boldsymbol{\rtimes}\mapsto \bar{p}$}};
		\node[right] at (-1,2.1) {\small{$\boldsymbol{\ltimes}\mapsto \bar{p}'$}};
		\draw[name path=bC1] (-8,1.2) -- (0.2,1.2);
		\node at (0,1.4) {\small{$\bar{C}_1$}};
		\node at (0,1) {\small{$-8$}};
		\node at (-1.4,1.2) {\small{$\spadesuit$}};
		\node at (-1.1,1.2) {\small{$\heartsuit$}};
		\node at (-0.8,1.2) {\small{$\diamondsuit$}};
		\node at (-0.5,1.2) {\small{$\clubsuit$}};
		\draw (-3.4,-0.4) -- (-3.6,1);
		\node at (-3.8,0.3) {\small{$-2$}};
		\draw (-3.6,0.8) -- (-3.4,2.2); 
		\node at (-3.75,1.5) {\small{$-1$}};
		\draw (-3.4,2) -- (-3.7,4.1);
		\node at (-3.8,2.7) {\small{$-3$}};
		\node at (-3.15,2.7) {\small{$\bar{C}_2'$}};
		\node at (-3.47,2.55) {\large{$\bullet$}};
		\node at (-3.52,2.9) {$\boldsymbol{\circ}$};
		\node at (-3.57,3.25) {$\boldsymbol{\ltimes}$};
		\node at (-3.62,3.6) {$\boldsymbol{\rtimes}$};
		\node at (-3.67,3.95) {\small{$\clubsuit$}};		
		\draw (-4.8,-0.4) -- (-5,1);
		\node at (-5.2,0.3) {\small{$-2$}};
		\draw (-5,0.8) -- (-4.8,2.2); 
		\node at (-5.15,1.5) {\small{$-1$}};
		\draw (-4.8,2) -- (-5,3.4);
		\node at (-5.2,2.7) {\small{$-3$}};
		\node at (-4.55,2.7) {\small{$\bar{L}_1$}};
		\node at (-4.87,2.55) {\large{$\bullet$}};
		\node at (-4.92,2.9) {$\boldsymbol{\rtimes}$};
		\draw (-5,3.2) -- (-4.7,5.3);
		\node at (-5.2,3.9) {\small{$-2$}};
		\node at (-4.55,3.9) {\small{$\bar{L}_{2}''$}};
		\node at (-4.75,4.95) {$\boldsymbol{\circ}$};
		\node at (-4.8,4.6) {\small{$\heartsuit$}};
		\node at (-4.85,4.25) {\small{$\ltimes$}};
		\draw (-6.2,-0.4) -- (-6.4,1);
		\node at (-6.6,0.3) {\small{$-2$}};
		\draw (-6.4,0.8) -- (-6.2,2.2); 
		\node at (-6.55,1.5) {\small{$-1$}};
		\draw (-6.2,2) -- (-6.4,3.4);
		\node at (-6.6,2.7) {\small{$-3$}};
		\node at (-5.95,2.7) {\small{$\bar{L}_2$}};
		\node at (-6.27,2.55) {$\boldsymbol{\circ}$};
		\node at (-6.32,2.9) {$\boldsymbol{\rtimes}$};
		\draw (-6.4,3.2) -- (-6.1,5.3);
		\node at (-6.6,3.9) {\small{$-2$}};
		\node at (-5.95,3.9) {\small{$\bar{L}_1''$}};
		\node at (-6.15,4.95) {\large{$\bullet$}};
		\node at (-6.2,4.6) {\small{$\spadesuit$}};
		\node at (-6.25,4.25) {\small{$\ltimes$}};
		\draw (-2,-0.4) -- (-2.2,1);
		\node at (-2.4,0.3) {\small{$-2$}};
		\draw (-2.2,0.8) -- (-2,2.2); 
		\node at (-2.35,1.5) {\small{$-1$}};
		\draw (-2,2) -- (-2.3,4.1);
		\node at (-2.4,2.7) {\small{$-3$}};
		\node at (-1.75,2.7) {\small{$\bar{C}_1'$}};
		\node at (-2.07,2.55) {\large{$\bullet$}};
		\node at (-2.12,2.9) {$\boldsymbol{\circ}$};
		\node at (-2.17,3.25) {$\boldsymbol{\ltimes}$};
		\node at (-2.22,3.6) {$\boldsymbol{\rtimes}$};
		\node at (-2.27,3.95) {\small{$\diamondsuit$}};
		\draw[name path=bC2] (-7.8,0.8) -- (-7.3,4.3); 
		\node at (-7.95,1.5) {\small{$-2$}};
		\node at (-7.45,1.5) {\small{$\bar{C}_{2}$}};
		\node at (-7.35,3.95) {\large{$\bullet$}};
		\node at (-7.4,3.6) {$\boldsymbol{\circ}$};
		\node at (-7.45,3.25) {\small{$\spadesuit$}};
		\node at (-7.5,2.9) {\small{$\heartsuit$}};
		\node at (-7.55,2.55) {\small{$\diamondsuit$}};
		\node at (-7.6,2.2) {\small{$\clubsuit$}};
		\node at (-7.65,1.85) {$\boldsymbol{\rtimes}$};
		\path [name intersections={of=bC1 and bC2}] (intersection-1) coordinate (U);
		\draw[->, gray] ($(U)+(0.4,-1)$) to[out=90,in=-45] ($(U)+(0.1,-0.1)$);
		\node[below] at ($(U)+(0.4,-1)$) {\small{blow up}};
		\node[below] at ($(U)+(0.5,-1.3)$) {\small{to get $\bar{L}$}};
		\draw[->, dashed] (0.9,3.3) -- (2.5,3.3);
		\node at (1.7,3.5) {\small{$\bar{\pi}$}}; 
	\begin{scope}[shift={(5,2)}, scale=0.5];
	\coordinate (P) at (0,4);
	\coordinate (P') at (0,1);
	\draw[add= 0.2 and 1.8, name path = L] (P) to (P');
	\draw[name path = C1] (0,0) ellipse (4 and 1);
	\draw[name path = C2] (0,0) ellipse (2 and 4);
	\coordinate (P1) at (-3.873,0.25); 
	\coordinate (P2) at (3.873,0.25); 
	\draw[add= .2 and .4, name path = L1] (P) to (P1);
	\draw[add= .2 and .4, name path = L2] (P) to (P2);
	\path [name intersections={of=C1 and C2}]
		(intersection-2) coordinate (S1) 
		(intersection-1) coordinate (S2)
		(intersection-4) coordinate (Q1') 
		(intersection-3) coordinate (Q2');
	\filldraw (P) circle (0.1);
	\filldraw (P') circle (0.1);
	\filldraw (S1) circle (0.1);
	\filldraw (S2) circle (0.1);
	\filldraw (Q1') circle (0.1);
	\filldraw (Q2') circle (0.1);
	\filldraw (P1) circle (0.06);
	\filldraw (P2) circle (0.06);
	\draw[add= 1 and 1, name path = L1'] (P') to (Q1');
	\draw[add= 1 and 1, name path = L2'] (P') to (Q2');
	\path [name intersections={of=L1 and L1'}] (intersection-1) coordinate (Q1); 
	\path [name intersections={of=L2 and L2'}] (intersection-1) coordinate (Q2);
	\filldraw (Q1) circle (0.1);
	\filldraw (Q2) circle (0.1);
	\coordinate (R1) at (-0.8,-0.98);
	\filldraw (R1) circle (0.06);
	\coordinate (R2) at (0.8,-0.98);
	\filldraw (R2) circle (0.06);
	\draw[red] (P) to[out=-120,in=0] (Q1) to[out=180,in=135] ($(S1)+(-0.8,0.8)$);
	\draw[red] ($(S1)+(-0.5,0.5)$) -- (S1) to[out=-45,in=100] ($(S1)+(0.6,-0.9)$);
	\draw[red] ($(S1)+(0.7,-1.3)$) to[out=-80,in=175] (R1) to[out=-5,in=-100] ($(R1)+(0.35,0.4)$) to[out=80,in=-120] (P') to[out=60,in=-160] (Q2) to[out=20,in=-45] (3,6) to[out=135,in=60] (P);
	\draw[blue] (P) to[out=-60,in=180] (Q2) to[out=0,in=45] ($(S2)+(0.8,0.8)$);
	\draw[blue] ($(S2)+(0.5,0.5)$) -- (S2) to[out=-135,in=80] ($(S2)+(-0.6,-0.9)$);
	\draw[blue] ($(S2)+(-0.7,-1.3)$) to[out=-100,in=5] (R2) to[out=-175,in=-80] ($(R2)+(-0.35,0.4)$) to[out=100,in=-60] (P') to[out=120,in=-20] (Q1) to[out=160,in=-135] (-3,6) to[out=45,in=120] (P);
	\draw[dotted] ($(Q2')!-1!(P')$) to[out=-135,in=0] (-5,-3.5);
	\draw (-5,-3.5) -- (-6.4,-3.5);
	\draw[dotted] ($(P1)!-0.4!(P)$) to[out=-135,in=90] (-6,-2.5);
	\draw (-6,-2.5) -- (-6,-3.9);
	\draw[dotted] ($(Q1')!-1!(P')$) to[out=-45,in=180] (5,-3.5);
	\draw (5,-3.5) -- (6.4,-3.5);
	\draw[dotted] ($(P2)!-0.4!(P)$) to[out=-45,in=90] (6,-2.5);
	\draw (6,-2.5) -- (6,-3.9);
	\node at (-4.6,-1.1) {\small{$\bar{\ll}_1$}};
	\node at (-3.2,-2.7) {\small{$\bar{\ll}''_2$}};
	\node at (4.7,-1.1) {\small{$\bar{\ll}_2$}};
	\node at (3.1,-2.7) {\small{$\bar{\ll}''_1$}};
	\node at (1.4,-3.9) {\small{$\bar{\cc}_2$}};
	\node at (3.5,-1) {\small{$\bar{\cc}_1$}};
	\node at (0.4,-2.7) {\small{$\bar{\ll}$}};
	\node at (-4,5) {\textcolor{blue}{\small{$\bar{\cc}'_2$}}};
	\node at ($(Q2')+(0.5,-0.4)$) {\small{$\bar{q}'_2$}};
	\node at ($(S1)+(-0.35,-0.4)$) {\small{$\bar{s}_1$}};
	\node at ($(Q1)+(-0.7,0.1)$) {\small{$\bar{q}_1$}};
	\node at (4,5) {\textcolor{red}{\small{$\bar{\cc}'_1$}}};
	\node at ($(Q1')+(-0.45,-0.4)$) {\small{$\bar{q}'_1$}};
	\node at ($(S2)+(0.4,-0.4)$) {\small{$\bar{s}_2$}};
	\node at ($(Q2)+(0.8,0.1)$) {\small{$\bar{q}_2$}};
	\node at ($(P)+(0,0.8)$) {\small{$\bar{p}$}};
	\node at ($(P')+(0.8,0.3)$) {\small{$\bar{p}'$}};
	\end{scope}
	\end{tikzpicture}			
	\caption{Construction of the surface \ref{def:F2n0} in Example \ref{ex:F2n0_Aut} via a symmetric planar divisor.}
	\label{fig:F2n0_Aut}
\end{figure}	
	
	First, we construct the configuration $\bar{\pp}\subseteq \P^2$, shown in the right of Figure \ref{fig:F2n0_Aut} above. We use Notation \ref{not:C1} and coordinates \eqref{eq:coordinates}, but we put a bar over each symbol, so that the notation does not conflict with the one used in Configuration \ref{conf:F2n0} and Figure  \ref{fig:F2n0}. Let us recall this notation here. We fix coordinates $[x:y:z]$ on $\P^2$ and put $\bar{p}'=[1:1:1]$, $\bar{\ll}_1=\{z=0\}$, $\bar{\ll}_2=\{x=0\}$ and $\bar{\cc}_1=\{y^2=xz\}$. The lines $\bar{\ll}_1$, $\bar{\ll}_2$ meet at $\bar{p}\de [0:1:0]$. The conic $\bar{\cc}_1$ passes through $\bar{p}'$ and is tangent to $\bar{\ll}_1$, $\bar{\ll}_2$ at $[1:0:0]$ and $[0:0:1]$, respectively. We define the involution $\tau\colon \P^2\to \P^2$ by $[x:y:z]\mapsto[z:y:x]$. Then  $\tau(\bar{p})=\bar{p}$, $\tau(\bar{p}')=\bar{p}'$, $\tau(\bar{\cc}_1)=\bar{\cc}_1$ and  $\tau(\bar{\ll}_1)=\bar{\ll}_2$.
	
	Let $\bar{q}_1=[3:1:0]\in \bar{\ll}_1$, $\bar{q}_2=\tau(\bar{q}_1)=[0:1:3]\in \bar{\ll}_2$. For $i\in \{1,2\}$ let $\bar{\ll}''_{i}$ be the line joining $\bar{q}_i$ with $\bar{p}'$. We have  $\bar{\ll}''_1=\{3y=x+2z\}$, $\bar{\ll}_2''=\{3y=z+2x\}$. For $i\in \{1,2\}$ write $\bar{\ll}_i''\cap \bar{\cc}_1=\{\bar{p}',\bar{q}_{i}\}$, so $\bar{q}_{1}'=[4:2:1]$, $\bar{q}_2'=[1:2:4]$. Let $\bar{\cc}_2$ be the conic passing through $\bar{p}$, $\bar{q}_1$, $\bar{q}_2$, $\bar{q}_1'$, $\bar{q}_2'$, that is
	\begin{equation*}
		\bar{\cc}_2=\{144\cdot (3y-x-2z)(3y-z-2x)=9\cdot (12y-4x-5z)(12y-4z-5y)\}.
	\end{equation*}
	We have $\bar{\cc}_1\cap \bar{\cc}_2=\{\bar{q}_1',\bar{q}_2',\bar{s}_1,\bar{s}_2\}$, where $\bar{s}_1=[\alpha^2:\alpha:1]$, $\bar{s}_2=[1:\alpha:\alpha^2]$ for $\alpha=\frac{1}{4}(1+\sqrt{-15})$. Now for $i\in \{1,2\}$ let $\bar{\cc}_i'$ be the conic passing through $\bar{p},\bar{p}',\bar{q}_1,\bar{q}_2$ and $\bar{s}_i$. Hence $\bar{\cc}_2'=\tau(\bar{\cc}_1')$ and
	\begin{equation*}
		\bar{\cc}_1'=\{\alpha^2(3\alpha-\alpha^2-2)\cdot z\cdot (3y-z-2x)=(3\alpha-1-2\alpha^2)\cdot x\cdot (3y-x-2z)\}.
	\end{equation*}
	Substituting a parametrization $\P^1\ni [s:t]\mapsto [s^2:st:t^2]\in \P^2$ of $\cc_1$ to the above equation of $\bar{\cc}_1'$, we compute that $\bar{\cc}_1'$ is tangent to $\bar{\cc}_1$ at $[1:-\alpha^2:\alpha^4]\neq \bar{p}',\bar{s}_1$. Applying $\tau$, we get that $\bar{\cc}_2'$ is tangent to $\bar{\cc}_1$ at $[\alpha^4:-\alpha^2:1]$.
	\smallskip
	
	Let $\bar{\ll}$ be the line joining $\bar{p}$ with $\bar{p}'$. We put $\bar{\pp}=\bar{\ll}+\sum_{i=1}^2 \bar{\cc}_i+\bar{\ll}_i+\bar{\ll}_i''+\bar{\cc}_i'$. Then $\tau(\bar{\pp})=\bar{\pp}$ and $\tau(\bar{\ll})=\bar{\ll}$. 
	
	Let $\tilde{\pi}\colon (\tilde{X},\tilde{D}')\to (\P^2,\bar{\pp})$ be the minimal log resolution. Define $\tilde{D}$ as $\tilde{D}'$ minus the $(-1)$-curves over $\bar{p}$, $\bar{p}'$, $\bar{q}_i$, $\bar{q}_i'$, $\bar{s}_i$, $i\in \{1,2\}$. Now $\tilde{\pi}^{-1}\circ\tau\circ \tilde{\pi}$ extends to a regular involution $\tilde{\iota} \in \Aut(\tilde{X},\tilde{D})$
	
	The proper transform of $\bar{\ll}$, say $\tilde{L}$, is a superfluous $(-1)$-curve in $\tilde{D}$. Since $\tau(\bar{\ll})=\bar{\ll}$, we have $\tilde{\iota}(\tilde{L})=\tilde{L}$. Let $\beta\colon (\tilde{X},\tilde{D})\to (\bar{X},\bar{D})$ be the contraction of $\tilde{L}$, and let  $\bar{\pi}=\tilde{\pi}\circ \beta^{-1}\colon \bar{X}\map\P^2$. Now, the birational map $\beta\circ \tilde{\iota}\circ \beta^{-1}=\bar{\pi}^{-1}\circ\tau\circ \bar{\pi}\colon \bar{X}\map \bar{X}$ extends to a  regular involution $\iota\in \Aut(\bar{X},\bar{D})$.
	
	Note that $\iota\in \Aut(\bar{X},\bar{D})$  interchanges the subchains $\bar{T}_i=[2,3,1,2]$ of $\bar{D}$. Indeed, after renaming $\bar{T}_{i}$'s if necessary, for $i\in \{1,2\}$ the subchain $\bar{T}_i$ contains the proper transform of $\bar{\ll}_i+\bar{\ll}_{3-i}''$, see Figure \ref{fig:F2n0_Aut}. Hence $\iota(\bar{T}_1)=\bar{T}_{2}$. Similarly, $\iota$ interchanges the subchains $[3,1,2]$ containing the proper transforms of $\bar{\cc}_{1}',\bar{\cc}_2'$. 
	\smallskip
	
	It remains to show that the pair $(\bar{X},\bar{D})$ is isomorphic to the minimal log smooth completion $(X,D)$ of the surface \ref{def:F2n0}. Note that the weighted graphs of $D$ and $\bar{D}$ agree, see Figures \ref{fig:F2n0} and \ref{fig:F2n0_Aut}. Thus the claim can be inferred from the classification Theorem \ref{CLASS}. We will now give a direct argument.
	
	 Let $\sigma\colon \P^2\map \P^2$ be a standard quadratic transformation centered at $\bar{p}$, $\bar{p}'$, $\bar{s}_2$.  
	Let 
	$\cc_1$, $\cc_2$, $\cc_3$, $\ll_1$, $\ll_2$, $\ll_3$, $\ll_{rq_1}$, $\ll_{rq_2}$, 
	$r$, $q_3$ 
	and 
	 $q_1$, $q_2$, $q_1'$, $q_2'$, $s_1$ be the images of 
	$\bar{\cc}_1$, $\bar{\cc}_2$, $\bar{\cc}_1'$, $\bar{\ll}_2$, $\bar{\ll}_1$, $\bar{\cc}_2'$, $\bar{\ll}_1''$, $\bar{\ll}_2''$,
	$\ll_{\bar{p}\bar{s}_2}$, $\ll_{\bar{p}'\bar{s}_2}$ and  
	$\bar{q}_1$, $\bar{q}_2$, $\bar{q}_1'$, $\bar{q}_2'$, $\bar{s}_1$.
	Then $\pp\de \cc_1+\cc_2+\cc_3+\ll_1+\ll_2+\ll_3+\ll_{rq_1}+\ll_{rq_2}$ is as in Configuration \ref{conf:F2n0}. Since the latter is unique up to a projective equivalence, we can identify it with $\pp$. Recall that the pair $(X,D)$ is constructed in \ref{def:F2n0} as follows: take the minimal log resolution $\pi\colon (X,D')\to (\P^2,\pp)$, and define $D$ as $D'$ minus the $(-1)$-curves over $q_1,q_2,q_3,q_1',q_2',r,s_1$. Now, the map $\pi^{-1}\circ \sigma\circ \bar{\pi}$ extends to the required isomorphism $(\bar{X},\bar{D})\to (X,D)$.
\end{ex}

\begin{rem}\label{rem:F2n0_Aut-on-F2}
	Let again $(X,D)$ be a minimal log smooth completion of the surface \ref{def:F2n0}. Let $\upsilon\colon X\to \F_2$ be the contraction of $D-\bar{C}_1-\bar{C}_2$, where $\bar{C}_1$, $\bar{C}_2$ are as in Figure \ref{fig:F2n0_Aut}. Then $\upsilon(\bar{C}_2)$ is the negative section, and $R\de \upsilon(\bar{C}_1)$ is as in Proposition \ref{prop:n0}\ref{item:n0_F2}, that is, $R$ is a rational curve of type $(1,3)$, with four cusps of multiplicity two, say $p_1,p_2,p_1',p_2'$, such that $\delta_{p_i}=2$, $\delta_{p_i'}=1$, $i\in \{1,2\}$. The involution $\iota\in \Aut(X,D)$ constructed in Example \ref{ex:F2n0_Aut} induces an involution $\hat{\iota}\de \upsilon\circ\iota\circ \upsilon^{-1}\in \Aut(\F_2,R)$ such that $\hat{\iota}(p_1)=p_2$, $\hat{\iota}(p_1')=p_2'$.  
\end{rem}

\begin{rem}\label{rem:F2n0_Q}
	The construction of the surface \ref{def:F2n0} given in Example \ref{ex:F2n0_Aut} is invariant under the action of $\Aut_{\Q}(\Q(\sqrt{-15}))$. Therefore, the surface \ref{def:F2n0} is defined over $\Q$. This fact also follows from Remark \ref{rem:Moe}, which views $S$ as a complement of a curve $\bar{R}\subseteq \P(1,1,2)$ given by an equation with rational coefficients.
\end{rem}

\begin{ex}\label{ex:P2n3}
	Let $S$ be as in \ref{def:P2n3}, see Figure \ref{fig:P2n3}. The group $\GAut(D)$ is nontrivial if and only if $\psi$ is centered at $(L_1,L_2;v)$, $(L_2,L_3;v)$, $(L_3,L_1;v)$ for some $v>1$. We claim that in this case, $\cG(S)=\Aut(S)\cong \Z_3$.
	
	We have $\GAut(D)=S_3$. The configuration $\pp$ consists of a conic $\cc_1$ inscribed a triangle $\ll_1+\ll_2+\ll_3$, so $\Aut(\P^2,\pp)\cong S_3$. Clearly, the $3$-cycle in $\Aut(\P^2,\pp)$ lifts to an automorphism of $(X,D)$, hence of $S$. Suppose that a transposition $\tau\in \GAut(D)$ lies in $\cG(S)$. Say that $\tau(L_1)=L_2$, $\tau(L_3)=L_3$. Then $\tau$ maps the bubble $A$ over $\ll_{1}\cap \ll_3$ to a bubble $\tau(A)$ meeting $D$ in the preimages of $\ll_1\cap \ll_3$ and $\ll_1\cap \ll_3$. The image of $\tau(A)$ on $\P^2$ meets $\pp$ only in the vertices of the triangle $\ll_1+\ll_2+\ll_3$,  so it is disjoint from $\cc_1$; a contradiction.
\end{ex}

We now proceed with the proof of Corollaries \ref{cor:uniq}\ref{item:n} and \ref{cor:aut}. 
Lemma \ref{lem:fiber_formula} shows that to prove Corollary \ref{cor:uniq}\ref{item:n}, we need to describe $\cG(S)$ and $\Aut(S)$ for \QHPs $S$ as in Theorem \ref{CLASS}. This is done in Proposition \ref{prop:aut} below. The following observation will be useful.

\begin{lem}\label{lem:Aut_easy}
	\begin{enumerate}
		\item\label{item:aut_three} Fix $\sigma\in \Aut(X,D)$. Let $C$ be a component of $D$. If $\sigma$ fixes three points of $C$ then $\sigma|_{C}=\id_{C}$.
		\item\label{item:aut_twigs} Fix $\sigma\in \cG(S)$. Let $T$ be a $(-2)$-twig of $D$ and let $A\subseteq E$ be a $(-1)$-curve meeting $T$ only once, in a tip of $D$. Assume that $\sigma(B)=B$ for every component $B$ of $D-T$ meeting $A+T$. Then $\sigma(T)=T$.		
	\end{enumerate}
\end{lem}
\begin{proof}
	Part \ref{item:aut_three} follows from the fact that $C\cong \P^1$. Suppose \ref{item:aut_twigs} fails, so  $\sigma(T)\neq T$. Then $\sigma(T)$ is another $(-2)$-twig of $D$, disjoint from $T$. By definition of $\cG(S)$, the class $\sigma(A)$ is represented by a $(-1)$-curve. Since $\sigma(T)\neq T$, we have $\sigma(A)\neq A$, so $\sigma(A)\cdot A\geq 0$. Let $\phi\colon X\to Z$ be the contraction of $T+\sigma(T)$ and let $\sigma'$ be the induced automorphism of $\NS_{\Q}(Z)$. By assumption, we have $\sigma'(\phi(A))\equiv \phi(A)$ and $A\cdot \sigma(T)=0$. Since $A+T$ is negative definite, we get   $0>\phi(A)^{2}=\sigma'(\phi(A))\cdot \phi(A)\geq \sigma(A)\cdot A\geq 0$; a contradiction.
\end{proof}

We can now give a description of $\Aut(S)$ and $\cG(S)$, which implies Corollary \ref{cor:aut}.

\begin{prop}\label{prop:aut} 
	Let $S$ be one of the $\Q$-homology planes listed in Theorem \ref{CLASS}. Then the group $\Aut(S)$ is nontrivial if and only if one of the following holds.
	\begin{enumerate}
		\item\label{item:quartic} We have $\Aut(S)\cong S_{3}$, and $S$ is constructed in \ref{def:C**_2} by expansions at $\{(C_{1},L_{pp'};2), (E_{p},L_{1},1),(E_{p},L_{2},1)\}$. Here $S$ is the complement of the tricuspidal quartic, see e.g. \cite[Proposition 2.9]{tDieck_optimal-curves} or \cite[Lemma 4.5]{PaPe_delPezzo}.
		\item \label{item:aut_3} We have $\Aut(S)\cong \Z_3$, $S$ is constructed in \ref{def:F2_n1-cusp} by an  expansion at $(C_{1},C_{2},1)$, and $\Aut(S)$ is described in Example \ref{ex:7_bis}, cf.\ Example \ref{ex:7_uniq}.
		\item \label{item:aut_3-1} We have $\Aut(S)\cong \Z_3$, the generator of $\Aut(S)$ is induced by the rotation of $\P^2$, defined using Notation \ref{not:C1} by $(p,p_{1},p_{2},p')\mapsto (\ll_{2}\cap\ll_{p'},p_{2},p',p_{1})$; and $S$ is one of the following:
		\begin{enumerate}[leftmargin=1.5em]
			\item[\ref{def:P2n1-cuspidal}] with any expansion, see \cite[Theorem 1]{tDieck_symmetric_hp}.
			\item[\ref{def:P2n3}] with expansions at  $(L_{1},L_{2};v), (L_{2},L_{3};v), (L_{3},L_{1};v)$ for some $v>1$, see Example \ref{ex:P2n3}.
			\item[\ref{def:C**_1a}] with expansions at $(\ll_{1},\ll_{2};v),(\ll_{2},\ll_{1}';v),(\ll_{1}',\ll_{1};v)$ for some $v>1$.
		\end{enumerate}
		\item\label{item:aut_2} We have $\Aut(S)=\langle \iota\rangle \cong \Z_{2}$, and $(S,\iota)$ is one of the following.
		\begin{enumerate}[leftmargin=1.5em]
			\item[\ref{def:F2n0}.] Here $\iota$ is constructed in Example \ref{ex:F2n0_Aut}.
			\item[\ref{def:F2_n2-transversal}] with expansions at $(C_2,L_1,v)$, $(C_2,L_1,v^{-1})$ for some $v\neq 1$. Here $\iota$ is as in Example \ref{ex:33_Aut}.
			\item[\ref{def:C**_1}] with expansions at  $(L_{1},L_{2};t)$, $(E_{q_{1}},L_{q_{1}q_{2}},u)$, $(L_{j}',L_{j},(u+1)^{3-2j})$, $(L_{3-j},E_{q_{3-j}},t)$ for some $t,u$ and $j\in \{1,2\}$. To see $\iota$, let $\phi\colon X'\to \P^{1}\times \P^{1}$ be the contraction of $L_{q_{1}q_{2}}$. Now $\iota$ comes from an involution of $\P^{1}\times \P^{1}$ mapping vertical lines $\phi(L_{1})$, $\phi(L_{2}')$, $\phi(E_{q_{2}})$ to $\phi(E_{q_{2}})$, $\phi(L_{2}')$, $\phi(L_{1})$; and horizontal lines $\phi(L_{2})$, $\phi(L_{1}')$, $\phi(E_{q_{2}})$ to $\phi(L_{2})$, $\phi(E_{q_{1}})$, $\phi(L_{1}')$.
			\item[\ref{def:C**_2}] with expansions at $(C_{1},L_{pp'},w)$ and, for both $j\in \{1,2\}$, at $(L_{j},E_{p_{j}},v)$ or $(C_{1},E_{p_{j}},v)$ or $(E_{p},L_{j},v)$, for some $v,w$; in the latter case $(v,w)\neq (1,2)$. Here $\iota$ comes from a reflection $\tau\in \Aut(\P^2)$ given by $\tau\colon (p,p_{1},p_{2},p')\mapsto (p,p_{2},p_{1},p')$.
			\item[\ref{def:C**_3}] with expansions at one of: $(L_{1}', E_{p'}; v)$, $(L_{qp'}', E_{p'}; v)$ or $(E_{p'}, C_{1}; v)$, $(E_{p'},C_{2};v)$ or $(C_{2}, L_{1}'; v)$, $(C_{1}, L_{qp'}'; v)$ for some $v$. Here $\iota$ comes from a reflection of $\P^{2}$ mapping  $(p_{1},p_{2},p',q)$ to $(q,p_{2},p',p_{1})$, see Figure \ref{fig:C**_3_conf}.
		\end{enumerate}
	\end{enumerate}
	The quotient $\cG(S)/\Aut(S)$ is trivial in all cases except \ref{def:F2_n1-cusp}, 
	where $\cG(S)/\Aut(S)=\Z_{2}$, see Example \ref{ex:7_bis}.
\end{prop}

\begin{proof}
	We check directly that nontrivial elements of $\GAut(D)$ which satisfy Lemma \ref{lem:Aut_easy}\ref{item:aut_twigs} occur only in towers \ref{def:F2n0}, \ref{def:A1A2_c=3}, \ref{def:F2_n1-cusp}, \ref{def:F2-hor-ccc-41}, \ref{def:P2n1-nodal}, \ref{def:P2n1-cuspidal}, \ref{def:F2_n2-transversal}, \ref{def:nodal-cubic_P2} and \ref{def:P2n3}--\ref{def:C**_3}. Cases \ref{def:F2_n1-cusp}, \ref{def:F2_n2-transversal}, \ref{def:P2n3} are settled in Examples \ref{ex:7_bis}, \ref{ex:33_Aut}, \ref{ex:P2n3}. Consider all the remaining ones together. In these cases, we have $\#\tst{\cP}=1$ by Proposition \ref{prop:conf_uniq}\ref{item:conf_number}, so $\cG(S)=\Aut(S)$ by Lemma \ref{lem:fiber_formula}. Recall that $\Aut(S)=\Aut(X,D)$ by Lemma \ref{lem:unique_completion}. Now, the elements of $\Aut(X,D)$ which satisfy Lemma \ref{lem:Aut_easy}\ref{item:aut_three} are exactly the ones listed in the statement. 
\end{proof}

\begin{rem}
	\label{rem:sym_examples}
	Exactly two of the symmetric towers from Proposition \ref{prop:aut} contain \ZHPs, namely \ref{def:P2n1-cuspidal} with $\Aut=\Z_3$ and \ref{def:C**_2} with $\Aut=\Z_2$. These examples were first constructed by tom Dieck in \cite[Theorems 2 and 6]{tDieck_symmetric_hp}. The latter was also obtained in \cite[\S 4]{MiySu-Cstst_fibrations_on_Qhp} by a different method. 
	
	The surface \ref{def:F2n0} is a new example of a symmetric $\Q$-homology plane. All the remaining ones were previously known, see \cite[Remark 4]{tDieck_symmetric_hp}. The towers \ref{def:C**_1a}, \ref{def:P2n3} with $\Aut=\Z_{3}$ were constructed e.g.\ in \cite[$A$, $B$]{tDieck_optimal-curves}. The towers  \ref{def:F2_n2-transversal}, \ref{def:C**_3} with $\Aut=\Z_{2}$ were constructed in \cite[$F,N$]{tDieck_optimal-curves}; the latter also in \cite[$(T3C_{2})$]{MiySu-Cstst_fibrations_on_Qhp}. The remaining surfaces are isomorphic to $\P^2\setminus \bar{E}$, where $\bar{E}$ is a rational tricuspidal curve: in case \ref{def:F2_n1-cusp}, $\bar{E}$ is the quintic $\Qb$, see Example \ref{ex:7}; in case \ref{def:C**_2}, $\bar{E}$ is the quartic as in Proposition \ref{prop:aut}\ref{item:quartic}, see \cite[Remark 3]{tDieck_symmetric_hp}.
\end{rem}

\begin{proof}[Proof of Corollary \ref{cor:uniq}\ref{item:n}]
	We check directly 
	that different combinatorial data from \tables lead to different weighted graphs of $D$. Thus to prove Corollary \ref{cor:uniq}\ref{item:n}, it remains to show that for every row $\rst$ of \tables, we have $\#\Phi(\tst{\cP})=\tst{\ngr}$, see Notation \ref{not:P} and formula \eqref{eq:Phi}.
	
	By Lemma \ref{lem:fiber_formula}, for $[S]\in\Phi(\tst{\cP})$ we have $\#\Phi^{-1}[S]=\#\cG/\Aut(S)$, which by Proposition \ref{prop:aut} equals $2$ for $\rst=$ \ref{def:F2_n1-cusp} and $1$ otherwise. Hence 
	$\#\Phi(\tref{\cP}{def:F2_n1-cusp})=\tfrac{1}{2}\#\tref{\cP}{def:F2_n1-cusp}$ and $\#\Phi(\tst{\cP})=
	\#\tst{\cP}$ for $\rst\neq$ \ref{def:F2_n1-cusp}. By Proposition \ref{prop:conf_uniq}\ref{item:conf_number}, we get $\#\Phi(\tst{\cP})=\tst{\ngr}$, as claimed.
\end{proof}

We conclude this section with the following result, which implies Corollary \ref{cor:Marco}. 

\begin{prop}\label{prop:Marco}
	Let $\tst{\cS}$ be the set of isomorphism classes of $\Q$HPs obtained by tom Dieck--Petrie algorithm from data given by the row $\rst$ of \tables, including fixed weights of expansions. Then the following holds. 
	\begin{enumerate}
		\item\label{item:same-graphs} Let $(X,D)$ be the minimal log smooth completion of a surface $S\in \tst{\cS}$. Then the weighted graph of $D$ does not depend on $S$.
		\item\label{item:number} We have $\#\tst{\cS}=\tst{\ngr}$, where $\tst{\ngr}$ is the number in the column \enquote{$\ngr$} of $\rst$.
		\item\label{item:conj} All surfaces in $\tst{\cS}$ are conjugate by $\Aut_{\Q}(\tst{\kk})$, where $\tst{\kk}$ is a finite field extension of $\Q$, specified in column \enquote{$\kk$} of $\rst$. In particular, their fundamental groups have the same profinite completions.
		\item\label{item:diffeo-pairs} If $\tst{\ngr}=2$ and $\rst\neq$\ref{def:A1A2_q-cn_22},\ref{def:F2_n1-node-3},\ref{def:F2-5} then the two surfaces in $\tst{\cS}$ are complex-conjugate, hence diffeomorphic.
		\item\label{item:diffeo-3} If $\tst{\ngr}=3$ then $\tst{\cS}$ consist of a pair of complex-conjugate surfaces, and a surface defined over $\R$. 
		\item \label{item:diffeo-4} If $\tst{\ngr}=4$ then $\tst{\cS}$ consists of two pairs of complex-conjugate surfaces.
	\end{enumerate}	
\end{prop}
\begin{proof}
	Part \ref{item:same-graphs} follows from the construction of $(X,D)$ by the tom Dieck--Petrie algorithm, see Definition \ref{def:TDP}. Part \ref{item:number} is a restatement of Corollary \ref{cor:uniq}\ref{item:n}, proved above. The remaining parts follow from Proposition \ref{prop:conf_uniq}; for the assertion about fundamental groups see \cite[Corollary 12.11]{AM_etale}.
\end{proof}

\begin{rem}[cf.\ Remark \ref{rem:candidate-Zariski}]\label{rem:Marco}
	In view of Proposition \ref{prop:Marco}, it is natural to pose the following question. Are there two non-diffeomorphic $\Q$HPs of log general type, whose minimal log smooth completion share the same weighted boundary graph?
	
	This question is open. Proposition \ref{prop:Marco} implies that if such $\Q$HPs exist and satisfy the Negativity Conjecture \ref{conj:negativity}, then they are Galois conjugate, and in fact belong to the set $\tst{\cS}$ with $\tst{\ngr}>2$ or $\rst=$ \ref{def:A1A2_q-cn_22}, \ref{def:F2_n1-node-3}, \ref{def:F2-5}. We remark that 
	examples of non-diffeomorphic, but Galois conjugate varieties do exist, they were constructed by various methods e.g.\ in \cite{Serre_conjugate-non-diffeo,Abelson_conjugate-non-diffeo,Shimada_non-homeo}.
\end{rem}

\begin{rem}\label{rem:fake-projectve-planes}
	Projective surface with the same Betti numbers as $\P^2$ are classified in \cite{CS-fake,PY-fake}. Like $\Q$HPs in Proposition \ref{prop:Marco}\ref{item:diffeo-pairs}, those \emph{fake projective planes} come in complex-conjugate pairs, there is $50$ of them. 
\end{rem}

\section{Negativity Conjecture implies Rigidity Conjecture}\label{sec:rig}

In this section we complete the proof of Theorem \ref{thm:conjectures}, by proving that \QHPs in Theorem \ref{CLASS} satisfy the Rigidity Conjecture \ref{conj:classical}\ref{item:strong_rig}, that is, $H^{i}(\lts{X}{D})=0$ for $i\geq 0$. We prove this result in Proposition \ref{prop:strong_rig}, after some preparations. Our argument is based on \cite[Proposition 2.2]{FlZa_cusps_d-3}.

Recall that for a log smooth pair $(V,T)$, the logarithmic tangent sheaf $\lts{V}{T}$ is a subsheaf of $\mathcal{T}_{V}$ consisting of those derivations which preserve the ideal sheaf of $T$. Its cohomology controls the deformation theory of $(V,T)$, see \cite{FZ-deformations,Kawamata_deformations}. Below, we list some of its properties. 

 \begin{lem}
 	\label{lem:rig}  
 	Let $V_{0}$ be a smooth surface and let $(V,T)$ be some log smooth completion of $V_{0}$. Then
 	\begin{enumerate}
 		\item\label{item:rig_bl}
 		The number $h^{2}(\lts{V}{T})$ depends only on $V_{0}$. 
 		\item\label{item:rig_surg}
 		If $L\subseteq T$ is a $(-1)$-curve then $h^{i}(\lts{V}{(T-L)})=h^{i}(\lts{V}{T})$ for all $i\geq 0$. 
		\item\label{item:rig_chi_QHP} If $V_{0}$ is a \QHP then $\chi(\lts{V}{T})=K_{V}\cdot (K_{V}+T)$.
		\item\label{item:rig_weak} If $V_{0}$ is a \QHP of log general type then $\chi(\lts{V}{T})=h^0(2K_{V}+T)\geq 0$.
 	\end{enumerate}
 \end{lem}
\begin{proof}Parts \ref{item:rig_bl},  \ref{item:rig_surg} and \ref{item:rig_chi_QHP} are proved in \cite{FZ-deformations}, Lemma 1.5(5), Proposition 1.7(3) and Lemma 1.3(5), respectively. Part \ref{item:rig_weak} is shown in \cite[Lemma 4.3(i)]{Palka-minimal_models} for $V_0$ being a complement of a planar rational cuspidal curve; however, the argument works in our generality, too. Let us recall it here for completeness. By \ref{item:rig_chi_QHP} and by the Riemann--Roch theorem, we have $\chi(\lts{V}{T})=K_{V}\cdot (K_{V}+T)=\chi(\cO_{V}(2K_{V}+T))$. Let $K_{V}+T=\mathcal{P}+\mathcal{N}$ be the Zariski--Fujita decomposition.  Lemma \ref{lem:no_lines} implies that $\mathcal{N}=\Bk_{T}T'$, where $T'$ is the sum of maximal twigs of $T$, see \cite[Lemma 2.1.(iv)]{Palka-minimal_models}. We have $\lfloor \mathcal{N} \rfloor =\lfloor \Bk_{T}T' \rfloor=0$ by \cite[II.3.5.2]{Miyan-OpenSurf}, so $\lceil \mathcal{P} \rceil=K_{V}+T$. Because $\mathcal{P}$ is big and nef, the Kawamata-Viehweg vanishing theorem \cite[9.1.18]{Lazarsfeld-positivity_2} gives $h^{i}(2K_{V}+T)=h^{i}(K_{V}+\lceil \mathcal{P} \rceil )=0$ for $i>0$, which ends the proof.
\end{proof}

To prove that the Negativity Conjecture \ref{conj:negativity} implies the Rigidity Conjecture  \ref{conj:classical}\ref{item:strong_rig}, we need to show that the number $h^2\de H^2(\lts{X}{D})$ is zero for pairs $(X,D)$ in Theorem \ref{CLASS}. To do this, we use specific $\C^{(t*)}$-fibrations of $X\setminus D$, where $t\leq 3$. If $t\leq 2$, then $h^2=0$ by \cite[Proposition 6.2]{FZ-deformations}. If $t=3$, \cite[Proposition 2.2]{FlZa_cusps_d-3} gives $h^2=0$ provided $D$ is a particular configuration on $X=\F_m$. We will need this result in a more general setting, where the same argument works. For the readers' convenience, we repeat it here in detail.

\begin{lem}[{cf.\ \cite[Proposition 6.2]{FZ-deformations} or  \cite[Proposition 2.2]{FlZa_cusps_d-3}}]
	\label{lem:rig_fibr} Let $S$ be a smooth surface which admits a $\C^{(t*)}$-fibration for some $t\leq 3$. If $t=3$, assume additionally that some fiber has a component $F_0$ isomorphic to $\C^{**}$. Then $H^{2}(\lts{V}{T})=0$ for any log smooth completion $(V,T)$ of $S$.
\end{lem}
\begin{proof}
	By Lemma \ref{lem:rig}\ref{item:rig_bl}, the number $h^{2}(\lts{V}{T})$ does not depend on the choice of the log smooth completion $(V,T)$. Choose one such that the $\C^{(t*)}$-fibration of $S$ extends to a $\P^1$-fibration $\pi$ of $V$; so $F_{\mathrm{gen}}\cdot T=t+1$ for a general fiber $F_{\mathrm{gen}}$ of $\pi$. Let $F$ be the closure of $F_0$ if $t=3$ and $F=F_{\mathrm{gen}}$ if $t<3$. In any case, $F\cdot T>\tfrac{1}{2} F_{\mathrm{gen}}\cdot T$, so $F$ has multiplicity one in the fiber. Hence there is a morphism $\phi\colon (V,T)\to (V',T')$ which contracts the fibers of $\pi$ to $0$-curves, such that $F\not\subseteq \Exc\phi$. Put $F'=\phi(F)$, and denote by $\pi'$ the induced $\P^1$-fibration of $V'$. 
	
	By Serre duality, we have an isomorphism $H^{2}(\lts{V}{T})\cong H^{0}(\Omega_{V}^{1}(\log T)\otimes \cO_{V}(K_{V}))$. Using the  perfect pairing $\Omega^{1}_{V}(\log T)^{\otimes 2} 
	\to \cO_{V}(K_{V}+T)$, we get 
	\begin{equation*}
	\Omega^{1}_{V}(\log T)\cong Hom(\Omega^{1}_{V}(\log T),\cO_{V}(K_{V}+T))\cong \lts{V}{T}\otimes \cO_{V}(K_{V}+T),
	\end{equation*} hence $\Omega^{1}_{V}(\log T) \otimes \cO_{V}(K_{V})\cong \lts{V}{T}\otimes \cO_{V}(2K_{V}+T)$.
	It follows that 
	\begin{equation*}
	H^{2}(\lts{V}{T})\cong H^{0}(\lts{V}{T}\otimes \cO_{V}(2K_{V}+T))\subseteq H^{0}(\lts{V'}{T'}\otimes \cO_{V}(2K_{V'}+T')).
	\end{equation*}
	Suppose the latter space contains a vector field $\xi\neq 0$. 	Let $\xi_0$ be the image of $\xi$ under the map $\mathcal{T}_{V'}\to \cN_{F'/V'}$, where $\cN_{F'/V'}\cong \cO_{F'}$ is the normal bundle to $F'$ in $V'$. Hence  $ \xi_0\in H^{0}(\cN_{F'/V'}\otimes \cO_{F'}((2K_{V'}+T')\cdot F'))=H^{0}(\cO_{\P^1}(t-3))$. If $t<3$ then $\xi_0\equiv 0$. If $t=3$ then by assumption $F'$ is tangent to $T'$ at some $r\in F'$, hence $\xi_0|_{r}=0$, so since $\xi_0\in H^0(\cO_{\P^1})$ is constant, we get $\xi_0\equiv 0$, too.
	
	Therefore, $\xi$ is tangent to $F'$, so $\xi|_{F'}\in H^{0}(\mathcal{T}_{F'}\otimes  \cO_{F'}((2K_{V'}+T')\cdot F'))=H^{0}(\cO_{\P^1}(t-1))$. After shrinking $\pi'(V')$ if needed, we can divide $\xi$ by the equation of $F'$ and assume $\xi|_{F'}\not\equiv 0$. 
	
	Suppose $t<3$. Then $F'$ is a general fiber, so it meets $T'$ transversally, at $t+1$ points. Since $\xi$ is tangent to $T'+F'$, each of these points is a zero of $\xi|_{F'}\in H^0(\cO_{\P^{1}}(t-1))$, a contradiction. 
		
	Therefore, $t=3$, so $\xi|_{F'}$ has two zeros. By definition, $F'$ meets $T'$ in three points, say $p,q,r$. We have $T'\cdot F'=4$, so, say, $(T'\cdot F')_{p}=(T'\cdot F)_{q}=1$, $(T'\cdot F')_{r}=2$. Since $\xi$ is tangent to $T'+F'$, we have $\xi|_{p}=\xi|_{q}=0$. To get a contradiction, we will show that $\xi|_{r}=0$, too. We argue as in \cite[Lemma 6.5]{FZ-deformations}. If $r\in \Sing T'$ then clearly $\xi|_{r}=0$, so assume $r\not\in \Sing T'$. In some local analytic coordinates $(x,y)$ at $r$, we have $F'=\{x=0\}$, $T'=\{x=y^2\}$, and $\xi=f\tfrac{\d}{\d x}+g\tfrac{\d}{\d y}$ for some $f,g\in \C[\![x,y]\!]$. Since $\xi$ is tangent to $F'$ and $T'$, we have $\xi(x)=f=x\cdot a$ and $\xi(y^2-x)=2yg-f=(y^2-x)\cdot b$ for some $a,b\in \C[\![x,y]\!]$. The equality $f=xa$ implies that $f(0,0)=0$. Moreover, substituting $f=xa$ to the equality $2yg-f=(y^2-x)b$, we infer that $y|(a-b)$ and $g=y\cdot \tfrac{1}{2}b+x\cdot\tfrac{a-b}{2y}$, so $g(0,0)=0$, too. Thus $\xi|_{r}=0$, as required.
\end{proof}

\begin{prop}\label{prop:strong_rig}
	Let $S$ be a $\Q$-homology plane of log general type, and let $(X,D)$ be its log smooth completion. Assume that $\kappa(K_{X}+\frac{1}{2}D)=-\infty$. Then for every $i\geq 0$ we have $H^{i}(\lts{X}{D})=0$.
\end{prop}
\begin{proof} 
Put $h^{i}\de h^{i}(\lts{X}{D})$. Recall that $\kappa(X\setminus D)=2$ implies that $h^{0}=0$. Indeed, by \cite[Theorem 11.12]{Iitaka_AG} in this case the group $\Aut(X,D)$ is finite, so $(X,D)$ admits no infinitesimal automorphisms. Hence $h^{2}-h^{1}=\chi(\lts{X}{D})\geq 0$ by Lemma \ref{lem:rig}\ref{item:rig_weak}. It remains to show $h^{2}=0$. 

By Lemma \ref{lem:rig_fibr} we can assume that $S$ has no $\C^{**}$-fibration, so $(X,\tfrac{1}{2}D)$ has an almost minimal model $(X\am,\frac{1}{2}D\am)$ as in Section \ref{sec:classification}. We use the notation introduced there. In particular, $(X',D')\to (X\am,D\am)$ is the minimal log resolution. Lemma \ref{lem:rig} implies that $h^{2}=h^{2}(\lts{X}{(\psi^{*}D')\redd})=h^{2}(\lts{X'}{D'})$. 
\smallskip

Consider the case $X\am\cong \P^{2}$, see Section \ref{sec:P2}. Then $\deg D\am=5$. Assume that $D\am$ has an ordinary node, say $p$. Let $\pi\colon Y\to X'$ be a blowup at the preimage of $p$. Then $h^{2}=h^{2}(\lts{Y}{(\pi^{-1}_{*}D')})$ by Lemma \ref{lem:rig}. The pencil of lines through $p$ pulls back to a $\P^{1}$-fibration of $Y$ with a fiber $F$ such that $F\cdot\pi^{-1}_{*}D'=\pi(F)\cdot D'-2=\deg D\am-2=3$, so $h^{2}=0$ by Lemma \ref{lem:rig_fibr}. Thus we can assume that $D\am$ does not have ordinary nodes. 

In this case, $D\am$ has at most one line. Using condition \eqref{eq:P2_ordinary}, we infer that $D\am=\ll+\qq$, where $\ll$ is a line tangent to a quartic $\qq$ at two points; and $\qq$ has three ordinary singularities. By Lemma \ref{lem:R}\ref{item:R}, the singular points of $\qq$ are cusps. Let $\F_{1}\to \P^{1}$ be a blowup at a cusp $p\in \qq$. The fibration $\pi_{\F_1}$ is given by the pencil of lines through $p$. By Hurwitz formula, neither of these lines is tangent to $\qq\setminus \{p\}$. Hence the line joining $p$ with some other cusp of $\qq$ meets $D\am$ in exactly three points. Thus $h^{2}=0$ by Lemma \ref{lem:rig_fibr}.
\smallskip

Consider the remaining cases together. 
In case $X_{\min}\cong \P(1,2,3)$, let $\phi\colon X\am\to \F_{2}$ be the contraction of the $(-1)$-curve $U\subseteq \Exc\theta$ and the short $(-2)$-twig in $\Delta\am$, see Figure \ref{fig:A1+A2}; and put $R\de \phi_{*}R\am$. We have $\phi_{*}\Delta\am=\Sec_2+F_{0}$ for some fiber $F_0$, such that $F_0\cap R$ is an ordinary node of $R$. In case $X_{\min}\cong \P(1,1,2)$ put $U=0$, $F_0=0$ and $R= R\am$. Now in both cases, the divisor $R\subseteq \F_2$ is of type $(1,3)$. The pair $(X',D'+U)$ is a log smooth completion of $S'\de \F_{2}\setminus (R+\Sec_2+F_{0})$, and $h^{2}=h^{2}(\lts{X'}{(D'+U)})$ by Lemma \ref{lem:rig}\ref{item:rig_surg}.

 Write $R\cap\Sec_2=\{p\}$. For a point $r\in \F_{2}$, we denote by $F_{r}$ the fiber passing through $r$. 

Let $r$ be a singular point of $R\hor\setminus F_{p}$. Since $R$ has no triple points, we have $F_r\not\subseteq R$. Hence $\#(F_{r}\cap R)\leq F_r\cdot R-1=2$. If $\#(F_{r}\cap R)=2$ then by construction $F_{r}\neq F_0$, and since $F_{r}\neq F_{p}$, we have $\#(F_{r}\cap (R\cup \Sec_2))=3$. Hence $F_{r}\cap S'\cong \C^{**}$, so  $h^2=0$ by Lemma \ref{lem:rig_fibr}. Therefore, we can assume that 
\begin{equation}\label{eq:no_C**_strong}
\mbox{for every $r\in \Sing R\hor \setminus F_{p}$ we have $F_r\cap R=\{r\}$.}	
\end{equation}

Suppose $R\hor$ is reducible. Then $R\hor$ contains a $1$-section, say $H_1$. Put $H_2=R\hor-H_1$. Numerical properties of $\F_2$ imply that $H_1\cdot H_2\geq 4$. Moreover, if $H_2$ is irreducible then $\delta_{H_2}\geq 1$; so in any case, $H_2$ is singular. Since all singularities of $R\hor$ are double points, we have $H_1\cap \Sing H_2=\emptyset$. In particular, $R\hor$ has at least two singular points. Every fiber contains at most one singular point of $R\hor$, so $\Sing R\hor \setminus F_p\neq \emptyset$.

Fix $r\in \Sing R\hor \setminus F_p$. Condition \eqref{eq:no_C**_strong} implies that $r$ is a common point of all components of $R\hor$. In particular, $H_2$ is irreducible, $(F_r\cdot H_2)_{r}=2$ and $r\in H_{2}\reg$. It follows that $\Sing H_2\subseteq F_p$. Since $H_2$ is a $2$-section, we have $H_2\cap F_p=\Sing H_2$, so $H_2\cap H_1\subseteq \F_2\setminus F_p$. Thus at every point $r\in H_2\cap H_1$, the fiber $F_r$ is tangent to $H_2$.  Since $H_1$ is a $1$-section, it is not tangent to $F_r$, hence $(H_1\cdot H_2)_{r}=1$ Applying Hurwitz formula to $\pi_{\F_2}|_{H_2}$ we infer that $2\geq \#(H_2\cap H_1)=H_2\cdot H_1\geq 4$; a contradiction.

Therefore, $R\hor$ is irreducible. Suppose $F_{p}\subseteq R$. Then $R\hor$ is of type $(0,3)$, so $\delta_{R\hor}=4$ by Lemma \ref{lem:F2-adjunction}. Moreover, singular points of $R\hor$ do not lie on $F_{p}$.  Condition \eqref{eq:no_C**_strong} implies that for every $r\in \Sing R\hor$, we have $(R\hor \cdot F_{r})_{r}=3$, so $\delta_r=1$. By Lemma \ref{lem:R}\ref{item:R}, $D\am$ has only two nodes, which are mapped to $\Sec_2\cap F_p$ and $R\hor\cap F_p$. Hence the only possible node of $R\hor$ is the one on $F_0$. Since $\delta_{R\hor}=4$, it follows that $R\hor$ has at least $3$ cusps, each tangent to a fiber. This is a contradiction with the Hurwitz formula.

Thus $F_{p}\not\subseteq R\hor$, so $R=R\hor$. Now, $R$ is as in Proposition \ref{prop:n0}. Its last statement  implies that condition \eqref{eq:no_C**_strong} fails in this case; a contradiction.
\end{proof}

\section{Surfaces in Theorem \ref{CLASS} are of log general type}\label{sec:k=2}

In this section, we complete the proof of Theorem \ref{CLASS}, by showing the following proposition.

\begin{prop}\label{prop:k=2}
	Let $(X,D)$ be a pair obtained by tom Dieck--Petrie algorithm from data in one of the rows of \tables. Then the surface $S\de X\setminus D$ is of log general type.
\end{prop}
\begin{proof}
A direct check shows that $D$ is snc-minimal, and all non-branching components of $D$ have negative self-intersection numbers. Hence by Lemma \ref{lem:unique_completion}\ref{item:graph}, $(X,D)$ is the unique minimal log smooth completion of $S$. 

Suppose $\kappa(S)=-\infty$. Then by \cite[III.1.3.2]{Miyan-OpenSurf} (see Proposition \ref{prop:kappa<2}) $S$ admits a $\C^{1}$-fibration over $\C^1$. The fiber at infinity of a minimal completion of such a fibration is a non-branching $0$-curve in $D$, but we have seen that $D$ contains no such; a contradiction.

Suppose now $\kappa(S)\in \{0,1\}$. We check directly that $S$ is not a Fujita surface. Indeed, these surfaces are constructed in \cite[8.64]{Fujita-noncomplete_surfaces} by expansion as in \ref{def:C**_1a}, with weights $(2,2,2)$, or as in \ref{def:C**_1}, with weights $(1,1,2,1)$ and $(1,1,2,2)$. These weights are explicitly excluded in Tables \ref{table:C**} and \ref{table:C**_1}. 

Therefore,  \cite[III.1.7.1]{Miyan-OpenSurf} 
implies that $S$ has a $\C^*$-fibration. By \cite[Proposition 4.2]{PaPe_Cstst-fibrations_singularities} we may choose one which extends to a $\P^{1}$-fibration $p\colon X\to\P^{1}$. Suppose $D$ contains a fiber of $p$, say $F$. Since all $0$-curves in $D$ are branching, \cite[7.5]{Fujita-noncomplete_surfaces} implies that $F=[2,1,2]$ and $D\hor$ meets $F$ once, in the middle component, say $C$. Put $C'=\psi(C)$. If $\psi$ is an isomorphism in some neighborhood of $C$ then $\beta_{D'}(C')=3$ and $C'$ meets two twigs of $D'$ of type $[2]$: we check directly that this is not the case. Hence $(C')^{2}\geq 0$, so $D$ is as in Figure \ref{fig:Cstst}. Suppose $(C')^2=0$. Then $C'$ contains only one center of $\psi$. Since $\beta_{D'}(C')=3$, it follows that  both components of $F\redd-C$ are proper transforms of components of $D'$ meeting $C'$. We check directly (see Figure \ref{fig:Cstst}) that those components are disjoint and branching in $D'$, so each of them contains two other centers of $\psi$. This is impossible, since $\#\Bs\psi^{-1}=n\leq 4$. Thus $(C')^{2}\geq 1$. It follows that $S$ is as in \ref{def:C**_1a} with expansions of weights $2$ and $\tfrac{1}{2}$. Those weights are excluded in Table \ref{table:C**}; a contradiction.

Thus $D$ contains no fibers. 
Since $\kappa(S)\geq 0$, the description of $\C^{*}$-fibrations in \cite[Theorem 4.9(C)]{Palka-classification2_Qhp} implies that $D\hor$ consists of two $1$-sections, say $H_{1}$, $H_{2}$; $\beta\de \beta_{D}(H_{1})=\beta_{D}(H_{2})\geq 3$ and each $H_{i}$ meets $\beta-1$ twigs of $D$, say $T_{i,1},\dots, T_{i,\beta-1}$, such that for $j\in \{1,\dots, \beta-1\}$, the chain $[\rev{T_{1,j}},1,T_{2,j}]$ is a fiber of $p$. In particular, $T_{2,j}=T_{1,j}^{*}$. Here we write $\rev{T}$ for the chain $T$ with the opposite order, and $T^{*}$ for the chain adjoint to $T$, see \cite[Section 4]{Fujita-noncomplete_surfaces} for definitions.

By Lemma \ref{lem:kappa_expansion}\ref{item:kappa_expansion} we can assume that all expansions have weight $1$ (if such weight is allowed).

Consider case  \ref{def:C**_1a}. Then $D$ contains a $(+1)$-curve $L_2'$ meeting all branching components of $D$, so, say, $L_2'=H_1$. Thus $\beta=3$, and, say, $L_{j}\subseteq T_{1,j}$. In particular, $\beta_{D}(L_{j})=2$, so for some $i\in \{1,2\}$, the chain $[1,T_{2,i}^{*}]$ is the preimage of one of the centers of $\psi$. Hence $T_{2,i}^{*}$ is a $(-2)$-chain, so $T_{2,i}$ is irreducible, i.e.\ $T_{2,i}=L_{i}$. Thus the corresponding weight equals one, which is excluded in Table \ref{table:C**}; a contradiction.

Consider case \ref{def:C**_1} with center $(L_1,E_{q_1})$. Then $L_1'$ is the only branching component of $D$; a contradiction. 

Now, consider all the remaining cases. Then $D$ contains a branching $(-1)$-curve $C$ meeting at most one other branching component of $D$. Such $C$ is horizontal, say, $C=H_{1}$. Indeed, suppose $C$ is a component of a fiber $F$. Then $\beta_{D}(C)=3$ implies that $C$ meets $D\hor$, hence $C$ has multiplicity $1$ in $F$. This is possible only if $C$ is a tip of $F$, so $C$ meets $H_{1}$ and $H_{2}$, but the latter are branching in $D$; a contradiction. 

Now we check directly that in each case one of the following conditions hold:
\begin{enumerate}
	\item\label{item:bad_H2} There is another $(-1)$-curve $C'\neq C$ meeting at most one other branching component of $D$ (hence $C'=H_{2}$), and $\beta_{D}(C')\neq \beta_{D}(C)$ or $\beta_{D}(C)=\beta_{D}(C')$, but $C'$ does not meet twigs of type $T_{1,j}^{*}$. 
	\item\label{item:no_H2} There is no branching component of $D$ meeting twigs of type $T_{1,j}^{*}$. 
\end{enumerate}
Either one yields a contradiction. Note that \ref{item:bad_H2} holds e.g. when $R\am$ has at least two cusps, and we take $C$, $C'$ in their resolutions; or when the only twigs meeting $C$, $C'$ are $(-2)$-twigs, at least one reducible. In turn, 
\ref{item:no_H2} holds e.g.\ when $R\am$ has a non-ordinary cusp, and all twigs of $D$ off its resolution are $(-2)$-twigs.

Let us follow this argument more closely in one of the more complicated cases, say \ref{def:P2n2_cuspidal}, see Figure \ref{fig:P2n2_cuspidal}. In all other cases we argue in a similar, straightforward manner. We use Notation \ref{not:P2}, and the same letters for the curves on $X'$ and their proper transforms on $X$. Choose $C=E_{p_{1}}$, so $T_{1,1}=[2]$, $T_{1,2}=[3]$. We can assume that one expansion is centered on $E_{p_{2}}$, for otherwise $C'=E_{p_{2}}$ does not meet a twig of type $[2]^{*}=[2]$, so \ref{item:bad_H2} holds. If the other center is $(E_{p'},L_{rp'})$ then \ref{item:no_H2} holds, and if it is $(C_{1},E_{p'})$ then $C'$ meets a twig of type $[3]$ or $[2,2,2,2]$, and  \ref{item:bad_H2} holds. Hence the other expansion is centered at $(C_{1},L_{rp'})$. We cannot take it with weight $1$, so let us denote this weight by $w$. Now \ref{item:bad_H2} holds for $C'=E_{p'}$: indeed, if $w\not\in \Z$ then $C'$ meets two branching components of $D$, and if $w\in \Z$ then $C'$ meets twigs of type $[2]$ and either $[3,w+1]$ or $[2,2,2,2,w+1]$, neither of which is of type $[3]^{*}=[2,2]$. 
\end{proof}

\section*{Tables}

We now list the combinatorial data $(\pp,P,\cC,\boldsymbol{v})$ needed to construct all $\Q$-homology planes in Theorem \ref{CLASS}. For planar divisors $\pp\subseteq \P^2$ and points in $P\subseteq \Sing \pp$, we use notation introduced in Section \ref{sec:constructions}. In particular, in each case: $\cc_{1}$ is a conic, $\ll_1,\ll_2,\ll_{p'}$ are lines tangent to $\cc_1$ at points $p_1$, $p_2$, $p'$; $\ll_{i}'$ are lines joining $p_i$ with $p'$; and for given points $a,b$, we denote by $\ll_{ab}$ the line joining them. Notation for the remaining components and points of $\pp$ is introduced separately for each particular configuration from Section \ref{sec:constructions}, specified in the column \enquote{config.}. 

The towers are arranged  according to the number $n$ of centers of expansions used in the tom Dieck--Petrie algorithm, see Definition \ref{def:TDP}. Towers with no $\C^{**}$-fibrations are listed in Tables \ref{table:smooth_0}--\ref{table:smooth_23}: in this case, results of Section \ref{sec:classification} show that $n\leq 3$. The $\C^{**}$-fibered case, studied in Section \ref{sec:Cstst} is summarized in Table \ref{table:C**}, here $n\leq 4$. 

In case $n>0$, possible centers and weights $\boldsymbol{v}=(v,w,\dots)$ are listed in the second-to-last column: if the indicated components meet at two points, we chose one of them. The restrictions on weights guarantee that the resulting surface is a $\Q$HP, see Lemma \ref{lem:expansions}\ref{item:expansions_det}. 

The column \enquote{$\ngr$} shows the number of non-isomorphic $\Q$HPs obtained from the same combinatorial data (including fixed weights of the expansions), hence sharing the same weighted graph of $D$, see Corollary \ref{cor:uniq}\ref{item:n}. 

The column \enquote{$\kk$} shows a finite extension of $\Q$ over which the corresponding $\Q$HPs are defined. We denote by $\omega\in \C$ a primitive third root of unity. Recall that by Proposition \ref{prop:Marco}\ref{item:conj}, the $\ngr$-tuples of $\Q$HPs obtained using the same combinatorial data are conjugate by the Galois group of $\kk$. In particular, if $\kk\not\subseteq \R$, we get pairs of complex-conjugate, hence diffeomorphic $\Q$HPs, see Proposition \ref{prop:Marco}\ref{item:diffeo-pairs}--\ref{item:diffeo-4}.

In the column \enquote{$\Z$} we put a \enquote{$\tick$} when the corresponding tower contains infinitely many $\Z$HPs, and a \enquote{$\none$} if it contains none, see Remark \ref{rem:ZHP}. For towers 	\ref{def:F2n2-cuspidal}, \ref{def:F2_n2-tangent}, \ref{def:P2n2_cuspidal}, \ref{def:P2n3}, \ref{def:C**_1}, \ref{def:C**_2}, \ref{def:C**_3} the corresponding symbols are listed in Tables \ref{table:F2n2-cuspidal}--\ref{table:C**_3}. If a cell is entirely crossed out (with a large \enquote{\ \  \diagbox{}{}\slashbox{}{}\ \ }) then the corresponding set of centers does not occur.
	
\medskip

\begin{small}
		\begin{longtable}{c|c|c|c|c|rl|c}\phantomsection\label{table:smooth_0}
			& Fig.
			& $\ngr$
			& $\kk$
			& $\# H_{1}(S;\Z) $
			& config.
			& planar divisor $\pp \subseteq \P^{2}$
			& $P$
			\endhead
			\caption{Smooth $\Q$HPs with $\kappa(K_{X}+D)=2$, $\kappa(K_X+\frac{1}{2}D)=-\infty$, and no $\C^{**}$-fibration, case $n=0$. 
				The field $\tst{\kk}$ is an extension of $\Q$ by a root of a cubic polynomial $\tst{g}$, where $\tref{g}{def:A1A2_c=1}(t)=4t^3-8t^2+4t-1$, $\tref{g}{def:A1A2_c=2_32}(t)=t^3+7t^2+15t+5$, $\tref{g}{def:A1A2_c=2_41}(t)=t^3+3t^2+3t+5$.}
			\endfoot
			\hline
			\nextconstr\label{def:F2n0} & \ref{fig:F2n0}
			& 1 & $\Q$ & $7$
			& \ref{conf:F2n0}
			& $\cc_{1}+\cc_{2}+\cc_{3}+\ll_{1}+\ll_{2}+\ll_{3}+\ll_{rq_{1}}+\ll_{rq_{2}}$
			& $q_{1}, q_{2}, q_{3}, q_{1}', q_{2}',r,s_{1}$ 
			\\ \hline 
			\nextconstr\label{def:A1A2_c=3} & \ref{fig:A1A2_c=3}
			& 1 & $\Q(\sqrt{21})$ & $11$
			& \ref{conf:311}
			& $\cc_{1}+\cc_{2}+\cc_{3}+\ll_{1}+\ll_{2}+\ll_{3}+\ll_{rq_{3}}$
			& $q_{1}, q_{2}, q_{3}, q_{3}',r,s_{2}$ 
			\\ \hline 
			\nextconstr\label{def:A1A2_c=1} & \ref{fig:A1A2_c=1}
			& 3 & $\tref{\kk}{def:A1A2_c=1}$ & $11$
			& \ref{conf:A1A2_c=1}
			& $\cc_{1}+\cc_{2}+\ll_{2}+\ll_{2}'+\ll_{rp'}+\ll_{p_{1}p_{2}}$
			& $p_{1}, p_{2}, p', q, r$ 
			\\ \hline
			\nextconstr\label{def:A1A2_c=2_32} & \ref{fig:A1A2_c=2_32}
			& 3 & $\tref{\kk}{def:A1A2_c=2_32}$ & $11$
			& \ref{conf:A1A2_c=2}\ref{item:32}
			& $\cc_{1}+\cc_{2}+\ll_{1}+\ll_{2}+\ll_{pp'}+\ll+\ll_{r_{1}p'}+\ll_{r_{2}p'}$
			& $p, p', p'', r_{1}, r_{2}, s_{1}, s_{2}$ 
			\\ \hline
			\nextconstr\label{def:A1A2_c=2_41} & \ref{fig:A1A2_c=2_41}
			& 3 & $\tref{\kk}{def:A1A2_c=2_41}$ &  $11$
			& \ref{conf:A1A2_c=2}\ref{item:41}
			& $\cc_{1}+\cc_{2}+\ll_{1}+\ll_{2}+\ll_{pp'}+\ll+\ll_{r_{1}p'}$ 
			& $p, p', p'', r_{1}, r_{2}, s_{1}$ 
	\end{longtable}
\smallskip
	
	\begin{longtable}{c|c|c|c|c|rl|c|cc|c}\hlabel{table:smooth_1}
			& Fig.
			& $\ngr$
			& $\kk$
			& $\Z$
			& config.
			& planar divisor $\pp \subseteq \P^{2}$
			& $P$
			& \multicolumn{2}{c|}{expansion w/weight $v$}
			& \cite{tDieck_optimal-curves}  
			\endhead
			\caption{(continued)}
			\endlastfoot 
			\caption{Smooth $\Q$HPs with $\kappa(K_{X}+D)=2$, $\kappa(K_X+\frac{1}{2}D)=-\infty$, and no $\C^{**}$-fibration, case $n=1$. Here $\omega$ is a primitive third root of unity, $\alpha_{\pm}=(-5\pm 2\sqrt{5})^{1/2}$, $\beta_{\pm}=(1\pm 2\imath)^{1/2}$ and $\zeta_{\pm}=(\omega^{\pm 1}-4)^{1/2}$.}
			\endfoot	
			\hline
			\nextconstr\label{def:A1A2_C2C3-node} & \ref{fig:A1A2_C2C3-node}
			& 4 & $\Q(\sqrt{5},\alpha_{\pm})$ & $\none$ 
			& \ref{conf:4}
			& $\cc_{1}+\cc_{2}+\ll_{2}+\ll_{1}'$
			& $p',p_{2},$
			& $(C_{1},L_{qt})$ 
			& & \\*
			&&& &&&
			$\phantom{\cc_{1}}+\ll_{2}'+\ll_{qt}$& $q,t$ &&&
			\\ \hline
			\nextconstr\label{def:A1A2_C2C3-cusp-41} & \ref{fig:A1A2_C2C3-cusp-41}
			& 1 & $\Q$ & $\tick$
			&\ref{conf:31}\ref{item:A1A2_C2C3-cusp-41}
			& $\cc_{1}+\cc_{2}+\cc_{3}+\ll_{1}+\ll_{qp'}$
			& $p_{2},p',q$
			& $(C_{1},E_{r})$ & $v \neq 4/5$  & \\*
			&&&& $\none$ &&	&& $(E_{r},C_{3})$ &  & \\*
			&&&& $\tick$ &&&& $(C_{3},C_{1})$ & $v\neq 5/6,1$ 
			& \\ \hline
			\nextconstr\label{def:A1A2_C2C3-cusp-32} & \ref{fig:A1A2_C2C3-cusp-32}
			& 1 & $\Q$ & $\tick$
			&\ref{conf:31}\ref{item:A1A2_C2C3-cusp-32}
			& $\cc_{1}+\cc_{2}+\cc_{3}+\ll_{1}+\ll_{qp'}$
			& $p_{2},p',q$
			& $(C_1,E_{r})$ & $v\neq 3/5$ & \\*
			&&&& $\tick$ && && $(C_{1},E_{r'})$ & $v\neq 2/5$ & \\*
			&&&& $\none$ &&&& $(C_3,E_{r})$ &  & \\*
			&&&& $\none$ &&&& $(C_{3},E_{r'})$ &  & 
			\\ \hline
			\nextconstr\label{def:A1A2_q-nnc} & \ref{fig:A1A2_q-nnc}
			& 2 & $\Q(\omega)$ & $\none$
			& \ref{conf:P2n1-nodal}
			& $\cc_{1}+\cc_{2}+\ll_{1}+\ll_{2}$ 
			& $p,p_{1},$
			& $(C_{1},L_{pr})$ &
			& \\* 
			&&&&&& $\phantom{\cc_{1}}+\ll_{q_{2}p_{1}}+\ll_{pr}$ & $q_{2},r$ &&& \\ \hline
			\nextconstr\label{def:A1A2_q-nc} & \ref{fig:A1A2_q-nc}
			& 4 & $\Q(\imath,\beta_{\pm})$ & $\none$
			& \ref{conf:A1A2_q-nc}
			& $\cc_{1}+\cc_{2}+\ll_{2}+\ll_{2}'+\ll_{rq}$
			& $p_{2},r,q$
			& $(C_{1},L_{qr})$ &
			& \\ \hline 
			\nextconstr\label{def:A1A2_3-cusp} & \ref{fig:A1A2_3-cusp}
			& 2 & $\Q(\sqrt{-7})$ & $\none$
			& \ref{conf:31}\ref{item:A1A2_3-cusp}
			& $\cc_{1}+\cc_{2}+\cc_{3}+\ll_{1}$
			& $p_{2},q$
			& $(C_{1},E_{r})$ & $v\neq 7$  & \\*
			&&&& $\none$ &&&& $(C_3,E_{r})$ & & \\* 
			&&&& $\tick$ &&&& $(C_{1},C_{3})$ & $v\neq 9/2$ 
			& \\ \hline
			\nextconstr\label{def:A1A2_3-node} & \ref{fig:A1A2_3-node}
			& 4 & $\Q(\omega,\zeta_{\pm})$ & $\none$
			& \ref{conf:A1A2_3-node}
			& $\cc_{1}+\cc_{2}+\ll_{1}+\ll_{1}'$
			& $p_{1}, p',$
			& $(C_{2},L_{rp_{1}})$ &
			& \\* 
			&&&&&& $\phantom{\cc_{1}}+\ll_{qp'}+\ll_{rp_{1}}$ & $q, r$ &&& \\ \hline
			\nextconstr\label{def:A1A2_q-cn_31} & \ref{fig:A1A2_q-cn_31}
			& 2 & $\Q(\sqrt{-15})$ & $\tick$
			& \ref{conf:A1A2_q-cn}\ref{item:A1A2_q-cn_31}
			& $\cc_{1}+\cc_{2}+\ll_{1}+\ll_{2}+\ll_{rp_{1}}$
			& $p,p_{1},r$
			& $(C_{1},C_{2})$ & $v\neq 8/3, 1$ & \\*
			&&&& $\none$ && && $(C_{1},E_{p'})$ & $v\neq 5$ & \\*
			&&&& $\none$ &&&& $(C_2,E_{p'})$ &  & \\ \hline
			\nextconstr\label{def:A1A2_q-cn_22} & \ref{fig:A1A2_q-cn_22}
			& 2 & $\Q(\sqrt{5})$ & $\tick$
			& \ref{conf:A1A2_q-cn}\ref{item:A1A2_q-cn_22}
			& $\cc_{1}+\cc_{2}+\ll_{1}+\ll_{2}+\ll_{rp_{1}}$
			& $p,p_{1},r$
			& $(C_{1},E_{p'})$ & $v\neq 10/3$ & \\*
			&&&& $\none$ &&&& $(C_{2},E_{p'})$ & & \\*
			\hline
			\nextconstr\label{def:A1A2_2n1c} & \ref{fig:A1A2_2n1c}
			& 2 & $\Q(\sqrt{-7})$ & $\none$
			& \ref{conf:31}\ref{item:A1A2_2n1c}
			& $\cc_{1}+\cc_{2}+\cc_{3}+\ll_{1}+\ll_{qr}$
			& $p',q,r$
			& $(C_{1},L_{qr})$ &
			& \\ \hline
			\nextconstr\label{def:A1A2_21} & \ref{fig:A1A2_21}
			& 1 & $\Q$ & $\none$
			& \ref{conf:31}\ref{item:A1A2_21}
			& $\cc_{1}+\cc_{2}+\cc_{3}+\ll_{1}+\ll_q$
			& $p_{2},q,r$
			& $(C_{1},E_{s})$ & $v\neq 7$ & \\*
			&&&& $\none$ &&&& $(C_3,E_{s})$ & & \\*
			&&&& $\tick$ &&&& $(C_{1},C_{3})$ & $v\neq 9/2$ 
			& \\ \hline
			\nextconstr\label{def:A1A2_2c1n} & \ref{fig:A1A2_2c1n}
			& 2 & $\Q(\sqrt{-7})$ & $\none$
			& \ref{conf:A1A2_2c1n}
			& $\cc_{1}+\cc_{2}+\ll_{1}+\ll_{2}+\ll_{qp_{2}}$
			& $p,p',q,$
			& $(C_{1},L_{sr})$ & &  \\*
			&&&&&& $\phantom{\cc_{1}}+\ll_{sp'}+\ll_{rq}+\ll_{sr}$ & $r,s,t$ &&&
			\\ \hline
			\nextconstr\label{def:F2_n1-node} & \ref{fig:F2_n1-node}
			& 2 & $\Q(\sqrt{-15})$ & $\none$
			& \ref{conf:F2_n1}\ref{item:node}
			& $\cc_{1}+\cc_{2}+\ll_{1}+\ll_{2}+\ll_{r}$
			& $p,p_{1},r$
			& $(C_{1},L_{r})$ &
			& \\ \hline			
			\nextconstr\label{def:F2_n1-cusp} & \ref{fig:F2_n1-cusp}
			& 1 & $\Q(\sqrt{5})$ & $\tick$
			& \ref{conf:F2_n1-cusp}
			& $\cc_{1}+\cc_{2}+\ll_{1}+\ll_{2}+\ll_{p'}$
			& $p,p_{1},r$ 
			& $(C_{1},C_{2})$ & $v\neq 6$  & $L$ \\*
			&&&& $\tick$ &&&& $(C_{1},E_{q})$ & $v\neq 10$  & \\*
			&&&& $\none$ &&&& $(C_{2},E_{q})$ &  &
			\\ \hline
			\nextconstr\label{def:F2_n1-node-3} & \ref{fig:F2_n1-node-3}
			& 2 & $\Q(\sqrt{5})$ & $\none$
			& \ref{conf:F2_n1}\ref{item:node-3}
			& $\cc_{1}+\cc_{2}+\ll_{1}+\ll_{2}$
			& $p,p_{1},$
			& $(C_{1},L_{pq})$ &
			& \\* 
			&&&&&& $\phantom{\cc_1}+\ll_{r}'+\ll_{pq}$ & $q,r$ &&& \\ \hline
			\nextconstr\label{def:F2-5} & \ref{fig:F2-5}
			& 2 & $\Q(\sqrt{5})$ & $\none$
			& \ref{conf:4}
			& $\cc_{1}+\cc_{2}+\ll_{2}+\ll_{1}'+\ll_{2}'$
			& $p,q,p_{2}$
			& $(C_{2},L_{2}')$ &
			& \\ \hline
			\nextconstr\label{def:F2_n1-cusp-hor_32} & \ref{fig:F2_n1-cusp-hor_32}
			& 1 & $\Q$ & $\none$
			&  \ref{conf:F2_n1-cusp-hor}\ref{item:F2_32}
			& $\cc_{1}+\cc_{2}+\ll_{2}+\ll_{1}'+\ll_{2}'$
			& $p_{2},p',q$
			& $(C_{1},E_{s})$ & & \\*
			&&&& $\none$ &&&& $(E_{s},C_{2})$ & $v\neq 1/3$ & \\*
			&&&& $\tick$  &&&& $(C_{2},E_{p_{1}})$ & $v\neq 9/2$ & \\*
			&&&& $\none$ &&&& $(E_{p_{1}},C_{1})$ &   
			& \\ \hline
			\nextconstr\label{def:F2_n1-cusp-hor_41} & \ref{fig:F2_n1-cusp-hor_41}
			& 2 & $\Q(\imath)$ &  $\tick$
			& \ref{conf:F2_n1-cusp-hor}\ref{item:F2_41}
			& $\cc_{1}+\cc_{2}+\ll_{2}+\ll_{1}'+\ll_{2}'$
			& $p_{2},p',q$
			& $(C_{1},C_{2})$ & $v\neq 2/5$ & \\*
			&&&& $\none$ &&&& $(C_{2},E_{p_1})$ & $v\neq 6$ & \\*
			&&&& $\none$ &&&& $(E_{p_1},C_{1})$ &  
			& \\ \hline
			\nextconstr\label{def:F2-5_cn} & \ref{fig:F2-5_cn}
			& 2 & $\Q(\sqrt{-15})$ & $\none$
			& \ref{conf:4}
			& $\cc_{1}+\cc_{2}+\ll_{2}+\ll_{1}'$
			& $p',q,$
			& $(C_{1},L_{qs})$ &
			& \\* 
			&&&&&& $\phantom{\cc_1}+\ll_{rp'}+\ll_{qs}$ & $r,s$ &&& \\ \hline
			\nextconstr\label{def:F2-hor-ccc-41} & \ref{fig:F2-hor-ccc-41}
			& 1 & $\Q(\sqrt{3})$ & $\tick$
			& \ref{conf:F2-hor-ccc-41}
			& $\cc_{1}+\cc_{2}+\ll_{1}+\ll_{2}$
			& $p,p'',$
			& $(C_{1},C_{2})$ & $v\neq 5/2$& $J$\\*
			&&&& $\none$ && $\phantom{\cc_1}+\ll_{pp'}+\ll$ &$q_{1},q_{2}$ & $(C_{2},E_{p'})$ &  & \\*
			&&&& $\none$ &&&& $(E_{p'},C_{1})$ & $v\neq 1/6$ &\\ \hline
			\nextconstr\label{def:P2n1-nodal} & \ref{fig:P2n1-nodal}
			& 1 & $\Q(\sqrt{3})$ & $\none$
			& \ref{conf:P2n1-nodal}
			& $\cc_{1}+\cc_{2}+\ll_{1}+\ll_{2}+\ll_{q_{1}q_{2}}$
			&  $p, q_{1}, q_{2}$
			& $(C_{1},L_{q_{1}q_{2}})$ &
			& $K$ \\ \hline
			\nextconstr\label{def:P2n1-cuspidal} & \ref{fig:P2n1-cuspidal}
			& 1 & $\Q(\omega)$ & $\none$
			& \ref{conf:P2n1-cuspidal}
			& $\cc_{1}+\cc_{2}+\ll_{1}+\ll_{2}+\ll_{3}$
			& $r_1,r_{2},r_{3}$
			& $(C_2,E_{q_{1}})$ & & $I$ \\*
			&&&& $\tick$ &&&& $(C_1,E_{q_{1}})$ & $v\neq 6$  
			&	
		\end{longtable}

		\begin{longtable}{c|c|c|c|c|rl|c|ccc|c}	\hlabel{table:smooth_23}
			&&&&&&&&\multicolumn{3}{c|}{expansions} & \\
			& Fig.\ 
			& $\ngr$
			& $\kk$
			& $\Z$
			& config.
			& planar divisor $\pp \subseteq \P^{2}$
			& $P$
			& weight $v$ & weight $w$ && 
			\cite{tDieck_optimal-curves}
			\endhead
			\caption{Smooth $\Q$HPs with $\kappa(K_{X}+D)=2$, $\kappa(K_X+\frac{1}{2}D)=-\infty$, and no $\C^{**}$-fibration, case $n\geq 2$.}
			\endfoot
			 \hline
			\nextconstr\label{def:nodal-cubic-P2_un} & \ref{fig:nodal-cubic-P2_un}
			& 2 & $\Q(\omega)$ & $\none$ 
			& \ref{conf:nodal-cubic_P2}
			& $\cc_{1}+\ll_{1}+\ll_{2}+\ll_{1}'$
			& $p_{1}, p',$
			& $(L_{r_{1}r_{2}},C_{1})$ & $(L_{2},L)$ & $w\neq 3/2$& \\*
			&&&& $\none$ &&$\phantom{\cc_1}+\ll_{2}'+\ll_{r_{1}r_{2}}+\ll$& $r_{1}, r_{2}$ && $(L,E_{s})$ && \\*
			&&&& $\none$ &&&&& $(E_{s},C_{1})$ & & \\*
			&&&& $\none$ &&&&& $(C_{1},E_{p_{2}})$ & & \\*
			&&&& $\none$ &&&&& $(E_{p_{2}},L_{2})$ & &
			\\ \hline
			\nextconstr\label{def:F2n2-cuspidal} & \ref{fig:F2n2-cuspidal}
			& 1 & $\Q$ &
			& \ref{conf:31}\ref{item:A1A2_21}
			& $\cc_{1}+\cc_{2}+\ll_{1}+\ll_{q}$
			& $q$
			& \multicolumn{3}{c|}{Table \ref{table:F2n2-cuspidal}}
			& $H$ \\ \hline				
			\nextconstr\label{def:F2n2-nodal} & \ref{fig:F2n2-nodal}
			& 2 & $\Q(\imath)$ & $\none$
			& \ref{conf:A1A2_q-nc}
			& $\cc_{1}+\cc_{2}+\ll_{2}+\ll_{2}'$
			& $p_{2}$
			& $(L_{2},C_{2})$ & $(L_{2}',E_{r})$ & & $M$ \\*
			&&&& $\none$ &&&&& $(C_{2},E_{r})$ & & \\*
			&&&& $\none$ &&&&& $(C_{2},E_{p_{1}})$ & & \\*
			&&&& $\none$ &&&&& $(C_1,E_{p_{1}})$ & $vw+4v\neq $ & \\*
			&&&&&&&&&&$\neq 2w-2$ & \\*
			&&&& $\none$ &&&&& $(C_{1},L_{2}')$ & $w\neq 2$ &   \\ \hline
			\nextconstr\label{def:F2_n2-tangent} & \ref{fig:F2_n2-tangent}
			& 1 & $\Q$ &
			& \ref{conf:F2_n2-tangent}
			& $\cc_{1}+\cc_{2}+\ll_{1}+\ll_{2}$
			& $p_{2}$
			& \multicolumn{3}{c|}{Table  \ref{table:F2_n2-tangent}}
			& $E$ \\ \hline
			\nextconstr\label{def:F2_n2-transversal} & \ref{fig:F2_n2-transversal}
			& 2 & $\Q(\omega)$ & $\none$
			& \ref{conf:A1A2_3-node}
			& $\cc_{1}+\cc_{2}+\ll_{1}+\ll_{1}'$
			& $p'$
			& $(C_{2},L_{1})$ & $(C_{2},L_{1})$ & $v\geq w\neq 1$ & $F$ \\*
			&&&& $\none$ &&&& $v\neq 1$ & $(L_{1},E_{p_{1}})$ & & \\*
			&&&& $\tick$ &&&&& $(E_{p_{1}},C_{1})$ & $3wv+2\neq $& \\*
			&&&&&&&&&&$\neq 2v+6w$ & \\ \hline
			\nextconstr\label{def:nodal-cubic_P2} & \ref{fig:nodal-cubic_P2}
			& 1 & $\Q(\omega)$ & $\none$
			& \ref{conf:nodal-cubic_P2}
			& $\cc_{1}+\ll_{1}+\ll_{2}+\ll_{1}'$
			& $p',r_{1}, r_{2}$
			& $(L_{r_{1}r_{2}},C_{1})$ & $(C_1,E_{p_{1}})$ & & $G$ \\*
			&&&&  $\none$&&$\phantom{\cc_1}+\ll_{2}'+\ll_{r_{1}r_{2}}$&&& $(E_{p_{1}},L_{1})$ && \\*
			&&&&  $\none$ &&&&& $(L_{1},L_{2})$ & $w>1$ 
			&  \\ \hline
			\nextconstr\label{def:P2n2_cuspidal} & \ref{fig:P2n2_cuspidal}
			& 1 & $\Q$ & 
			& \ref{conf:P2n2_cuspidal}
			& $\cc_{1}+\ll_{1}+\ll_{2}+\ll_{pp'}$ 
			& $p, q, r$
			& \multicolumn{3}{c|}{Table \ref{table:P2n2_cuspidal}}
			& $D$ \\* 
			&&&&&& $\phantom{\cc_1}+\ll_{qp_{2}}+\ll_{rp'}$ &&&&&
			\\ \hline  \hline	
			\nextconstr\label{def:P2n3} & \ref{fig:P2n3}
			& 1 & $\Q$ &
			& \ref{conf:P2n1-cuspidal}
			& $\cc_{1}+\ll_{1}+\ll_{2}+\ll_{3}$
			&
			&  \multicolumn{3}{c|}{Table \ref{table:P2n3}} & $B$
	\end{longtable}
\smallskip
	
	\begin{longtable}{c|c|c|c|c|rl|c|c|c|c}\hlabel{table:C**}
		& Fig.
		& $\ngr$
		& $\kk$
		& $\Z$
		& config.
		& planar divisor $\pp \subseteq \P^{2}$
		& $P$
		& $n$
		& expansions
		& \cite{tDieck_optimal-curves}
		\endhead
		\caption{Smooth \QHPs of log general type admitting $\C^{**}$-fibrations, see \cite{MiySu-Cstst_fibrations_on_Qhp}.}
		\endfoot
		\hline
		\nextconstr\label{def:C**_1} & \ref{fig:C**_1}
		& 1 & $\Q$ & 
		& \ref{conf:nodal-cubic_P2}
		& $\ll_{1}+\ll_{2}+\ll_{1}'+\ll_{2}'+\ll_{r_{1}r_{2}}$
		&& 4
		& Table \ref{table:C**_1}
		& $A$\\ \hline
		\nextconstr\label{def:C**_1a} & \ref{fig:C**_1a}
		& 1 & $\Q$ & $\none$
		& \ref{conf:nodal-cubic_P2}
		& $\ll_{1}+\ll_{2}+\ll_{1}'+\ll_{2}'$
		&& 3
		& $(\ll_{1},\ll_{2}),\ (\ll_{1}',\ll_{1}),\ (\ll_{2},\ll_{1}')$
		& $A$ \\
		&&&&&&&&& weights $u\geq v \geq w$; $u\geq 1$, & \\
		&&&&&&&&& $\{1\},\{2,\tfrac{1}{2}\}\not\subseteq \{u,v,w\}$ & \\
		&&&&&&&&& $(u,v,w)\neq (2,2,2)$  
		& \\ \hline
		\nextconstr\label{def:C**_2} & \ref{fig:C**_2}
		& 1 & $\Q$ & 
		& \ref{conf:P2n2_cuspidal}
		& $\cc_{1}+\ll_{1}+\ll_{2}+\ll_{pp'}$
		&& 3
		& 
		Table \ref{table:C**_2}
		& $C$ \\ \hline
		\nextconstr\label{def:C**_3} & \ref{fig:C**_3}
		& 1 & $\Q$ & $\none$
		& \ref{conf:31}\ref{item:C**_3}
		& $\cc_{1}+\cc_{2}+\ll_{1}+\ll_{2}'+\ll_{qp'}+\ll_{1}'$
		& $p_{1},p_{2},q$
		& 2
		& Table \ref{table:C**_3}
		& $N$
	\end{longtable}
\smallskip
	
		\begin{longtable}{c||c|c|c|c}
		\caption{Expansions in  \ref{def:F2n2-cuspidal}.}
	\endfoot \hlabel{table:F2n2-cuspidal}
			\diagbox[width=2cm]{$v$}{weight}{$w$} & 
			$(C_{2},C_{1})$ &
			$(C_{1},E_{r})$ &
			$(E_{r},L_{q})$ & 
			$(L_{q},C_{2})$ \\* \hline\hline
			\multirow{2}{*}{$(C_{2},C_{1})$} & \multirow{2}{*}{\diagbox{}{}\slashbox{}{}} & $vw+4\neq 6v$ & $3vw+2w+1\neq 2v$ & $4vw+1\neq 2w, (v,w)\neq (1,1)$\\*
			&& $\tick$ & $\none$ & $\tick$ \\* \hline
			\multirow{2}{*}{$(C_{1},E_{p_2})$} & $vw+3w\neq 3$ & $2vw+6\neq 4v+3w$ & $v(2w+1)\neq 3(w+1)$ &  $2vw\neq 6w+v+3$ \\*
			& $\tick$ & $\tick$ & $\none$ & $\tick$ \\* \hline
			\multirow{2}{*}{$(E_{p_{2}},C_{2})$} & $3vw\neq 3v+1$ & $6(v+1)\neq w(3v+1)$ & $\any$ & $\any$ \\*
			& $\tick$ & $\tick$ & $\none$ &$\none$ 
		\end{longtable}
\smallskip

	\begin{longtable}{c||c|c|c|c|c} %
		\caption{Expansions in  \ref{def:F2_n2-tangent}.}
		\endfoot\hlabel{table:F2_n2-tangent}
			\diagbox[width=2cm]{$v$}{weight}{$w$} & 
			$(C_{1},C_{2})$ &
			$(C_{2},E_{p})$ &
			$(E_{p},L_{1})$ & 
			$(L_{1},E_{p_1})$ &
			$(E_{p_{1}},C_{1})$ \\* \hline\hline
			\multirow{2}{*}{$(C_{1},C_{2})$} & \multirow{2}{*}{\diagbox{}{}\slashbox{}{}} & $2vw+6\neq 3v$ & $3vw+3v\neq 6w+2$ & $3vw+6v\neq 2w$ & $\any$ \\*
			&& $\tick$ & $\tick$ & $\tick$ & $\none$ \\* \hline
			\multirow{2}{*}{$(C_{2},E_{p'})$} & $2vw+2w\neq 2$ & $\any$ & $3vw+3w+v\neq 1$ &  $vw\neq w+6$ & $\any$ \\*
			& $\none$ & $\none$ & $\none$ & $\none$ & $\none$ \\* \hline
			\multirow{2}{*}{$(E_{p'},C_{1})$} & $2vw\neq 2v+1$ & $4vw+6v+2w\neq 3$ & $6vw\neq 2v+3w+3$ & $\any$ & $\any$ \\*
			& $\none$ & $\tick$ & $\tick$ & $\none$ & $\none$ 
		\end{longtable}
\smallskip

\begin{longtable}{c||c|c|c|c} %
	\caption{Expansions in  \ref{def:P2n2_cuspidal}.}
	\endfoot\hlabel{table:P2n2_cuspidal}
		\diagbox[width=2cm]{$v$}{weight}{$w$} & 
		$(L_{rp'},L_{2}), w\neq 1$ &
		$(L_{2},E_{p_2})$ &
		$(E_{p_2},C_{1})$ & 
		$(C_{1},L_{rp'}),\ w\neq 1$ \\* \hline\hline
		$(C_{1},L_{rp'})$ & $\any$ & $vw+3v\neq 3w+6$ & $3vw+1\neq 6w$ & \multirow{2}{*}{\diagbox{}{}\slashbox{}{}}\\*
		$v \neq 1$ & $\none$ & $\tick$ & $\tick$ & \\* \hline
		\multirow{2}{*}{$(L_{rp'},E_{p'})$} & $\any$ & $\any$ & $6vw+6w\neq v+2$ &  $vw+2w\neq 2$ \\*
		& $\none$ & $\none$ & $\tick$ & $\tick$ \\* \hline
		\multirow{2}{*}{$(E_{p'},C_{1})$} & $2vw+4v+w\neq 1$ & $4vw+6v\neq w+3$ & $6vw\neq 2v+3w$ & $2vw\neq 1+2v$ \\*
		& $\tick$ & $\tick$ & $\tick$ &$\tick$ 
	\end{longtable}

\smallskip

\begin{longtable}{c||c|c|c|c|c} %
	\caption{Expansions in \ref{def:P2n3} other than the one at $(L_1,L_2;u)$, $u\neq 1$.}
	\endfoot	\hlabel{table:P2n3}
		\diagbox[width=2cm]{$v$}{weight}{$w$} & 
		$(L_{3},L_{1})$, $w\neq 1$ &
		$(E_{p_2},C_1)$ &
		$(E_{p_2},L_{2})$ & 
		$(E_{p_3},C_1)$ & 
		$(E_{p_3},L_{3})$ \\* \hline\hline
		$(L_{2},L_{3})$ & $u,v> 1$,\ $u\geq v,w$ & \multirow{2}{*}{\diagbox{}{}\slashbox{}{}} & \multirow{2}{*}{\diagbox{}{}\slashbox{}{}} & $2uvw+2uw+u\neq uv+2w+1$ & $uvw+uv+uw\neq v$ \\*
		$v\neq 1$ & $\none$ &&& $\tick$ & $\none$ \\* \hline
		\multirow{3}{*}{$(E_{p_1},C_1)$} & \multirow{3}{*}{\diagbox{}{}\slashbox{}{}} & \multirow{2}{*}{$u>1$} & \multirow{2}{*}{$\any$} & \multirow{2}{*}{$2uvw+uw\neq v+w$} & $4uvw+2uv+2uw+ $ \\*
		&&&&& $\ +\ u+2v\neq 2w+1$ \\*
		&& $\none$ & $\none$ & $\none$ & $\tick$ \\* \hline
		\multirow{2}{*}{$(E_{p_1},L_{1})$} &  \multirow{2}{*}{\diagbox{}{}\slashbox{}{}} & \multirow{2}{*}{\diagbox{}{}\slashbox{}{}} & $u>1$ & $4uvw+2w\neq 2v+1$ & $\any$  \\*
		&&& $\none$ & $\tick$ & $\none$ 
\end{longtable}

\smallskip
	
		\begin{longtable}{c||cc|cc}%
			\caption{Expansions in  \ref{def:C**_1} other than the ones at $(L_{1},L_{2};t)$ and $(E_{q_{1}},L_{q_{1}q_{2}};u)$.}
			\endfoot		\hlabel{table:C**_1}
			\diagbox[width=2cm]{$w$}{weight}{$v$} & 
			\multicolumn{2}{c|}{$(L_{1}',L_{1})\ \ v\neq 1$} &
			\multicolumn{2}{c}{$(L_{1},E_{r_1})$} \\* \hline\hline 
			$(L_{2}',L_{2})$ & 
			$tuvw + tvw+w\neq tv+uw+vw$ & \multirow{3}{*}{\cite[$(UP_{3-1})$]{MiySu-Cstst_fibrations_on_Qhp}}
			& \multirow{2}{*}{$w^{-1}\geq u+1$} & \multirow{3}{*}{\cite[$(UC_{2-1})'$]{MiySu-Cstst_fibrations_on_Qhp}} \\*
			$w\neq 1$ & $(t,u,v,w)\neq (1,1,2,2)$ & & & \\* 
			& $\tick$ && $\none$ & \\* \hline 
			\multirow{3}{*}{$(L_{2},E_{r_2})$} & $v\geq u+1$
			& \multirow{3}{*}{\cite[$(UC_{2-1})'$]{MiySu-Cstst_fibrations_on_Qhp}} 
			& $tw\neq uv+v$ & \multirow{3}{*}{\cite[$(UC_{2-1})$]{MiySu-Cstst_fibrations_on_Qhp}} \\*
			& $(t,u,v,w)\neq (1,1,2,1)$ && $w\geq u+1\geq v^{-1}$ & \\* 
			& $\none$ && $\tick$ & 
		\end{longtable}

\smallskip
	
	\begin{longtable}{c|c||c|c|c}%
		\caption{Expansions in  \ref{def:C**_2} other than the one at $(L_{pp'},C_{1};u)$, $u\neq 1$.}
		\endfoot\hlabel{table:C**_2}
			\cite{MiySu-Cstst_fibrations_on_Qhp} & 
			\diagbox[width=2cm]{$w$}{weight}{$v$} & 
			$(E_{p},L_{2})$ &
			$(L_{2},E_{p_{2}})$ &
			$(C_{1},E_{p_{2}})$
			\\* \hline\hline 
			\multirow{6}{*}{$(TP_{2})$}
			& \multirow{2}{*}{$(E_{p},L_{1})$}
			& $2uv+2uw+2u \neq 1, v\geq w$
			&  \multirow{2}{*}{\diagbox{}{}\slashbox{}{}}
			&  \multirow{2}{*}{\diagbox{}{}\slashbox{}{}}  \\*
			&& $\tick$ && \\* \cline{2-5}
			& \multirow{2}{*}{$(L_{1},E_{p_{1}})$}
			& $2uvw + 4uv + 2uw + 2u \neq w + 2$ 
			& $2uvw + 2uv + 2uw \neq  vw + 2v + 2w + 4, v\geq w$
			&  \multirow{2}{*}{\diagbox{}{}\slashbox{}{}}  \\*
			&& $\tick$ & $\tick$ & \\* \cline{2-5}
			& \multirow{2}{*}{$(C_{1},E_{p_{1}})$}
			& $4uv+2u\neq uw+2$
			& $uvw  + 2uw + 2v + 4\neq 2uv$
			& $v\geq w$  \\*
			&& $\tick$ & $\tick$ & $\none$  \\* \hline
			\multirow{4}{*}{$(TC_{2-1})$}
			& \multirow{2}{*}{$(E_{p},L_{pp'})$}
			& $\any$ & $\any$ & $\any$  \\*
			&& $\none$ & $\none$ & $\none$  \\* \cline{2-5}
			& $(L_{pp'},C_{1})$ & $u>w$ & $u>w$ & $u>w$  \\*
			& $w\neq 1$ & $\tick$ & $\tick$ & $\tick$
		\end{longtable}
\smallskip
	
	\begin{longtable}{c||c|c|c}
			\caption{Expansions in  \ref{def:C**_3} \cite[{$(T3C_2)$}]{MiySu-Cstst_fibrations_on_Qhp}.}
			\endfoot	\hlabel{table:C**_3}
			\diagbox[width=2cm]{$w$}{weight}{$v$} &
			$(C_{2},L_{1}')$ & 
			$(L_{1}',E_{p'})$ &
			$(E_{p'},C_{2})$  \\* \hline\hline 
			\multirow{2}{*}{$(C_{1},L_{qp'})$}
			& $v\geq w$ & \multirow{2}{*}{\diagbox{}{}\slashbox{}{}} & \multirow{2}{*}{\diagbox{}{}\slashbox{}{}} \\*
			& $\none$ & & \\* \hline
			\multirow{2}{*}{$(L_{qp'},E_{p'})$}
			& $\any$ & $v\geq w$ & \multirow{2}{*}{\diagbox{}{}\slashbox{}{}} \\*
			& $\none$ & $\none$ &  \\* \hline
			\multirow{2}{*}{$(E_{p'},C_{1})$}
			& $\any$ & $\any$ & $v\geq w$  \\*
			& $\none$ & $\none$ & $\none$ 
		\end{longtable}
\smallskip
					{\renewcommand{\arraystretch}{1.3}					
						\begin{longtable}{r|c|ccccc}\hlabel{table:rcc}
							curve
							& complement
							& \multicolumn{4}{c}{centers and weights of expansions}
							\endhead
							\caption{Complements of planar rational cuspidal curves in Theorem \ref{CLASS}. \\
						We use the notation from \cite{PaPe_Cstst-fibrations_singularities,PaPe_delPezzo}. In particular, $\gamma,p,s,k$ are some integer parameters and $(F_{j})_{j=0}^{\infty}$ is the Fibonacci sequence, that is $F_{0}=0$, $F_{1}=1$, $F_{j+2}=F_{j+1}+F_{j}$ for $j\geq 0$. }
							\endlastfoot
							%
							 \hline
							\multirow{2}{*}{$\Qb$} 
							
							& \multirow{2}{*}{\ref{def:F2_n1-cusp}}
							& $(C_1,C_2)$ &&&
							\\
							&& $1$ &&&	
							\\ \hline
							\multirow{2}{*}{$\Qa$}
							
							& \multirow{2}{*}{\ref{def:P2n1-cuspidal}}
							& $(C_1,E_{q_{1}})$ &&& 
							\\
							&& $1$ &&&		
							\\ \hline
							\multirow{2}{*}{$\FZa(d,k)$} 
							& \multirow{2}{*}{\ref{def:C**_2}}
							& $(C_{1},L_{pp'})$
							& $(E_{p},L_{1})$ 
							& $(E_{p},L_{2})$ &
							\\
							& 
							& $d-2$
							& $k$
							& $d-k-2$ & 
							\\ \hline
							\multirow{2}{*}{$\FZb(\gamma)$} 
							& \multirow{2}{*}{\ref{def:P2n2_cuspidal}}
							& $(L_{2},L_{rp'})$
							& $(C_{1},L_{rp'})$ &&
							\\
							&&  $\tfrac{\gamma-1}{\gamma-2}$
							& $\gamma-3$ && 
							\\ \hline
							\multirow{2}{*}{$\FE(\gamma)$}
							& \multirow{2}{*}{\ref{def:F2n2-cuspidal}}
							& $(L_{1},C_{2})$
							& $(C_{1},C_{2})$ &&
							\\
							&& $\tfrac{\gamma-2}{2\gamma-5}$
							& $\gamma-3$ &&
							\\ \hline
							
							& \multirow{5}{*}{\ref{def:C**_1}}						
							& $(L_{1},L_{2})$
							& $(E_{p_{1}'},L_{p_{1}'p_{2}'})$
							& $(L_{1}',L_{1})$ 
							& $(L_{2}',L_{2})$						
							\\
							$\cA(\gamma,p,s)$
							&& $\tfrac{\gamma s+s-1}{\gamma+1}$
							& $\tfrac{ps-s+1}{s}$
							& $\gamma$
							& $p$ 
							\\
							$\cB(\gamma,p,s)$
							&& $\tfrac{\gamma s+s-1}{\gamma+1}$
							& $p-1$
							& $\gamma$
							& $\tfrac{s}{ps+p+s-1}$ 
							\\
							$\cC(\gamma,p,s)$
							&& $s$
							& $\tfrac{(\gamma+1)(ps-s+1)+p-1}{\gamma s+s+1}$
							& $\gamma$
							& $\tfrac{1}{p}$ 
							\\ 
							$\cD(\gamma,p,s)$
							&& $s$
							& $p-1$
							& $\gamma$
							& $\tfrac{\gamma s +s+1}{(\gamma+1)(ps-1)+p}$ 
							\\
							&& \multicolumn{4}{c}{(for $\gamma=1$ $(X,D)$ is not snc-minimal, see note \ref{note:rcc} on p.\ \pageref{note:rcc})}
							\\ \hline
							
							& \multirow{3}{*}{\ref{def:C**_2}}
							& $(C_{1},L_{pp'})$
							& $(C_{1},L_{pp'})$
							& $(C_{1},E_{p_{2}})$ 
							& 						
							\\
							$\cE(k)$
							&& $\tfrac{2k+1}{4k+4}$
							& $\tfrac{k+1}{2k+1}$
							& $1$
							& 
							\\ 
							$\cF(k)$
							&& $\tfrac{2k+1}{4k}$
							& $\tfrac{k}{2k+1}$
							& $1$
							& 
							\\  \hline 
							\multirow{2}{*}{$\cG(\gamma)$}
							& \multirow{2}{*}{\ref{def:C**_2}}
							& $(C_{1},L_{pp'})$
							& $(E_{p},L_{2})$
							& $(C_{1},E_{p_{1}})$ 
							& 						
							\\
							& 
							& $\gamma-1$
							& $\gamma-2$
							& $1$
							& 
							\\ \hline
							\multirow{2}{*}{$\cH(\gamma)$}
							& \multirow{2}{*}{\ref{def:P2n3}}
							& $(E_{2},L_{2})$
							& $(E_{3},C)$
							& $(L_{1},L_{2})$ &						
							\\
							&& $\tfrac{\gamma}{2\gamma-1}$ 
							& $1$
							& $\gamma-1$ &	
							\\ \hline
							\multirow{2}{*}{$\cI$}
							
							& \multirow{2}{*}{\ref{def:F2_n2-tangent}}
							& $(C_{2},C_{1})$
							& $(C_{2},E_{p'})$ &&
							\\
							&& $2$
							& $\tfrac{3}{5}$ &&		
							\\ \hline
							\multirow{2}{*}{$\cJ(k)$}
							& \multirow{2}{*}{\ref{def:P2n3}}
							& $(L_{2},L_{3})$
							& $(E_{3},C)$ 						
							& $(L_{1},L_{2})$ &
							\\
							&& $1+\tfrac{1}{k}$
							& $1$						
							& $k$ &
							\\ \hline
							\multirow{2}{*}{$\cORa(k)$}
							& \multirow{2}{*}{\ref{def:C**_1}}
							& $(L_{1},L_{2})$
							& $(E_{p_{1}'},L_{p_{1}'p_{2}'})$
							& $(L_{1}',L_{1})$ 
							& $(L_{2}',L_{2})$						
							\\
							&& $1$
							& $1$
							& $\tfrac{F_{4k+4}}{F_{4k}}-3$
							& $2$ 
							\\ \hline
							\multirow{2}{*}{$\cORb(k)$}
							& \multirow{2}{*}{\ref{def:C**_2}}
							& $(C_{1},L_{pp'})$
							& $(C_{1},E_{p_{2}})$
							& $(E_{p},L_{pp'})$ 
							& 						
							\\
							&& $\tfrac{F_{4k+4}}{F_{4k}}-1$
							& $1$
							& $1$
							&			
						\end{longtable}
					}
	
\end{small}

\bibliographystyle{amsalpha}
\bibliography{bibl2019}
\end{document}